\documentclass[11pt]{article}
\usepackage{amsfonts,amssymb,amsmath,amsthm,epsfig,verbatim,graphicx}
\usepackage{textcomp}
\usepackage[shortlabels]{enumitem}
\usepackage{epsfig}
\usepackage{multicol}

\usepackage{tikz}
\usetikzlibrary{decorations.pathreplacing}
\usetikzlibrary{calc}
\usetikzlibrary{shapes.geometric}

\usepackage{subcaption}
\usepackage[autostyle]{csquotes}
\usepackage{lmodern}
\usepackage[T1]{fontenc}
\usepackage[mathscr]{euscript}
\usepackage{mathabx}
\usepackage{makecell}

\usepackage{algorithm}
\usepackage{algpseudocode}

\usepackage{bbold}
\let\boldmathfont\mathbb
\let\mathbb\relax

\usepackage{bm}
\usepackage{palatino}
\usepackage{mathpazo}
\usepackage{inconsolata}
\usepackage{hyperref}
\newcommand{\oeis}[1]{\href{https://oeis.org/#1}{#1}}
\definecolor{nicegray}{RGB}{80, 80, 130}
\usepackage{xcolor}

\begingroup
    \makeatletter
    \@for\theoremstyle:=definition,remark,plain\do{%
        \expandafter\g@addto@macro\csname th@\theoremstyle\endcsname{%
            \addtolength\thm@preskip\parskip
            }%
        }
\endgroup

\usepackage[margin=1.1in]{geometry}

\theoremstyle{definition}
\newtheorem{theorem}{Theorem}

\newtheorem{lemma}[theorem]{Lemma}

\newtheorem{definition}{Definition}
\newtheorem{example}{Example}
\newtheorem{nonexample}[example]{Non-example}

\numberwithin{theorem}{section}
\numberwithin{definition}{section}
\numberwithin{example}{section}
\numberwithin{remark}{section}

\makeatletter
\newfont{\footsc}{cmcsc10 at 8truept}
\newfont{\footbf}{cmbx10 at 8truept}
\newfont{\footrm}{cmr10 at 10truept}
\newcommand*\patchAmsMathEnvironmentForLineno[1]{%
  \expandafter\let\csname old#1\expandafter\endcsname\csname #1\endcsname
  \expandafter\let\csname oldend#1\expandafter\endcsname\csname end#1\endcsname
  \renewenvironment{#1}%
     {\linenomath\csname old#1\endcsname}%
     {\csname oldend#1\endcsname\endlinenomath}}%
\newcommand*\patchBothAmsMathEnvironmentsForLineno[1]{%
  \patchAmsMathEnvironmentForLineno{#1}%
  \patchAmsMathEnvironmentForLineno{#1*}}%
\AtBeginDocument{%
\patchBothAmsMathEnvironmentsForLineno{equation}%
\patchBothAmsMathEnvironmentsForLineno{align}%
\patchBothAmsMathEnvironmentsForLineno{flalign}%
\patchBothAmsMathEnvironmentsForLineno{alignat}%
\patchBothAmsMathEnvironmentsForLineno{gather}%
\patchBothAmsMathEnvironmentsForLineno{multline}%
}
\usepackage{mathtools}
\usepackage{parskip}
\usepackage{lineno}
\usepackage[small]{titlesec}

\usepackage[backend=biber, style=ext-numeric-comp, maxbibnames=99, articlein=false]{biblatex}

\DeclareFieldFormat[article]{citetitle}{#1}
\DeclareFieldFormat[article]{title}{#1}

\AtEveryBibitem{\clearfield{note}}
\AtEveryBibitem{\clearfield{month}}
\AtEveryBibitem{\clearfield{issn}}
\AtEveryBibitem{\clearfield{isbn}}
\AtEveryBibitem{\clearlist{language}}

\AtEveryBibitem{\ifentrytype{article}{\clearfield{doi}}{}}
\AtEveryBibitem{\ifentrytype{article}{\clearfield{url}}{}}
\AtEveryBibitem{\ifentrytype{book}{\clearfield{url}}{}}
\AtEveryBibitem{\ifentrytype{incollection}{\clearfield{url}}{}}
\AtEveryBibitem{\ifentrytype{incollection}{\clearfield{doi}}{}}
\AtEveryBibitem{\ifentrytype{article}{\clearfield{publisher}}{}}
\AtEveryBibitem{\ifentrytype{online}{}{\clearfield{eprint}}}

\DeclareFieldFormat{titlecase:title}{\MakeSentenceCase*{#1}}

\usepackage{color}
\definecolor{todocolor}{RGB}{205,235,139}
\definecolor{todo-idea}{RGB}{120,180,255}
\definecolor{todo-error}{RGB}{208,31,60}
\definecolor{todo-question}{RGB}{255,255,136}

\usepackage[colorinlistoftodos, color=todocolor, textsize=small]{todonotes}

\newcommand{\Av}{\operatorname{Av}}

\newcommand{\Co}{\operatorname{Co}}
\newcommand{\Grid}{\operatorname{Grid}}
\newcommand{\rank}{\operatorname{rank}}

\newcommand{\INS}{\operatorname{ins}}

\renewcommand{\AA}{\mathcal{A}}
\newcommand{\BB}{\mathcal{B}}
\newcommand{\C}{\mathbb{C}}
\newcommand{\CC}{\mathcal{C}}
\newcommand{\DD}{\mathcal{D}}
\newcommand{\EE}{\mathcal{E}}
\newcommand{\FF}{\mathcal{F}}
\newcommand{\GG}{\mathcal{G}}

\newcommand{\HH}{\mathcal{H}}
\newcommand{\II}{\mathcal{I}}
\newcommand{\JJ}{\mathcal{J}}
\newcommand{\KK}{\mathcal{K}}
\newcommand{\MM}{\mathcal{M}}
\newcommand{\NN}{\mathcal{N}}
\newcommand{\N}{\mathbb{N}}
\newcommand{\OO}{\mathcal{O}}

\newcommand{\RR}{\mathcal{R}}
\newcommand{\SSS}{\mathscr{S}}
\newcommand{\TT}{\mathcal{T}}

\newcommand{\Z}{\mathbb{Z}}
\newcommand{\ZZZ}{\mathscr{Z}}
\newcommand{\x}{\times}
\newcommand{\twobyfour}{2$\times$4}

\newcommand{\DNA}{\texttt{DNA}}
\newcommand{\eqrel}{\mathscr{R}}
\newcommand{\inso}[2]{\INS_\OO({#1}, {#2})}
\newcommand{\insr}[2]{\INS_\RR({#1}, {#2})}
\newcommand{\wt}{\widetilde}
\newcommand{\ReqIns}{\texttt{ReqIns}}
\newcommand{\ObsDel}{\texttt{ObsDel}}
\newcommand{\ReqDel}{\texttt{ReqDel}}
\newcommand{\ReqListDel}{\texttt{ReqListDel}}
\newcommand{\PointPl}{\texttt{PointPl}}
\newcommand{\RowSep}{\texttt{RowSep}}
\newcommand{\ColSep}{\texttt{ColSep}}
\newcommand{\Factor}{\texttt{Factor}}
\newcommand{\ObsInf}{\texttt{ObsInf}}
\newcommand{\ds}{\displaystyle}
\newcommand{\bb}{\square}

\newcommand{\CSS}{\texttt{CombSpecSearcher}}
\newcommand{\Queue}{\texttt{CSSQueue}}
\newcommand{\ClassDB}{\texttt{ClassDB}}
\newcommand{\RuleDB}{\texttt{RuleDB}}
\newcommand{\EquivDB}{\texttt{EquivDB}}
\newcommand{\SpecSearcher}{\texttt{SpecFinder}}
\usepackage{tilings}
\usetikzlibrary{patterns}
\usetikzlibrary{calc}
\tikzset{invisible/.style={minimum width=0mm,inner sep=0mm,outer sep=0mm}}
\tikzset{tiling/.style={invisible, anchor=north west}}
\tikzset{ptnode/.style={rounded corners=15, black!40, thick, dashed}}
\newcommand\ptedge[4]{
  \draw ($#1 + (0.5, -1.3) + #2$)..%
  controls ($#1 + #2 + (0.5, -1.3) + (0, -1)$)%
  and ($#3 + #4 + (0, 1) + (0.5, -0.3)$)..%
  ($#3 + #4 + (0.5, -0.3)$);
}

\renewenvironment{abstract}{
	\begin{list}{}%
	{\setlength{\rightmargin}{1in}%
	\setlength{\leftmargin}{1in}}%
	\item[]\ignorespaces\begin{small}}%
	{\end{small}\unskip\end{list}%
}

\newpagestyle{main}[\small]{
	\headrule
	\sethead[\usepage][][]
	{\sc Combinatorial Exploration}{}{\usepage}
}

\title{\sc Combinatorial Exploration:\\ An Algorithmic Framework For Enumeration}

\usepackage{makecell}
\author{
	\begin{tabular}{m{2.5in}m{2.5in}}
		\makecell{
			Michael H. Albert\\
			\small Department of Computer Science\\
			\small University of Otago\\
			\small Dunedin, New Zealand\\
			\small \texttt{malbert@cs.otago.ac.nz}
		}
		&
		\makecell{
			Christian Bean\\
			\small Department of Computer Science\\
			\small Reykjavik University\\
			\small Reykjavik, Iceland\\
			\small \texttt{christianbean@ru.is}
		}\\&\\
		\makecell{
			Anders Claesson\\
			\small Division of Mathematics\\
			\small The Science Institute\\
			\small University of Iceland\\
			\small Reykjavik, Iceland\\
			\small \texttt{akc@hi.is}
		}
		&
		\makecell{
			\'Emile Nadeau\\
			\small Department of Computer Science\\
			\small Reykjavik University\\
			\small Reykjavik, Iceland\\
			\small \texttt{emile19@ru.is}
		}\\&\\
		\makecell{
			Jay Pantone\\
			\small Department of Mathematical\\
			\small and Statistical Sciences\\
			\small Marquette University\\
			\small Milwaukee, WI, USA\\
			\small \texttt{jay.pantone@marquette.edu}
		}
		&
		\makecell{
			Henning Ulfarsson\\
			\small Department of Computer Science\\
			\small Reykjavik University\\
			\small Reykjavik, Iceland\\
			\small \texttt{henningu@ru.is}
		}
	\end{tabular}
}

\titleformat{\section}{\large\sc}{\thesection.}{1em}{}
\date{}

\begin{document}
\maketitle

\begin{abstract}
	%!TEX root = combinatorial-exploration.tex

Combinatorial Exploration is a new domain-agnostic algorithmic framework
to automatically and rigorously study the structure of combinatorial
objects and derive their counting sequences and generating functions.
We describe how it works and provide an open-source Python
implementation. As a prerequisite, we build up a new theoretical
foundation for combinatorial decomposition strategies and combinatorial
specifications.

We then apply Combinatorial Exploration to the domain of permutation
patterns, to great effect. We rederive hundreds of results in the literature
in a uniform manner and prove many new ones. These results can be found
in a new public database, the Permutation Pattern Avoidance Library
(PermPAL) at \url{https://permpal.com}. Finally, we give three
additional proofs-of-concept, showing examples of how Combinatorial
Exploration can prove results in the domains of alternating sign
matrices, polyominoes, and set partitions.

\end{abstract}

%% ==== %% ==== %% ==== %% ==== %% ==== %% ==== %% ==== %% ==== %% ==== %% ==== %%
%% ==== %% ==== %% ==== %% ====    SECTION ONE     ==== %% ==== %% ==== %% ==== %%
%% ==== %% ==== %% ==== %% ==== %% ==== %% ==== %% ==== %% ==== %% ==== %% ==== %%

\section{Introduction}
\label{section:introduction}
%!TEX root = combinatorial-exploration.tex

Combinatorial structures are ubiquitous throughout mathematics. Graphs,
permutations, words, polyominoes and other families of combinatorial objects
often play a central role in many different fields. Enumerative
combinatorics is concerned with the elucidation of structural properties of
these families such as counting, classification, and limiting behavior.

A \emph{combinatorial set} is a set of combinatorial objects, each of which is
assigned a nonnegative integer size, with the property that the number of objects
of any given size is finite. The analysis of a given combinatorial set often relies
on domain-specific methods and ad-hoc derivations. In this work, we develop an
algorithmic framework capable of automatically deriving rigorous proofs
about the structure of combinatorial sets and computing their
counting sequences and generating functions.	

At its heart, \emph{Combinatorial Exploration} is the systematic application of
structural strategies that break down a set of combinatorial objects
repeatedly into simpler parts, until a full decomposition into completely
understood parts is obtained. The quickly growing fields of \emph{symbolic
combinatorics} and \emph{analytic combinatorics} revolve around techniques to
derive enumerative information when the structure of a combinatorial set is
already understood.
The literature already contains several effective methods to describe the structure of a combinatorial set and the most fundamental is the theory of combinatorial species, developed by Joyal~\cite{joyal:species-1981, joyal:species-1986} in the 1980s and put into book form by Bergeron, Labelle, and Leroux~\cite{bergeron:species} in the 1990s. More recently, Pivoteau, Salvy, and Soria~\cite{pivoteau:algos-well-founded} have studied recursively defined species from a computational perspective. The constructible classes of Flajolet and Sedgewick~\cite{flajolet:ac} can be seen as a more narrowly focused form of combinatorial species that avoids the language of category theory. Another device for describing the construction of a combinatorial set is the generating tree of Chung, Graham, Hoggatt, and Kleiman~\cite{chung:gen-trees}; it was further formalized in the ECO method of Barcucci, Lungo, Pergola, and Pinzani~\cite{ECO}. The starting point of all of these approaches is foreknowledge of the structure of a combinatorial set.
Combinatorial Exploration addresses the absence of a uniform method for
completing that first step: understanding the structure.

This article is accompanied by an implementation of Combinatorial Exploration
that can be used to apply the algorithms presented here to specific enumerative
problems. That implementation consists of around 7,300 lines of Python code
and is available on Github~\cite{comb-spec-searcher}.
Under our ``plug-and-play'' framework, designed specifically to enable researchers
to use Combinatorial Exploration on new domains, one needs only to provide an
implementation of their specific combinatorial sets together with whichever
domain-specific structural strategies they wish to use. In this way,
we believe that Combinatorial Exploration constitutes the first
component of a new toolkit for enumerative combinatorics. The Github
repository~\cite{comb-spec-searcher} provides a short tutorial to
implement Combinatorial Exploration for the easy domain of binary words
using only about 200 lines of code.

To prove the efficacy of this new framework, we have also applied
Combinatorial Exploration to the domain of permutation patterns. The result
is that we are able to automatically and rigorously prove the results of
dozens of existing papers as well as many new ones. That implementation is
also available on Github~\cite{tilings}, and consists
of around 24,000 lines of Python code. We have additionally created an
online database presenting the results of this work in the field of
permutation patterns. It is called the Permutation Pattern Avoidance
Library (PermPAL)~\cite{permpal-biblatex}, and can be found at
\url{https://permpal.com}. It is inspired by Tenner's Database of
Permutation Pattern Avoidance~\cite{tenner:pp-database-biblatex}, the
\url{graphclasses.org} website~\cite{graphclasses-biblatex}, and of
course the Online Encyclopedia of Integer Sequences~\cite{oeis}.

Section~\ref{section:pp-intro} gives a brief introduction to the field of
permutation patterns, which will be used as an example domain throughout the
paper. In particular, Subsection~\ref{subsection:pp-success} catalogs many of
the successful applications of Combinatorial Exploration in this domain.

Section~\ref{section:combinatorial-exploration} lays down a theoretical
framework for concepts that have been widely used in enumerative
combinatorics for some time. When Combinatorial Exploration is successful,
the result is a \emph{combinatorial specification} that fully describes
the combinatorial set and whose validity is rigorously verified.
By introducing a new type of decomposition function called a
\emph{combinatorial strategy}, we are able to create a more formalized
version of combinatorial specifications. This section then introduces
Combinatorial Exploration and its associated suite of combinatorial algorithms. 

Section~\ref{section:productivity} provides results that guarantee that a
combinatorial specification produced by Combinatorial Exploration
is \emph{productive}, that is, it contains sufficient information to
uniquely determine the counting sequences of the involved sets. We believe
that the formalizations created and explored in
Sections~\ref{section:combinatorial-exploration} and~\ref{section:productivity}
represent, independently of Combinatorial Exploration, important advances
in the understanding of combinatorial
specifications and their use to solve enumerative problems.

Section~\ref{section:transfer-tools} describes the enumerative and analytic
tools that derive various enumerative products from a combinatorial specification,
including polynomial-time counting algorithms, generating functions, and
random sampling procedures.

Sections~\ref{section:pp-results}, \ref{section:asm-results},
\ref{section:polyomino-results},
and~\ref{section:set-partitions-results} demonstrate the impressive efficacy
and wide applicability of
Combinatorial Exploration by applying it to four domains: permutation patterns,
alternating sign matrices, polyominoes, and set partitions, respectively. The
domain of permutation patterns is afforded the most attention and presented
in the most detail, and this is the domain that we have actually
implemented~\cite{tilings}. For each of the remaining three domains, we only
briefly describe one possible way to approach them with Combinatorial
Exploration, and derive by hand one or two combinatorial specifications for
interesting combinatorial sets. For alternating sign matrices, we find
a specification for the set of $132$-avoiding alternating sign matrices, which
has an algebraic generating function. For polyominoes,
we present a specification that enumerates Ferrers diagrams and briefly discuss
one for moon (L-convex) polyominoes, both of which have non-D-finite generating
functions. For set partitions, we derive specifications
for both the $1212$-avoiding and the $111$-avoiding set partitions; the former has
an algebraic generating function, while the latter has a D-finite generating function.
We expect that Combinatorial Exploration would be equally successful in many
other domains, such as inversion sequences, ascent sequences, and more.

Finally, Section~\ref{section:algorithmic} discusses some algorithmic
aspects of our implementation of Combinatorial Explorations, and
Section~\ref{section:final-notes} provides a few concluding remarks and
suggests avenues for continued development.

%% ==== %% ==== %% ==== %% ==== %% ==== %% ==== %% ==== %% ==== %% ==== %% ==== %%
%% ==== %% ==== %% ==== %% ====    SECTION TWO     ==== %% ==== %% ==== %% ==== %%
%% ==== %% ==== %% ==== %% ==== %% ==== %% ==== %% ==== %% ==== %% ==== %% ==== %%

\section{Permutation Patterns}
\label{section:pp-intro}
%!TEX root = combinatorial-exploration.tex

Although Combinatorial Exploration is inherently a domain-agnostic tool, it will
prove useful in the following sections to have an interesting domain at hand. To
that end, this section aims to be a crash course on the field of
\emph{permutation patterns}, an active area of enumerative combinatorics for
several decades that has meaningful connections to computer science, statistical
mechanics, and algebraic geometry. In this section we give a brief introduction
to the field and provide a summary of the many cases in which Combinatorial
Exploration rigorously proves both new and old results.

\subsection{Introduction to Permutation Patterns}
A \emph{permutation} of length $n$ is an ordering of the numbers
$1, 2, \ldots, n$. The \emph{standardization} of an ordering of $n$
distinct positive integers is the permutation of the same length whose
entries have the
same relative order. For example, the standardization of $6283$ is $3142$,
since $2$ is the smallest entry in $6238$ and so it becomes $1$, $3$ is the
second smallest entry and so it becomes $2$, and so on.

We say that a larger permutation $\pi = \pi(1)\pi(2)\cdots\pi(n)$ \emph{contains}
a smaller permutation \(\sigma = \sigma(1)\sigma(2)\cdots\sigma(k)\)
\emph{as a pattern}, and write $\sigma \leq \pi$, if $\pi$ has a (not
necessarily consecutive) subsequence $\pi(i_1)\pi(i_2)\cdots\pi(i_k)$ whose
standardization is equal to $\sigma$. In this context $\sigma$ is called a
\emph{pattern} and the subsequence $\pi(i_1)\pi(i_2)\cdots\pi(i_k)$ is an
\emph{occurrence} of the pattern in $\pi$. For instance, $312 \leq 41257368$
as exhibited by the subsequence $413$ (among others), but one can check that
$321 \not\leq 41257368$. In this case we say that $41257368$ \emph{avoids}
$321$.

A geometric viewpoint turns out to be useful. We can associate each
permutation $\pi$ of length $n$ with the plot of the points $\{(i, \pi(i)): 1
	\leq i \leq n\}$ in the Cartesian plane. Conversely, any finite set of points in
the plane such that no two points lie on the same horizontal or vertical line
can be continuously deformed to a plot of this kind and associated with the
resulting permutation.  The permutation containment relation has a tidy
geometric definition: $\pi$ contains $\sigma$ if some finite subset of the
points in the plot for $\pi$ can be deleted to leave (up to continuous
deformation) the plot of $\sigma$. Figure~\ref{figure:perm-example} shows the
plot of $41257368$, and demonstrates graphically that $41257368$ contains the
permutation $312$.

\begin{figure}
	\begin{center}
		\begin{tikzpicture}[scale=0.4, x=1cm,y=1cm]
			\draw [dotted, gray] (0,0) grid (9,9);
			\draw[ultra thick, gray] (0,0) rectangle (9,9);
			\foreach \Point in {(1,4),
					(2,1), (3,2), %
					(4,5), (5,7), (6,3), %
					(7,6), (8,8)}{
					\node at \Point {\point{2pt}};
				}
			\draw[line width=0.25mm, black] (1,4) circle (10pt);
			\draw[line width=0.25mm, black] (2,1) circle (10pt);
			\draw[line width=0.25mm, black] (6,3) circle (10pt);
		\end{tikzpicture}
	\end{center}
	\caption{The permutation $41257368$ with a circled occurrence of the pattern
		$312$.}
	\label{figure:perm-example}
\end{figure}
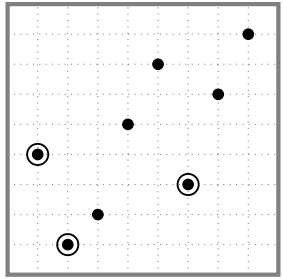

The notion of one permutation containing another is the engine that drives the
entire field of permutation patterns. A \emph{permutation class} (often
abbreviated as simply a \emph{class}) is a downward-closed set of permutations
under the containment order; that is, a set $\CC$ of permutations is a class if
it has the property that if $\pi \in \CC$ and $\sigma \leq \pi$, then $\sigma
	\in \CC$. Every class can be uniquely specified by the minimal set of
permutations that it avoids, called its \emph{basis}, and the class with basis
$B$ is denoted $\Av(B)$. Put plainly, $\Av(B)$ is the downward-closed set of
permutations that avoid all of the permutations in $B$.

In an exercise in Chapter 2 of his seminal work \emph{The Art of Computer
	Science}~\cite[Exercise 2.2.1.5]{knuth:taocp-1}, Knuth asks the reader to
prove that the set of permutations that can be sorted by a stack is
precisely $\Av(231)$\footnote{We omit the set braces in this notation,
	writing e.g., $\Av(123, 3412)$ instead of $\Av(\{123,3412\})$.} and to find
the number of permutations of each length in this set. This innocently
simple exercise set off a rapid and deep exploration of
permutation classes. We recommend Vatter's survey~\cite{vatter:perm-survey}
for a comprehensive introduction to the field.

\subsection{Sums and Skew Sums of Permutations}
\label{subsec:sums-and-skew-sums}

We will rely on two particular operations on permutations to give illustrative
examples in the following sections, and we introduce those now. Given
permutations $\sigma$ and $\tau$ of lengths $k$ and $\ell$ respectively, the
\emph{(direct) sum} $\sigma \oplus \tau$ is the permutation of length $k + \ell$
defined by
\[
	(\sigma \oplus \tau)(i) = \left\{\begin{array}{ll}
		\sigma(i),   & 1 \leq i \leq k        \\
		\tau(i-k)+k, & k+1 \leq i \leq k+\ell
	\end{array}\right..
\]
Geometrically, $\sigma \oplus \tau$ is formed by placing the entries of $\tau$
above and to the right of the entries of $\sigma$, as demonstrated in
Figure~\ref{figure:sum-and-skew-sum}.

\begin{figure}
	\begin{center}
		\begin{tikzpicture}[scale=0.4, x=1cm,y=1cm]
			\draw [dotted, gray] (0.5,0.5) grid (3.5,3.5);
			\draw[ultra thick, gray] (0.5,0.5)  rectangle (3.5,3.5);
			\foreach \Point in {(1,3),(2,1), (3,2)}{
					\node at \Point {\point{2pt}};
				}
			\draw [dotted, gray] (3.5,3.5) grid (5.5,5.5);
			\draw[ultra thick, gray] (3.5,3.5) rectangle (5.5,5.5);
			\foreach \Point in {(4,5),(5,4)}{
					\node at \Point {\point{2pt}};
				}
		\end{tikzpicture}
		\hspace{1in}
		\begin{tikzpicture}[scale=0.4, x=1cm,y=1cm]
			\draw [dotted, gray] (0.5,2.5) grid (3.5,5.5);
			\draw[ultra thick, gray] (0.5,2.5)  rectangle (3.5,5.5);
			\foreach \Point in {(1,5),(2,3), (3,4)}{
					\node at \Point {\point{2pt}};
				}
			\draw [dotted, gray] (3.5,0.5) grid (5.5,2.5);
			\draw[ultra thick, gray] (3.5,0.5) rectangle (5.5,2.5);
			\foreach \Point in {(4,2),(5,1)}{
					\node at \Point {\point{2pt}};
				}
		\end{tikzpicture}
	\end{center}
	\caption{On the left, the permutation $312 \oplus 21 = 31254$, and on the right, the permutation $312 \ominus 21 = 53421$.}
	\label{figure:sum-and-skew-sum}
\end{figure}
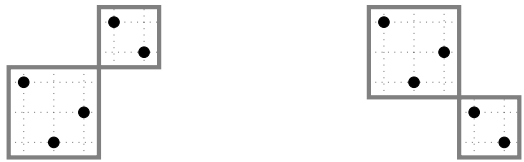

Symmetrically, the skew sum of $\sigma$ and $\tau$, denoted $\sigma \ominus
	\tau$ is formed geometrically but placing the entries of $\tau$ below and to
the right of the entries of $\sigma$. Formally,
\[
	(\sigma \ominus \tau)(i) = \left\{\begin{array}{ll}
		\sigma(i)+\ell, & 1 \leq i \leq k        \\
		\tau(i-k),   & k+1 \leq i \leq k+\ell
	\end{array}\right..
\]

These two operations can be extended to sets in a natural way. For any two sets
of permutations $\AA$ and $\BB$, we define
\[
	\AA \oplus \BB = \{\alpha \oplus \beta \;:\; \alpha \in \AA,\, \beta \in \BB\}
\]
and
\[
	\AA \ominus \BB = \{\alpha \ominus \beta \;:\; \alpha \in \AA,\, \beta \in \BB\}.
\]

\subsection{Enumeration of Permutation Classes}
\label{subsec:enum-perm-classes}

Letting $\CC_n$ denote the set of permutations of length $n$ in $\CC$, we define
the \emph{counting sequence} of a class to be the sequence $a_n = |\CC_n|$, with
the convention that $\CC_0$ contains the empty permutation of length $0$ and
thus has cardinality $1$. The \emph{generating function} of $\CC$ is the formal
power series for its counting sequence:
\[
	f_\CC(x) = \sum_{n \geq 0} |\CC_n|x^n.
\]

A \emph{principal class} is one whose basis contains only a single permutation.
The simplest non-trivial examples are $\Av(12)$ and $ \Av(21)$. The former
contains all decreasing permutations (e.g., $87654321$), while the latter
contains all increasing permutations. The counting sequences of these two
classes are clearly the same, with $a_n = 1$ for all $n$, and their generating
function $f_{\Av(12)} = f_{\Av(21)} = 1/(1-x)$ is rational.

\paragraph{Principal classes of length 3.}
The six principal classes avoiding a pattern of length $3$ are $\Av(123)$,
$\Av(132)$, $\Av(213)$, $\Av(231)$, $\Av(312)$, and $\Av(321)$. The outer two
are essentially the same\footnote{To be more precise, there are obvious
	bijections between them, derived from the symmetries of the square, that
	preserve both length and permutation containment.} when accounting for
symmetries of the permutation containment order, and so are the middle four.
Although these two groupings, called \emph{symmetry classes}, are structurally
very different\footnote{To mention just one difference, $\Av(123)$ contains
	infinite antichains under the containment order, while $\Av(132)$ does not.}, it
turns out that they are both counted by the famous Catalan numbers $a_n = {2n
	\choose n}/(n+1)$ and have the algebraic generating function $(1 -
	\sqrt{1-4x})/(2x)$. We will be making frequent references to the Online
Encyclopedia of Integer Sequences~\cite{oeis} (OEIS), which hosts integer sequences and
vast information about them, e.g., the Catalan numbers appear as sequence
\oeis{A000108} on the site.

\paragraph{Principal classes of length 4.}
We now arrive at the twenty four principal classes avoiding a pattern of length
$4$, and already we encounter the active edge of investigation. These twenty
four principal classes partition into seven symmetry classes. Two turn out to
have the same algebraic generating function (see B{\'o}na~\cite{bona:1342} and
Stankova~\cite{stankova:forbidden-subseqs}) and four turn out to have the same
D-finite generating function\footnote{A generating function is \emph{D-finite}
	if it is the solution of a nontrivial linear differential equation with
	polynomial coefficients.} (see Bousquet-M\'elou~\cite{bousquet-melou:1234-1243-1432-2143}, Gessel~\cite{gessel:symmetric-functions},
Stankova~\cite{Stankova:bijection-1234-4123}, and
West~\cite{West:bijection-1234-1243-2143}). The generating function for the last
symmetry class, containing $\Av(1324)$ and $\Av(4231)$, remains unknown. Through
about a dozen articles~\cite{bevan:1324-9.81, johansson:functional-1324,
	bona:new-upper-bound-1324, claesson:upper-bounds-1324, marinov:counting-1324,
	albert:sw-4231, bona:new-record-1324, conway:1324, conway:4231-50-terms,
	bevan:1324-staircase-conference, bevan:1324-staircase}, the first fifty terms
of the counting sequence of $\Av(1324)$ are known explicitly and rough bounds on
the exponential growth are known.

\paragraph{The $\bm{2\times4}$ classes.}
For the last $25$ years, a particular collection of permutation classes has been
perhaps the most important and influential in guiding the direction of research.
A \emph{\twobyfour{} class} is a permutation class whose basis consists of two
permutations of length four, e.g., $\Av(2413, 3142)$. Accounting for symmetries
of the pattern containment relation, there are essentially $56$ different
\twobyfour{} classes. The quest to understand these $56$ permutation classes has
led to the creation of a surprising number of sophisticated enumerative
techniques that we mention below.

\paragraph{Automatic methods in permutation patterns.}

Automatic and computational methods, both rigorous and empirical, have played a
prominent role in the study of permutation patterns for quite some time.
Figure~\ref{figure:comp-methods} shows many of these methods. An arrow from
method A to method B implies that method B is stronger, i.e.,
every permutation class that can be enumerated
automatically by method A can also be enumerated by method B. A dashed arrow
implies that we conjecture that method B is stronger than method A.
As the figure shows, several of these methods, e.g., Zeilberger's enumeration
schemes, Vatter's
implementation of the regular insertion encoding, and the substitution
decomposition applied to permutation classes with finitely many simples, are
subsumed by Combinatorial Exploration. 

\begin{figure}
	\begin{center}
		\begin{tikzpicture}[scale = 0.75, place/.style = {rectangle,draw = black!50,fill = gray!5,thick, minimum size=0.8cm, font=\scriptsize}, align=center, auto]
			\node [place] (poly) at (2,-1) {Polynomial\\classes~\cite{homberger:effective-automatic}};
			\node [place] (gentree) at (8,-1) {Finitely labeled\\generating tree~\cite{vatter:finitely-labeled}};
			\node [place] (temp) at (-1,0.7) {Templates~\cite{ZeilbergerTemplates}};
			\node [place] (struct) at (-1.5,2.5) {Struct-cover\\verified~\cite{Bean:method-struct}};
			\node [place] (rie) at (6,2.5) {Regular insertion\\encoding~\cite{vatter:regular-insertion-encoding}};
			\node [place] (schemeZ) at (10,0.7) {Zeilberger's finite\\enum. scheme~\cite{zeilberger:enumeration-schemes}};
			\node [place] (sea) at (15.5,0.7) {Scanning elements\\ algorithm~\cite{firro:scanning-elements}};
			\node [place] (simple)  at (2,2.5) {Finitely many simple\\ permutations~\cite{Bassino:method-fin-simp}};
			\node [place] (schemeV) at (10,2.5) {Vatter's finite\\enum. scheme~\cite{vatter:enum-scheme}};
			\node [place] (schemeF) at (10,4.5) {Flexible finite\\enum. scheme~\cite{beirs-ariel:flexible-schemes}};
			\node [place] (tile) at (4,4.5) {Combinatorial\\Exploration};

			\draw [->,semithick] (poly) to (temp);
			\draw [->,semithick] (temp) to (struct);
			\draw [->,semithick] (poly) to (rie);
			\draw [->,semithick] (poly) to (simple);
			\draw [->,semithick] (simple) to (tile);
			\draw [->,semithick] (rie) to (tile);
			\draw [->,semithick] (rie) to (schemeF);
			\draw [->,semithick] (gentree) to (rie);
			\draw [->,semithick] (gentree) to (schemeZ);
			\draw [->,semithick] (schemeZ) to (schemeV);
			\draw [->,semithick] (schemeV) to (schemeF);
			\draw [->,semithick, dashed] ($(schemeF.west)+(0,0)$) to (tile.east);
			\draw [<->,semithick, dashed] (schemeZ) to (sea);

		\end{tikzpicture}
		\caption{Comparison of algorithmic enumeration methods.}
		\label{figure:comp-methods}
	\end{center}
\end{figure}
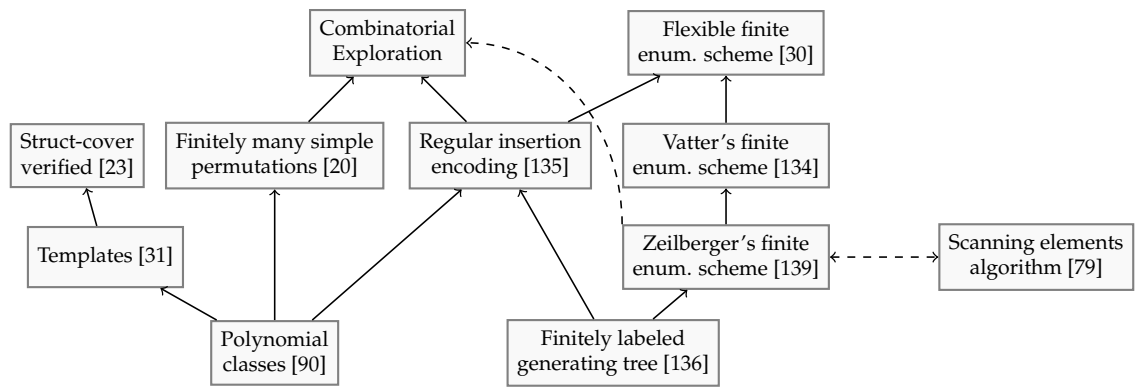

\subsection{Success of Combinatorial Exploration in Permutation Patterns}
\label{subsection:pp-success}

Before introducing Combinatorial Exploration over the next several sections, we
hope to first convince the reader that it is a powerful enumerative
tool by describing some of the new and old results that it can automatically and
rigorously prove. We strongly believe that no algorithmic effort can fully
substitute for the endless creativity of human mathematicians. There are,
however, many examples where the enumeration of a specific combinatorial set is
needed as a stepping stone on the way to a more important theorem, and for tasks
like this Combinatorial Exploration can be tremendously helpful.

As Section~\ref{section:combinatorial-exploration} describes, when the process
of Combinatorial Exploration is successful, the result is a rigorous description
of the permutation class called a \emph{combinatorial specification}, and as
Section~\ref{section:transfer-tools} discusses, this combinatorial specification
can be converted to a system of equations satisfied by the generating function
for the class.

In many cases, this is a univariate algebraic system---each equation expresses
one of the generating functions as the sum or product of others---and can be
solved explicitly or implicitly with an appropriate computer algebra
system\footnote{We find Maple particularly effective.}. In other cases, we
get a bivariate algebraic system with one catalytic variable (to use the
terminology of Zeilberger~\cite{zeilberger:umbral1}). These can be solved by
adapting the methods of Bousquet-M\'elou and
Jehanne~\cite{bousquet-melou:poly-eqs} from single equations to systems of
equations, although the symbolic computations required (in particular,
elimination via Gr\"obner bases or resultants) can be intense. These
systems generically have solutions that are algebraic.

Lastly, in some cases we obtain systems of equations with two or more catalytic
variables, and although we know of no way to solve these systems, they
nonetheless permit calculation of the terms of the counting sequences involved
in polynomial time. It is known that solutions of these systems can be
non-D-finite. In the remainder of this section, we use ``CV'' as an abbreviation
for ``catalytic variable''.

Let us finally say, most importantly, that the successes of Combinatorial
Exploration that are reproved results from the literature are not obtained by
simply taking the ideas from that literature and implementing them as code.
Instead, as described in Section~\ref{section:pp-results}, we have come up with
just a small handful of what we call \emph{combinatorial strategies} that are
sufficient to yield all of the results below in a surprisingly uniform manner,
automatically and rigorously.

\begin{itemize}[label=$\diamond$]
	\item We can find specifications automatically for six out of the seven
		  symmetry classes of permutations avoiding one pattern of length
		  $4$, all but $\Av(1324)$. Four of these six, $\Av(1234)$, 
		  $\Av(1243)$, $\Av(1432)$, and $\Av(2143)$, have the same counting
		  sequence, and we find for them specifications with two CVs,
		  which is consistent with their shared generating function
		  being known to be non-algebraic. Additionally, the first three of these
		  have specifications that are \emph{parallel}, which means they have a
	      similar enough structure that one can automatically produce
	      bijections between any pair of them. The remaining two of the six,
	      $\Av(1342)$ and $\Av(2413)$, also share a counting sequence, and
	      we find specifications with one CV that can
	      easily be solved to produce their algebraic generating function.
	      These are the first \emph{direct} enumerations of these two permutation
	      classes, as previous enumerations were via bijections to each other
	      and to other objects.
	      The final class, $\Av(1324)$, currently remains out of reach, but we are
	      optimistic that several not-yet-implemented strategies may lead to progress.

	\item Out of the $56$ symmetry classes of permutations avoiding two patterns
	      of length $4$---the classes that have been an important testbed
	      for new enumerative techniques over the years
	      \cite{kremer:finite, kremer:forbiddenps, vatter:regular-insertion-encoding, albert:enumeration-three-classes-grid, albert:enumeration-2143-4231, albert:enumeration-4321-1324, atkinson:skew-merged-enum, atkinson:3plus1-avoiding, miner:several-2x4s, le:wilf-classes-pairs, callan:enumeration-1243-2134, pantone:enumeration-3124-4312, albert:inflations-case-studies, bona:perm-class-smooth, bevan:two-2x4s, bevan:one-2x4, albert:c-machines, callan:4321-3241, bloom:two-vignettes, miner:completing-2x4s, kremer:forbidden}
	      ---we can find specifications for all $56$ of them. $53$ have
	      specifications with $0$ CVs or $1$ CV, and so their algebraic generating
	      functions can be determined. The remaining three were conjectured by
	      Albert, Homberger, Pantone, Shar, and Vatter~\cite{albert:c-machines} to
	      be non-D-finite, and correspondingly their specifications have $2$ or
	      more CVs. Figure~\ref{figure:heatmap} on page~\pageref{figure:heatmap}
	      shows heatmaps for the \twobyfour{} classes, derived by sampling
	      many long permutations uniformly at random from each class.

	\item Out of the $317$ symmetry classes of permutations avoiding three
	      patterns of length $4$, again we can find specifications for all of them.
	      One is conjectured to be non-D-finite, and for this class we find a
	      specification with $2$ CVs. For the remaining $316$ we find specifications
	      with $0$ CVs or $1$ CV. The original enumeration of these classes required
	      an enormous amount of work, largely by hand, and can be found
	      in~\cite{mansour:1243-2134-plus-one, mansour:1324-2143-plus-one,%
		  mansour:1324-and-two-other, mansour:1342-and-two-other,%
		  mansour:Wilf-triples-1, mansour:Wilf-triples-2, mansour:triple-all-done,%
		  albert:c-machines}.

	\item Similarly, we can find specifications for all symmetry classes
	      avoiding $n$ patterns of length $4$ for $4 \leq n \leq 24$, all with $0$ CVs
	      or $1$ CV, implying that we can compute all of their generating functions.
	      This work was originally done by Mansour with various
	      collaborators~\cite{mansour:4x4, mansour:4x4-prelim-1, mansour:4x4-prelim-2,%
		  toufik:5x4, mansour:6x4-and-7x4, mansour:eight-and-nine,
		  mansour:subsets-of-four}, requiring hundreds of pages of work, and even then
	      with a vast majority of the derivations left as exercises for the reader.

	\item As Figure~\ref{figure:comp-methods} explained, Combinatorial
	      Exploration can automatically enumerate any insertion encodable permutation
	      class. More significantly, it can do so much more quickly than Vatter's
	      original implementation, which requires a great deal of brute force
	      generation of permutations to determine which configurations are possible
	      and which are not. The particular data structure we use in Combinatorial
	      Exploration for permutation patterns (described in
	      Section~\ref{section:pp-results}) can automatically determine this
	      information without such generation. As a result, we can enumerate
	      insertion encodable classes that were previously out of reach due to
	      computational intensity.

	\item Bevan, Brignall, Elvey Price, and Pantone~\cite{bevan:1324-staircase}
	      found improved lower and upper bounds on the exponential growth rate of
	      $\Av(1324)$ by considering a set of gridded permutations that they called
	      ``domino permutations''. The enumeration of these was challenging,
	      requiring a bijection to a type of arch systems and several pages of
	      work to enumerate these arch systems. We can find a
	      specification and the algebraic generating function for the domino
	      permutations.

	\item Ikeda's thesis~\cite{ikeda:LCI-thesis} focuses on the enumeration of
	      $\Av(52341,\allowbreak 53241,\allowbreak 52431,\allowbreak
	      35142,\allowbreak 42513,\allowbreak 351624)$, which consist of permutations
	      representing the \emph{local complete intersection} Schubert varieties. These
	      Schubert varieties have singularities that are well-behaved, in the sense
	      that their local rings satisfy a certain algebraic condition, see e.g.,
	      Ulfarsson and Woo~\cite{ulfarsson-woo:lci}. The main
	      result in Ikeda's thesis, Theorem 6.9, is incorrect. We are able to find a specification with 1
	      CV and the correct algebraic generating function, which satisfies
	      the equation
	      \begin{align*}
	      	0 &= (13{x}^{9}-69{x}^{8}+162{x}^{7}-115{x}^{6}-496{x}^{5}+1401{x}^{4}-1415{x}^{3}+637{x}^{2}-97x+4)F(x)^{3}\\
	      		&\qquad + (50{x}^{8}-198{x}^{7}+376{x}^{6}+118{x}^{5}-1794{x}^{4}+2840{x}^{3}-1676{x}^{2}+281x-12)F(x)^{2}\\
	      		&\qquad +(60{x}^{7}-192{x}^{6}+248{x}^{5}+483{x}^{4}-1743{x}^{3}+1472{x}^{2}-272x+12)F(x)\\
	      		&\qquad + 24{x}^{6}-64{x}^{5}+40{x}^{4}+295{x}^{3}-432{x}^{2}+88x-4	
	      \end{align*}      
	% (13*x^9-69*x^8+162*x^7-115*x^6-496*x^5+1401*x^4-1415*x^3+637*x^2-97*x+4)*F(x)^3+(50*x^8-198*x^7+376*x^6+118*x^5-1794*x^4+2840*x^3-1676*x^2+281*x-12)*F(x)^2+(60*x^7-192*x^6+248*x^5+483*x^4-1743*x^3+1472*x^2-272*x+12)*F(x) +24*x^6-64*x^5+40*x^4+295*x^3-432*x^2+88*x-4

	\item Defant~\cite{defant:stack-sorting-preimages} studies the preimage of
	      various permutation classes under the West-stack-sorting operation,
	      derives that the preimage of $\Av(321)$ is $\Av(34251, 35241, 45231)$, and
	      gives rough bounds on its exponential growth rate, but is unable to
	      enumerate it. We find a specification with 2 CVs, allowing us to compute
	      $636$ terms in the counting sequence. We are unable to conjecture the
	      generating function from these terms, and thus we predict that it is
	      non-D-finite. Empirically, we estimate with quite a bit of confidence that
	      the growth rate is $6 + 2\sqrt{5}$.

	\item B\'ona and Pantone~\cite{bona:almost-monotone} used Combinatorial
	      Exploration to assist with the study of five classes avoiding four patterns
	      of length $5$, and one class avoiding five patterns of length $6$. All have
	      specifications with $2$ or more CVs, and they were able to predict that
	      several had D-finite generating functions, and give empirical estimates of
	      their asymptotic behavior.

	\item Dimitrov~\cite{dimitrov:box-classes} studies ``box classes''. An
	      incorrect theorem for the enumeration of  $\Av(1 \bb 3 \bb 2)$ is given~\cite[Theorem 4.4]{dimitrov:box-classes}. We
	      are able to find a specification with $1$ CV and compute the correct
	      algebraic generating function, which satisfies the equation
	      \begin{align*}
	      	0 &= {x}^{8}F(x)^6-{x}^{2}({x}^{6}+2{x}^{5}+5{x}^{4}+2{x}^{3}+4{x}^{2}-4x+5)F(x)^{4}\\
	      	& \qquad +(2{x}^{4}+2{x}^{3}+4{x}^{2}-2x+1)F(x)^{2} - 1
	      \end{align*}
	      
	      We can also enumerate many other box classes, including
	      $\Av(1 \bb 2 \bb 3)$, $\Av(1 \bb \bb\; 3 2)$, and $\Av(1 3 \bb \bb\; 2)$.
	% -1+F^6*x^8-x^2*(x^6+2*x^5+5*x^4+2*x^3+4*x^2-4*x+5)*F^4+(2*x^4+2*x^3+4*x^2-2*x+1)*F^2

	\item Egge~\cite{egge:unbalanced} conjectured that a group of permutation
	      classes defined by avoiding two patterns of length $4$ and one of length $6$
	      are all counted by the Schr\"oder numbers. Burstein and
	      Pantone~\cite{burstein:unbalanced} proved one of these conjectures, and then
	      Bloom and Burstein~\cite{bloom:egge} proved the remainder. We are able to
	      find specifications and generating functions for all of these classes.

	\item The skew-vexillary permutations are those in the class
		  $\Av(24153,\allowbreak 25143,\allowbreak 31524,\allowbreak
		   31542,\allowbreak 32514,\allowbreak 32541,\allowbreak 42153,\allowbreak
		  52143,\allowbreak 214365)$.
		  They were studied by Klein, Lewis, and
	      Morales~\cite{klein:skew-vexillary} who showed that their generating
	      function $S(x)$ can be written in terms of the generating function $V(x)$
	      for $\Av(2143)$ as
	      \[
		      S(x) = (1-x)V(x)^2 - V(x) - \frac{1}{1-x}.
	      \]
	      We are able to find a specification for the skew-vexillary
	      permutations expressed in terms of the subclass $\Av(2143)$.
	      This specification treats $\Av(2143)$ as a class whose generating
	      function is already independently known, and it
	      produces the same equation relating $S(x)$ and $V(x)$.
	      
	\item The juxtaposition of two classes $\CC$ and $\DD$ is the class of
	      permutations that can be formed by putting a permutation from $\CC$
	      side-by-side with a permutation from $\DD$ and vertically interleaving the
	      entries in some way. The juxtaposition of two finitely-based classes is
	      itself a finitely-based class. We can enumerate many juxtaposition classes,
	      including the interesting case $\CC = \Av(213, 231)$ and $\DD = \Av(132,
		  312)$, which is an example of two classes with rational generating functions
	      whose juxtaposition is algebraic (and non-rational).

	\item Staircase classes were considered by Albert, Pantone, and
	      Vatter~\cite{albert:growth-merges-staircases}. The $(\Av(21),\allowbreak
		  \Av(21))$-staircase is just $\Av(321)$, but the $(\Av(12),
		  \Av(21))$-staircase is much more interesting. It is predicted to be the class
	      $\Av(2341, 3421, 4231, 52143)$. We find a specification with $2$ CVs, and
	      compute $400$ terms of the counting sequence. We are unable to conjecture
	      its generating function, but we estimate the growth rate to be about
	      $4.1768325$. In the same paper, the class $\Av(4321, 321654, 421653, 431652, 521643, 531642)$ is given as a potential counterexample to a possible general
	      property---we find a $2$ CV specification that permits us to generate enough
	      terms of the counting sequence to determine that this class is unlikely
	      to be such a counterexample.

	\item Guo and Kitaev~\cite{gao:partially-ordered-patterns} explore the
	      notion of ``partially ordered permutations''.  We are able to find
	      specifications for many of the classes they consider.

	\item Elder~\cite{elder:stack-depth-2-and-infinite-series} determined the
	      basis of the class of permutations that can be sorted by a stack of depth $2$
	      and an infinite stack in series; it has basis elements of length at most $8$.
	      The algebraic generating function of this class was found later by
	      Elder, Lee, and Rechnitzer~\cite{elder:stack-depth-2-infinite-GF}. We find a
	      specification with $1$ CV, but the corresponding system of equations is too
	      large for us to be able to solve and produce the algebraic solution.

	\item The ``Widdershins spirals'' are a class of permutations whose
	      generating function was given without proof by Brignall, Engen, and
	      Vatter~\cite{brignall:lwqo-widdershins} and that plays an important role in
	      later work of Brignall and Vatter~\cite{brignall:lwqo}. We find a
	      specification with $0$ CVs and the corresponding rational generating
	      function.

	\item The Schubert varieties with a simple presentation for their
	      cohomology ring are studied by Gasharov and Reiner~\cite{gasharov-reiner:dbi}
	      and are called varieties \emph{defined by inclusion}. They showed that these
	      correspond to the permutation class $\Av(4231, 35142, 42513, 351624)$.
	      The class was enumerated in Albert and Brignall~\cite{albert-brignall:dbi}. We
	      find a specification with $0$ CVs and the corresponding algebraic generating
	      function.
	      
	\item Alland and Richmond~\cite{alland:grassmannian} show that for a 
		permutation $\pi$, the Schubert variety $X_\pi$ has a complete parabolic
		bundle structure if and only if $\pi \in \Av(3412, 52341, 635241)$. We
		are able to find a specification with 1 CV, guaranteeing that this class
		has an algebraic generating functions, but the system is too large for
		us to solve. We can, however, use the tree to generate the first $400$
		terms of the counting sequence and then use these terms to conjecture
		a value for the generating function; it appears to be algebraic with
		a minimal polynomial of degree $6$.
\end{itemize}

There are of course some permutation classes that appear in the literature
that we are not currently able to enumerate with Combinatorial Exploration,
although we are still in the process of creating new strategies that we
believe have potential to help. A few examples are listed below.

\begin{itemize}[label=$\diamond$]
	\item Burstein and Pantone~\cite{burstein:unbalanced} conjecture that
		$\Av(2143, 246135)$ is equinumerous to $\Av(2413)$. The generating
		function of the second class is already known, but we are unable to
		find a specification for the first class.
	\item Atkinson, Ru\v{s}kuc, and Smith~\cite{atkinson:sub-closure} prove
		that the substitution closure of the class is $\Av(123)$ is equal
		to the class $\Av(24153,\allowbreak 25314,\allowbreak 31524,\allowbreak 41352,\allowbreak 246135,\allowbreak 415263)$.
		We have not yet found a specification for this class.
	\item Fink, M\'esz\'aros, and St. Dizier~\cite{fink:zero-one} prove that 
		the Schubert polynomial $\mathfrak{S}_\pi$ is zero-one if and only if
		$\pi$ is in the class $\Av(12543,\allowbreak 13254,\allowbreak 13524, \allowbreak 13542,\allowbreak 21543,\allowbreak 125364,\allowbreak 125634,\allowbreak 215364,\allowbreak 215634,\allowbreak 315264,\allowbreak 315624,\allowbreak 315642)$. We have not found
		a specification for this class.
	\item The permutation class whose enumeration would be of the most interest
		is $\Av(1324)$, and we are unable to find a specification for it at this
		time.
\end{itemize}

\subsection{An Enumerative Example}
We end this subsection by demonstrating one way that the counting sequence and
generating function for $\Av(132)$ can be calculated. Our derivation closely
tracks the concept of proof trees that will be developed in
Section~\ref{section:combinatorial-exploration}.

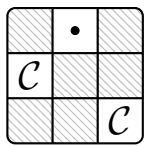
\begin{figure}
	\begin{center}
		\begin{tikzpicture}[scale=0.3, baseline=(current bounding box.center)]
			%		\input{132-node.tex}
			%		\specialnode
			\draw[pattern=north west lines, pattern color=lightgray]  (2,0) to[rounded corners=3pt] (0,0) to (0,2) to (2,2) to cycle;
			\draw[pattern=north west lines, pattern color=lightgray]  (2,4) to (0,4) to[rounded corners=3pt] (0,6) to (2,6) to cycle;
			\draw[pattern=north west lines, pattern color=lightgray] (2,0) rectangle (4,2);
			\draw[pattern=north west lines, pattern color=lightgray] (2,2) rectangle (4,4);
			\draw[pattern=north west lines, pattern color=lightgray] (4,2) rectangle (6,4);
			\draw[pattern=north west lines, pattern color=lightgray]  (6,4) to (4,4) to (4,6) to[rounded corners=3pt] (6,6) to cycle;

			\draw[thick, rounded corners=3pt] (0,0) rectangle (6,6);

			\draw[thick] (2,0) -- (2,6);
			\draw[thick] (4,0) -- (4,6);
			\draw[thick] (0,2) -- (6,2);
			\draw[thick] (0,4) -- (6,4);

			\node at (0,-0.5) {\color{white}.};

			\node at (1, 3) {\Large $\CC$};
			\node at (5, 1) {\Large $\CC$};
			\draw[fill] (3, 5) circle (5pt);
		\end{tikzpicture}
	\end{center}
	\caption{A graphical depiction of an arbitrary nonempty permutation in the class $\CC$ of $132$-avoiding permutations.}
	\label{figure:132-basic-proof}
\end{figure}

Let $\CC = \Av(132)$. Every permutation in $\CC$ is either empty or contains
a maximum entry. Let $\pi \in \CC$ be a nonempty permutation of length $n$
with maximum entry $\pi(k)$. Every entry with index less than $k$ must have
value larger than every entry with index greater than $k$, or else $\pi$
contains a $132$ pattern in which $\pi(k)$ acts as the $3$. Moreover, the
entries with index less than $k$ can together form any (non-standardized) permutation
that avoids $132$ and the entries with index greater than $k$ can also
together form any permutation that avoids $132$. More concretely, the
nonempty permutations in $\Av(132)$ are precisely those in the set
\[
	(\Av(132) \oplus \{1\}) \ominus \Av(132).
\]
This is shown pictorially in Figure~\ref{figure:132-basic-proof}.

This is a full structural description of $\Av(132)$ in the sense that the
set of $132$-avoiding permutations of length $n$ can be completely described
recursively in terms of shorter $132$-avoiding permutations. From this
description we can write down a functional equation satisfied by the
generating function $f(x)$ that counts $132$-avoiding permutations by
length. The statement ``Every $132$-avoiding permutation is either empty or
has a maximum entry with two non-interleaving $132$-avoiding permutations
(one on each side)'' translates directly to the equation
\[
	f(x) = 1 + xf(x)^2.
\]
Upon applying the quadratic equation and selecting the correct root, we find
$f(x) = (1 - \sqrt{1-4x})/(2x)$,	which is the generating function for the
Catalan numbers.

%% ==== %% ==== %% ==== %% ==== %% ==== %% ==== %% ==== %% ==== %% ==== %% ==== %%
%% ==== %% ==== %% ==== %% ====   SECTION THREE    ==== %% ==== %% ==== %% ==== %%
%% ==== %% ==== %% ==== %% ==== %% ==== %% ==== %% ==== %% ==== %% ==== %% ==== %%

\section{Combinatorial Exploration}
\label{section:combinatorial-exploration}
%!TEX root = combinatorial-exploration.tex

Combinatorial Exploration seeks to fully describe the structure of a
combinatorial set. In this context, a \emph{full description} is,
informally, one that accomplishes the following goal.

\begin{quote}
	The set of objects of size $n$ in a combinatorial set $\CC$ can be
	completely and uniquely constructed from the set of objects of size at most
	$n-1$ in $\CC$. That is to say, a set of rules is given to construct each
	object of size $n$ in $\CC$ exactly once.
\end{quote}

The combinatorial literature has often tried to capture this informal notion
using concepts like combinatorial specifications and admissible operators (see,
e.g.,~\cite{flajolet:ac}, which we discuss in more detail later in this
section). In order to give the most complete picture of
Combinatorial Exploration and justify the correctness of its products, we are
required to develop---for, to our knowledge, the first time---a full theoretical
framework that rigorously captures these oft-cited concepts.

Before we begin to build this framework we introduce, again informally,
\emph{proof trees}, a combinatorial structure that represents the output of
successful Combinatorial Exploration and represents the idea of a full description
mentioned above.

We rely on the domain of permutation patterns heavily in this section to explain
concepts and provide examples, but both Combinatorial Exploration and the theoretical
framework we develop can be applied to any type of combinatorial objects that
can be algorithmically constructed and manipulated.

\subsection{Proof Trees}

%Before explaining the algorithmic steps that constitute Combinatorial
%Exploration, we introduce the product it produces upon successful termination.
%A \emph{proof tree} is a tree-like structure that represents the notion of a
%``full description'' described above.

A \emph{proof tree} for a combinatorial set $\CC$ is a rooted tree in the
graph-theoretical sense. Each vertex $v$ represents a combinatorial set $\DD^{(v)}$, the
root representing $\CC$ itself. Each internal vertex $v$ with children
$u_1, \ldots, u_m$ implicitly (in a way we will make clear later) defines a structural
relationship between the parent set $\DD^{(v)}$ and the child sets
$\DD^{(u_1)}, \ldots, \DD^{(u_m)}$.

This relationship can take many forms. In the easiest of cases, it may be that $\DD^{(v)}$
is the disjoint union $\DD^{(u_1)} \sqcup \cdots \sqcup \DD^{(u_m)}$. More generally,
the parent-child relationship represents some way in which $\DD^{(v)}$ is decomposed
into simpler sets $\DD^{(u_1)}, \ldots, \DD^{(u_m)}$, or equivalently, how the child sets
$\DD^{(u_1)}, \ldots, \DD^{(u_m)}$ can be combined to reconstruct the parent $\DD^{(v)}$.

Before we bring these concepts into sharper focus, let us examine the procedure
by which we described and enumerated $\Av(132)$ in the previous section through
this informal lens. Figure~\ref{figure:132-proof-tree} shows a proof tree for
$\CC = \Av(132)$. This proof tree is essentially a pictorial representation of
the following structural description.

\begin{figure}[ht]
	\centering
	\begin{tikzpicture}
		\input{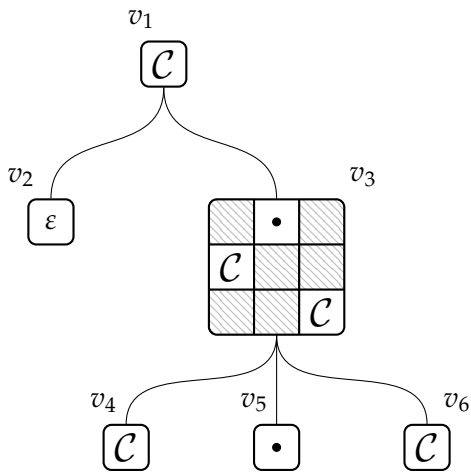}
	\end{tikzpicture}
	\caption{A proof tree for the class $\CC = \Av(132)$.}
	\label{figure:132-proof-tree}
\end{figure}

Every permutation in $\CC = \Av(132)$ is either the empty permutation of size $0$ (denoted $\varepsilon$), or
contains a topmost entry. This is represented in the proof tree by the root
$v_1$, with $\DD^{(v_1)} = \CC$, having a left child $v_2$ representing the set consisting of
only the empty permutation, and a right child $v_3$ representing the set of all
nonempty permutations in $\CC$. As described in the derivation of the previous
section, one may observe that every entry to the left of the maximum must have
greater value than every entry to the right of the maximum, and therefore all
nonempty permutations are represented by the given picture. The parent-children
relationship between these vertices is $\DD^{(v_1)} = \DD^{(v_2)} \sqcup \DD^{(v_3)}$.

The vertex $v_3$ has three children, $v_4$, $v_5$ and $v_6$, the first and third of
which represent the same
set as the root. The second represents the set containing the single permutation of
size $1$. This relationship is \emph{not} a disjoint union---rather, it represents
that any permutation in $\DD^{(v_3)}$ can be formed by selecting any two permutations
$\alpha,\beta \in \CC$, placing all entries in $\alpha$ above and to the left of
all entries in $\beta$, and placing a topmost entry between the two. Moreover,
every choice of $\alpha,\beta$ results in a distinct element of $\DD^{(v_3)}$.

We have introduced the concept of proof trees specifically because they enable
efficient enumeration of the objects represented by the root. We discuss this in
detail in future sections, for now only showing how the proof tree above
allows one to write down both a polynomial-time counting algorithm for $\CC$ and a
system of equations that can be solved to find its generating function.

Each parent-children relationship carries with it sufficient structural information
to determine the number of permutations of size $n$ in the parent from the number of
permutations of various sizes in the children. When the entire proof tree is taken
together, this permits the recursive computation of the number of permutations in the
root of any size.

Recall that if $\DD$ is a combinatorial set, then $\DD_n$ denotes the set of objects in
$\DD$ with size $n$. As we have observed that $\DD^{(v_1)} = \DD^{(v_2)} \sqcup \DD^{(v_3)}$ it
clearly follows that $|\DD^{(v_1)}_n| = |\DD^{(v_2)}_n| + |\DD^{(v_3)}_n|$. Furthermore,
the relationship between $v_3$ and its children implies the equality
\[
	|\DD^{(v_3)}_n| = \sum_{i=0}^{n-1} |\DD^{(v_4)}_i||\DD^{(v_6)}_{n-1-i}|.
\]
For legibility we use the fact that $|\DD^{(v_5)}_n| = 1$ for $n=1$, and $0$ otherwise, to avoid writing a double sum.

Noting that $\DD^{(v_2)}$ contains only the empty permutation of size $0$, we can derive from the proof tree for
$\Av(132)$ the system of recurrences
\begin{align*}
	|\DD^{(v_1)}_n| & = |\DD^{(v_2)}_n| + |\DD^{(v_3)}_n|                     \\
	|\DD^{(v_2)}_n| & = \left\{\begin{array}{ll} 1,&\text{ if $n=0$}\\ 0,&\text{ otherwise}\end{array}\right.               \\
	|\DD^{(v_3)}_n| & = \sum_{i=0}^{n-1} |\DD^{(v_4)}_i||\DD^{(v_6)}_{n-1-i}| \\
	|\DD^{(v_4)}_n| & = |\DD^{(v_1)}_n|                                       \\
	|\DD^{(v_6)}_n| & = |\DD^{(v_1)}_n|
\end{align*}

To derive the generating function of the class we associate to each vertex $v$
the generating function $F_v(x) = \displaystyle\sum_{n \geq 0}
	|\DD^{(v)}_n|x^n$. The structure described above allows one to write down the
system of equations
\begin{align*}
		F_{v_1}(x) &= F_{v_2}(x) + F_{v_3}(x)\\
		F_{v_2}(x) &= 1\\
		F_{v_3}(x) &= xF_{v_4}(x)F_{v_6}(x)\\
		F_{v_4}(x) &= F_{v_1}(x)\\
		F_{v_6}(x) &= F_{v_1}(x)\\
\end{align*}
which can be solved to find that the generating function for $\Av(132)$ is
$F_{v_1}(x) = (1-\sqrt{1-4x})/(2x)$, as expected. We choose to return now to the
viewpoint of recurrences, and we say nothing more about generating functions until
Section~\ref{section:transfer-tools}.

The reader who wishes to see a larger example of a proof tree right now can refer
to Figure~\ref{fig:av1243_1342_2143} on page~\pageref{fig:av1243_1342_2143}. The
combinatorial sets in this proof tree for $\Av(1243, 1342, 2143)$ are represented
by a structure we call \emph{tilings},
introduced in Section~\ref{section:pp-results}.

\subsection{Combinatorial Strategies}

The time has finally arrived to lay the foundation for the theoretical
framework upon which Combinatorial Exploration operates.

Each parent-children relationship in the proof tree for $\Av(132)$ in
Figure~\ref{figure:132-proof-tree} represents a structural decomposition
of the parent combinatorial set into the child combinatorial sets. In this
section we formally define this concept under the name \emph{combinatorial
	strategy}.

An $m$-ary \emph{combinatorial strategy} consists of three components, which we
now discuss in broad terms before stating the full formal definition. The first
and most important component of a strategy $S$ is a \emph{decomposition
	function} $d_S$, which takes as input a combinatorial set $\AA$. If strategy $S$
cannot be meaningfully applied to $\AA$ to decompose it into other sets, then
the output of $d_S$ is the symbol $\DNA$, short for ``does not apply''.
Otherwise, the output is an $m$-tuple of combinatorial sets $(\BB^{(1)}, \ldots,
	\BB^{(m)})$.

The other two components that comprise a combinatorial strategy capture the
requirement that if $d_S(\AA) = (\BB^{(1)}, \ldots, \BB^{(m)})$, then it must be
possible to calculate, in a manner uniform over all input sets $\AA$, the
enumeration of $\AA$ from the enumerations of the sets
$\BB^{(1)}, \ldots, \BB^{(m)}$. The \emph{reliance profile function}
$r_S : \N \to \Z^m$ for a
strategy $S$ encodes which terms in the enumeration of each $\BB^{(i)}$ are
necessary to calculate $|\AA_n|$. Given an input
$n \in \N$, $r_S(n) = (r^{(1)}(n), \ldots, r^{(m)}(n))$, where each
$r^{(i)}(n) \in \Z$ is an upper bound for the largest size of
elements in $\BB^{(i)}$ that are necessary in order to compute $|\AA_n|$. A
negative value indicates that no values $\BB^{(i)}$ are needed. For
instance, if $d_S(\AA) = (\BB^{(1)}, \BB^{(2)}, \BB^{(3)})$ and $r_S(10) = (2, 9, -1)$, this
means that $|\AA_{10}|$ can be computed from the quantities
$|\BB^{(1)}_0|, |\BB^{(1)}_1|, |\BB^{(1)}_2|$ and
$|\BB^{(2)}_0|, |\BB^{(2)}_1|, \ldots, |\BB^{(2)}_{9}|$; no values
$|\BB^{(3)}_i|$ are needed.

The third component of a strategy, the spectrum of \emph{counting functions},
governs precisely how the enumeration of $|\AA_n|$ is computed from the
quantities $|\BB^{(i)}_j|$. Each strategy $S$ possesses a counting function
$c_{S,(n)}$ for each $n \in \N$ such that when the quantities $|\BB^{(i)}_j|$
determined by the reliance profile function are given as input to $c_{(n),S}$,
then the output is $|\AA_n|$.

We want to emphasize at this point that the reliance profile function and the
counting functions of a strategy are fixed functions that do not vary with the
combinatorial set $\AA$ to which the strategy is applied. We now give the full
formal definition of a strategy.

\begin{definition} \label{def:strat}
	Let $\ZZZ$ be the collection of all combinatorial sets. An $m$-ary
	\emph{combinatorial strategy} $S$ consists of three components.
	\begin{enumerate}
		\item A \emph{decomposition function} $d_S : \ZZZ \to \ZZZ^m \cup \{\DNA\}$
		      whose input is a combinatorial set $\AA$ (the parent set), and whose output is
		      either an ordered $m$-tuple of combinatorial sets $(\BB^{(1)}, \ldots,
			      \BB^{(m)})$ (the child sets) or the symbol $\DNA$. When the output is
		      $d_S(\AA) = \DNA$, short for ``does not apply'', we say that $S$ \emph{cannot be
			      applied} to the combinatorial set $\AA$.
		\item A \emph{reliance profile function} $r_S : \N \to \Z^m$ whose input
		      is a natural number $n$ and whose output is an ordered
                      $m$-tuple of integers.
                      We use $r_S^{(i)}(n)$ to denote the $i$th component of $r_S(n)$,
		      i.e.,
		      \[
			      r_S(n) =  (r_S^{(1)}(n), \ldots, r_S^{(m)}(n)).
		      \]
		\item An infinite sequence of \emph{counting functions} $c_{S,(n)}$ indexed by $n \in \N$, each of whose input is $m$ tuples of integers $w^{(1)}, \ldots, w^{(m)}$
		      and whose output is a natural number. The counting functions must have the
		      property that if $d_S(\AA) = (\BB^{(1)}, \ldots, \BB^{(m)})$ and
		      $r_S(n) = (r_S^{(1)}(n), \ldots, r_S^{(m)}(n))$,
		      then for input tuples
		      \[
			      w^{(i)}(n) = \bigl(|\BB^{(i)}_0|, \ldots, |\BB^{(i)}_{r_S^{(i)}(n)}|\bigr)
		      \]
		      we have
		      \[
			      c_{S,(n)}(w^{(1)}(n), \ldots, w^{(m)}(n)) = |\AA_n|.
		      \]
		      To be overly explicit, the domain of $c_{S,(n)}$ is
		      $\N^{D_1} \times \cdots \times \N^{D_m}$, where
		      \[
				  D_k = \max(0,r_S^{(k)}(n)+1),
		      \]
			while the codomain is simply $\N$.
	\end{enumerate}
\end{definition}

The counting functions of a strategy $S$ describe, for any combinatorial set
$\AA$ to which $S$ can be applied and for each $n$ separately, the uniform
method of calculating $|\AA_n|$ from the quantities $|\BB^{(i)}_j|$ indicated by
the reliance profile function. As a result, an application of a strategy
$d_S(\AA) = (\BB^{(1)}, \ldots, \BB^{(m)})$ implies the statement:
\begin{quote}
	The number of objects in $\AA$ of size $n$ can be calculated from the number of
	objects in the $\BB^{(i)}$ of various sizes, without regard to actual nature of
	the objects,
\end{quote}
where the ``various sizes'' are dictated by the (fixed) reliance profile function $r_S$.

Many readers are no doubt familiar with the theory of symbolic combinatorics so
beautifully presented by Flajolet and Sedgewick~\cite{flajolet:ac}, to whom we
owe a substantial debt of gratitude for inspiring much of the framework
constructed here. In Subsection~\ref{subsection:comb-specs} we will say much more
about the similarities and differences between their framework and ours and will
discuss how what we are calling proof trees, for the moment, are simply an
alternative, more pictorial, formulation of a ``combinatorial specification''.

Meanwhile, it is instructive at this point to give several examples
and non-examples of strategies, some of which rely on notions not completely
defined until Section~\ref{section:pp-results}, but which we hope are clear
enough in context.

\begin{example}
	\label{example:size-0-or-not}
	Consider a binary strategy $Z$ that we will call ``size-$0$-or-not''
	constructed as follows.
	\begin{enumerate}
		\item The decomposition function $d_Z$ is defined by $d_Z(\AA) = (\BB, \CC)$
		      where $\BB$ contains precisely the elements of $\AA$ with size $0$ and $\CC$
		      contains all others. It follows that $\AA$ is the disjoint
		      union $\BB \sqcup \CC$. This strategy applies to all possible input sets, and
		      thus the output is never $\DNA$.
		\item The reliance profile function $r_Z$ is $n \mapsto (n, n)$.
%		\item The reliance profile function $r_Z$ is $n \mapsto (0, n)$.
%		\item The counting functions are defined by
%		      \[
%			      c_{Z,(n)}((b_0), (c_0, \ldots, c_n)) = \left\{
%			      	\begin{array}{l}
%			      		b_0, \quad n=0\\
%			      		c_n, \quad n > 0
%			      	\end{array}
%			      	\right..
%		      \]
%		      It is clear by the definition of $d_Z$ that the counting
%		      	functions satisfy
%		      \[
%			     |\AA_n| = c_{Z,(n)}((|\BB_0|), (|\CC_0|,  \ldots, |\CC_n|))
%		      \]
%		      for all $n \in \N$, and therefore they meet conditions of Definition~\ref{def:strat}.
		 \item The counting functions are defined by
		      \[
			      c_{Z,(n)}((b_0, \ldots, b_n), (c_0, \ldots, c_n)) = b_n + c_n
		      \]
		      and since $\AA$ is the disjoint union $\BB \sqcup \CC$, it is clear that the
		      counting functions satisfy
		      \[
			      c_{Z,(n)}((|\BB_0|, \ldots, |\BB_n|), (|\CC_0|,  \ldots, |\CC_n|)) = |\BB_n| + |\CC_n| = |\AA_n|
		      \]
		      for all $n \in \N$, and therefore they satisfy the condition in Definition~\ref{def:strat}.

	\end{enumerate}
	This is the strategy that is applied to the root of the proof tree for
	$\Av(132)$ in Figure~\ref{figure:132-proof-tree}; the output $(\BB,\CC)$ of the
	decomposition function gives the two children of the root.

	Note that since we are guaranteed that $\BB$ will never contain any
	objects of size greater than $0$ we could also have defined the
	reliance profile function to be $n \mapsto (0,n)$, with counting
	functions
	\[
		c_{Z,(n)}((b_0), (c_0, \ldots, c_n)) = \left\{
			\begin{array}{l}
	      		b_0, \quad n=0\\
	      		c_n, \quad n > 0
	      	\end{array}\right.,
	\]
	but we chose to highlight the disjoint union nature of this strategy.

\end{example}

\begin{example}
	\label{example:factor-around-max-entry}
	We now describe the other strategy used in the proof tree of $\Av(132)$,
	although we cannot give all of the details until
	Section~\ref{section:pp-results} (here we describe just a specific case of
	a much more general strategy). We have also slightly simplified it by not
	considering the point itself to be one of the children.

	Consider a binary strategy $F$ that we will call ``factor-around-max-entry''
	constructed as follows.
	\begin{enumerate}
		\item Let $\AA$ be a set of permutations. If there exist sets $\BB$ and $\CC$
		      such that $\AA = (\BB \oplus \{1\}) \ominus \CC$, then they are unique and we
		      define $d_F(\AA) = (\BB, \CC)$. Otherwise, $d_F(\AA) = \DNA$.
		\item Since the objects in $\AA$ of size $n$ are built from pairs $(\beta,
			      \gamma)$ where $|\beta|+|\gamma|=n-1$, the reliance profile function is $n
			      \mapsto (n-1, n-1)$. (This example demonstrates why the values in the reliance
			      profile function come from $\Z$ rather than just $\N$.)
		\item The counting functions are $c_{F,(0)} = 0$ and if $n > 0$
		      \[
			      c_{F,(n)}((b_0, \ldots, b_{n-1}), (c_0, \ldots, c_{n-1}) =
			      \sum_{j=0}^{n-1}b_jc_{n-1-j}.
		      \]
		      We must show that the counting function satisfies the equality
		      \[
			      |\AA_n| = c_{F,(n)}((|\BB_0|,\ldots, |\BB_{n-1}|), (|\CC_0|, \ldots, |\CC_{n-1}|));
		      \]
		      every permutation $\pi \in \AA_n$ can be written as $(\beta \oplus 1) \ominus
			      \gamma$ for $\beta \in \BB$ and $\gamma \in \CC$ in precisely one way, because
		      $\beta$ must involve the entries to the left of the maximum entry and $\gamma$
		      must involve the entries to the right of the maximum entry.
	\end{enumerate}
\end{example}

\begin{nonexample}
	\label{nonexample:strategy}
	In the previous example, the identification of the maximum entry of a
	permutation as a splitting point is necessary for there to be a uniform
	counting function. Suppose instead that we had attempted to define a
	decomposition function $d_F(\AA) = (\BB,\CC)$ if $\AA = \BB \oplus \CC$, with
	$d_F(\AA) = \DNA$ otherwise. This leads to two problems. First, the
	decomposition is not unique; for example $\{1\} \oplus \{12, 132\}$ and
	$\{12\} \oplus \{1, 21\}$ both equal $\{123, 1243\}$, so what should the value of
	$d_F(\{123, 1243\})$ be?

	Secondly and separately, it is not possible to define
	counting functions that are correct uniformly over all possible inputs $\AA$.
	The examples
	\[
		\{12,123,1234\} = \{1,12\} \oplus \{1,12\}
	\]
	and
	\[
		\{12,123,132,1243\} = \{1,12\} \oplus \{1,21\}
	\]
	demonstrate that there is no counting function that can compute the number of
	objects in $\AA$ of each size from the number of objects in $\BB$ and $\CC$ of
	each size, as the first example would require that
	\[
		c_{(F),3}((0,1,1),(0,1,1)) = 1
	\]
	while the second would require that
	\[
		c_{(F),3}((0,1,1),(0,1,1)) = 2,
	\]
	i.e., the two examples have different numbers of permutations of length $3$
	on the left hand side, despite both having the same number of permutations of
	each length on each component of the right-hand side.
\end{nonexample}

\begin{example}
	A two-colored permutation is a permutation in which each of the entries is
	assigned a color, red or blue. Given two sets of (uncolored) permutations $\BB$
	and $\CC$, we define $\BB \oasterisk \CC$ to be the set of two-colored
	permutations where the red entries form a permutation in $\BB$ and the blue
	entries form a permutation in $\CC$. For example, if $132 \in \BB$ and $21 \in
		\CC$, then $\overline{\color{red} 1}{\color{blue} 4}\overline{\color{red}
			5}\overline{\color{red} 3}{\color{blue} 2} \in \CC$.\footnote{For those who are
		reading this in black-and-white, the overlined entries are red while the
		remaining entries are blue.}

	We now define a strategy $M$ that we call ``colored-merge''.
	\begin{enumerate}
		\item The input is a set $\AA$ of two-colored permutations. If there exist sets of
		      (ordinary) permutations $\BB$ and $\CC$ such that $\BB \oasterisk \CC = \AA$,
		      then the choice of $\BB$ and $\CC$ is unique and we define $d_M(\AA) =
			      (\BB,\CC)$. Otherwise $d_M(\AA) = \DNA$.
		\item In order to form the elements of $\AA_n$, it is necessary to know the
		      permutations in $\BB$ and $\CC$ of sizes $0, \ldots, n$. Therefore, the
		      reliance profile function is $n \mapsto (n, n)$.
		\item Each element of $\AA_n$ is formed by selecting an element from $\BB$ of
		      size $j$, an element from $\CC$ of size $n-j$, and then interleaving the two
		      permutations by selecting which $j$ of the $n$ positions and which $j$ of the
		      $n$ values will be occupied by entries from the red permutation. Thus, the
		      counting functions are
		      \[
			      c_{M,(n)}((b_0, \ldots, b_n), (c_0, \ldots, c_n)) = \sum_{j=0}^n{n \choose j}^2b_jc_{n-j}.
		      \]
	\end{enumerate}
\end{example}

Finally, we provide an additional non-example to demonstrate that operations that do not preserve the ability to count are not strategies.

\begin{nonexample}
	For a permutation $\pi$ of length at least $1$, let $\pi^{-}$ be the
	permutation that remains after deleting the largest entry of $\pi$. Extending
	this to sets of permutations, define $\AA^{-} = \{\pi^{-} : \pi \in \AA\}$.

	Consider the potential decomposition function $d_D(\AA) = \AA^{-}$. The equalities
	\[
		d_D(\{123,132,312\}) = \{12\} = d_D(\{123\})
	\]
	demonstrate that it is not possible to uniformly compute the counting sequence
	of $\AA$ from the counting sequence of $\AA^{-}$. Therefore we cannot define a
	strategy with this decomposition function.
\end{nonexample}

When a strategy $S$ is applied to a combinatorial set $\AA$ the
output of the decomposition function is either $d_S(\AA) = \DNA$, in which case
no information has been learned, or $d_S(\AA) = (\BB^{(1)}, \ldots, \BB^{(m)})$,
in which case the strategy has produced what we call a \emph{combinatorial rule}
\[
	\AA \xleftarrow{S} (\BB^{(1)}, \ldots, \BB^{(m)}),
\]
which records the fact that $\AA$ can be decomposed into sets $\BB^{(1)},
\ldots, \BB^{(m)}$ under strategy $S$. Since $S$ must have a reliance
profile function and valid counting functions in order to be considered a
strategy, this implies that the counting sequence of $\AA$ is a function of
the counting sequences of $\BB^{(1)}, \ldots , \BB^{(m)}$.

\subsection{Combinatorial Specifications}
\label{subsection:comb-specs}

In the first chapters of Flajolet and Sedgewick~\cite{flajolet:ac}, the authors introduce
several operations whose input is a tuple of combinatorial sets and whose output
is a single combinatorial set. Among these are the disjoint union
$\AA = \BB \sqcup \CC$, the Cartesian product $\AA = \BB \times \CC$, the sequence operator
$\AA = \text{\scshape Seq}(\BB)$, and several others. Our decomposition
functions are in some sense the ``mirror image'' of their operators; for them,
$\BB$ and $\CC$ are the inputs to a function that outputs
$\AA = \BB \times \CC$, while for us, any set $\AA$ may be the input of a decomposition function
giving $\BB$ and $\CC$ as outputs \emph{if such a decomposition is possible}.
Although this is not a major deviation, it is this perspective that truly drives
the theory of Combinatorial Exploration---combinatorial sets are broken down as
much as possible via \emph{decomposition}, rather than built up from atoms, and
fingers are crossed that after some number of such decompositions a proof tree
can be formed. In the standard viewpoint of symbolic combinatorics, one must
figure out from scratch how to express a combinatorial set as the image of some operator;
Combinatorial Exploration provides a framework of systematic tools to aid
discovering when a combinatorial set can be decomposed into (hopefully) simpler
sets.

Flajolet and Sedgwick~\cite[Definition I.5]{flajolet:ac} additionally define an
operator to be \emph{admissible} if the counting sequence of the output set only
depends on the counting sequence of the input sets, and a \emph{combinatorial
specification} is defined to be a system of operators (more detailed information
is given below). They then define and prove the admissibility
of several useful operators that they call $+$, $\times$, \textsc{Seq},
\textsc{PSet}, \textsc{MSet}, and \textsc{Cyc}, and they call a combinatorial
set \emph{constructible} if it is a component of a combinatorial specification
made up entirely of these particular operators. By now some readers have surely
realized that our notion of a proof tree is similar to the concept of a
combinatorial specification, and this section makes that link more precise.

The situation considered by Flajolet and Sedgewick is that you have a given set of objects with known structure, or perhaps defined in terms of a combinatorial specification. Two important issues, whether the specification defines a unique set of objects and whether the resulting system of generating function equations uniquely defines a power series solution, can be confirmed by inspection, and therefore are typically not explicitly addressed. Our approach is to attempt to search for a specification inside a larger set of combinatorial rules, and thus, we need to develop a theory to prove that our approach does not run into these uniqueness issues.
Our framework, particularly the consideration of reliance
profile functions and counting functions as inherent to a strategy, permits us
to study such questions. In Section~\ref{section:productivity} we define a
strategy to be \emph{productive} if it satisfies certain criteria; we then prove
that a proof tree
composed entirely of rules created by productive strategies is guaranteed to contain sufficient information
to uniquely define the counting sequences of the sets involved.

Before describing in detail the correspondence between proof trees and combinatorial
specifications, we need to point out a particularly useful kind of strategy. As
defined in the previous subsection, a combinatorial rule is a tuple
$(\AA, (\BB^{(1)}, \ldots, \BB^{(m)}),S)$ where $\AA$ and $\BB^{(i)}$ are
combinatorial sets and $S$ is a $m$-ary strategy. Stylistically, we write
\[
	\AA \xleftarrow{S} (\BB^{(1)}, \ldots, \BB^{(m)}).
\]
Although we did not specifically mention this earlier, it is permissible for a
strategy to be $0$-ary, or nullary. We call such a strategy a \emph{verification
strategy} because it signifies that the counting sequence of the input set is
known independently of any structural decomposition. Recall that in the strict
framework we have created for strategies, for a given strategy $S$ the counting
functions $c_{S,(n)}$ are \emph{independent of the input set $\AA$}. As a
result, we can define for each combinatorial set $\AA$ whose counting sequence
is independently known a verification strategy $V_{\AA}$ with the following
properties:
\begin{itemize}[label=$\diamond$]
	\item $d_{V_{\AA}}(\AA) = ()$ (the $0$-tuple) and $d_{V_{\AA}}(\RR) = \DNA$ for all $\RR \neq \AA$,
	\item $r_{V_{\AA}}$ is the function $n \mapsto ()$ (again, the $0$-tuple), and
	\item for all $n$, $c_{V_{\AA},(n)}$ is a $0$-ary function with output $|\AA_n|$.
\end{itemize}

Of course, one does not want to indiscriminately allow the verification
strategies $V_{\AA}$ to be applied for all $\AA$, or else a trivial proof tree
would always be immediately found with no decomposition at all. It is up to the
user who is employing Combinatorial Exploration to select which verification
strategies to allow---in most cases it makes sense to employ any verification
strategy $V_{\AA}$ for which the counting sequence for $\AA$ is either already
known or can be independently calculated. A major benefit of this approach is
that once a proof tree for a combinatorial set $\RR$ is found, future
applications of Combinatorial Exploration can activate the verification strategy
$V_{\RR}$ to use this knowledge without the extra work of rediscovering the
proof tree for $\RR$.

In any case, most applications require the use of at least a few verification
strategies. In the proof tree for $\Av(132)$ in
Figure~\ref{figure:132-proof-tree}, like in nearly all of the proof trees for
pattern-avoiding permutation classes, we require the use of
$V_{\{\varepsilon\}}$, providing the combinatorial rule
\[
	\{\varepsilon\} \xleftarrow{V_{\{\varepsilon\}}} (),
\]
and in most other proof trees we also require the verification strategy $V_{\{1\}}$ which creates the rule
\[
	\{1\} \xleftarrow{V_{\{1\}}} ()
\]
verifying the set containing the single permutation of length $1$.

A \emph{combinatorial specification}, as defined by Flajolet and
Sedgewick~\cite[Definition I.7]{flajolet:ac},
is a set of combinatorial rules with the
property that every combinatorial set that appears on the right-hand side of a
rule appears as the left-hand side of a rule exactly once. The proof tree for
$\CC = \Av(132)$ in Figure~\ref{figure:132-proof-tree}, for instance, may be
written as the specification\label{132-specification}
\begin{align*}
	\CC             & \xleftarrow{Z} (\{\varepsilon\}, \DD) \\
	\DD             & \xleftarrow{F} (\CC, \CC)             \\
	\{\varepsilon\} & \xleftarrow{V_{\{\varepsilon\}}} ()
\end{align*}
Here, $Z$ is the strategy ``size-0-or-not'' from
Example~\ref{example:size-0-or-not}, while $F$ is the strategy
``factor-around-max-entry'' from Example~\ref{example:factor-around-max-entry}.

Proof trees and combinatorial specifications are two mostly equivalent structures used
to represent the same information. Each non-leaf vertex $v$ with children
$u_1, \ldots, u_m$ derived from strategy $S$ corresponds to a combinatorial rule
\[
	D^{(v)} \xleftarrow{S} (D^{(u_1)},\ldots,D^{(u_m)}),
\]
while each leaf $w$ corresponds to either a verification strategy $V$ and rule
\[
	D^{(w)} \xleftarrow{V} ()
\]
or the set $D^{(w)}$ represented by $w$ already appears as a non-leaf vertex in
the tree (equivalently, $D^{(w)}$ is already the left-hand side of some rule).
One difference between these two structures is that proof trees have a
combinatorial set designated as the root, while combinatorial specifications
do not (although in practice there is usually one particular set whose enumeration
is sought, and this acts as a root). Another difference is that the union
of two specifications whose combinatorial sets are disjoint would be
considered a specification, while the same is not true for proof trees. Neither
of these differences is problematic for the work done here.

The proof tree model gives a more intuitive picture of the structural
hierarchy of a combinatorial set as well as better graphical depictions, while
the combinatorial specification model is better suited for proving results and
describing algorithms. Going forward, we use the two terms
largely interchangeably, depending on the context.

As we mentioned in Subsection~\ref{subsection:pp-success} about our successes
in applying Combinatorial Exploration to the field of permutation patterns, we
employ strategies not discussed in this paper that work with counting functions
not just in the size variable $n$ but also in additional variables
$k_1, k_2, \ldots$, leading to catalytic variables in the corresponding system
of generating function equations. The framework thus requires a multivariate
generalization to combinatorial specifications, which we address in future work.

\subsection{Strategic Exploration}

In the proof tree for $\Av(132)$ shown in Figure~\ref{figure:132-proof-tree},
the root $\CC$ and its two children constitute the first step in a structural
description of $\Av(132)$: \emph{every permutation in $\Av(132)$ is either the
	empty permutation, or contains a maximum entry}. Why did we choose the maximum
entry rather than the minimum entry, or the leftmost or rightmost entry? We did
so solely because, with the benefit of hindsight, it is that choice that resulted
in the nice proof tree we have presented.

In order to discover proof trees like the one for $\Av(132)$, the process of
Combinatorial Exploration applies to any given combinatorial set many strategies
simultaneously.
\begin{figure}
	\begin{center}
		%!TEX root = combinatorial-exploration.tex
\begin{tikzpicture}[baseline=(current bounding box.north)]
	\node (root) at (-1, 0.7) {\specialnoderoot};
	\node (empty) at (-2, -2) {\specialnodeempty};
	\begin{scope}[scale=0.3, shift={(-4, -10.65)}]
		\draw[pattern=north west lines, pattern color=lightgray]  (2,0) to[rounded corners=3pt] (0,0) to (0,2) to (2,2) to cycle;
		\draw[pattern=north west lines, pattern color=lightgray]  (2,4) to (0,4) to[rounded corners=3pt] (0,6) to (2,6) to cycle;
		\draw[pattern=north west lines, pattern color=lightgray] (2,0) rectangle (4,2);
		\draw[pattern=north west lines, pattern color=lightgray] (2,2) rectangle (4,4);
		\draw[pattern=north west lines, pattern color=lightgray] (4,2) rectangle (6,4);
		\draw[pattern=north west lines, pattern color=lightgray]  (6,4) to (4,4) to (4,6) to[rounded corners=3pt] (6,6) to cycle;
	
		\draw[thick, rounded corners=3pt] (0,0) rectangle (6,6);
	
		\draw[thick] (2,0) -- (2,6);
		\draw[thick] (4,0) -- (4,6);
		\draw[thick] (0,2) -- (6,2);
		\draw[thick] (0,4) -- (6,4);
		
		\node at (1, 3) {\Large $\CC$};
		\node at (5, 1) {\Large $\CC$};
		\draw[fill] (3, 5) circle (5pt);
	\end{scope}
	
	\ptedge{(root)}{(-0.5,0.7)}{(empty)}{(-0.5,0.9)}
	\ptedge{(root)}{(-0.5,0.7)}{(-0.3,-2)}{(-0.5,0.9)}
\end{tikzpicture}
\qquad
\begin{tikzpicture}[baseline=(current bounding box.north)]
	\node (root) at (-1, 0.7) {\specialnoderoot};
	\node (empty) at (-2, -2) {\specialnodeempty};
	\begin{scope}[scale=0.3, shift={(-3, -10.65)}]
		\draw[pattern=north west lines, pattern color=lightgray]  (2,0) to[rounded corners=3pt] (0,0) to (0,2) to (2,2) to cycle;
		\draw[pattern=north west lines, pattern color=lightgray]  (2,4) to (0,4) to[rounded corners=3pt] (0,6) to (2,6) to cycle;
		\draw[pattern=north west lines, pattern color=lightgray] (2,2) rectangle (4,4);
	
		\draw[thick, rounded corners=3pt] (0,0) rectangle (4,6);
	
		\draw[thick] (2,0) -- (2,6);
		\draw[thick] (0,2) -- (4,2);
		\draw[thick] (0,4) -- (4,4);
		
		\node at (3, 1) {\Large $\CC$};
		\draw[fill] (1, 3) circle (5pt);
		\draw[thick] (2.3, 4.3) -- (3.7, 5.7);
	\end{scope}
	
	\ptedge{(root)}{(-0.5,0.7)}{(empty)}{(-0.5,0.9)}
	\ptedge{(root)}{(-0.5,0.7)}{(-0.3,-2)}{(-0.5,0.9)}
\end{tikzpicture}
\qquad
\begin{tikzpicture}[baseline=(current bounding box.north)]
	\node (root) at (-1, 0.7) {\specialnoderoot};
	\node (empty) at (-2, -2) {\specialnodeempty};
	\begin{scope}[scale=0.3, shift={(-4, -8.65)}]
		\draw[pattern=north west lines, pattern color=lightgray]  (2,0) to[rounded corners=3pt] (0,0) to (0,2) to (2,2) to cycle;
		\draw[pattern=north west lines, pattern color=lightgray]  (4,0) to[rounded corners=3pt] (6,0) to (6,2) to (4,2) to cycle;
		\draw[pattern=north west lines, pattern color=lightgray] (2,2) rectangle (4,4);
	
		\draw[thick, rounded corners=3pt] (0,0) rectangle (6,4);
	
		\draw[thick] (2,0) -- (2,4);
		\draw[thick] (4,0) -- (4,4);
		\draw[thick] (0,2) -- (6,2);
		
		\node at (1, 3) {\Large $\CC$};
		\draw[fill] (3, 1) circle (5pt);
		\draw[thick] (4.3, 2.3) -- (5.7, 3.7);
	\end{scope}
	
	\ptedge{(root)}{(-0.5,0.7)}{(empty)}{(-0.5,0.9)}
	\ptedge{(root)}{(-0.5,0.7)}{(-0.3,-2)}{(-0.5,0.9)}
\end{tikzpicture}
\qquad
\begin{tikzpicture}[baseline=(current bounding box.north)]
	\node (root) at (-1, 0.7) {\specialnoderoot};
	\node (empty) at (-2, -2) {\specialnodeempty};
	\begin{scope}[scale=0.3, shift={(-4, -10.65)}]
		\draw[pattern=north west lines, pattern color=lightgray]  (2,0) to[rounded corners=3pt] (0,0) to (0,2) to (2,2) to cycle;
		\draw[pattern=north west lines, pattern color=lightgray]  (4,2) to (4,0) to[rounded corners=3pt] (6,0) to (6,2) to cycle;
		\draw[pattern=north west lines, pattern color=lightgray] (2,4) rectangle (4,6);
		\draw[pattern=north west lines, pattern color=lightgray] (2,2) rectangle (4,4);
		\draw[pattern=north west lines, pattern color=lightgray] (0,2) rectangle (2,4);
		\draw[pattern=north west lines, pattern color=lightgray]  (6,4) to (4,4) to (4,6) to[rounded corners=3pt] (6,6) to cycle;
	
		\draw[thick, rounded corners=3pt] (0,0) rectangle (6,6);
	
		\draw[thick] (2,0) -- (2,6);
		\draw[thick] (4,0) -- (4,6);
		\draw[thick] (0,2) -- (6,2);
		\draw[thick] (0,4) -- (6,4);
		
		\node at (1, 5) {\Large $\CC$};
		\node at (3, 1) {\Large $\CC$};
		\draw[fill] (5, 3) circle (5pt);
	\end{scope}
	
	\ptedge{(root)}{(-0.5,0.7)}{(empty)}{(-0.5,0.9)}
	\ptedge{(root)}{(-0.5,0.7)}{(-0.3,-2)}{(-0.5,0.9)}
\end{tikzpicture}
	\end{center}
	\caption{Four different combinatorial rules that result from applying
		different strategies---in this case, isolating a point in the four
		different directions---to the same set $\CC$.}
	\label{figure:four-combinatorial-rules}
\end{figure}
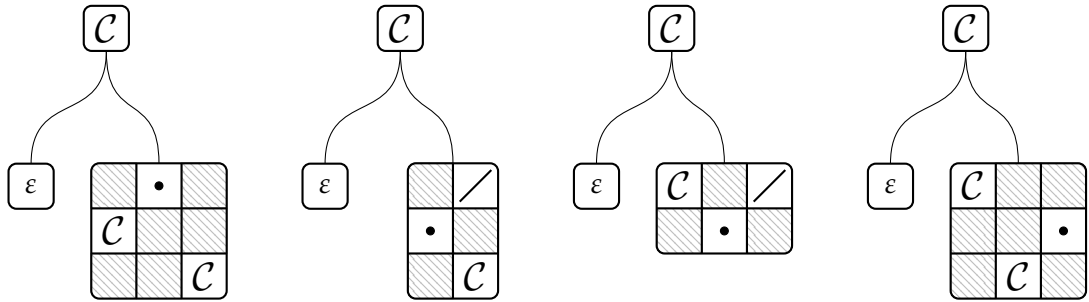

Figure~\ref{figure:four-combinatorial-rules} depicts graphically
the four combinatorial rules that result from splitting $\Av(132)$ into
$\{\varepsilon\}$ and the permutations of length at least 1, and choosing an
extreme direction to draw a point.%
\footnote{Note that in the middle two cases, the two non-point cells do
not separate into their own rows or columns. These pictures only give a
rough idea of the permutations involved, and should not be interpreted too
literally. Section~\ref{section:pp-results} introduces a representation
called a \emph{tiling} that makes these ideas much more concrete.}
Thus, applying these four strategies (isolating a point in each of the four
extreme directions) produces a combinatorial rule. Each rule has the same parent,
and each rule has the child $\{\varepsilon\}$, but the second child of each rule
is different.

Combinatorial Exploration works by repeatedly applying a collection of strategies,
first to the root (the combinatorial set you seek to understand), then to the
children of the rules that are produced by these applications, then to the children
of those rules, and so on. While it sounds at first like this will suffer from
a problem of combinatorial explosion, we shall see that it is surprisingly
effective.

We call the ever-growing collection of combinatorial rules that are discovered
the \emph{decomposition universe}, and we can depict them pictorially by pushing
the graph/tree model just slightly past its useful limit. In a proof tree, the
children of a vertex are just the combinatorial sets produced by one single
strategy applied to the parent, but in the decomposition universe each parent
vertex may have several groups of children, each group coming from the application
of a different strategy.

\begin{figure}
	\centering
	\tikzset{
    trinode/.style={
        draw,
        regular polygon,
        regular polygon sides=3,
        fill=black,
        minimum height=1em
    }
}
\begin{tikzpicture}
	\node[rectangle, draw, thick] (A) at (0, 20) {$\AA$};

	\node[trinode] (N1) at (-3, 18.5) {};
	\node[trinode] (N2) at (0, 18.5) {};
	\node[trinode] (N3) at (3, 18.5) {};
	
	\node[rectangle, draw, thick] (B) at (-3.5, 17) {$\BB$};
	\node[rectangle, draw, thick] (C) at (-2.5, 17) {$\CC$};
	
	\node[rectangle, draw, thick] (D) at (-0.75, 17) {$\DD$};
	\node[rectangle, draw, thick] (E) at (0, 17) {$\EE$};
	\node[rectangle, draw, thick] (F) at (0.75, 17) {$\FF$};
	
	\node[rectangle, draw, thick] (G) at (3, 17) {$\GG$};
	
	\node[rectangle, draw, thick] (H) at (-4, 14) {$\HH$};
	\node[rectangle, draw, thick] (I) at (-3, 14) {$\II$};
	
	\node[rectangle, draw, thick] (J) at (-2, 14) {$\JJ$};
	\node[rectangle, draw, thick] (K) at (-1, 14) {$\KK$};
	
	\node[rectangle, draw, thick] (M) at (0.25, 14) {$\MM$};
	\node[rectangle, draw, thick] (N) at (1.25, 14) {$\NN$};
	\node[rectangle, draw, thick] (O) at (2.25, 14) {$\OO$};
	
	\node[trinode] (N4) at (-3.5, 15.5) {};
	\node[trinode] (N5) at (-1.5, 15.5) {};
	\node[trinode] (N6) at (0, 15.5) {};
	\node[trinode] (N7) at (1.25, 15.5) {};	
	
	\draw[thick] (A.south) to[out=270, in=90] ($(N1.north) + (0, -0.1)$);
	\draw[thick] (A.south) to[out=270, in=90] ($(N2.north) + (0, -0.1)$);
	\draw[thick] (A.south) to[out=270, in=90] ($(N3.north) + (0, -0.1)$);
	
	\draw[thick] (N1.south) to[out=270, in=90] (B.north);
	\draw[thick] (N1.south) to[out=270, in=90] (C.north);
	
	\draw[thick] (N2.south) to[out=270, in=90] (D.north);
	\draw[thick] (N2.south) to[out=270, in=90] (E.north);
	\draw[thick] (N2.south) to[out=270, in=90] (F.north);
	
	\draw[thick] (N3.south) to[out=270, in=90] (G.north);
	
	\draw[thick] (B.south) to[out=270, in=90] ($(N4.north) + (0, -0.1)$);
	\draw[thick] (D.south) to[out=270, in=90] ($(N5.north) + (0, -0.1)$);
	\draw[thick] (E.south) to[out=270, in=90] ($(N6.north) + (0, -0.1)$);
	\draw[thick] (F.south) to[out=270, in=90] ($(N7.north) + (0, -0.1)$);
	
	\draw[thick] (N4.south) to[out=270, in=90] (H.north);
	\draw[thick] (N4.south) to[out=270, in=90] (I.north);
	
	\draw[thick] (N5.south) to[out=270, in=90] (J.north);
	\draw[thick] (N5.south) to[out=270, in=90] (K.north);
	
	\draw[thick] (N7.south) to[out=270, in=90] (M.north);
	\draw[thick] (N7.south) to[out=270, in=90] (N.north);
	\draw[thick] (N7.south) to[out=270, in=90] (O.north);
\end{tikzpicture}
	\caption{An abstract decomposition universe with seven combinatorial rules.}
	\label{figure:universe}
\end{figure}
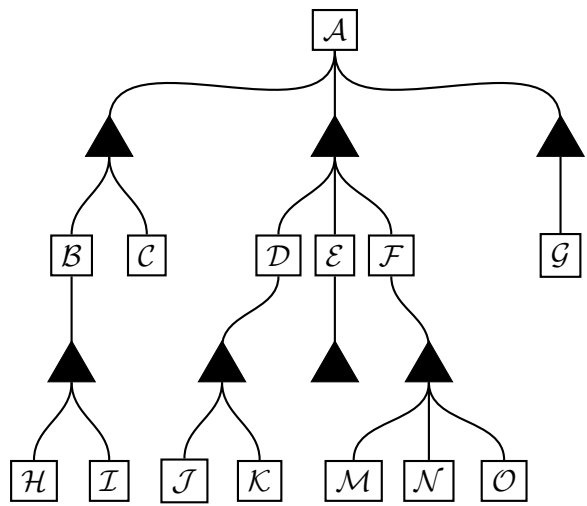

Figure~\ref{figure:universe} shows an abstract decomposition universe that starts
with the root combinatorial set $\AA$. Suppose that some strategy $S_1$ decomposes
$\AA$ into $\BB$ and $\CC$. In the figure, we use a triangular node to group $\BB$
and $\CC$ together. Suppose then that some other strategy $S_2$ decomposes $\AA$
into $\DD$, $\EE$, and $\FF$, and that some third strategy $S_3$ decomposes $\AA$ into
$\GG$. There may be other strategies that do not apply to $\AA$ at all, and it can
be possible that some of the child sets are actually the same set. After
attempting to apply all of one's strategies to $\AA$, the next step is to try
to apply them all to $\BB$, perhaps producing some more child groupings as shown
in the figure, then apply all of them to $\CC$, then $\DD$, and so on. It is possible
that there are sets to which no strategies apply. Moreover, there may be
verification strategies that apply to some of these sets, which would
correspond to a grouping of zero children, as can be seen with set $\EE$ in the
figure.

A simpler and less pictorial but more mathematically manageable way to think of the
decomposition universe is simply as a large set of combinatorial rules. The
universe shown in Figure~\ref{figure:universe} corresponds to the set of
rules (omitting the strategy names)
\begin{align*}
	\AA &\leftarrow (\BB,\CC) & \AA &\leftarrow (\DD, \EE, \FF) & \AA &\leftarrow (\GG) & \BB & \leftarrow (\HH,\II)\\
	\DD &\leftarrow (\JJ, \KK) & \EE &\leftarrow () & \FF &\leftarrow (\MM, \NN, \OO) .
\end{align*}

The goal, of course, is to find within the large decomposition universe a
proof tree with $\AA$ as its root, or equivalently, to find within the large
set of combinatorial rules a combinatorial specification (recall, this means
a set in which all of the symbols on a right-hand side appear exactly once
on a left-hand side) that involves $\AA$. It turns out that there is an efficient
algorithm to detect the presence of a combinatorial specification among a set of
combinatorial rules, as we discuss in the next subsection.

Once a combinatorial specification is detected, Combinatorial Exploration
is complete, and that specification can be used to determine whatever enumerative
information is desired; Section~\ref{section:productivity} explores why a
specification found by Combinatorial Exploration is guaranteed to contain
sufficient information to compute the counting sequences of the sets involved,
and Section~\ref{section:transfer-tools} explains how this is actually done.

Because the framework of Combinatorial Exploration is domain-agnostic, it turns
out to be an exceptionally powerful algorithmic approach.
Given any domain of combinatorial sets, e.g., permutation classes
(Section~\ref{section:pp-results}),
alternating sign matrices (Section~\ref{section:asm-results}),
polyominoes (Section~\ref{section:polyomino-results})
and set partitions (Section~\ref{section:set-partitions-results}),
in order to apply the
Combinatorial Exploration framework, one only needs to
\begin{itemize}[label=$\diamond$]
	\item decide how to effectively represent the combinatorial objects, and sets of
		these objects, as a data structure,
	\item implement structural decomposition strategies that discover
	      combinatorial rules,
	\item optionally tune internal parameters, e.g., how often various
	      strategies are applied, in which order, whether depth or breadth is
	      preferred.
\end{itemize}

Section~\ref{section:algorithmic} provides more detailed information
about the inner workings of our implementation of Combinatorial
Exploration.

\subsection{Detecting a Proof Tree In a Decomposition Universe}
\label{subsection:tree-searcher}

The process of Combinatorial Exploration produces a large set $U$ of
combinatorial rules that we call the decomposition universe and that we may
picture as in Figure~\ref{figure:universe}. While every rule in $U$ is a valid
application of a strategy, not every rule is useful. Among the rules in $U$
there may or may not be a subset $U'$ that is a combinatorial specification involving
the root of the universe (the combinatorial set of interest). If
so, and if we can find this subset, then Combinatorial Exploration has
successfully done its job, and every combinatorial set that is a part of $U'$
can be enumerated!\footnote{This is true as long as the combinatorial
	specification has a unique solution, an issue we address in great detail in
	Section~\ref{section:productivity}.}

This raises the question of how a large\footnote{In our applications, $U$ has
	occasionally contained over a hundred million rules.} set $U$ can be
efficiently searched for a subset $U'$ that is a combinatorial specification.
Viewing the decomposition universe simply as a set of rules, rather than as the
graph-like structure in Figure~\ref{figure:universe}, yields a fast and
simple algorithm to accomplish this.

The defining property of a combinatorial specification $U'$ is that every set
appearing on the right-hand side of a rule also appears on the left-hand side of
exactly one rule. It follows that any rule in $U$ that contains a set on the
right-hand side that does not appear on any left-hand side cannot be
contained in any combinatorial specification $U' \subseteq U$.
Algorithm~\ref{algorithm:specfinder} below works by repeatedly removing such
rules from $U$. We prove in Theorem~\ref{theorem:specfinder} that when the algorithm terminates,
the remaining set $V$ has the
property that every rule in $V$ is in a combinatorial specification within
$U$---that is, $V$ is the union of all combinatorial specifications contained
in $U$. Moreover, if the initial combinatorial set $\AA$ that we are trying
to enumerate is one of the left-hand sides in $V$, then there is at least one
subset of $U'$ of $V$ that is a combinatorial specification for $\AA$.
This combinatorial specification $U'$ is the successful output of the
Combinatorial Exploration process.

\begin{algorithm}[H]
	\caption{Combinatorial Specification Searcher}
	\label{algorithm:specfinder}
	\begin{algorithmic}[1]
		\State \textbf{Input:} A set of combinatorial rules $U$
		\State \textbf{Output:} The union of all combinatorial specifications
			contained in $U$
		\State
		\State $\textit{changed} \gets \textbf{True}$
		\While{$\textit{changed}$}
		\State $\textit{changed} \gets \textbf{False}$
		\For{$\AA \xleftarrow{S} (\BB^{(1)}, \ldots, \BB^{(m)}) \in U$}
		\If{any $\BB^{(j)}$ is not on the left-hand side of any rule in $U$}%not in some $\left(\BB^{(j)} \xleftarrow{S'} S\right) \in U$}
		\State $U \gets U \smallsetminus \{\AA \xleftarrow{S} (\BB^{(1)}, \ldots, \BB^{(m)})\}$
		\State $\textit{changed} \gets \textbf{True}$
		\EndIf
		\EndFor
		\EndWhile
		\State $V \gets U$
		\State \Return $V$
	\end{algorithmic}
\end{algorithm}

\begin{theorem}
\label{theorem:specfinder}
	For any set of combinatorial rules $U$, the set $V$ returned by
	Algorithm~\ref{algorithm:specfinder} is equal to the union of all
	combinatorial specifications that are contained in $U$.
\end{theorem}
\begin{proof}
	Let $U$ be a set of combinatorial rules and let $T$ denote the union of all
	combinatorial specifications contained in $U$. Let $V$ be the set output by
	Algorithm~\ref{algorithm:specfinder} when $U$ is given as input. We will show
	that $V = T$.

	To see that $V \subseteq T$, choose a combinatorial rule $R_1 \in V$. We will
	show that there is a combinatorial
	specification $U' \subseteq V$ that contains $R_1$, thereby ensuring
	$R_1 \in T$.

	Start by defining $U' = \{R_1\}$ and repeat the following steps. Pick a
	combinatorial set $\BB$ that is on the right-hand side of some rule $R_2$ in $U'$
	but is not on any left-hand side. If no such set exists, then $U'$ is
	a combinatorial specification, and we're finished. Otherwise,
	we now search for a rule in $V \smallsetminus U'$ that has $\BB$ on its
	left-hand side. There must be some such rule, say $R_3$, because otherwise
	Algorithm~\ref{algorithm:specfinder} would have removed $R_2$ from $V$.
	Add $R_3$ to $U'$, and continue to repeat these steps. Since $U$ is finite,
	this process is guaranteed to finish in finitely many steps, and as mentioned above,
	the resulting $U'$ must be a combinatorial specification. Therefore
	$R_1 \in T$ and the inclusion $V \subseteq T$ is proved.

	To see that $T \subseteq V$, let $R \in T$ be a combinatorial rule. Suppose
	that $R \not\in V$. This implies that at some point in the execution of
	Algorithm~\ref{algorithm:specfinder}, $R$ was removed because there was a
	set on its right-hand side that was not on the left-hand side of any remaining
	rule. Since $R \in T$, there exists a combinatorial specification
	$U' \subseteq U$ that contains $R$, and since $R$ was removed, there must have
	some rule (possibly $R$ itself) that was the first rule in $U'$ removed
	by Algorithm~\ref{algorithm:specfinder}. This is a contradiction, as every
	set on the right-hand side of that first-removed rule must still have been
	on the left-hand side of some rule in $U'$ that had not yet been removed.
\end{proof}

The output of Algorithm~\ref{algorithm:specfinder} has now been proved to be the
union of all combinatorial specifications in the universe $U$ of rules.
Typically, one wants to obtain just a single combinatorial specification, and
Algorithm~\ref{algorithm:specgetter} below describes precisely how to quickly
extract from the union of all combinatorial specifications
a random specification $U'$ involving the combinatorial set of interest.
Algorithm~\ref{algorithm:specgetter} is essentially the procedure used to show
that $V \subseteq T$ in the proof of Theorem~\ref{theorem:specfinder}.
It is also possible to extract, e.g., the smallest specification involving a
particular set, but we do not know a fast way to do this.

\begin{algorithm}[H]
	\caption{Combinatorial Specification Extractor}
	\label{algorithm:specgetter}
	\begin{algorithmic}[1]
		\State \textbf{Input:} A set of combinatorial rules $V$ output from Algorithm~\ref{algorithm:specfinder} and a combinatorial set $\AA$ on the left-hand side of some rule in $V$
		\State \textbf{Output:} A combinatorial specification involving $\AA$
		\State
		\State $\textit{seen} \gets \emptyset$
		\State $\textit{spec} \gets \emptyset$
		\State initialize $queue$
		\State push $\AA$ to $queue$
		\While{$queue$}
		\State pop $\AA'$ from $queue$
		\State choose arbitrarily any one rule of the form $\AA' \xleftarrow{S} (\BB^{(1)}, \ldots, \BB^{(m)}) \in V$
		\State $spec \gets spec \cup \{\AA' \xleftarrow{S} (\BB^{(1)}, \ldots, \BB^{(m)})\}$
		\State $seen \gets seen \cup \{\AA'\}$
		\For{$\BB$ in $\BB^{(1)}, \ldots, \BB^{(m)}$}
		\If{$\BB$ not in $seen$}
		\State push $\BB$ to $queue$
		\EndIf
		\EndFor
		\EndWhile
		\State \Return $spec$
	\end{algorithmic}
\end{algorithm}

Our software to perform Combinatorial Exploration is available on GitHub. The main repository can be found at
\url{https://github.com/PermutaTriangle/comb_spec_searcher}~\cite{comb-spec-searcher}
and a second
repository that implements Combinatorial Exploration for the domain of
permutations (see Section~\ref{section:pp-results}) can be found at
\url{https://github.com/PermutaTriangle/Tilings}~\cite{tilings}.
Despite the simplicity
behind the idea of Combinatorial Exploration, we have had to develop a number of
novel algorithms to produce an efficient implementation.
Section~\ref{section:algorithmic} briefly describes just a few of the computational
challenges that needed to be addressed.

%% ==== %% ==== %% ==== %% ==== %% ==== %% ==== %% ==== %% ==== %% ==== %% ==== %%
%% ==== %% ==== %% ==== %% ====    SECTION FOUR    ==== %% ==== %% ==== %% ==== %%
%% ==== %% ==== %% ==== %% ==== %% ==== %% ==== %% ==== %% ==== %% ==== %% ==== %%

\section{Productive Proof Trees and Combinatorial Specifications}
\label{section:productivity}
%!TEX root = combinatorial-exploration.tex

Although combinatorial specifications have appeared often in the literature,
there does not seem to be any uniform treatment of the inconvenient fact that
some combinatorial specifications convey no enumerative information---they
cannot be used to count the number of objects of size $n$ in each combinatorial
set. We call such a combinatorial specification \emph{trivial}.

Consider, for example, the combinatorial specification
\begin{align*}
	\AA & \xleftarrow{S_1} (\BB,\CC)              \\
	\BB & \xleftarrow{S_2} (\{\varepsilon\}, \CC) \\
	\CC & \xleftarrow{S_3} (\{\}, \BB)
\end{align*}
where $\varepsilon$ is a combinatorial object of size $0$, $S_1$ and $S_3$ are any strategies identifying that $|\AA_n| = |\BB_n| + |\CC_n|$ and $|\CC_n| = |\{\}_n| + |\BB_n|$, and $S_2$ is a strategy  identifying that
\[
	|\BB_n| = \sum_{j=0}^n |\{\varepsilon\}_j||\CC_{n-j}|.
\]
Upon simplification, this system of recurrences becomes
\begin{align*}
	|\AA_n| & = |\BB_n| + |\CC_n| \\
	|\BB_n| & = |\CC_n|           \\
	|\CC_n| & = |\BB_n|,
\end{align*}
which of course cannot be used to calculate any terms in the counting sequences
of the combinatorial sets.

Although triviality is obvious in this simple example, avoiding triviality
becomes more delicate as strategies become more intricate.

One might worry that in order to prove that a combinatorial specification is
\emph{productive} (that is, non-trivial), it could be necessary to examine the
global structure of its corresponding proof tree. Surprisingly, this is not the
case. In this section, we present a set of local conditions that can be
placed on strategies, and we call a strategy satisfying these a \emph{productive
	strategy}. We then prove that when each combinatorial rule in a specification is
generated by a productive strategy, the result is a productive combinatorial
specification. We believe that this general result is, on its own, a significant
new contribution to the literature on combinatorial specifications.

% ================================== %
% ================================== %
% ================================== %

\subsection{Reliance Graphs}
\label{subsection:reliance-graph}

Suppose $P$ is a proof tree whose root is
the combinatorial set $\CC$. In order to compute the size of $\CC_{10}$, we can
imagine ``asking'' the root: ``How many elements do you contain of size $10$?''
Suppose that the children of $\CC$ are $\BB^{(1)}, \ldots, \BB^{(m)}$ derived by
the strategy $S$, i.e., that there is a combinatorial rule
\[
	\CC \xleftarrow{S} (\BB^{(1)}, \ldots, \BB^{(m)}).
\]
In order for the root $\CC$ to ``answer'' our question about the number of
elements of size $10$, the information contained in the strategy $S$ tells it to
ask its children $\BB^{(1)}, \ldots, \BB^{(m)}$ for the number of elements they
each contain of particular sizes (according to the reliance profile function),
and how to combine those numbers to get the answer (according to the counting
functions). To be more explicit, suppose $r_S(10) = (r_S^{(1)}(10), \ldots,
r_S^{(m)}(10))$. Then, our request for $|\CC_{10}|$ leads to requests for
$|\BB^{(1)}_j|$ for $j \in \{0, \ldots, r_S^{(1)}(10)\}$, $|\BB^{(2)}_j|$ for $j
\in \{0, \ldots, r_S^{(2)}(10)\}$, etc., and those quantities are passed as
input into $c_{S,(10)}$, whose output is then $|\CC_{10}|$ as desired.  Each
request for $|\BB^{(i)}_j|$ is handled recursively in the same manner---if
$|\BB^{(i)}_j|$ has already been calculated then the answer is known; if there
is a verification strategy $V_{\BB^{(i)}}$ and the corresponding rule $\BB^{(i)}
\xleftarrow{V_{\BB^{(i)}}} ()$, then $|\BB^{(i)}_j|$ is already known;
otherwise, $\BB^{(i)}$ is decomposed via a strategy into its children, and the
recursion continues.

If this process terminates, then the root $\CC$ is able to answer our question,
and we are told the number of elements in $\CC$ of size $10$. We aim to
characterize when this process terminates, and when it does not. In order to do
so, we extend the notion of reliance profile functions from strategies to 
specifications (and, equivalently, proof trees).
The \emph{reliance graph} $R$ of a combinatorial specification $C$
is an infinite directed
graph defined as follows. For each combinatorial set $\BB^{(i)}$, there
is an infinite family of vertices $\{\BB^{(i)}_n : n \in \N\}$. There is a
directed edge from $\BB^{(i_1)}_{j_1}$ to $\BB^{(i_2)}_{j_2}$ if $\BB^{(i_2)}$
is on the right-hand side of the rule whose left-hand side is $\BB^{(i_1)}$,
and the reliance profile
function $r_S$ dictates that $\BB^{(i_1)}_{j_1}$ relies on $\BB^{(i_2)}_{j_2}$,
i.e., if $\BB^{(i_2)}$ is the $\ell$th element of $d_S(\BB^{(i_1)})$, then $j_2
\leq r_S^{(\ell)}(j_1)$. Informally a directed edge from $\BB^{(i_1)}_{j_1}$ to
$\BB^{(i_2)}_{j_2}$ conveys that in order to find the count $|\BB^{(i)}_{j_1}|$
it is required to first calculate $|\BB^{(i_2)}_{j_2}|$.

\definecolor{superlightgray}{rgb}{0.85, 0.85, 0.85}

We take this opportunity to examine a couple of examples, first looking once
again at the specification for $\Av(132)$ on page~\pageref{132-specification}
using strategies $Z$, $F$, $V_{\{\varepsilon\}}$, and $V_{\{1\}}$ described in
Subsection~\ref{subsection:comb-specs}.

The strategy $Z$ applied to the root carries the reliance profile function $n
	\mapsto (n, n)$, the strategy $F$ carries the reliance profile function $n
	\mapsto(n-1, n-1)$, while the other two strategies are verification
	strategies and carry the reliance profile function $n \mapsto ()$. An
	abbreviated portion of the reliance graph for this specification is shown
	on the left in
	Figure~\ref{figure:reliance-graph-examples}. In this graph, each vertex
	$\CC_n$ is the source of a directed edge to each of $\{\varepsilon\}_i$ and
	$\DD_i$ for $i \leq n$, each vertex $\DD_n$ is the source of a directed edge
	to each of $\CC_i$ for $i \leq n-1$, and each vertex $\{\varepsilon\}_n$ is
	the source of no directed edges (being the parent of a verification
	strategy). In an attempt to make the graph more readable, we have drawn edges
	in gray whose reliance is not actually used in the counting formulas,
	e.g., although $|\DD_1|$ is technically an input into the counting formula
	for $|\CC_2|$, it is not used as part of the computation in that counting
	formula.

\begin{figure}
	\centering
	\ \hfill
	\begin{tikzpicture}

	\node[circle, fill, draw, label={[label distance=-0.2cm]265:$\{\varepsilon\}_2$}] (e2) at (0,0) {};
	\node[circle, fill, draw, label={[label distance=-0.2cm]265:$\{\varepsilon\}_1$}] (e1) at (1.5,0) {};
	\node[circle, fill, draw, label={[label distance=-0.2cm]265:$\{\varepsilon\}_0$}] (e0) at (3,0) {};
	
	\node[circle, fill, draw, label={[label distance=0cm]265:$\CC_2$}] (c2) at (0,1.5) {};
	\node[circle, fill, draw, label={[label distance=0cm]265:$\CC_1$}] (c1) at (1.5,1.5) {};
	\node[circle, fill, draw, label={[label distance=0cm]265:$\CC_0$}] (c0) at (3,1.5) {};
	
	\node[circle, fill, draw, label={[label distance=0cm]265:$\DD_2$}] (d2) at (0,3) {};
	\node[circle, fill, draw, label={[label distance=0cm]265:$\DD_1$}] (d1) at (1.5,3) {};
	\node[circle, fill, draw, label={[label distance=0cm]265:$\DD_0$}] (d0) at (3,3) {};
	
	\draw[-latex, ultra thick, shorten >=0.8, shorten <=0.8, superlightgray] (c2) to[out=-65, in=160] (e1);
	\draw[-latex, ultra thick, shorten >=0.8, shorten <=0.8, superlightgray] (c2) to[out=-55, in=165] (e0);
	\draw[-latex, ultra thick, shorten >=0.8, shorten <=0.8, superlightgray] (c1) to[out=-35, in=130] (e0);
	
	\draw[-latex, ultra thick, shorten >=0.8, shorten <=0.8, superlightgray] (c2) to[out=65, in=200] (d1);
	\draw[-latex, ultra thick, shorten >=0.8, shorten <=0.8, superlightgray] (c2) to[out=55, in=195] (d0);
	\draw[-latex, ultra thick, shorten >=0.8, shorten <=0.8, superlightgray] (c1) to[out=35, in=230] (d0);
	
	% Redraw so it's over the gray arrows
	\node[circle, fill, draw, label={[label distance=0cm]265:$\DD_1$}] (d1) at (1.5,3) {};
	\node[circle, fill, draw, label={[label distance=0cm]265:$\DD_0$}] (d0) at (3,3) {};
	% end redraw
	
	\draw[-latex, ultra thick, shorten >=0.8, shorten <=0.8] (c2) -- (d2);
	\draw[-latex, ultra thick, shorten >=0.8, shorten <=0.8] (c1) -- (d1);
	\draw[-latex, ultra thick, shorten >=0.8, shorten <=0.8] (c0) -- (d0);
	\draw[-latex, ultra thick, shorten >=0.8, shorten <=0.8] (c2) -- (e2);
	\draw[-latex, ultra thick, shorten >=0.8, shorten <=0.8] (c1) -- (e1);
	\draw[-latex, ultra thick, shorten >=0.8, shorten <=0.8] (c0) -- (e0);
	
	\draw[-latex, ultra thick, shorten >=0.8, shorten <=0.8] (d2) to[out=-65, in=160] (c1);
	\draw[-latex, ultra thick, shorten >=0.8, shorten <=0.8] (d2) to[out=-55, in=165] (c0);
	\draw[-latex, ultra thick, shorten >=0.8, shorten <=0.8] (d1) to[out=-35, in=130] (c0);

	\node at (-0.7, 0) {\large $\cdots$};
	\node at (-0.7, 1.5) {\large $\cdots$};
	\node at (-0.7, 3.0) {\large $\cdots$};
\end{tikzpicture}
	\hfill
	\begin{tikzpicture}
	\node[circle, fill, draw, label={[label distance=0cm]265:$\CC_2$}] (c2) at (0,0) {};
	\node[circle, fill, draw, label={[label distance=0cm]265:$\CC_1$}] (c1) at (1.5,0) {};
	\node[circle, fill, draw, label={[label distance=0cm]265:$\CC_0$}] (c0) at (3,0) {};
	
	\node[circle, fill, draw, label={[label distance=0cm]200:$\BB_2$}] (b2) at (0,1.5) {};
	\node[circle, fill, draw, label={[label distance=0cm]200:$\BB_1$}] (b1) at (1.5,1.5) {};
	\node[circle, fill, draw, label={[label distance=0cm]200:$\BB_0$}] (b0) at (3,1.5) {};
	
	\node[circle, fill, draw, label={[label distance=0cm]265:$\AA_2$}] (a2) at (0,3) {};
	\node[circle, fill, draw, label={[label distance=0cm]265:$\AA_1$}] (a1) at (1.5,3) {};
	\node[circle, fill, draw, label={[label distance=0cm]265:$\AA_0$}] (a0) at (3,3) {};

	\draw[-latex, ultra thick, shorten >=0.8, shorten <=0.8, superlightgray] (a2) to[out=-65, in=160] (b1);
	\draw[-latex, ultra thick, shorten >=0.8, shorten <=0.8, superlightgray] (a2) to[out=-55, in=165] (b0);
	\draw[-latex, ultra thick, shorten >=0.8, shorten <=0.8, superlightgray] (a1) to[out=-35, in=130] (b0);
	
	\draw[-latex, ultra thick, shorten >=0.8, shorten <=0.8, superlightgray] (b2) to[out=-65, in=160] (c1);
	\draw[-latex, ultra thick, shorten >=0.8, shorten <=0.8, superlightgray] (b2) to[out=-55, in=165] (c0);
	\draw[-latex, ultra thick, shorten >=0.8, shorten <=0.8, superlightgray] (b1) to[out=-35, in=130] (c0);
	
	\draw[-latex, ultra thick, shorten >=0.8, shorten <=0.8, superlightgray] (c2) to[out=65, in=200] (b1);
	\draw[-latex, ultra thick, shorten >=0.8, shorten <=0.8, superlightgray] (c2) to[out=55, in=195] (b0);
	\draw[-latex, ultra thick, shorten >=0.8, shorten <=0.8, superlightgray] (c1) to[out=35, in=230] (b0);
	
	\draw[-latex, ultra thick, shorten >=0.8, shorten <=0.8, superlightgray] (a2) -- (c1);
	\draw[-latex, ultra thick, shorten >=0.8, shorten <=0.8, superlightgray] (a1) -- (c0);
	\draw[-latex, ultra thick, shorten >=0.8, shorten <=0.8, superlightgray] (a2) to[out=-70, in=160] (c0);
	
	% redraw
	\node[circle, fill, draw, label={[label distance=0cm]200:$\BB_1$}] (b1) at (1.5,1.5) {};
	\node[circle, fill, draw, label={[label distance=0cm]200:$\BB_0$}] (b0) at (3,1.5) {};
	% end

	\draw[-latex, ultra thick, shorten >=0.8, shorten <=0.8] (a2) -- (b2);
	\draw[-latex, ultra thick, shorten >=0.8, shorten <=0.8] (a1) -- (b1);
	\draw[-latex, ultra thick, shorten >=0.8, shorten <=0.8] (a0) -- (b0);
	
	\draw[-latex, ultra thick, shorten >=0.8, shorten <=0.8] (c2) to[out=70, in=-70] (b2);
	\draw[-latex, ultra thick, shorten >=0.8, shorten <=0.8] (b2) to[out=-110, in=110] (c2);
	
	\draw[-latex, ultra thick, shorten >=0.8, shorten <=0.8] (c1) to[out=70, in=-70] (b1);
	\draw[-latex, ultra thick, shorten >=0.8, shorten <=0.8] (b1) to[out=-110, in=110] (c1);
	
	\draw[-latex, ultra thick, shorten >=0.8, shorten <=0.8] (c0) to[out=70, in=-70] (b0);
	\draw[-latex, ultra thick, shorten >=0.8, shorten <=0.8] (b0) to[out=-110, in=110] (c0);
	
	\draw[-latex, ultra thick, shorten >=0.8, shorten <=0.8] (a2) to[out=-65, in=45] (c2);
	\draw[-latex, ultra thick, shorten >=0.8, shorten <=0.8] (a1) to[out=-65, in=45] (c1);
	\draw[-latex, ultra thick, shorten >=0.8, shorten <=0.8] (a0) to[out=-65, in=45] (c0);
	
	\node at (-0.7, 0) {\large $\cdots$};
	\node at (-0.7, 1.5) {\large $\cdots$};
	\node at (-0.7, 3.0) {\large $\cdots$};
\end{tikzpicture}
	\hfill\ 
	\caption{Partial depictions of two reliance graphs.}
	\label{figure:reliance-graph-examples}
\end{figure}
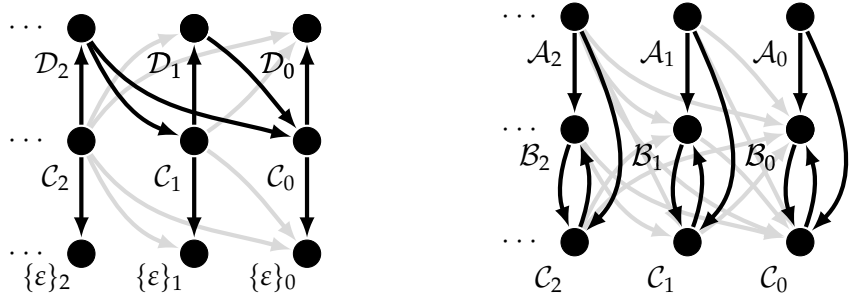

On the other hand, the trivial combinatorial
specification at the beginning of this section has the
reliance graph that is partially shown on the right
in Figure~\ref{figure:reliance-graph-examples}.
It is clear from this reliance graph that the recursive enumeration procedure
that we have described fails. When the vertex $\BB$ is asked for the number of its
elements of size $2$, it poses the same question to $\CC$, which in turns poses
the same question to $\BB$, \emph{ad infinitum}.

The characteristic of a reliance graph that determines whether our recursive
enumeration procedure is effective is the absence or presence of infinite
directed walks; the reliance graph on the right
in Figure~\ref{figure:reliance-graph-examples}
possesses infinite walks (e.g., $\AA_2 \to \BB_2 \to \CC_2 \to \BB_2 \to \CC_2
\to \cdots$), while the reliance graph on the left
does not. In this subsection, we prove
that if the reliance graph of a combinatorial specification has no infinite
walks, then the combinatorial specification uniquely determines the counting
sequences of all sets involved.

We start by discussing what it means for a combinatorial specification to
uniquely determine its counting sequences. For each combinatorial rule $\AA
\xleftarrow{S} (\BB^{(1)}, \ldots, \BB^{(m)})$ and each $n \in \N$ we can obtain
a symbolic equation relating $|\AA_n|$ with the various quantities
$|\BB^{(i)}_j|$ dictated by the reliance profile function of $S$. It is
convenient here to take the ``combinatorial specification'' perspective instead
of the ``proof tree'' perspective. Let $P$ be a combinatorial specification
involving $N$ combinatorial sets $\BB^{(1)}, \ldots, \BB^{(N)}$. This implies
that there are $N$ combinatorial rules:
\begin{align*}
	\BB^{(1)} & \xleftarrow{S_1} (\BB^{(i_{1,1})}, \BB^{(i_{1,2})}, \ldots, \BB^{(i_{1,m_1})})  \\
	\BB^{(2)} & \xleftarrow{S_2} (\BB^{(i_{2,1})}, \BB^{(i_{2,2})}, \ldots, \BB^{(i_{2,m_2})})  \\
	          & \; \vdots                                                                      \\
	\BB^{(N)} & \xleftarrow{S_N} (\BB^{(i_{N,1})}, \BB^{(i_{N,2})}, \ldots, \BB^{(i_{N,m_N})}).
\end{align*}
To each of these rules we associate an infinite sequence of equations, one for
each $n \in \N$, derived from the corresponding counting function:
\[
	b^{(j)}_n = c_{S_j,(n)}\left(\left(b^{(i_{j,1})}_0, b^{(i_{j,1})}_1, \ldots, b^{(i_{j,1})}_{r_{S_j}^{(1)}(n)}\right), \ldots, \left(b^{(i_{j,m_j})}_0, b^{(i_{j,m_j})}_1, \ldots, b^{(i_{j,m_j})}_{r_{S_j}^{(m_j)}(n)}\right)\right).
\]

For example, the strategy $Z$ from Example~\ref{example:size-0-or-not}, when
producing a combinatorial rule $\BB^{(1)} \xleftarrow{Z} (\BB^{(2)},\BB^{(3)})$
produces symbolic equations
\[
	b^{(1)}_0 = b^{(2)}_0 + b^{(3)}_0, \qquad b^{(1)}_1 = b^{(2)}_1 + b^{(3)}_1, \qquad b^{(1)}_2 = b^{(2)}_2 + b^{(3)}_2, \qquad \ldots.
\]
The strategy $F$ from Example~\ref{example:factor-around-max-entry}, when
producing a combinatorial rule $\BB^{(1)} \xleftarrow{F} (\BB^{(2)}, \BB^{(3)})$
produces symbolic equations
\[
	b^{(1)}_0 = 0, \qquad b^{(1)}_1 = b^{(2)}_0b^{(3)}_0, \qquad b^{(1)}_2 = b^{(2)}_0b^{(3)}_1 + b^{(2)}_1b^{(3)}_0, \qquad \ldots.
\]
For a verification strategy, the right-hand sides of the symbolic equations are
just constants. For instance, the verification strategy $V_{\varepsilon}$ used
in the combinatorial specification for $\Av(132)$ produces symbolic equations
\[
	|\{\varepsilon\}_0| = 1, \qquad |\{\varepsilon\}_1| = 0, \qquad |\{\varepsilon\}_2| = 0, \qquad \ldots.
\]

We are now ready to state an important definition that will be the focus of
the remainder of this section.

\begin{definition}
	A proof tree involving $N$ combinatorial sets and its corresponding
	combinatorial specification are called \emph{productive} if the infinite
	system of equations produced by the counting functions has a unique solution
	in $\left(\C^{\N}\right)^N$.
\end{definition}

Productivity is clearly the property that one wants a combinatorial
specification to possess, otherwise the specification may not be useful.
The following theorem shows that the productivity of a specification can be
checked by examining its reliance graph.

\begin{theorem}
	\label{theorem:unique-counting-series-solution}
	Let $P$ be a proof tree (or the corresponding specification) 	involving combinatorial sets $\BB^{(1)}, \BB^{(2)},
		\ldots, \BB^{(N)}$ and whose reliance graph contains no infinite
		directed walks. Let $\SSS(P)$ be the system of equations in the
		indeterminates $\{b^{(i)}_j \,:\, j \in \N, \; 1 \leq i \leq N\}$.
		There exists a unique solution to the system
	\[
		\left(\left(\wt{b}_0^{(1)},\wt{b}_1^{(1)},\ldots\right),\left(\wt{b}_0^{(2)},\wt{b}_1^{(2)},\ldots\right), \ldots, \left(\wt{b}_0^{(N)},\wt{b}_1^{(N)},\ldots\right)\right) \in \left( \C^\N \right)^N.
	\]
	In other words, $P$ is a productive proof tree. 
\end{theorem}
\begin{proof}
	By construction, one solution to the system $\SSS(P)$ is the one which
	enumerates the combinatorial sets in the proof tree. We call this solution
	\[
		\wt{b} = \left(\left(\wt{b}_0^{(1)},\wt{b}_1^{(1)},\ldots\right),\left(\wt{b}_0^{(2)},\wt{b}_1^{(2)},\ldots\right), \ldots, \left(\wt{b}_0^{(N)},\wt{b}_1^{(N)},\ldots\right)\right).
	\]
	Suppose there were a different solution
	\[
		\wt{d} = \left(\left(\wt{d}_0^{(1)},\wt{d}_1^{(1)},\ldots\right),\left(\wt{d}_0^{(2)},\wt{d}_1^{(2)},\ldots\right), \ldots, \left(\wt{d}_0^{(N)},\wt{d}_1^{(N)},\ldots\right)\right).
	\]
	Let $R$ be the reliance graph of $P$. For any vertex $v$ in $R$,
	define $\rank(v)$ to be the length of
	the longest walk in $R$ that starts at $v$.  
	
	We claim that $\rank(v)$ exists and is finite. To justify this,
	we must rule out the possibility that there are infinitely many
	walks that start at $v$ with lengths that are arbitrarily large,
	but finite. The outdegree of any vertex in a reliance graph is
	finite, and so if there were infinitely many walks that started at
	$v_1$, then there would be at least one child of $v_1$, say $v_2$
	from which infinitely many walks started. Carrying on in this
	way, we can form a walk $v_1, v_2, v_3, \ldots$, which has infinite
	length, contradicting our assumption that $R$ has no infinite walks.
	Therefore, finitely many walks start from $v$, and thus $\rank(v)$
	is finite.
	
	Note that each indeterminate $b_j^{(i)}$ corresponds to one
	vertex in $R$, namely the one labeled $\BB_j^{(i)}$. The vertices
	that are leaves in the reliance graph are precisely those for which
	$\BB^{(i)}$ is on the
	left-hand side of a verification rule, and the corresponding
	equation in $\SSS(P)$ with $b_j^{(i)}$ on the left-hand side has
	as its right-hand side an explicit natural number.
	
	Now we consider the two distinct solutions $\wt{b}$ and $\wt{d}$.
	Among all pairs $\wt{b}_j^{(i)}, \wt{d}_j^{(i)}$ where the solutions 
	differ (that is, $\wt{b}_j^{(i)} \neq \wt{d}_j^{(i)}$), choose one
	pair $\wt{b}_J^{(I)}, \wt{d}_J^{(I)}$ that is minimal with respect
	to the rank of the corresponding vertex in $R$ (the one labeled
	$\BB_J^{(I)}$). That minimal rank is guaranteed to be at least
	one, because the rank zero vertices are leaves, and by the
	discussion above it is clear that the solutions $\wt{b}$ and
	$\wt{d}$ must agree at all indeterminates corresponding to leaves.
	
	There is a single equation in $\SSS(P)$ with $b_J^{(I)}$ on the 
	left-hand side. The right-hand side involves
	some nonempty set of variables of the form $b_i^{(j)}$
	(we will call these
	variables the \emph{children} of $b_I^{(J)}$ for the moment).
	The vertices in $R$ that are labeled by these children (the
	\emph{child vertices}) are precisely the ones at the other end of 
	all outgoing edges from the vertex labeled $\BB_J^{(I)}$, and
	therefore the rank of each child vertex must be strictly less than
	the rank of their parent vertex $\BB_J^{(I)}$.
	
	By our minimality assumption, it follows that the solutions
	$\wt{b}$ and $\wt{d}$ agree for each of the child variables. But
	this cannot be, since $\wt{b}_J^{(I)}$ and $\wt{d}_J^{(I)}$ are
	each computed from the values of these child indeterminates in the same
	manner, and thus this would imply $\wt{b}_J^{(I)} = \wt{d}_J^{(I)}$.
\end{proof}

% ================================== %
% ================================== %
% ================================== %

\subsection{Productive Strategies}

We have just proved that any specification whose reliance graph contains no
infinite directed walks is productive; that is, it contains sufficient
information to uniquely determine the counting sequences of all of its
combinatorial sets. It is natural to now ask how one can determine when the
reliance graph of a specification contains no infinite directed walks.

Given any particular concrete combinatorial specification, it can surely be
determined by hand (though perhaps with great effort for large specifications)
whether the reliance graph contains infinite directed walks.   However, in order
for Combinatorial Exploration to be effective, we want to be certain ahead of
time that any specification produced will be productive. This is not just to
avoid the need to check an output specification, but also because the
efficient method we described in Section~\ref{subsection:tree-searcher} is not
capable of simply ``ignoring'' an unproductive specification and continuing to
search for a productive one.

It is perhaps surprising that it is possible to place local restrictions on
individual strategies that then guarantee that any specification that employs
only those strategies is itself productive. To accomplish this, we introduce the
concept of a \emph{productive strategy}.

\begin{definition}
	\label{definition:productive-strategy}
	We call an $m$-ary strategy $S$ a \emph{productive strategy} if the following
	two conditions hold for all combinatorial sets $\AA$ with corresponding
	decomposition $d_S(\AA) = (\BB^{(1)}, \ldots, \BB^{(m)})$, and for all $i
	\in \{1,\ldots,m\}$.
	\begin{enumerate}
		\item For all $N \in \N$, if $\AA_N$ relies on $\BB_j^{(i)}$,
				then $j \leq N$.
		\item  If $\AA_N$ relies on $\BB_N^{(i)}$ for some $N \in \N$, then
		      \begin{enumerate}[(a)]
			      \item $|\AA_n| \geq |\BB_n^{(i)}|$ for all $n \in \N$, and
			      \item $|\AA_\ell| > |\BB_\ell^{(i)}|$ for some $\ell \in \N$.
		      \end{enumerate}
	\end{enumerate}
	As before, we use the phrase ``$\AA_n$ relies on $\BB^{(i)}_j$'' or the
	diagram $\AA_n \to \BB^{(i)}_j$ as a simplified way of stating the formal
	information contained in the reliance profile function of $S$, namely that
	$j \leq r_S^{(i)}(n)$.
\end{definition}

In other words, each reliance either moves from a larger object
size to a smaller object size ($N > j$), in which case $|\BB^{(i)}_j|$ can be
smaller, larger, or the same size as $|\AA_N|$, or stays at the same size of
object ($N=j$), in which case the child set cannot be larger than the parent
at \emph{any} size, and must actually be properly smaller for at least one
size.

Although these conditions may at first seem restrictive, it turns out that they
are satisfied by many natural strategies. We now prove that the reliance graph
of any proof tree that is composed entirely of productive strategies has no
infinite directed walks.

\begin{theorem}
	\label{theorem:productive-strats-implies-no-inf-walks}
	Let $P$ be a proof tree, or the equivalent combinatorial specification,
	composed entirely of rules derived from productive strategies. Then the reliance graph of
	$P$ has no infinite
	directed walks.
\end{theorem}
\begin{proof}
	Let $R$ be the reliance graph of $P$. Suppose toward a contradiction that
	there is an infinite walk $W$. The set of vertices involved in $W$ is
	finite---indeed, the only vertices that can be involved in a walk $W$ 
	that starts at vertex $\AA_\ell$ are those of
	the form $\DD_j$ where $\DD$ is any combinatorial set involved in $P$ and $j
	\leq \ell$. As $W$ is an infinite walk but involves only a finite set of
	vertices, $W$ must contain a cycle
	\[
		\AA_N \to \EE^{(1)}_{N_1} \to \EE^{(2)}_{N_2} \to \cdots \to \EE^{(m)}_{N_m} \to \AA_N.
	\]
	Condition 1 further implies that the sizes of objects involved in this cycle
	are all the same, i.e., $N = N_1 = N_2 = \cdots = N_m$. Hence,
	\[
		\AA_N \to \EE^{(1)}_N \to \EE^{(2)}_N \to \cdots \to \EE^{(m)}_N \to \AA_N.
	\]
	Since we cannot have $\AA_N \to \AA_N$ (this would fail Condition 2(b)), we
	must have $m \geq 1$.

	It now follows from Condition 2(a) that
	\[
		|\AA_n| \geq |\EE^{(1)}_n| \geq |\EE^{(2)}_n| \geq \cdots \geq |\EE^{(m)}_n| \geq |\AA_n|
	\]
	and thus
	\[
		|\AA_n| = |\EE^{(1)}_n| = |\EE^{(2)}_n| = \cdots = |\EE^{(m)}_n| = |\AA_n|
	\]
	for all $n \in \N$. As a result $\AA$ is equinumerous to $\EE^{(1)}$,
	contradicting Condition 2(b).
\end{proof}

Combining Theorems~\ref{theorem:unique-counting-series-solution}
and~\ref{theorem:productive-strats-implies-no-inf-walks} gives the following key
result.

\begin{theorem}
	\label{theorem:prod-strats-implies-prod-spec}
	Let $P$ be a proof tree, or the equivalent combinatorial specification,
	composed entirely of rules derived from productive strategies. Then $P$ is productive, i.e.,
	the infinite system of equations derived from its counting functions
	has a unique solution.
\end{theorem}

Furthermore, the discussion at the beginning of
Section~\ref{subsection:reliance-graph} demonstrates how the unique counting
sequences for a productive specification can actually be computed.

% ================================== %
% ================================== %
% ================================== %

\subsection{Equivalence Strategies}
\label{subsection:equivalencestrategies}

The formalization of productive strategies and proof trees has one significant
deficiency. An \emph{equivalence strategy} is any $1$-ary strategy that
identifies that two combinatorial sets are equinumerous. More precisely, $S$ is
an equivalence strategy if whenever $d_S(\AA) = \BB$, we have $|\AA_n| =
|\BB_n|$ for all $n$. The reliance profile function for any equivalence strategy
can be assumed to be $r_S(n) = (n)$.

The following sections demonstrate that equivalence strategies are
powerful tools for identifying structural decompositions of combinatorial sets, yet
according to our definition they are clearly not productive strategies---they
always fail Condition 2 of~\ref{definition:productive-strategy}.

It is easy to see why equivalence strategies lead to non-productive proof trees.
Suppose some equivalence strategy $S$ produces the rule $\AA \xleftarrow{S} \BB$
and some other equivalence strategy $T$ produces the rule $\BB \xleftarrow{T}
\AA$. These two rules together form a combinatorial specification (because every
set on a right-hand side appears exactly once on a left-hand side), although all
we can conclude is that $|\AA_n| = |\BB_n|$. There is not sufficient information
to compute the counting sequence for either $\AA$ or $\BB$.

One possible solution to this issue is to ensure that whenever a rule of the
form $\AA \xleftarrow{S} \BB$ for an equivalence strategy $S$ exists in a proof
tree, no rule $\BB \xleftarrow{} \AA$ is permitted. This is impractical for
several reasons. First, while Combinatorial Exploration is expanding the
universe of combinatorial rules, it is not clear which direction of a rule will
prove to be more useful when searching for a proof tree. More importantly,
significant power is lost by ignoring that the sets are equinumerous
and thus that their counting sequences can be substituted for each other as
needed.

There is another solution that is simpler and yet preserves all information
provided by the equivalence strategies. Up to this point, we have defined a
combinatorial rule in a specification (equivalently, a parent-children
relationship in a proof tree) to be of the form
\[
	\AA \xleftarrow{S} (\BB^{(1)}, \ldots, \BB^{(m)}),
\]
where $\AA$ and the $\BB^{(i)}$ are combinatorial sets. To generalize this, we
sort all combinatorial sets involved into equivalence classes in the following
way. First, we say that $\AA$ is related to $\BB$ if $\AA = \BB$ or if there is
an equivalence strategy $S$ that either produces the rule $\AA \xleftarrow{S}
\BB$ or the rule $\BB \xleftarrow{S} \AA$. Then we take the transitive closure
of this relation to form the equivalence relation $\eqrel$. Two combinatorial
sets in the same equivalence class of $\eqrel$ are guaranteed to be
equinumerous, but it is not necessarily true that any two equinumerous
combinatorial sets are in the same equivalence class.

To mitigate the productivity issue caused by equivalence strategies, we
henceforth consider the vertices of a proof tree to represent not combinatorial
sets, but equivalence classes of combinatorial sets from $\eqrel$.\footnote{An
important feature of the Combinatorial Exploration framework is that as
strategies are repeatedly applied, we may continue to add new sets to
equivalence classes or merge two existing equivalence classes together.}
Equivalently, the symbols on the left- and right-hand sides of combinatorial
rules do not represent combinatorial sets, but equivalence classes of
combinatorial sets. The combinatorial rules produced by equivalence strategies
themselves are not added to the ever-growing set of combinatorial rules, since
they are not productive.

Often the identification of two sets as equinumerous can lead to the discovery
of proof trees much smaller than might otherwise be found. Many specific
examples are seen in later sections, but we give a generic example here. Suppose
that while in the process of Combinatorial Exploration we have produced the
following rules (in which the symbols on each side are still single
combinatorial sets):
\begin{align*}
	\AA & \xleftarrow{S} (\BB,\CC) \\
	\BB & \xleftarrow{V_{\BB}} ()  \\
	\DD & \xleftarrow{E} (\CC)     \\
	\DD & \xleftarrow{T} (\AA,\BB)
\end{align*}
where $S$ and $T$ are productive strategies, $V_{\BB}$ is a verification
strategy, and $E$ is an equivalence strategy. This set of rules does not contain
a combinatorial specification, let alone a productive one.

However, when grouping the combinatorial sets into equivalence classes as
described above, we obtain the productive combinatorial specification
\begin{align*}
	\{\AA\}     & \xleftarrow{S} (\{\BB\}, \{\CC,\DD\}) \\
	\{\BB\}     & \xleftarrow{V_{\BB}} ()               \\
	\{\CC,\DD\} & \xleftarrow{T} (\{\AA\}, \{\BB\}).
\end{align*}

%The use of equivalence strategies in the domain of permutation patterns can be
%seen in Figures~\ref{fig:placed-point-from-large-tree},
%\ref{fig:placed-point-from-large-tree-12} and~\ref{fig:bigsep}, and how they
%contribute to a proof tree in Figure~\ref{fig:av1243_1342_2143}.

%% ==== %% ==== %% ==== %% ==== %% ==== %% ==== %% ==== %% ==== %% ==== %% ==== %%
%% ==== %% ==== %% ==== %% ====    SECTION FIVE    ==== %% ==== %% ==== %% ==== %%
%% ==== %% ==== %% ==== %% ==== %% ==== %% ==== %% ==== %% ==== %% ==== %% ==== %%

\section{Transfer Tools}
\label{section:transfer-tools}
%!TEX root = combinatorial-exploration.tex

Finding a combinatorial specification for a set is often an intermediate step in
the quest to fully understand the set. This brief section outlines several ways
in which a combinatorial specification can be transferred into other
products including polynomial-time counting algorithms, systems of
equations for the generating function, and uniform random sampling routines.

\subsection{Polynomial-Time Counting Algorithms}
\label{subsection:poly-time-counting}

A polynomial-time counting algorithm for a combinatorial set $\AA$ is a routine
that can calculate $|\AA_n|$ in polynomially-many (in $n$) arithmetic
operations. Which operations are included in this definition can vary according
to one's preferences or needs; here we choose addition, subtraction,
multiplication, and division. Brute force computation of $|\AA_n|$ takes
exponential time or more in many applications, but often possession of a
combinatorial specification enables fast computation of many terms in the
counting sequence.

A combinatorial strategy $S$ is called a \emph{polynomial-time strategy} if, for
all $n \in \N$, computing the output of the counting function $c_{S,(n)}$
requires only polynomially-many (in $n$) arithmetic operations. All strategies
discussed in this article are polynomial-time strategies. 

If $C$ is a combinatorial specification made up entirely of productive
polynomial-time strategies, then the
routine discussed in Section~\ref{subsection:reliance-graph} for computing terms
in the counting sequence of any set involved in $C$ takes only polynomially-many
arithmetic operations.

\subsection{Generating Functions}
\label{subsection:generating-functions}

Possession of a combinatorial specification also often allows one to write down
a system of equations satisfied by the generating functions of each of the
combinatorial sets involved. These systems can then sometimes be solved to yield
the generating function for the specific set of interest.

To produce such a system of equations, one must associate to each strategy $S$
used a function $g_S$ that describes how the generating functions of each input
set can be combined to produce the generating function of the output set. For
example, suppose $S$ is a strategy with decomposition function $d_S(\AA) =
(\BB^{(1)}, \BB^{(2)})$ and counting functions $a_n = b^{(1)}_n + b^{(2)}_n$,
then we can assign $g_S(B_1, B_2) = B_1 + B_2$. Now, if $A(x)$, $B^{(1)}(x)$,
and $B^{(2)}(x)$ are the generating functions of $\AA$, $\BB^{(1)}$, and
$\BB^{(2)}$, then we have the equality
\[
	A(x) = g_S(B^{(1)}(x), B^{(2)}(x)).
\]
Another common example is a binary strategy $S$ with counting function $a_n =
\ds\sum_{k=0}^n b^{(1)}_kb^{(2)}_{n-k}$, which corresponds to the equation
$g_S(B_1, B_2) = B_1B_2$.

If $V$ is a verification strategy, then we consider $g_V$ to simply have no
inputs, while the output is the generating function for any combinatorial set
that $V$ verifies (recall that for a fixed verification strategy $V$, all sets
to which it applies must, by definition, have the same counting sequence).

In this introductory work, many of the strategies we describe lead to generating
function equations that are either sums, products, or explicit generating
functions. The reader should not get the impression that this is a limitation of
the Combinatorial Exploration framework. It is simply a result of the fact that
we have chosen to defer the introduction of more complicated strategies, and a
multivariate generalization of the entire framework, to future work.

Still, we cannot resist showing a
morsel here. It is often desirable to consider combinatorial sets indexed
not just by size $n$, but by some additional statistic $k$. For example,
$\AA_{n,k}$ could be the set of binary words of length $n$ that contain exactly
$k$ occurrences of the symbol $1$. In many cases, these additional statistics are actually required in
order to find a specification at all. Accordingly, the whole theory of
strategies (particularly, their reliance profile functions and counting
functions), the concepts of reliance graphs and productivity, and the transfer
tools described in this section all require suitable generalization. The result
is that systems of equations now involve multiple variables in complicated ways,
with equations such as
\[
	A(x,y) = \frac{yB(x,y) - B(x,1)}{y-1}
\]
and
\[
	C(x,w,y,z) = D\left(x,\frac{wy}{z},w\right),
\]
leading to systems that either require advanced methods to solve (e.g.,
extensions of the ideas of Bousquet-M\'elou and
Jehanne~\cite{bousquet-melou:poly-eqs}) or that we do not
know how to solve at all!

\subsection{Random Sampling}
\label{subsection:random-sampling}

It is computationally difficult in general to sample an object of size $n$ from
a combinatorial set $\AA$ uniformly at random because the brute force method of
doing so requires computing $\AA_n$ explicitly. However, the reader will not be
surprised that having a combinatorial specification involving the set $\AA$
often leads to random sampling routines that can be performed in polynomial
time. Although our implementation of Combinatorial
Exploration~\cite{comb-spec-searcher} is already capable of performing
random sampling, we defer the details, including which properties that
strategies in a specification must have for this to be possible, to a
later article.

Figure~\ref{figure:heatmap} demonstrates this ability by showing heatmaps
for all classes up to symmetry defined by avoiding two patterns of length
4 (except for $\Av(1234, 4321)$, which contains a finite number of
permutations), and three of the seven symmetry classes defined by
avoiding one pattern of length $4$. Each heatmap is formed by sampling
one million permutations of length $300$ uniformly at random from the
class and forming a picture 300 pixels wide and tall such that the
darkness of the pixel at location $(i,j)$ corresponds to how many of the
one million sampled permutations have an entry at index $i$ and value $j$ 
(i.e.,\ $\pi(i) = j$). Lighter pixels indicated fewer such entries, and
darker, more. Higher resolution pictures can be viewed on the pages for
each corresponding class on the PermPAL website~\cite{permpal-biblatex}.

Less formally, these pictures can be thought of as taking
the one million sampled permutations, forming their plots as described in
Section~\ref{section:pp-intro}, making them mostly transparent, and
stacking them on top of each other. In this way, a heatmap conveys
information about what permutations in a particular class tend to
``look like''. This concept has been studied much more formally by
Miner and Pak~\cite{miner:shapes-with-pak}, Hoffman, Rizzolo, and
Slivken~\cite{hoffman:shapes}, and many others.

\begin{figure}
	\begin{center}
	\begin{tabular}{cccccc}
		\makecell{\frame{\includegraphics[width=0.6in]{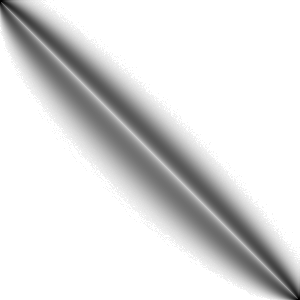}} \\[-4pt]\footnotesize$\Av(1234, 1243)$}
		&
		\makecell{\frame{\includegraphics[width=0.6in]{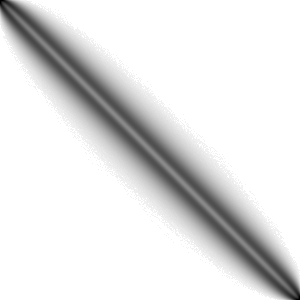}} \\[-4pt]\footnotesize$\Av(1234, 1324)$}
		&
		\makecell{\frame{\includegraphics[width=0.6in]{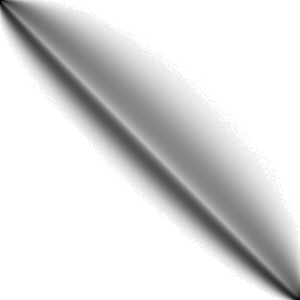}} \\[-4pt]\footnotesize$\Av(1234, 1342)$}	
		&
		\makecell{\frame{\includegraphics[width=0.6in]{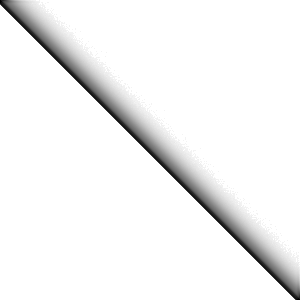}} \\[-4pt]\footnotesize$\Av(1234, 1432)$}
		&
		\makecell{\frame{\includegraphics[width=0.6in]{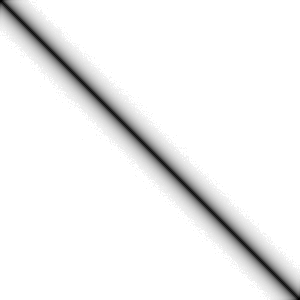}} \\[-4pt]\footnotesize$\Av(1234, 2143)$}
		&
		\makecell{\frame{\includegraphics[width=0.6in]{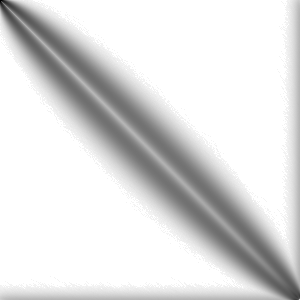}} \\[-4pt]\footnotesize$\Av(1234, 2341)$}
		\\[26pt]
		\makecell{\frame{\includegraphics[width=0.6in]{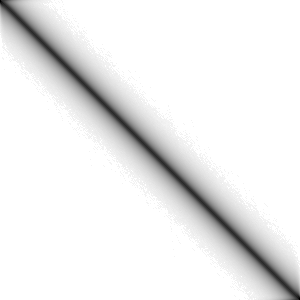}} \\[-4pt]\footnotesize$\Av(1234, 2413)$}
		&
		\makecell{\frame{\includegraphics[width=0.6in]{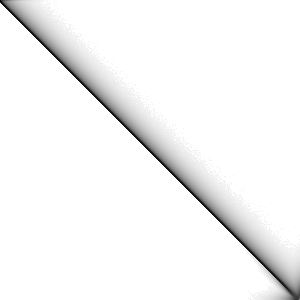}} \\[-4pt]\footnotesize$\Av(1234, 2431)$}
		&
		\makecell{\frame{\includegraphics[width=0.6in]{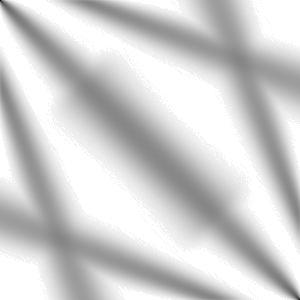}} \\[-4pt]\footnotesize$\Av(1234, 3412)$}
		&
		\makecell{\frame{\includegraphics[width=0.6in]{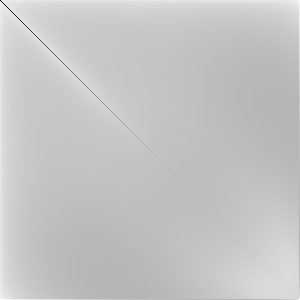}} \\[-4pt]\footnotesize$\Av(1234, 3421)$}
		&
		\makecell{\frame{\includegraphics[width=0.6in]{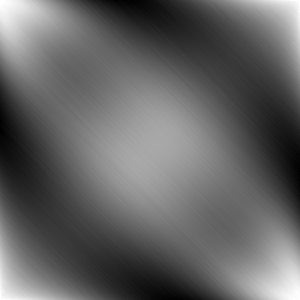}} \\[-4pt]\footnotesize$\Av(1234, 4231)$}
		&
		\makecell{\frame{\includegraphics[width=0.6in]{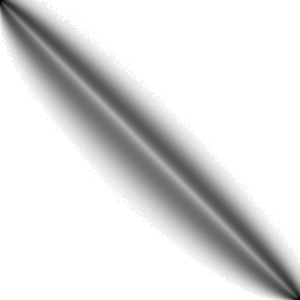}} \\[-4pt]\footnotesize$\Av(1243, 1324)$}
		\\[26pt]
		\makecell{\frame{\includegraphics[width=0.6in]{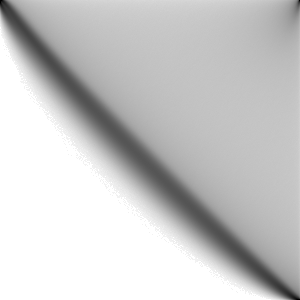}} \\[-4pt]\footnotesize$\Av(1243, 1342)$}
		&
		\makecell{\frame{\includegraphics[width=0.6in]{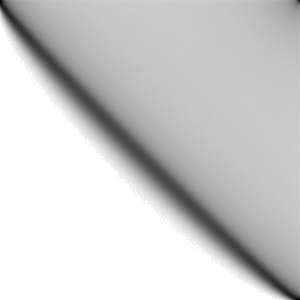}} \\[-4pt]\footnotesize$\Av(1243, 1432)$}
		&
		\makecell{\frame{\includegraphics[width=0.6in]{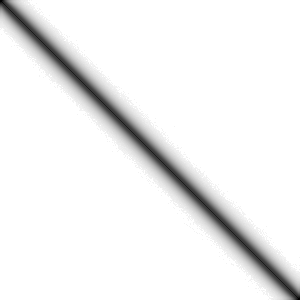}} \\[-4pt]\footnotesize$\Av(1243, 2134)$}
		&
		\makecell{\frame{\includegraphics[width=0.6in]{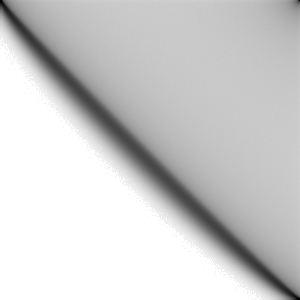}} \\[-4pt]\footnotesize$\Av(1243, 2143)$}
		&
		\makecell{\frame{\includegraphics[width=0.6in]{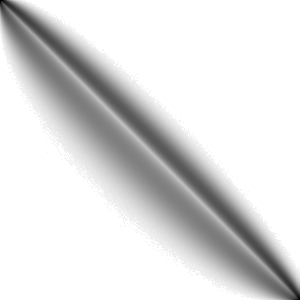}} \\[-4pt]\footnotesize$\Av(1243, 2314)$}
		&
		\makecell{\frame{\includegraphics[width=0.6in]{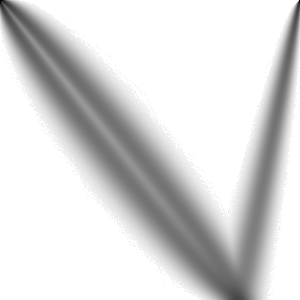}} \\[-4pt]\footnotesize$\Av(1243, 2341)$}
		\\[26pt]
		\makecell{\frame{\includegraphics[width=0.6in]{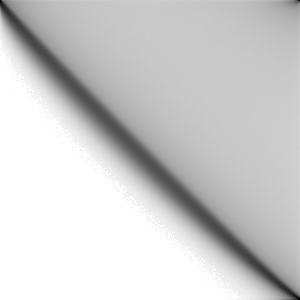}} \\[-4pt]\footnotesize$\Av(1243, 2413)$}
		&
		\makecell{\frame{\includegraphics[width=0.6in]{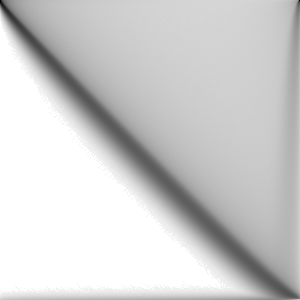}} \\[-4pt]\footnotesize$\Av(1243, 2431)$}
		&
		\makecell{\frame{\includegraphics[width=0.6in]{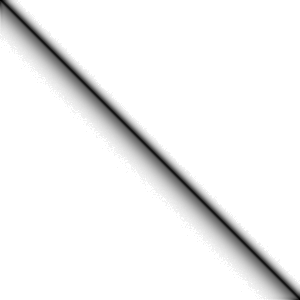}} \\[-4pt]\footnotesize$\Av(1243, 3214)$}
		&
		\makecell{\frame{\includegraphics[width=0.6in]{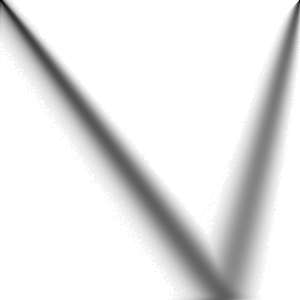}} \\[-4pt]\footnotesize$\Av(1243, 3241)$}
		&
		\makecell{\frame{\includegraphics[width=0.6in]{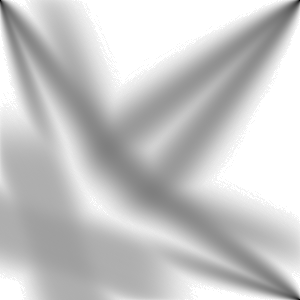}} \\[-4pt]\footnotesize$\Av(1243, 3412)$}
		&
		\makecell{\frame{\includegraphics[width=0.6in]{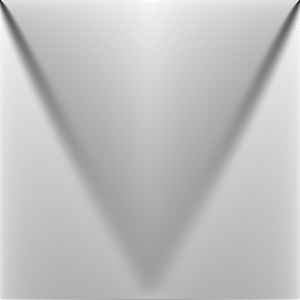}} \\[-4pt]\footnotesize$\Av(1243, 3421)$}
		\\[26pt]
		\makecell{\frame{\includegraphics[width=0.6in]{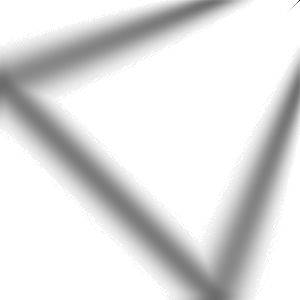}} \\[-4pt]\footnotesize$\Av(1243, 4231)$}
		&
		\makecell{\frame{\includegraphics[width=0.6in]{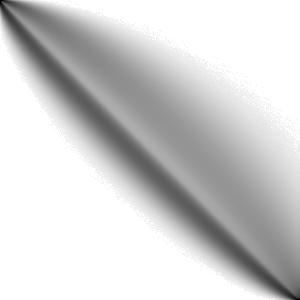}} \\[-4pt]\footnotesize$\Av(1324, 1342)$}
		&
		\makecell{\frame{\includegraphics[width=0.6in]{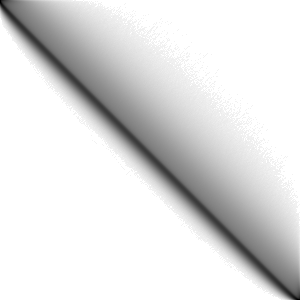}} \\[-4pt]\footnotesize$\Av(1324, 1432)$}
		&
		\makecell{\frame{\includegraphics[width=0.6in]{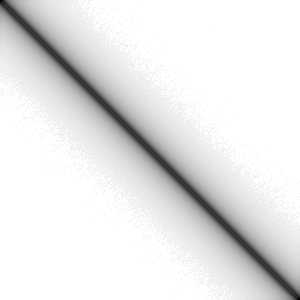}} \\[-4pt]\footnotesize$\Av(1324, 2143)$}
		&
		\makecell{\frame{\includegraphics[width=0.6in]{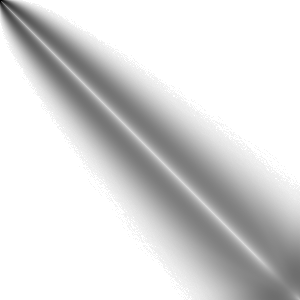}} \\[-4pt]\footnotesize$\Av(1324, 2341)$}
		&
		\makecell{\frame{\includegraphics[width=0.6in]{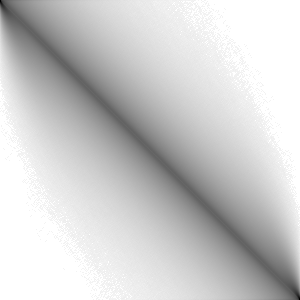}} \\[-4pt]\footnotesize$\Av(1324, 2413)$}
		\\[26pt]
		\makecell{\frame{\includegraphics[width=0.6in]{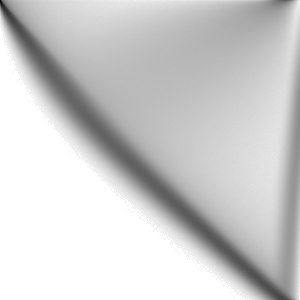}} \\[-4pt]\footnotesize$\Av(1324, 2431)$}
		&
		\makecell{\frame{\includegraphics[width=0.6in]{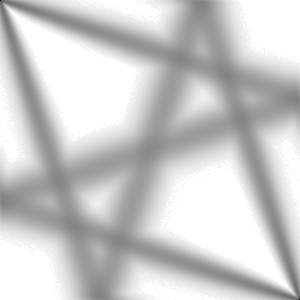}} \\[-4pt]\footnotesize$\Av(1324, 3412)$}
		&
		\makecell{\frame{\includegraphics[width=0.6in]{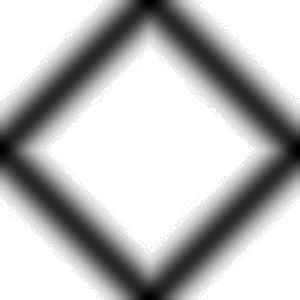}} \\[-4pt]\footnotesize$\Av(1324, 4231)$}
		&
		\makecell{\frame{\includegraphics[width=0.6in]{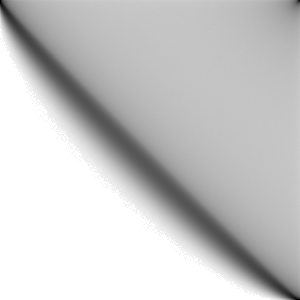}} \\[-4pt]\footnotesize$\Av(1342, 1423)$}
		&
		\makecell{\frame{\includegraphics[width=0.6in]{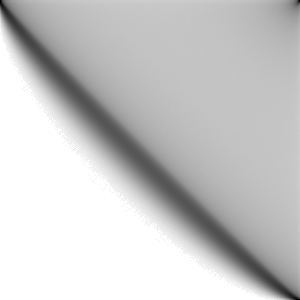}} \\[-4pt]\footnotesize$\Av(1342, 1432)$}
		&
		\makecell{\frame{\includegraphics[width=0.6in]{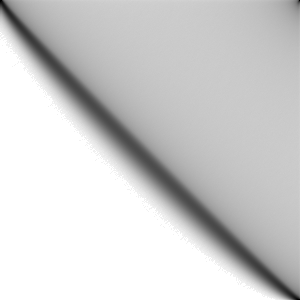}} \\[-4pt]\footnotesize$\Av(1342, 2143)$}
		\\[26pt]
		\makecell{\frame{\includegraphics[width=0.6in]{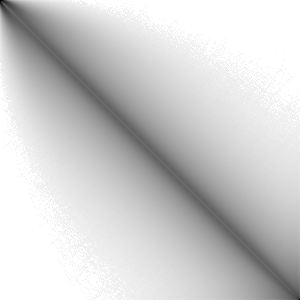}} \\[-4pt]\footnotesize$\Av(1342, 2314)$}
		&
		\makecell{\frame{\includegraphics[width=0.6in]{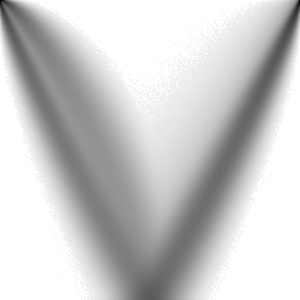}} \\[-4pt]\footnotesize$\Av(1342, 2341)$}
		&
		\makecell{\frame{\includegraphics[width=0.6in]{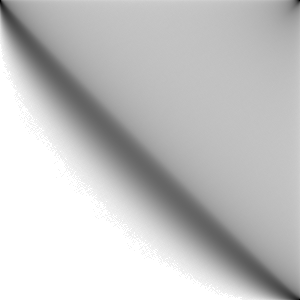}} \\[-4pt]\footnotesize$\Av(1342, 2413)$}
		&
		\makecell{\frame{\includegraphics[width=0.6in]{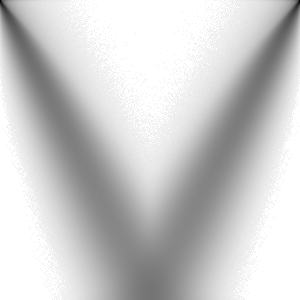}} \\[-4pt]\footnotesize$\Av(1342, 2431)$}
		&
		\makecell{\frame{\includegraphics[width=0.6in]{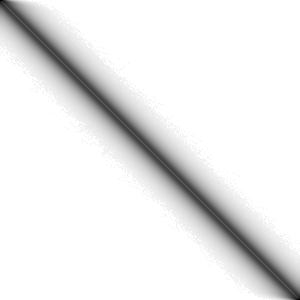}} \\[-4pt]\footnotesize$\Av(1342, 3124)$}
		&
		\makecell{\frame{\includegraphics[width=0.6in]{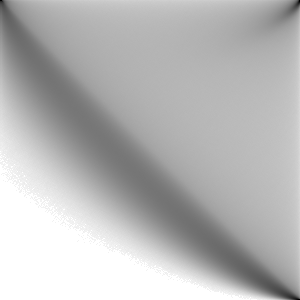}} \\[-4pt]\footnotesize$\Av(1342, 3142)$}
		\\[26pt]
		\makecell{\frame{\includegraphics[width=0.6in]{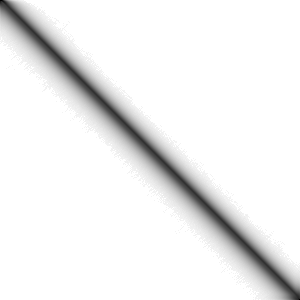}} \\[-4pt]\footnotesize$\Av(1342, 3214)$}
		&
		\makecell{\frame{\includegraphics[width=0.6in]{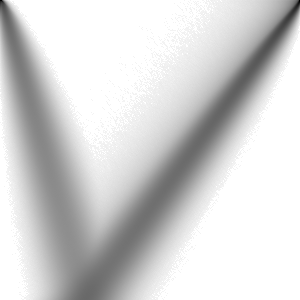}} \\[-4pt]\footnotesize$\Av(1342, 3241)$}
		&
		\makecell{\frame{\includegraphics[width=0.6in]{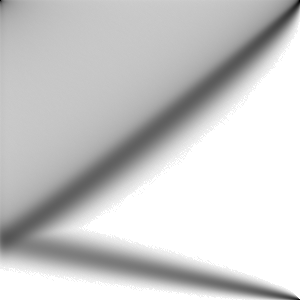}} \\[-4pt]\footnotesize$\Av(1342, 3412)$}
		&
		\makecell{\frame{\includegraphics[width=0.6in]{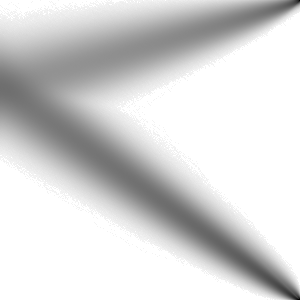}} \\[-4pt]\footnotesize$\Av(1342, 4123)$}
		&
		\makecell{\frame{\includegraphics[width=0.6in]{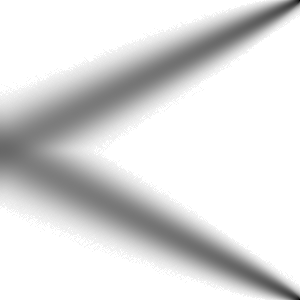}} \\[-4pt]\footnotesize$\Av(1342, 4213)$}
		&
		\makecell{\frame{\includegraphics[width=0.6in]{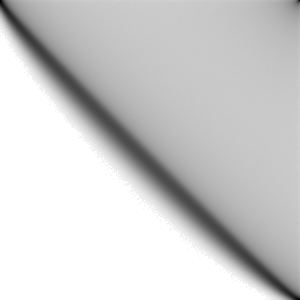}} \\[-4pt]\footnotesize$\Av(1432, 2143)$}
		\\[26pt]
		\makecell{\frame{\includegraphics[width=0.6in]{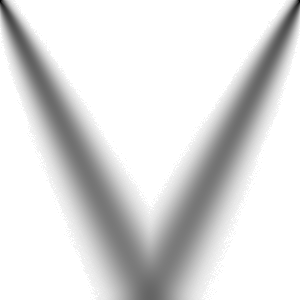}} \\[-4pt]\footnotesize$\Av(1432, 2341)$}
		&
		\makecell{\frame{\includegraphics[width=0.6in]{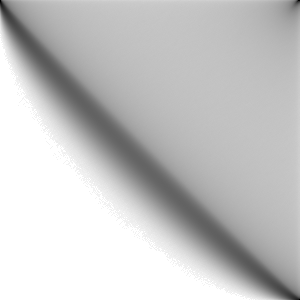}} \\[-4pt]\footnotesize$\Av(1432, 2413)$}
		&
		\makecell{\frame{\includegraphics[width=0.6in]{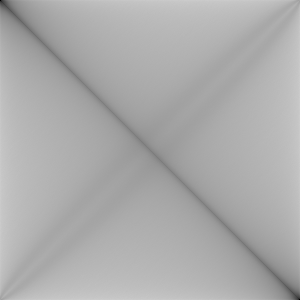}} \\[-4pt]\footnotesize$\Av(1432, 3214)$}
		&
		\makecell{\frame{\includegraphics[width=0.6in]{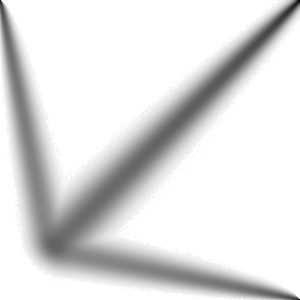}} \\[-4pt]\footnotesize$\Av(1432, 3412)$}
		&
		\makecell{\frame{\includegraphics[width=0.6in]{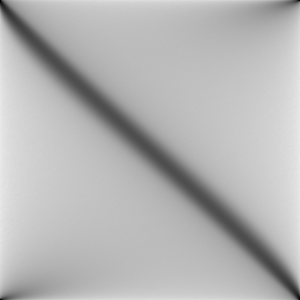}} \\[-4pt]\footnotesize$\Av(2143, 2413)$}
		&
		\makecell{\frame{\includegraphics[width=0.6in]{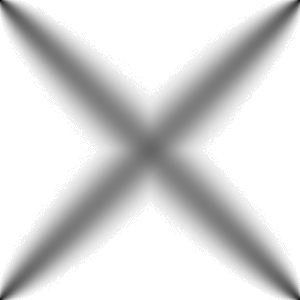}} \\[-4pt]\footnotesize$\Av(2143, 3412)$}
		\\[26pt]
		\makecell{\frame{\includegraphics[width=0.6in]{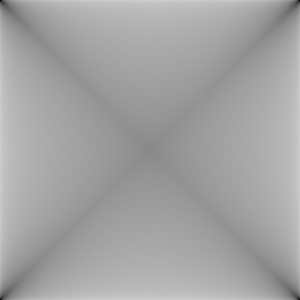}} \\[-4pt]\footnotesize$\Av(2413, 3142)$}
		&&&
		\makecell{\frame{\includegraphics[width=0.6in]{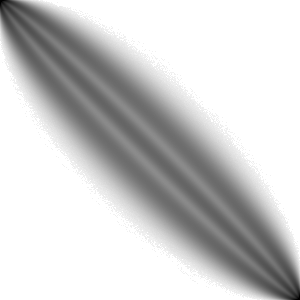}} \\[-4pt]\footnotesize$\Av(1234)$}
		&
		\makecell{\frame{\includegraphics[width=0.6in]{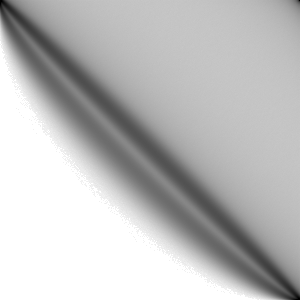}} \\[-4pt]\footnotesize$\Av(1243)$}
		&
		\makecell{\frame{\includegraphics[width=0.6in]{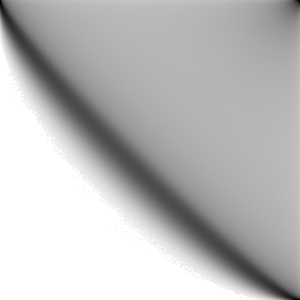}} \\[-4pt]\footnotesize$\Av(1432)$}
	\end{tabular}
	\end{center}
	\caption{Heatmaps of 58 permutation classes.}
	\label{figure:heatmap}	
\end{figure}

%% ==== %% ==== %% ==== %% ==== %% ==== %% ==== %% ==== %% ==== %% ==== %% ==== %%
%% ==== %% ==== %% ==== %% ====    SECTION SIX     ==== %% ==== %% ==== %% ==== %%
%% ==== %% ==== %% ==== %% ==== %% ==== %% ==== %% ==== %% ==== %% ==== %% ==== %%

\section{Applying Combinatorial Exploration to Permutation Patterns}
\label{section:pp-results}
%!TEX root = combinatorial-exploration.tex

This and the several following sections demonstrate the potency of Combinatorial
Exploration by applying it to a number of popular combinatorial objects. For
each new domain to which Combinatorial Exploration is to be applied we must do
the following.
\begin{enumerate}
	\item Provide a representation of combinatorial sets in this domain.
	\item Invent strategies that decompose (some of the) combinatorial sets into
	other combinatorial sets.
	\item Prove the productivity of any non-equivalence strategy.
\end{enumerate}

As will become clear over the next few sections, there are choices to be made in
the first two steps that may determine the efficacy of Combinatorial
Exploration. Consider for example the combinatorial domain explored in this
section, pattern-avoiding permutation classes (see
Section~\ref{section:pp-intro} for the relevant definitions). Before describing
strategies that decompose combinatorial sets, we must decide what the objects in
the combinatorial sets actually are. The most obvious choice is that the objects
should be precisely the things we are counting: permutations. It turns out that
such a simple structure does not lend itself to discovering or describing
decomposition strategies.

Taking some inspiration from the literature, we will define \emph{gridded
permutations} below as a generalization of permutations with a geometric flavor.
Our representations of the combinatorial domains in later sections also have a
geometric bent. This is certainly not a requirement to employ Combinatorial
Exploration, but it is our hope that these sections convey how a geometric
structure often leads to effective decomposition strategies that are easier to
describe. The task of making and implementing design decisions such as choosing
representations for objects and devising structural strategies is the portion
of using Combinatorial Exploration where knowledge and experience in the given
domain is critical. Once this task is done, one can sit back and allow
Combinatorial Exploration to enumerate combinatorial sets fully rigorously
and automatically.

Of the combinatorial domains explored in this article, we have spent by
far the most time with permutation patterns. While this section demonstrates the
depth of success that Combinatorial Exploration can bring, the succeeding
sections display its corresponding breadth, offering samples of representations
of the combinatorial objects and a few examples of strategies and resulting
proof trees, sometimes derived by hand. We expect that experts in those
combinatorial domains may have a better sense of effective object
representations and decomposition strategies than we have given here, and the
``plug-and-play'' structure of our software enables researchers to easily
implement new domains, new representations, and new strategies.

As discussed in Section~\ref{subsection:pp-success}, Combinatorial
Exploration has been a successful tool for the enumeration of permutation
classes. In this section, we will outline in more detail many of the particular
strategies that have been used for this success.

\subsection{Gridded permutations}

Murphy and Vatter~\cite{Murphy:2002vb} introduced the notion of \emph{grid
classes} as a means of defining one permutation class as a geometric combination
of others. We omit their definition here because we will use a new variation
with a small but crucial change.

We have found that when working with geometric descriptions of sets of
permutations, it is easier and vastly more effective to endow them with an
additional geometric structure that we call a \emph{gridding}. Extending this
notion, we work with \emph{gridded classes} instead of the grid classes of
Murphy and Vatter (which from our viewpoint may be more
appropriate to call \emph{griddable} classes), and we will later define a
combinatorial object called a \emph{tiling} that implicitly defines a set
of gridded permutations.

A \emph{gridded permutation} of size $n$ is a pair $(\pi, P)$, where $\pi$ is a
permutation of length $n$ called the \emph{underlying permutation} and $P =
(c_1, \ldots, c_n)$ for $c_i \in \N \times \N$ is the tuple of
\emph{positions}. The positions represent a placement of $\pi$ onto the positive
quadrant of $\mathbb{R}^2$ where the point corresponding to $(i, \pi(i))$ with
$c_i = (x, y)$ has been drawn in the square $[x, x+1) \times [y, y+1)$. For
example, Figure~\ref{fig:gridded-perm} depicts a gridded permutation of size
$9$.

\begin{figure}
	\begin{center}
		\begin{tikzpicture}[scale=0.6, x=1cm,y=1cm]

			\foreach \Point in {(0.25, 0.75), (0.75, 3.75), (1.25, 1.75), %
								(1.75, 1.25), (2.25, 3.25), (2.75, 2.75), %
								(3.25, 4.25), (3.75, 0.25), (4.25, 2.25)}{\node
				at \Point {\point{2pt}};}
			% def shrunk(p): return [(i/2*1.+0.25,v/2*1.+0.25) for i, v in
			%     enumerate(p)]
			%
			% % to ensure that the points are being properly centered:
			\draw [dashed, very thin, gray] (0.01,0.01) grid (5.5,5.5);
			\draw[-latex, thin, gray] (-0.5,0)--(6,0);
			\draw[-latex, thin, gray] (0,-0.5)--(0,6);
		\end{tikzpicture}
	\end{center}
	\caption{The gridded permutation
		$(284376915, ((0, 0), (0, 3), (1, 1), (1, 1), (2, 3), (2, 2), (3, 4),$
		$ (3,0), (4, 2)))$.
		For example, the entry $9$, which occurs at index $7$ in the permutation,
		lies in the unit cell whose lower left corner has coordinates $(3,4)$,
		reflected by the fact that $(3,4)$ is the seventh cell in the tuple of cell positions.}
	\label{fig:gridded-perm}
\end{figure}
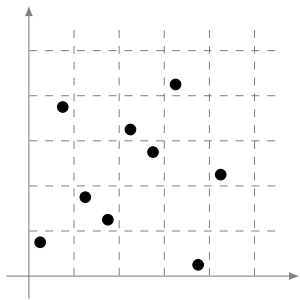

Not every pair $(\pi, P)$ is a valid gridded permutation, as the position tuple
may not be consistent with the permutation. For example, $(21, ((0, 0), (1,
1)))$ is not a valid gridded permutation---there is no way to place one entry in
the region $[0,1) \times [0,1)$ and another entry in the region $[1,2) \times
[1,2)$ such that the two entries form a $21$ pattern. 

A pair $(\pi, P)$ forms a valid gridded permutation when the relative positions of the entries in $\pi$ are consistent with the relative positions of the cells in $P$. More formally, letting $|\pi| = n$ and $P = (c_1, \ldots, c_n)$ with $c_i = (x_i, y_i)$, the pair $(\pi, P)$ forms a gridded permutation if ${x_1 \leq x_2 \leq \cdots \leq x_n}$ and ${y_{\pi^{-1}(1)} \leq y_{\pi^{-1}(2)} \leq \cdots \leq y_{\pi^{-1}(n)}}$.

When all of the entries of a gridded permutation have the same position $c$, we
abbreviate our notation by writing $(\pi, c)$ instead of $(\pi, (c, c,\dots,
c))$. In this case we say the gridded permutation is \emph{localized in cell
$c$}.

Let $\GG$ denote the set of all gridded permutations. The set $\GG$ is
inconvenient from a combinatorial perspective as, for example, it contains an
infinite number of gridded permutations of size $1$, and so we further define
the refinements $\GG^{(t, u)}$ to be the set of gridded permutations whose
positions all lie within the rectangle $[0, t) \times [0, u)$. Informally,
$\GG^{(t,u)}$ is the set of those gridded permutations that can be drawn
on a grid with width $t$ (i.e., $t$ columns) and height $u$ (i.e., $u$ rows).
Let $\GG^{(t, u)}_n$ denote the set of gridded permutations in $\GG^{(t, u)}$
with size $n$.
We can determine $|\GG^{(t, u)}_n|$ by considering where the  horizontal and
vertical grid lines can pass between the points of each permutation of length
$n$. One finds that the total number is
\[
	n! \binom{t + n - 1}{t - 1} \binom{u + n - 1}{u - 1}.
\]
In particular, the finiteness of $|\GG^{(t, u)}_n|$ implies that $\GG^{(t,u)}$
is a combinatorial set.

The notion of pattern containment in permutations extends naturally to the realm
of gridded permutations by requiring the entries of the smaller permutation to
lie in precisely the same cells in the larger permutation. Formally, we say that
the size $n$ gridded permutation $g = (\pi, (c_1, \ldots, c_n))$ \emph{contains}
the size $k$ gridded permutation $h = (\sigma, (d_1, \ldots, d_k))$ if there is
a subsequence $\pi(i_1)\pi(i_2) \cdots \pi(i_k)$ of $\pi$ whose standardization
is equal to $\sigma$ and $c_{i_j} = d_j$ for $1 \leq j \leq k$. Note that the
cells of $h$ are required to be actually equal to the cells of $g$ where $h$
occurs; there is no standardization of cells. We often use the
term \emph{pattern} colloquially to refer to the smaller permutations whose
containment is under consideration. We call the subsequence $(i_1, \ldots, i_k)$
an \emph{occurrence} of $h$ in the gridded permutation $g$. If $g$ does not
contain $h$ we say it \emph{avoids} $h$.
Figure~\ref{fig:gridded-containment-and-avoidance} shows a gridded permutation
that contains two occurrences of the pattern $(231, ((0,0), (1,1), (3,0)))$ and
avoids the pattern $(231, ((1,3), (3,4), (4,2)))$.

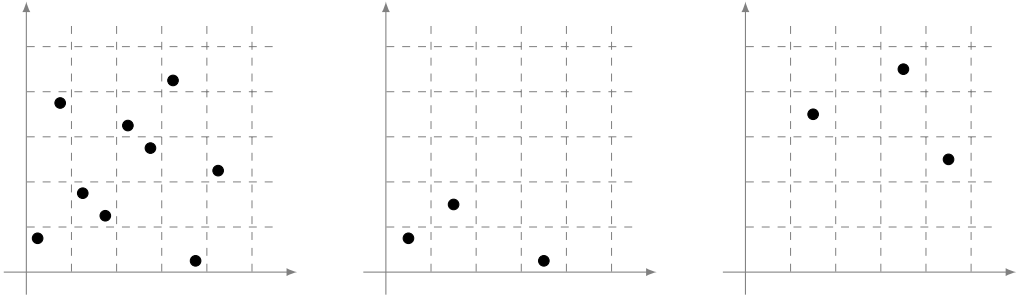
\begin{figure}
	\begin{center}
		\begin{tikzpicture}[scale=0.6, x=1cm,y=1cm]

			\foreach \Point in {(0.25, 0.75), (0.75, 3.75), (1.25, 1.75), %
								(1.75, 1.25), (2.25, 3.25), (2.75, 2.75), %
								(3.25, 4.25), (3.75, 0.25), (4.25, 2.25)}{\node
				at \Point {\point{2pt}};}
			\draw [dashed, very thin, gray] (0.01,0.01) grid (5.5,5.5);
			\draw[-latex, thin, gray] (-0.5,0)--(6,0);
			\draw[-latex, thin, gray] (0,-0.5)--(0,6);
		\end{tikzpicture}
		\qquad
		\begin{tikzpicture}[scale=0.6, x=1cm,y=1cm]
			\foreach \Point in {(0.5, 0.75), (1.5, 1.5), (3.5, 0.25)}{\node at
				\Point {\point{2pt}};}
			\draw [dashed, very thin, gray] (0.01,0.01) grid (5.5,5.5);
			\draw[-latex, thin, gray] (-0.5,0)--(6,0);
			\draw[-latex, thin, gray] (0,-0.5)--(0,6);
		\end{tikzpicture}
		\qquad
		\begin{tikzpicture}[scale=0.6, x=1cm,y=1cm]
			\foreach \Point in {(1.5, 3.5), (3.5, 4.5), (4.5, 2.5)}{\node at
				\Point {\point{2pt}};}
			\draw [dashed, very thin, gray] (0.01,0.01) grid (5.5,5.5);
			\draw[-latex, thin, gray] (-0.5,0)--(6,0);
			\draw[-latex, thin, gray] (0,-0.5)--(0,6);
		\end{tikzpicture}
	\end{center}
	\caption{The gridded permutation from Figure~\ref{fig:gridded-perm} (left)
	contains two occurrences of the gridded pattern $(231, ((0,0),(1,1), (3,0)))$ (center),
	but avoids $(231, ((1,3),(3,4),(4,2)))$ (right).}
	\label{fig:gridded-containment-and-avoidance}
\end{figure}

Extending the notion of containment to sets of patterns, we say a gridded
permutation $g$ \emph{avoids} a set of gridded patterns $\OO$ if it avoids every
gridded pattern in $\OO$ and we define $\Av(\OO)$ to be the set of gridded
permutations that avoid $\OO$. On the other hand, we say a gridded permutation
\emph{contains} a set of gridded patterns $\RR$ if it does not avoid $\RR$. Note
that this means $g$ contains $\RR$ if it contains \emph{at least one} pattern in
$\RR$; $g$ need not contain all patterns in $\RR$. The set of gridded
permutations containing $\RR$ is denoted $\Co(\RR)$.\footnote{This definition of
containment may seem odd at first, but its purpose will become clear in the
following subsection, as will our use of the letters $\OO$ and $\RR$, which will
come to represent what we call ``obstructions'' and ``requirements''
respectively.}

%%%%% ===== JP: I commented this out because I don't think we actually use it anywhere.
%Finally, as we wish to work with combinatorial sets that by definition may
%contain only a finite number of objects of each size, we define
%\[
%  \Av^{(t, u)}(\OO) = \Av(\OO) \cap \GG^{(t,u)}
%  \quad \text{ and } \quad
%  \Co^{(t, u)}(\RR) = \Co(\RR) \cap \GG^{(t,u)}.
%\]

\subsection{Tilings}

Rather than working with arbitrary combinatorial sets of gridded permutations,
we restrict ourselves to those with a certain kind of structure that we now
explain. In doing so, we are able to devise novel and efficient algorithms for
manipulating these sets that exploit this structure. We start by defining a new
combinatorial object that represents structured sets of gridded permutations.

\begin{definition}
	A \emph{tiling} is a triple $\TT = ((t, u), \OO, \RR)$ where $t$ and $u$ are
	integers, $\OO$ is a set of gridded permutations that we call
	\emph{obstructions} and $\RR = \{ \RR_1, \RR_2, \ldots, \RR_k \}$ is a set of
	sets of gridded permutations that we call \emph{requirements}.

	A tiling $\TT$ represents the combinatorial set of gridded permutations $g \in
	\GG^{(t,u)}$ such that $g$ avoids $\OO$ and $g$ contains each $\RR_i$ for $1
	\leq i \leq k$. We call this set $\Grid(\TT)$. In other words, $\Grid(\TT)$ is
	the set gridded permutations in the region $[0, t) \times [0, u)$ that avoid
	\emph{all} of the patterns in $\OO$ and contain \emph{at least one} of the
	patterns in each $\RR_i$. We call the individual gridded permutations in each
	$\RR_i$ {requirements} and we call each set $\RR_i$ a
	\emph{requirement list}.
\end{definition}

An \emph{interval} of a poset $P$ is a set $I$ with the property that for all
$a,c \in I$ and $b \in P$, if $a \leq b \leq c$, then $b \in I$. The following
proposition shows that the sets of gridded permutations that are represented by
tilings are precisely intervals in the poset of gridded permutations when
restricting to tilings of a fixed dimension. We see this as justification that
tilings are a natural structure to use for Combinatorial Exploration.

\begin{theorem}
\label{theorem:tilings-are-intervals}
	Fix positive integers $t$ and $u$, and let $\HH$ be a set of gridded
	permutations in $\GG^{(t,u)}$. The following are equivalent:
	\begin{enumerate}
		\item The set $\HH$ is an interval in the poset of gridded
		permutations $\GG^{(t,u)}$.
		\item There exists a tiling $\TT$ such that $\Grid(\TT) = \HH$.
	\end{enumerate}
\end{theorem}

\begin{proof}
To show that Condition 1 implies Condition 2 we construct the tiling
$\TT$. Let $\OO$ be the subset of gridded permutations in $\GG^{(t,u)}$ that are
not contained in any gridded permutation in $\HH$. These will be the obstructions
of $\TT$; this set may be infinite, which is not a problem. The
requirements will be a single requirement list containing every gridded
permutation in $\HH$, i.e.,  $\RR = \{ \HH \}$. Now define
$\TT = ((t,u), \OO, \RR)$. It remains to show that $\Grid(\TT) = \HH$.

Let $g \in \Grid(\TT)$. By the definition of $\TT$, $g$ contains at least one
permutation in $\HH$ and avoids every obstruction in $\OO$. Assume toward a
contradiction that $g \not\in \HH$. Because $\HH$ is an interval, $g$ is not
contained in any permutation in $\HH$. Thus by construction, $g \in \OO$.
This is a contradiction thus $g$ must be in $\HH$. Conversely, consider a
permutation $g \in \HH$. This satisfies the requirements of $\TT$ by the
definition of $\RR = \{\HH\}$. Moreover, $\OO$ was defined to be the set of
permutations that are not contained in any permutation in $\HH$, and so $g$
does not contain any permutation in $\OO$. Thus $g \in \Grid(\TT)$, confirming
that Condition 1 implies Condition 2.

To show that Condition 2 implies Condition 1, let $\TT = ((t,u), \OO, \RR)$ be
a tiling and define $\HH = \Grid(\TT)$. With the aim of showing that $\HH$ is an
interval, suppose that $a,c \in \HH$, $b \in \GG^{(t,u)}$ and $a \leq b \leq c$.
Since $c \in \HH$, we know that $c$ does not contain any obstruction in $\OO$, and
as $b \leq c$ the same must be true for $b$. Moreover $a$ must contain at least one
gridded permutation from each requirement list in $\RR$, and as $b \geq a$ the same
must be true for $b$. Since $b$ avoids all of the obstructions of $\TT$ and
satisfies all of the requirements, we conclude $b \in \Grid(\TT) = \HH$ and
therefore $\HH$ is an interval.
\end{proof}

Henceforth, all tilings used in this paper have the property that each requirement
list $\RR_i$ contains only finitely many gridded permutations. Under this assumption,
one useful property of the tiling representation is that for a
tiling $\TT$, it can be determined in finite time whether $\Grid(\TT)$ is empty.
This is algorithmically useful, as later parts of this section will
demonstrate.
\begin{theorem}
\label{theorem:is-empty-check}
	Consider a tiling $\TT = ((t,u), \OO, \RR)$ where
	$\RR = \{\RR_1, \ldots, \RR_k\}$. Let $\ell_i$ be the size of the largest
	gridded permutation in $R_i$ and define $L = \ell_1 + \cdots + \ell_k$.
	If $\Grid(\TT)$ is nonempty, then it contains a gridded permutation
	whose size is at most $L$.
\end{theorem}
\begin{proof}
	Suppose $\Grid(\TT)$ is nonempty and let $\pi \in \Grid(\TT)$. Since
	$\pi$ satisfies all of the requirements in $\TT$, there exist gridded
	permutations $\sigma_1 \in \RR_1, \ldots, \sigma_k \in \RR_k$ such that $\pi$
	contains each $\sigma_i$. Choose one particular occurrence of each $\sigma_i$,
	let $I_i$ be the indices of $\pi$ where this occurrence is located, and let
	$I = I_1 \cup \cdots \cup I_k$. By design $|I| \leq L$.

	Consider the subpermutation $\tau$ of $\pi$ formed just from the entries
	at the indices of $I$. The permutation $\tau$ avoids all of the obstructions
	in $\OO$ and satisfies all of the requirements in $\RR$, and therefore
	$\tau \in \Grid(\TT)$ and has size at most $L$.
\end{proof}

This theorem implies that to check whether $\Grid(\TT)$ is empty, one only needs
to check whether each gridded permutation in $\GG^{(t,u)}$ of size at most $L$
avoids the obstructions and contains the requirements. While this can be done
in finite time, it is still slower than desirable. In practice, we have made
several significant optimizations to speed up this algorithm, but we will not
go into detail here.

Given any (ungridded) permutation class $\Av(B)$, define the tiling
$\TT_{\Av(B)}$ to be the $1 \times 1$ tiling with obstructions mimicking the
basis elements $B$ and no requirements:
\[
  \TT_{\Av(B)} = ((1,1), \{(\beta, (0,0)) : \beta \in B\}, \{\}).
\]
Since the map $\pi \mapsto (\pi, (0,0))$ from $\Av(B)$ to $\Grid(\TT_{\Av(B)})$
is a size-preserving bijection, we can enumerate $\Av(B)$ by applying
Combinatorial Exploration to the set $\Grid(\TT_{\Av(B)})$. The strategies we
describe in the remainder of this section are designed to be applied to sets of
gridded permutations coming from tilings, that is, sets $G$ such that $G =
\Grid(\TT)$ for some tiling $\TT$. The result is that we can define the
decomposition function of a strategy as a function whose input is a tiling $\TT$
and whose output is a sequence of tilings $(\TT^{(1)}, \ldots, \TT^{(m)})$, even
though the decomposition function really acts on the corresponding sets of
gridded permutations.

In order to explain the strategies we use for gridded permutations, we make
liberal use of figures to depict tilings graphically. The entries of gridded
permutations are drawn as points, and those points are connected by lines so
it is clear which points belong to which gridded permutations. To maximize visual
distinction, obstructions are drawn in red, with solid lines and solid round points,
while requirements are drawn in blue, with dotted lines and hollow square points.
Since the requirements of a tiling consist of a set of sets of gridded permutations
$\{\RR_1, \ldots, \RR_k\}$---a structure that is hard to convey visually---we
typically only draw tilings when each $\RR_i$ has size $1$. Thus, in this article,
whenever a tiling shows two requirements $r_1$ and $r_2$, for example, this should
be interpreted as $\{\{r_1\}, \{r_2\}\}$ (two requirement lists of length one), not
$\{\{r_1, r_2\}\}$ (one requirement list of length two), i.e., all gridded
permutations that can be drawn on the tiling must contain both $r_1$ and $r_2$.

Consider, for example, the tiling $\TT = ((3,2), \OO, \RR)$
where\footnote{Recall that when the entries of a gridded permutation all lie in
the same cell $c$, we write $(\pi, c)$ instead of $(\pi, (c,c,\ldots,c))$.}
\begin{align*}
    \OO &= \{(1, (0,1)), (1,(1,0)), (1,(2,1)), (12, (1,1)), (21, (1,1)),\\
        &\qquad (132, (0,0)), (132, (2,0)), (123, ((0,0),(2,0),(2,0)))
\end{align*}
and
\[
  \RR = \{\{(1, (1,1))\}, \{(21, (0,0))\}\}
\]
depicted on the left-hand side of Figure~\ref{fig:tiling-ex}. The gridded
permutations in $\Grid(\TT)$ are those contained in $\GG^{(3,2)}$ that:
\begin{itemize}
  \renewcommand\labelitemi{--}
  \item have no entries in the cells $(0,1)$, $(1,0)$ and $(2,1)$,
  \item have exactly one entry in the cell $(1,1)$,
  \item have two entries in cell $(0,0)$ that form a $21$ pattern,
  \item avoid the pattern $132$ fully within cell $(0,0)$ or fully within cell
  $(2,0)$, and
  \item avoid a $123$ pattern with the first entry in cell $(0,0)$ and the
  second two entries in cell $(2,0)$.
\end{itemize}

\begin{figure}
	\centering
  \begin{tikzpicture}[baseline=(current bounding box.center)]
      \node (nonemptytop3) at (12, -2) {\tiling{1.0}{3}{2}{}%
      {%
        {3/{(0.15, 0.25), (0.35, 0.75), (0.55, 0.50)}},%
        {3/{(2.45, 0.25), (2.65, 0.75), (2.85, 0.50)}},%
        {3/{(0.75, 0.1), (2.15, 0.6), (2.35, 0.675)}},%
        {2/{(1.1, 1.25), (1.4, 1.75)}},%
        {2/{(1.6, 1.75), (1.9, 1.25)}},%
        {1/{(0.5, 1.5)}},%
        {1/{(2.5, 1.5)}},%
        {1/{(1.5, 0.5)}}%
      } {%
        {1/{(1.5, 1.5)}},%
        {2/{(0.5, 0.85), (0.85, 0.5)}}%
      }};
	\end{tikzpicture}
  \qquad = \qquad
	\begin{tikzpicture}[baseline=(current bounding box.center)]
      \node (nonemptytop3) at (12, -2) {\tiling{1.0}{3}{2}{1/1}%
      {%
        {3/{(0.15, 0.25), (0.35, 0.75), (0.55, 0.50)}},%
        {3/{(2.45, 0.25), (2.65, 0.75), (2.85, 0.50)}},%
        {3/{(0.75, 0.1), (2.15, 0.6), (2.35, 0.675)}}%
      } {%
        {2/{(0.5, 0.85), (0.85, 0.5)}}%
      }};
	\end{tikzpicture}
	\caption{On the left, a tiling with all obstructions and requirements
	depicted. On the right, the same tiling shown using visual shortcuts.}
	\label{fig:tiling-ex}
\end{figure}
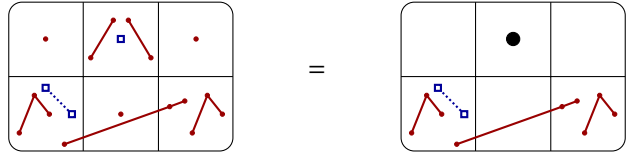

A cell is called \emph{empty} if it contains the obstruction of length $1$,
and \emph{nonempty} otherwise. It is so common to have cells that are empty or
required to
contain exactly one entry that we have visual shortcuts for these two cases:
cells that are required to be empty are drawn without the single-point
obstruction, while cells required to contain exactly one entry are drawn with a
solid, larger black circle. Figure~\ref{fig:tiling-ex-empty-and-point} shows
these two visual shortcuts. The tiling on the left-hand side of
Figure~\ref{fig:tiling-ex} is thus simplified to the tiling on the right-hand
side.

\begin{figure}
	\centering
	\begin{tikzpicture}
      \node (empty) at (12, -2) {\tiling{1.0}{1}{1}{}%
      {%
        {1/{(0.5, 0.5)}}%
      } {%
      }}; \node (equiv2) at (13, -2) {$=$}; \node (point2) at (14, -2)
      {\tiling{1.0}{1}{1}{}%
      {%
      } {%
      }};
	\end{tikzpicture}
	\qquad\qquad
	\begin{tikzpicture}
      \node (point1) at (12, -2) {\tiling{1.0}{1}{1}{}%
      {%
        {2/{(0.1, 0.25), (0.4, 0.75)}},%
        {2/{(0.6, 0.75), (0.9, 0.25)}}%
      } {%
      	{1/{(0.5, 0.5)}}%
      }}; \node (equiv2) at (13, -2) {$=$}; \node (point2) at (14, -2)
      {\tiling{1.0}{1}{1}{0/0}%
      {%
      } {%
      }};
	\end{tikzpicture}
	\caption{The two visual shortcuts that we use to make pictures of
	tilings more legible.}
	\label{fig:tiling-ex-empty-and-point}
\end{figure}
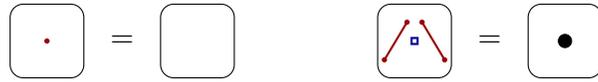

\subsection{Descriptions of Six Strategies}
\label{subsection:six-strategies}

To enumerate the (ungridded) permutation class $\Av(B)$, we start with the
combinatorial set $\TT_{\Av(B)}$ defined in the previous subsection and apply
strategies that specifically decompose sets of gridded permutations of the form $\Grid(\TT)$
for a tiling $\TT$.\footnote{As a result, all of our strategies implicitly have
decomposition functions that output $\DNA$ when the input set is not a
combinatorial set of this form.}

In this subsection we give intuitive, graphical descriptions of six fundamental
strategies, saving the details and the proofs of their productivity for
Subsection~\ref{subsection:strats-details}. The reader who chooses to skip
Subsection~\ref{subsection:strats-details} should still have a fairly complete
understanding of these six strategies. There are many more strategies that can
be efficaciously applied to the study of permutation patterns whose description
we defer to a future work whose sole focus is on Combinatorial Exploration in
the realm of permutation patterns.

In contrast to many algorithms that work by directly manipulating sets of
permutations, like the enumeration schemes of
Zeilberger~\cite{zeilberger:enumeration-schemes} and the insertion
encoding of Vatter~\cite{vatter:regular-insertion-encoding}, we can apply the strategies
defined in this section by simply manipulating the tilings themselves,
without the need to actually generate large sets of permutations. This
is a major benefit of using tilings to represent sets of
gridded permutations.

The six strategies we now introduce are, in order: requirement insertion,
obstruction/requirement simplification, point placement, row and column
separation, factorization, and obstruction inferral. The second, third, fourth,
and sixth of these strategies are equivalence strategies; to prove their
correctness in Subsection~\ref{subsection:strats-details} we describe a
size-preserving bijection between the input and output combinatorial sets. The
first and fifth of these strategies are not equivalence strategies and we
prove their productivity as outlined in Section~\ref{section:productivity}.

We will consider the enumeration of $\CC = \Av(1243, 1342, 2143)$ as a running
example. A full proof tree for $\CC$ is shown in
Figure~\ref{fig:av1243_1342_2143} on page~\pageref{fig:av1243_1342_2143}. The
root of the tree is $\TT_{\CC} = \TT_1$ and each strategy used in the tree will
be described in the remainder of this section.

\subsubsection{Requirement Insertion}
\label{subsubsection:requirement-insertion-intuition}

The strategy ``size-0-or-not'' in Example~\ref{example:factor-around-max-entry}
used the fact that every gridded permutation either has size $0$, or not, to
decompose a combinatorial set $\AA$ into a pair $d_Z(\AA) = (\BB,\CC)$ where
$\BB$ contains those elements with size $0$ and $\CC$ contains all others. This
can be reformulated as follows: every gridded permutation $\pi$ either contains
a gridded permutation of size $1$ (and thus $|\pi| \geq 1$) or avoids all
gridded permutations of size $1$ (and thus has size $0$).

This idea can be easily generalized from containing or avoiding gridded
permutations of size $1$ to containing or avoiding any set $H$ of gridded
permutations. Two applications of this strategy are seen in
Figure~\ref{fig:req-ins-from-large-tree}, which shows the top portion of the full proof tree
in Figure~\ref{fig:av1243_1342_2143} that is serving as our running example. We start
with the tiling $\TT_1$
and decompose it into the disjoint union of the tiling $\TT_2$, which avoids the
gridded pattern $(1, (0,0))$, and the tiling $\TT_3$ which contains this gridded
pattern. The same strategy is then applied to $\TT_3$ with the gridded pattern
$(12, (0,0))$. Note that $\TT_5$ shows only the requirement $(12, (0,0))$, and
not the previous requirement $(1, (0,0))$. This is because the new larger
requirement makes the smaller requirement redundant, and so we have quietly removed
it. We discuss this strategy in more detail in
Subsection~\ref{subsubsection:obs-req-simp}.

The tilings $\TT_2$ and $\TT_4$ are subject to verification strategies that we
discuss in Subsection~\ref{subsubsec:verification}.
That leaves the tiling $\TT_5$, containing those permutations in $\TT_1$ that are not
strictly decreasing, as the only unexplored tiling so far.

\begin{figure}
	\centering
	\begin{tikzpicture}
	\input{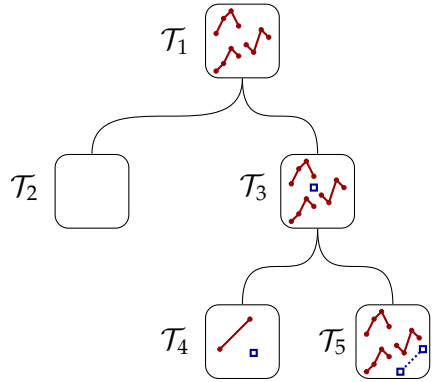}
	\end{tikzpicture}
	\caption{The top portion of the proof tree for $\Av(1243, 1342, 2143)$ shown in
	full in Figure~\ref{fig:av1243_1342_2143} on page~\pageref{fig:av1243_1342_2143}.
	The requirement $(1, (0,0))$ is inserted into $\TT_1$, and the requirement
	$(12, (0,0))$ is inserted into $\TT_3$. Note that the point requirement from
	$\TT_3$ has not been shown on $\TT_5$ because it is redundant.}
	\label{fig:req-ins-from-large-tree}
\end{figure}

More formally, let $H$ be a set of gridded patterns and $\TT = ((t,u), \OO,
\RR)$ a tiling. We define two tilings
\[
	\inso{H}{\TT} = ((t,u), \OO \cup H, \RR), \qquad \insr{H}{\TT} = ((t,u), \OO, \RR \cup \{H\}).
\]

In the first one we have added all the gridded patterns in $H$ as obstructions,
and in the second one we have added $H$ as one of the requirement lists, so every
permutation in $\Grid(\insr{H}{\TT})$ contains at least one gridded pattern in
$H$. Clearly $\Grid(\TT)$ is equal to the disjoint union
$\Grid(\inso{H}{\TT}) \sqcup \Grid(\insr{H}{\TT})$.
When $H = \{ h \}$ we drop the brackets and instead write $\inso{h}{\TT}$ and
$\insr{h}{\TT}$.

We now give the actual definition of the strategy. Given any set of gridded
permutations $H$, the strategy $\ReqIns_H$ is defined as follows\footnote{Recall
that we describe strategies by their action on tilings even though they actually
act on combinatorial sets of gridded permutations.}:
\begin{itemize}
  \renewcommand\labelitemi{--}
  \item If $\TT$ is a tiling with dimensions $t \times u$ and $H \subseteq
  \GG^{(t,u)}$, and both sets
  \[
	\inso{H}{\TT} \qquad \text{ and } \qquad \insr{H}{\TT}
  \]
  are nonempty, then $d_{\ReqIns_H}(\TT) = (\inso{H}{\TT}, \insr{H}{\TT})$.
  Otherwise $d_{\ReqIns_H}(\TT) = \DNA$.
  \item The reliance profile function is $r_{\ReqIns_H}(n) = (n, n)$.
  \item The counting functions are $c_{\ReqIns_H, (n)}((a_0,\ldots,a_n), (b_0,\ldots,b_n)) = a_n + b_n$.
\end{itemize}

Requirement Insertion is an example of what we call a \emph{disjoint-union-type
strategy}, which is a strategy $S$ such that whenever $d_S(\AA) = (\BB^{(1)},
\ldots, \BB^{(m)})$ for $m > 1$ and all $\BB^{(i)}$ are nonempty,
we have that $\AA$ is the disjoint union $\BB^{(1)} \sqcup
\cdots \sqcup \BB^{(m)}$ and so $|\AA_n| = |\BB^{(1)}_n| + \cdots +
|\BB^{(m)}_n|$. We additionally require that the reliance profile function is
$n \mapsto (n, \ldots, n)$.
Under these conditions, disjoint-union-type strategies are easily seen to
be productive; we justify the productivity of Requirement Insertion in particular
in Theorem~\ref{theorem:requirement-insertion-productive} in
Subsection~\ref{subsection:strats-details}.

\subsubsection{Obstruction and Requirement Simplification}
\label{subsubsection:obs-req-simp}
In Figure~\ref{fig:req-ins-from-large-tree}, $\TT_5$ is formed from $\TT_3$ by
adding the requirement list $\{(12, (0,0))\}$, so that the full requirements of
$\TT_5$ are
\[
  \{\;\;\{(1, (0,0))\},\;\;\{(12, (0,0))\}\;\;\}.
\]
We have not drawn the requirement of size $1$ in the figure because it is
redundant---all gridded permutations that contain $(12, (0,0)))$ also contain
$(1, (0,0))$---and so it is omitted from the figure. Throughout this work, we
will frequently simplify obstructions and requirements in the way described
here, often without mentioning it.

It is advantageous to remove redundant obstructions and requirements from
tilings. In addition to a gain in computational efficiency, it also boosts the
theoretical strength of Combinatorial Exploration by increasing our ability to
identify when two tilings represent the same set of gridded permutations. There
are several ways in which obstructions and requirements can be deleted, altered,
or even added, leading to simpler tiling representations for the same sets of
gridded permutations.

We want to first point out a subtlety around the way we have framed
requirements. As we described in the example above, the requirement $(1, (0,0))$
is redundant in the sense that it is a strictly weaker condition than the
requirement $(12, (0,0))$. It is incorrect to remove this redundancy just by
deleting the requirement from its list, yielding the full requirements
\[
  \{\;\;\{\;\},\;\;\{(12, (0,0))\}\;\;\}.
\]
A gridded permutation can be drawn on a tiling if it contains at least one
requirement from each set, and so if a list becomes empty this condition is
impossible to satisfy, implying that no gridded permutations can be drawn. The
correct way to handle this redundancy is to instead delete the entire
requirement list:
\[
  \{\;\;\{(12, (0,0))\}\;\;\}.
\]
To summarize, deleting individual requirements out of their lists ostensibly
leads to stricter conditions, while deleting an entire requirement list leads to
a weaker condition.

\paragraph{Obstruction Deletion}\ \\
The case of obstructions is more straightforward, and so we start by defining a
strategy that removes an obstruction if its removal does not change the
underlying set of gridded permutations. We call this strategy Obstruction
Deletion. Formally, for any gridded permutation $h$, define the equivalence
strategy $\ObsDel_h$ as follows:
\begin{itemize}
  \renewcommand\labelitemi{--}
  \item If $\TT = ((t,u), \OO, \RR)$ is a tiling and if $h \in \OO$, then define
  $\TT' = ((t,u), \OO \smallsetminus \{h\}, \RR)$. If $\Grid(\TT) =
  \Grid(\TT')$, then we define $d_{\ObsDel_h}(\TT) = \TT'$. Otherwise
  $d_{\ObsDel_h}(\TT) = \DNA$.
  \item The reliance profile function is $r_{\ObsDel_h}(n) = (n)$.
  \item The counting functions are $c_{\ObsDel_h, (n)}((a_0, \ldots, a_n)) = a_n$.
\end{itemize}

We should confess here that it may seem strange that we have defined a strategy
that, by its very definition, makes absolutely no change to the actual
combinatorial set under consideration---the more typical situation is that an
equivalence strategy outputs a set different from its input, but that the two
sets are equinumerous. Indeed, since this strategy only alters the
representation of the set (the tiling), and not the set itself, we could have
just described this as a computational step to simplify our representation
completely independent of the strategic framework. However, there are cases
where there are multiple tilings that could be used to describe the same set of
gridded permutations, and (1) it is not clear whether one should be considered
``simpler'' than the other and (2) it may be algorithmically expensive to detect
this. By writing such transformations as combinatorial rules, this information
is preserved in our universe of rules and any version of the tiling can be used
to construct a combinatorial specification.

Since $\ObsDel_h$ is claimed to be an equivalence strategy, we do not need to
prove that it is productive. In order to justify that it is an equivalence
strategy, we would only need to verify that $|\Grid_n(\TT)| = |\Grid_n(\TT')|$, but
this is true by definition; we only apply the strategy in cases where the
underlying gridded permutations do not change.

What this theoretical definition does not even begin to make clear is how, in
the process of Combinatorial Exploration, we detect when it can be applied. It
is, perhaps surprisingly, not always obvious when an obstruction can be deleted.
(This is due to the complication added by the notion of requirements.) What we
describe here is one sufficient condition that guarantees an obstruction can be
deleted---when we detect this condition, we apply the strategy.

Suppose a tiling $\TT$ has obstructions $h_1, h_2 \in \OO$ with $h_1 \leq h_2$.
Every gridded permutation that avoids $h_1$ also avoids $h_2$. Therefore, to any
tiling $\TT$ we may apply the strategy $\ObsDel_h$ for any non-minimal $h \in
\OO$.

\paragraph{Requirement Deletion}\ \\
There is a similar circumstance in which requirements can be deleted without
altering the underlying set of gridded permutations. As before, we first define
the strategy $\ReqDel$, which gives no sense of in which situations it may be
applied, and then we give a sufficient condition for a valid application of the
strategy.

\begin{itemize}
  \renewcommand\labelitemi{--}
  \item If $\TT = ((t,u), \OO, \RR)$ is a tiling and if $r \in \RR_i$, then
  define
  \[
    \TT' = ((t,u), \OO, \{\RR_1, \ldots, \RR_{i-1}, \RR_i \smallsetminus \{r\},
    \RR_{i+1}, \ldots, \RR_k\}).
  \]
  If $\Grid(\TT) =  \Grid(\TT')$, then we define  $d_{\ReqDel_{r,i}}(\TT) =
  \TT'$. Otherwise  $d_{\ReqDel_{r,i}}(\TT) = \DNA$.
  \item The reliance profile function is $r_{\ReqDel_{r,i}}(n) = (n)$.
  \item The counting functions are $c_{\ReqDel_{r,i}, (n)}((a_0, \ldots, a_n)) = a_n$.
\end{itemize}

As with $\ObsDel$, $\ReqDel$ is an equivalence strategy that does not change the
underlying set of gridded permutations. Suppose that a requirement list $\RR_i$
contains two gridded permutations $r_1$ and $r_2$ with $r_1 \leq r_2$. Then, any
gridded permutation that contains $r_2$ also contains $r_1$, and so $r_2$ can be
deleted from $\RR_i$ without changing the underlying gridded permutations.
Therefore, a condition sufficient to ensure that $\ReqDel_{r,i}$ can be applied
to $\TT$ is that $r \in \RR_i$ is non-minimal among $\RR_i$.

\paragraph{Requirement List Deletion}\ \\
It is sometimes possible that an entire requirement list can be deleted without
changing the underlying set of gridded permutations. Once again, we start by
formally defining the strategy $\ReqListDel$:
\begin{itemize}

  \item If $\TT = ((t,u), \OO, \RR)$ is a tiling and $1 \leq i \leq |\RR|$ then
  define
  \[
    \TT' = ((t,u), \OO, \{\RR_1, \ldots, \RR_{i-1}, \RR_{i+1}, \ldots, \RR_k\}).
  \]
  If $\Grid(\TT) =  \Grid(\TT')$, then we define  $d_{\ReqListDel_{i}}(\TT) =
  \TT'$. Otherwise  $d_{\ReqListDel_{i}}(\TT) = \DNA$.
  \item The reliance profile function is $r_{\ReqListDel_{i}}(n) = (n)$.
  \item The counting functions are $c_{\ReqListDel_{i}, (n)}((a_0, \ldots, a_n)) = a_n$.
\end{itemize}

$\ReqListDel$ is an equivalence strategy that can be applied in the following
scenario. Suppose $\RR_i \in \RR$ and there exists $j \neq i$ such that every
$r' \in \RR_j$ contains at least one $r \in \RR_i$. Then, every gridded
permutation that contains a requirement in the list $\RR_j$ also contains a
requirement in the list $\RR_i$, and so $\RR_i$ can be deleted without changing
the underlying set of gridded permutations.

In each of the descriptions of $\ObsDel$, $\ReqDel$, and $\ReqListDel$, we gave
a sufficient but not necessary condition ensuring that each strategy could be
applied. We discuss more general approaches in Subsection~\ref{subsubsection:obs-inf}.

In the figures in this section that show examples of the application of
strategies, we will often implicitly apply these obstruction and requirement
simplification strategies in order to make the pictures more readable.

\subsubsection{Point Placement}
\label{subsubsection:point-placement}

In describing the combinatorial specification for \(\Av(132)\) in
Section~\ref{section:combinatorial-exploration}, we made the observation that
the topmost point of any \(132\)-avoiding permutation could be isolated, and we
then made several inferences about the structure of such permutations. The
strategy described in this subsection, point placement, is a vast generalization
of this concept.

Point placement is an equivalence strategy that acts on a tiling by isolating
one point of a singleton requirement in a cell of its own and forcing that point
to be extreme in one of four directions: topmost, bottommost, leftmost, or
rightmost. Before a more detailed definition, consider the example shown in
Figure~\ref{figure:first-point-placement} which starts with a \(1 \times 1\)
tiling containing a \(132\) obstruction and a \(1\) requirement. The single
point of the requirement could be placed in any of the four extreme directions;
here we place it topmost. The result is a \(3 \times 3\) tiling in which the
point of the requirement is in the middle cell. To ensure that this point has
been placed in its own row and column, we add size \(1\) obstructions in cells
\((1,0)\), \((0,1)\), \((2,1)\), and \((1,2)\). To ensure that the placed point
is truly the topmost point, we add size \(1\) obstructions in cells \((0,2)\)
and \((2,2)\). Lastly, to be sure that the gridded permutations on this new
tiling still avoid \(132\), we add obstructions with the underlying \(132\) in
all possible ways\footnote{Many are redundant for obvious reasons, so we have
chosen for clarity to not draw these.}. This second tiling can be simplified to
the third tiling using a useful form of obstruction simplification that we
elaborate on in Subsection~\ref{subsubsection:obs-inf}: any gridded
permutation that can be drawn on this tiling contains a point in the cell
\((1,1)\), and thus avoiding the gridded permutation \((132,
((0,0),(1,1),(2,0)))\) is equivalent to avoiding the gridded subpermutation \((12,
((0,0), (2,0)))\). We can thus replace the former with the latter. The fourth
tiling is obtained using the Obstruction Deletion strategy to remove those
obstructions that contain \((12, ((0,0), (2,0)))\), and the fifth and final tiling
is the result of deleting the topmost row which cannot contain any points of any
gridded permutations anyway.

\begin{figure}
	\centering
	%!TEX root = combinatorial-exploration.tex
\begin{tikzpicture}[scale=1, baseline=(current bounding box.center)]
	\node at (0, 0) {
	\tiling{0.9}{1}{1}{}%
		{%
			{3/{(0.25, 0.25), (0.5, 0.75), (0.75, 0.5)}}%
		}{%
			{1/{(0.6, 0.4)}}%
		}};
\end{tikzpicture}
$\cong$
\begin{tikzpicture}[scale=1, baseline=(current bounding box.center)]
	\node at (0, 0) {
	\tiling{0.9}{3}{3}{1/1}%
		{%
			{3/{(0.25, 0.25), (0.5, 0.75), (0.75, 0.5)}},%
			{3/{(2.25, 0.25), (2.5, 0.75), (2.75, 0.5)}},%
			{3/{(0.8, 0.7), (1.75, 1.2), (2.25, 0.9)}},%
			{3/{(0.9, 0.5), (2.1, 0.75), (2.2, 0.65)}},%
			{3/{(0.6, 0.1), (0.75, 0.4), (2.1, 0.25)}}%
		}{}};
\end{tikzpicture}
$=$
\begin{tikzpicture}[scale=1, baseline=(current bounding box.center)]
	\node at (0, 0) {
	\tiling{0.9}{3}{3}{1/1}%
		{%
			{3/{(0.25, 0.25), (0.5, 0.75), (0.75, 0.5)}},%
			{3/{(2.25, 0.25), (2.5, 0.75), (2.75, 0.5)}},%
			{2/{(0.8, 0.7), (2.25, 0.9)}},%
			{3/{(0.9, 0.5), (2.1, 0.75), (2.2, 0.65)}},%
			{3/{(0.6, 0.1), (0.75, 0.4), (2.1, 0.25)}}%
		}{}};
\end{tikzpicture}
$=$
\begin{tikzpicture}[scale=1, baseline=(current bounding box.center)]
	\node at (0, 0) {
	\tiling{0.9}{3}{3}{1/1}%
		{%
			{3/{(0.25, 0.25), (0.5, 0.75), (0.75, 0.5)}},%
			{3/{(2.25, 0.25), (2.5, 0.75), (2.75, 0.5)}},%
			{2/{(0.8, 0.7), (2.25, 0.9)}}%
		}{}};
\end{tikzpicture}
$=$
\begin{tikzpicture}[scale=1, baseline=(current bounding box.center)]
	\node at (0, 0) {
	\tiling{0.9}{3}{2}{1/1}%
		{%
			{3/{(0.25, 0.25), (0.5, 0.75), (0.75, 0.5)}},%
			{3/{(2.25, 0.25), (2.5, 0.75), (2.75, 0.5)}},%
			{2/{(0.8, 0.7), (2.25, 0.9)}}%
		}{}};
\end{tikzpicture}
	\caption{The evolution of tilings that result from performing the
		Point Placement strategy on the leftmost tiling. The first and second
		tiling are equivalent in the sense that there is a size-preserving
		bijection between the sets of gridded permutations associated with each.
		The third, fourth,
		and fifth tilings arise after applying the discussed simplifications, and
		are all equal to the second tiling in the sense that the sets of gridded
		permutations associated with each are all equal.}
	\label{figure:first-point-placement}
\end{figure}
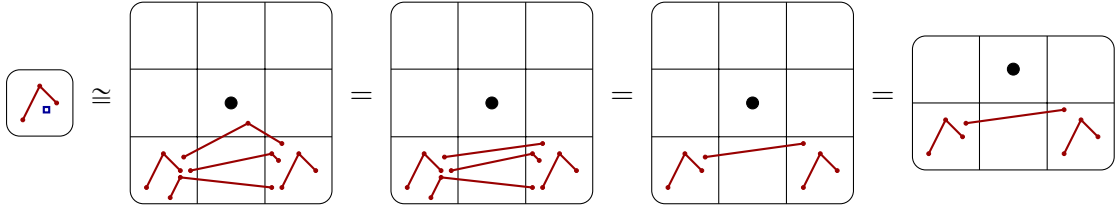

We took the time to point out each of these simplification steps individually
because in future examples they will always be applied, often without
comment, because otherwise the tilings produced by point placement have so many
obstructions that pictures of them become useless. Applications of point
placement can be more complicated than the example above in three ways: first,
we may be placing a point of a requirement that has size greater than \(1\);
second, there may be other requirement lists that have to be duplicated across
new cells in a manner similar to obstructions; and third, we often apply point
placement to tilings whose dimensions are larger than \(1 \times 1\).

More formally, consider a tiling \(\TT = ((t,u), \OO, \RR)\).
\begin{itemize}
	\renewcommand\labelitemi{--}
	\item If $\TT$ contains a singleton requirement list \(\RR_i = \{h\}\), if
		\(\ell\) is an index with \(1 \leq \ell \leq |h|\), and \(d \in
		\{\leftarrow, \rightarrow, \uparrow, \downarrow\}\) is a direction, then
		we define
		\[
			d_{\PointPl_{h, \ell, d}}(\TT) = \TT',
		\]
		where $\TT'$ is formed as described in the example above, a process
		fully explained in
		Subsection~\ref{subsubsection:point-placement-details}. Otherwise
		\(d_{\PointPl_{h, \ell, d}}(\TT) = \DNA\).
	\item The reliance profile function is \(r_{\PointPl_{h,\ell,d}}(n) = (n)\).
	\item The counting functions are \(c_{\PointPl_{h, \ell, d},(n)}((a_0, \ldots, a_n)) = a_n\).
\end{itemize}

The proof tree in Figure~\ref{fig:av1243_1342_2143} uses point placement several
times, so we take this opportunity to present these applications as further
examples.

\begin{figure}
	\centering
	\begin{tikzpicture}
	  \input{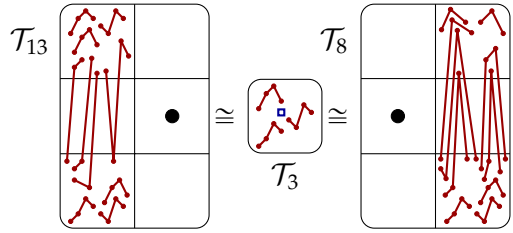}
	\end{tikzpicture}
	\caption{The tiling $\TT_3$ is equivalent to each of the tilings $\TT_8$
		and $\TT_{13}$, which are both produced by applying the point placement
		strategy to the same requirement in $\TT_3$, but with different
		directions.}
	\label{figure:big-tree-t3-to-t8}
\end{figure}

In Figure~\ref{figure:big-tree-t3-to-t8}, the strategy \(\PointPl_{(1,(0,0)), 1,
\leftarrow}\) is applied to the tiling $\TT_3$ to produce $\TT_8$, while
the strategy \(\PointPl_{(1,(0,0)), 1, \rightarrow}\)
is applied to the same tiling to produce $\TT_{13}$.

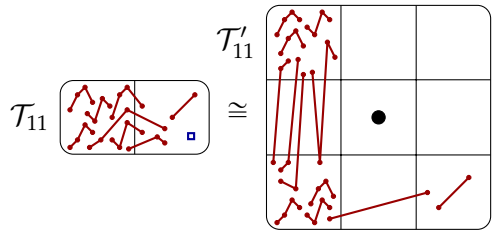
\begin{figure}
	\centering
	\begin{tikzpicture}[baseline=(current bounding box.center)]
	\node (t11) at (2.6, -11) {$\TT_{11}$};
	
	\node (middlenonempty) at (4, -11) {
		\tiling{1.0}{2}{1}{}%
		{%
			{4/{(0.14, 0.10), (0.24, 0.20), (0.34, 0.40), (0.44, 0.30)}},%
			{4/{(0.14, 0.60), (0.24, 0.80), (0.34, 0.90), (0.44, 0.70)}},%
			{4/{(0.37, 0.55), (0.47, 0.45), (0.57, 0.75), (0.67, 0.65)}},%
			{4/{(0.70, 0.20), (0.80, 0.10), (0.90, 0.43), (1.10, 0.30)}},%
			{4/{(0.40, 0.10), (0.55, 0.20), (0.90, 0.60), (1.40, 0.35)}},%
			{4/{(0.70, 0.50), (0.80, 0.80), (0.90, 0.90), (1.10, 0.65)}},%
			{3/{(0.92, 0.08), (1.30, 0.24), (1.40, 0.16)}},%
			{2/{(1.50, 0.50), (1.80, 0.80)}}%
		}{%
			{1/{(1.75, 0.25)}}%
		}%
	};
\end{tikzpicture}
$\cong$
\begin{tikzpicture}[baseline=(current bounding box.center)]
	\node (placedmiddle) at (7, -11) {
		\tiling{1.0}{3}{3}{1/1}%
		{%
			{4/{(0.15, 0.10), (0.25, 0.20), (0.35, 0.40), (0.45, 0.30)}},%
			{4/{(0.60, 0.35), (0.70, 0.55), (0.80, 0.65), (0.90, 0.45)}},%
			{4/{(0.55, 0.20), (0.65, 0.10), (0.75, 0.40), (0.85, 0.30)}},%
			{4/{(0.15, 2.60), (0.25, 2.80), (0.35, 2.90), (0.45, 2.80)}},%
			{4/{(0.55, 2.70), (0.65, 2.60), (0.75, 2.90), (0.85, 2.80)}},%
			{4/{(0.20, 2.35), (0.30, 2.55), (0.40, 2.65), (0.50, 2.45)}},%
			{4/{(0.62, 2.10), (0.72, 0.90), (0.82, 2.50), (0.92, 2.30)}},%
			{3/{(0.20, 0.65), (0.40, 0.55), (0.50, 2.07)}},%
			{3/{(0.20, 0.80), (0.30, 0.90), (0.43, 2.27)}},%
			{3/{(0.10, 0.90), (0.20, 2.15), (0.30, 2.25)}},%
			{2/{(2.30, 0.30), (2.70, 0.70)}},%
			{2/{(0.85, 0.15), (2.15, 0.50)}}%
		}{}%
	};
	\useasboundingbox (current bounding box.south west) rectangle (current bounding box.north east);
	\node (t11p) at (5.1, -10) {$\TT_{11}'$};
\end{tikzpicture}
	\caption{The tiling $\TT_{11}'$ is the result of applying point placement to $\TT_{11}$.}
	\label{figure:big-tree-t11-to-t11p}
\end{figure}

For an example in which the input tiling has dimensions larger than \(1 \times
1\), the tiling \(\TT_{11}'\) shown in Figure~\ref{figure:big-tree-t11-to-t11p} is the
result of applying \(\PointPl_{(1,(1,0)), 1, \leftarrow}\) to \(T_{11}\).

\begin{figure}
	\centering
	\begin{tikzpicture}[baseline=(current bounding box.center)]
	\node (contains12) at (3, -4) {
		\tiling{1.0}{1}{1}{}%
		{%
			{4/{(0.15, 0.10), (0.25, 0.20), (0.35, 0.40), (0.45, 0.30)}},%
			{4/{(0.15, 0.60), (0.25, 0.80), (0.35, 0.90), (0.45, 0.70)}},%
			{4/{(0.55, 0.45), (0.65, 0.35), (0.75, 0.65), (0.85, 0.55)}}%
		}{%
			{2/{(0.65, 0.15), (0.85, 0.35)}}%
		}%
	};
%	\useasboundingbox (current bounding box.south west) rectangle (current bounding box.north east);
	\node (t5) at (2.1, -4) {$\TT_5$};
\end{tikzpicture}
$\cong$
\begin{tikzpicture}[baseline=(current bounding box.center)]
	\node (nonemptyright) at (5.6, -4) {
		\tiling{1.0}{3}{3}{1/1}%
		{%
			{4/{(0.15, 0.10), (0.25, 0.20), (0.35, 0.40), (0.45, 0.30)}},%
			{4/{(0.60, 0.35), (0.70, 0.55), (0.80, 0.65), (0.90, 0.45)}},%
			{4/{(0.55, 0.20), (0.65, 0.10), (0.75, 0.40), (0.85, 0.30)}},%
			{4/{(0.15, 2.60), (0.25, 2.80), (0.35, 2.90), (0.45, 2.80)}},%
			{4/{(0.55, 2.70), (0.65, 2.60), (0.75, 2.90), (0.85, 2.80)}},%
			{4/{(0.20, 2.35), (0.30, 2.55), (0.40, 2.65), (0.50, 2.45)}},%
			{4/{(0.62, 2.10), (0.72, 0.90), (0.82, 2.50), (0.92, 2.30)}},%
			{3/{(0.20, 0.65), (0.40, 0.55), (0.50, 2.07)}},%
			{3/{(0.20, 0.80), (0.30, 0.90), (0.43, 2.27)}},%
			{3/{(0.10, 0.90), (0.20, 2.15), (0.30, 2.25)}},%
			{2/{(0.85, 0.15), (2.25, 0.75)}},%
			{2/{(2.25, 0.25), (2.75, 0.75)}}%
		}{%
			{1/{(0.55, 0.55)}}%
		}};
	\useasboundingbox (current bounding box.south west) rectangle (current bounding box.north east);
	\node (t55) at (3.75, -3) {$\TT_5'$};
\end{tikzpicture}
$\cong$
\begin{tikzpicture}[baseline=(current bounding box.center)]		
	\node (placed12) at (10.5, -4) {
		\tiling{1.0}{5}{5}{1/1, 3/3}%
		{%
			{2/{(2.85, 0.15), (4.25, 0.75)}},%
			{2/{(2.464, 0.75), (2.826, 4.50)}},%
			{2/{(2.40, 2.90), (2.60, 4.25)}},%
			{4/{(2.15, 0.10), (2.25, 0.20), (2.35, 0.40), (2.45, 0.30)}},%
			{4/{(2.60, 0.35), (2.70, 0.55), (2.80, 0.65), (2.90, 0.45)}},%
			{4/{(2.55, 0.20), (2.65, 0.10), (2.75, 0.40), (2.85, 0.30)}},%
			{4/{(2.13, 0.45), (2.23, 0.55), (2.31, 2.07), (2.42, 0.93)}},%
			{4/{(2.65, 0.66), (2.71, 0.90), (2.74, 2.10), (2.86, 0.76)}},%
			{4/{(2.61, 0.80), (2.67, 2.25), (2.80, 2.40), (2.93, 0.93)}},%
			{4/{(2.07, 0.80), (2.14, 0.67), (2.29, 2.64), (2.54, 0.93)}},%
			{3/{(2.08, 2.65), (2.20, 2.92), (2.40, 2.75)}},%
			{3/{(2.60, 2.75), (2.75, 2.90), (2.90, 2.60)}},%
			{3/{(2.07, 0.93), (2.22, 2.78), (2.48, 2.61)}},%
			{4/{(0.14, 4.10), (0.24, 4.20), (0.34, 4.40), (0.44, 4.30)}},%
			{4/{(0.14, 4.60), (0.24, 4.80), (0.34, 4.90), (0.44, 4.70)}},%
			{4/{(0.37, 4.55), (0.47, 4.45), (0.57, 4.75), (0.67, 4.65)}},%
			{4/{(0.70, 4.20), (0.80, 4.10), (0.90, 4.43), (2.10, 4.30)}},%
			{4/{(0.40, 4.10), (0.55, 4.20), (0.90, 4.60), (2.40, 4.35)}},%
			{4/{(0.70, 4.50), (0.80, 4.80), (0.90, 4.90), (2.10, 4.65)}},%
			{3/{(0.92, 4.08), (2.30, 4.24), (2.40, 4.16)}},%
			{2/{(2.50, 4.50), (2.80, 4.80)}},%
			{2/{(4.25, 0.25), (4.75, 0.75)}}%
		}{}
	};
	\useasboundingbox (current bounding box.south west) rectangle (current bounding box.north east);
	\node (t6) at (7.6, -2) {$\TT_5''$};
\end{tikzpicture}
	\caption{The tiling $\TT_5''$ is the result of applying point placement twice to $\TT_5$.}
	\label{figure:big-tree-t5-to-t6}
\end{figure}
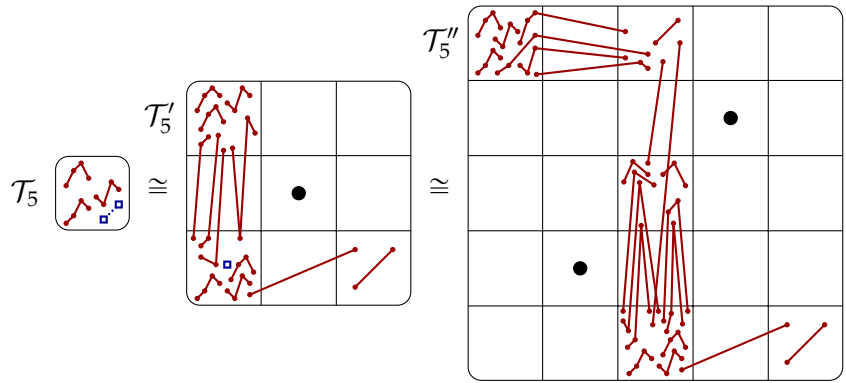

Finally, as demonstrated in Figure~\ref{figure:big-tree-t5-to-t6} the tiling
\(\TT_5''\) is produced by twice applying point placement to \(\TT_5\), placing
both points of a size two requirement. First, \(\PointPl_{(12,(0,0)), 2,
\rightarrow}\) is applied, placing the \(2\) in the requirement \((12,(0,0))\)
as far to the right as possible. To the result, \(\PointPl_{(1, (0,0)), 1,
\leftarrow}\) is applied, placing the sole point of the size \(1\) requirement
as far to the left as possible.

To prove that point placement is an equivalence strategy, we are required to
show the existence of a size-preserving bijection between any input tiling and
its corresponding output tiling. This requires rather a lot of bookkeeping, and
is verified in Subsection~\ref{subsubsection:point-placement-details}.

\subsubsection{Row Separation and Column Separation}
\label{subsubsection:row-col-sep}

The rightmost tiling in Figure~\ref{figure:first-point-placement} is the result
of placing a point into a \(1 \times 1\) tiling and performing several
simplifications. The strategies described in this subsection, row separation and
column separation, permit even further simplification of this tiling and others.

The presence of the \((12, (0,0), (2,0))\) obstruction in that tiling implies
that for any gridded permutation drawn on the tiling, any entries in cell
\((0,0)\) must lie above any entries in cell \((2,0)\). To capture this
information, we create a new tiling, shown on the right in
Figure~\ref{figure:row-col-sep-first-example}, in which the content of these two
cells has been separated into two rows.

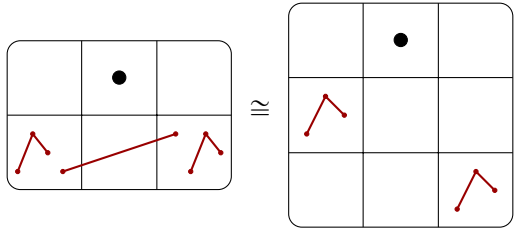
\begin{figure}
	\centering
	\begin{tikzpicture}[baseline=(current bounding box.center)]
		\node (nonemptytop3) at (4, -5) {
			\tiling{1.0}{3}{2}{1/1}%
			{%
				{3/{(0.15, 0.25), (0.35, 0.75), (0.55, 0.50)}},%
				{3/{(2.45, 0.25), (2.65, 0.75), (2.85, 0.50)}},%
				{2/{(0.75, 0.25), (2.25, 0.75)}}%
			}{}%
		};
	\end{tikzpicture}
	$\cong$
	\begin{tikzpicture}[baseline=(current bounding box.center)]
		\node (nonemptytop4) at (8, -5){
			\tiling{1.0}{3}{3}{1/2}%
			{%
				{3/{(0.25, 1.25), (0.5, 1.75), (0.75, 1.50)}},%
				{3/{(2.25, 0.25), (2.5, 0.75), (2.75, 0.50)}}%
			}{}%
		};
	\end{tikzpicture}
	\caption{An application of row separation.}
	\label{figure:row-col-sep-first-example}
\end{figure}

Another example of this strategy is found in the tiling \(\TT_5''\) (which itself was the
result of twice applying point placement to \(\TT_5\)) in the proof tree in
Figure~\ref{fig:av1243_1342_2143}. As we show in
Figure~\ref{figure:row-col-sep-second-example}, the obstruction
\((12,((2,0),(4,0)))\) implies that the two nonempty cells in the bottom row can
be separated. Further, the two obstructions \((12, ((2,0),(2,4)))\) and \((12,
((2,2),(2,4)))\) imply that any entries in cells \((2,0)\) and \((2,2)\) must
lie to the right of any entries in cell \((2,4)\), and so the cells in this
column can be separated, yielding the tiling $\TT_6$ in Figure~\ref{fig:av1243_1342_2143}.

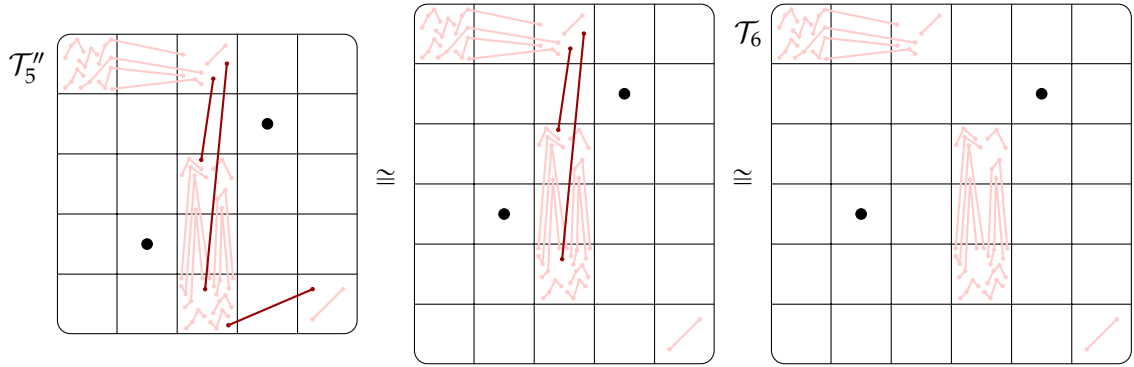
\begin{figure}
	\centering
	\begin{tikzpicture}[baseline=(current bounding box.center)]
	\node (n0) at (0, 0.45) {
		\tilinglight{0.8}{5}{5}{1/1, 3/3}%
		{%
			{2/{(2.85, 0.15), (4.25, 0.75)}},%
			{2/{(2.464, 0.75), (2.826, 4.50)}},%
			{2/{(2.40, 2.90), (2.60, 4.25)}}%
		}{}{%
			{4/{(2.15, 0.10), (2.25, 0.20), (2.35, 0.40), (2.45, 0.30)}},%
			{4/{(2.60, 0.35), (2.70, 0.55), (2.80, 0.65), (2.90, 0.45)}},%
			{4/{(2.55, 0.20), (2.65, 0.10), (2.75, 0.40), (2.85, 0.30)}},%
			{4/{(2.13, 0.45), (2.23, 0.55), (2.31, 2.07), (2.42, 0.93)}},%
			{4/{(2.65, 0.66), (2.71, 0.90), (2.74, 2.10), (2.86, 0.76)}},%
			{4/{(2.61, 0.80), (2.67, 2.25), (2.80, 2.40), (2.93, 0.93)}},%
			{4/{(2.07, 0.80), (2.14, 0.67), (2.29, 2.64), (2.54, 0.93)}},%
			{3/{(2.08, 2.65), (2.20, 2.92), (2.40, 2.75)}},%
			{3/{(2.60, 2.75), (2.75, 2.90), (2.90, 2.60)}},%
			{3/{(2.07, 0.93), (2.22, 2.78), (2.48, 2.61)}},%
			{4/{(0.14, 4.10), (0.24, 4.20), (0.34, 4.40), (0.44, 4.30)}},%
			{4/{(0.14, 4.60), (0.24, 4.80), (0.34, 4.90), (0.44, 4.70)}},%
			{4/{(0.37, 4.55), (0.47, 4.45), (0.57, 4.75), (0.67, 4.65)}},%
			{4/{(0.70, 4.20), (0.80, 4.10), (0.90, 4.43), (2.10, 4.30)}},%
			{4/{(0.40, 4.10), (0.55, 4.20), (0.90, 4.60), (2.40, 4.35)}},%
			{4/{(0.70, 4.50), (0.80, 4.80), (0.90, 4.90), (2.10, 4.65)}},%
			{3/{(0.92, 4.08), (2.30, 4.24), (2.40, 4.16)}},%
			{2/{(2.50, 4.50), (2.80, 4.80)}},%
			{2/{(4.25, 0.25), (4.75, 0.75)}}%
		}%
	};
	\useasboundingbox (current bounding box.south west) rectangle (current bounding box.north east);
	\node (t8) at (-2.4, 2) {$\TT_5''$};
\end{tikzpicture}
$\cong$
\begin{tikzpicture}[baseline=(current bounding box.center)]
	\node (n1) at (5.25, 0) {
		\tilinglight{0.8}{5}{6}{1/2, 3/4}%
		{%
			{2/{(2.464, 1.75), (2.826, 5.50)}},%
			{2/{(2.40, 3.90), (2.60, 5.25)}}%
		}{}{%
			{4/{(2.15, 1.10), (2.25, 1.20), (2.35, 1.40), (2.45, 1.30)}},%
			{4/{(2.60, 1.35), (2.70, 1.55), (2.80, 1.65), (2.90, 1.45)}},%
			{4/{(2.55, 1.20), (2.65, 1.10), (2.75, 1.40), (2.85, 1.30)}},%
			{4/{(2.13, 1.45), (2.23, 1.55), (2.31, 3.07), (2.42, 1.93)}},%
			{4/{(2.65, 1.66), (2.71, 1.90), (2.74, 3.10), (2.86, 1.76)}},%
			{4/{(2.61, 1.80), (2.67, 3.25), (2.80, 3.40), (2.93, 1.93)}},%
			{4/{(2.07, 1.80), (2.14, 1.67), (2.29, 3.64), (2.54, 1.93)}},%
			{3/{(2.08, 3.65), (2.20, 3.92), (2.40, 3.75)}},%
			{3/{(2.60, 3.75), (2.75, 3.90), (2.90, 3.60)}},%
			{3/{(2.07, 1.93), (2.22, 3.78), (2.48, 3.61)}},%
			{4/{(0.14, 5.10), (0.24, 5.20), (0.34, 5.40), (0.44, 5.30)}},%
			{4/{(0.14, 5.60), (0.24, 5.80), (0.34, 5.90), (0.44, 5.70)}},%
			{4/{(0.37, 5.55), (0.47, 5.45), (0.57, 5.75), (0.67, 5.65)}},%
			{4/{(0.70, 5.20), (0.80, 5.10), (0.90, 5.43), (2.10, 5.30)}},%
			{4/{(0.40, 5.10), (0.55, 5.20), (0.90, 5.60), (2.40, 5.35)}},%
			{4/{(0.70, 5.50), (0.80, 5.80), (0.90, 5.90), (2.10, 5.65)}},%
			{3/{(0.92, 5.08), (2.30, 5.24), (2.40, 5.16)}},%
			{2/{(2.50, 5.50), (2.80, 5.80)}},%
			{2/{(4.25, 0.25), (4.75, 0.75)}}%
		}%
	};
\end{tikzpicture}
$\cong$
\begin{tikzpicture}[baseline=(current bounding box.center)]
	\node (n2) at (11, 0) {
		\tilinglight{0.8}{6}{6}{1/2, 4/4}%
		{}{}{%
			{4/{(3.15, 1.10), (3.25, 1.20), (3.35, 1.40), (3.45, 1.30)}},%
			{4/{(3.60, 1.35), (3.70, 1.55), (3.80, 1.65), (3.90, 1.45)}},%
			{4/{(3.55, 1.20), (3.65, 1.10), (3.75, 1.40), (3.85, 1.30)}},%
			{4/{(3.13, 1.45), (3.23, 1.55), (3.31, 3.07), (3.42, 1.93)}},%
			{4/{(3.65, 1.66), (3.71, 1.90), (3.74, 3.10), (3.86, 1.76)}},%
			{4/{(3.61, 1.80), (3.67, 3.25), (3.80, 3.40), (3.93, 1.93)}},%
			{4/{(3.07, 1.80), (3.14, 1.67), (3.29, 3.64), (3.54, 1.93)}},%
			{3/{(3.08, 3.65), (3.20, 3.92), (3.40, 3.75)}},%
			{3/{(3.60, 3.75), (3.75, 3.90), (3.90, 3.60)}},%
			{3/{(3.07, 1.93), (3.22, 3.78), (3.48, 3.61)}},%
			{4/{(0.14, 5.10), (0.24, 5.20), (0.34, 5.40), (0.44, 5.30)}},%
			{4/{(0.14, 5.60), (0.24, 5.80), (0.34, 5.90), (0.44, 5.70)}},%
			{4/{(0.37, 5.55), (0.47, 5.45), (0.57, 5.75), (0.67, 5.65)}},%
			{4/{(0.70, 5.20), (0.80, 5.10), (0.90, 5.43), (2.10, 5.30)}},%
			{4/{(0.40, 5.10), (0.55, 5.20), (0.90, 5.60), (2.40, 5.35)}},%
			{4/{(0.70, 5.50), (0.80, 5.80), (0.90, 5.90), (2.10, 5.65)}},%
			{3/{(0.92, 5.08), (2.30, 5.24), (2.40, 5.16)}},%
			{2/{(2.50, 5.50), (2.80, 5.80)}},%
			{2/{(5.25, 0.25), (5.75, 0.75)}}%
		}%
	};
	\useasboundingbox (current bounding box.south west) rectangle (current bounding box.north east);
	\node (t8) at (8.3, 2) {$\TT_6$};
\end{tikzpicture}
	\caption{The tiling $\TT_6$ is the result of applying row separation and then
		column separation to $\TT_5''$.}
	\label{figure:row-col-sep-second-example}
\end{figure}

We now give the formal definition of the row separation equivalence strategy;
the column separation strategy is defined similarly. Consider a tiling \(\TT =
((t,u), \OO, \RR)\). For a particular row \(r\), let \(S\) be a nonempty subset
of the nonempty cells in \(r\) and let \(S'\) denote the remaining nonempty
cells in \(r\). We can apply row separation to row \(r\), splitting it in two
rows in which the lower row inherits the cells from \(S\) and the upper row
inherits the cells from \(S'\), if there is no gridded permutation that can be
drawn on \(\TT\) that possesses an entry in a cell in \(S\) whose value is larger
than an entry in a cell in \(S'\). A sufficient condition to ensure this is the
following: for every pair of cells \((c,c') \in S \times S'\), if \(c\) is to
the left of \(c'\) then \(\OO\) contains the obstruction \((21, (c, c'))\),
otherwise if \(c\) is to the right of \(c'\) then \(\OO\) contains the
obstruction \((12, (c', c))\).

As the examples above suffice to understand the idea of row and column
separation, our formal definition below omits the full details, which are
deferred to Subsection~\ref{subsubsection:row-col-sep-details}.
\begin{itemize}
	\renewcommand\labelitemi{--}
	\item For a tiling \(\TT = ((t,u), \OO, \RR)\), a row \(r\), and a nonempty
	subset \(S\) of the nonempty cells in row \(r\), if the obstructions
	described above are all present, then we define
	\[
		d_{\RowSep_{r,S}}(\TT) = \TT'
	\]
	where \(\TT'\) is the tiling in which the cells in \(S\) have been moved
	into a separate row below the cells in row \(r\) but not in \(S\), and the
	obstructions and requirements have been adjusted accordingly. Otherwise
	\(d_{\RowSep_{r,S}}(\TT) = \DNA\).
	\item The reliance profile function is \(r_{\RowSep_{r,S}}(n) = (n)\).
	\item The counting functions are \(c_{\RowSep_{r,S}}((a_0, \ldots, a_n)) = a_n\).
\end{itemize}

The column separation strategy \(\ColSep_{c,S}\) is defined analogously.

\subsubsection{Factor}
\label{subsubsection:factor}

Most of the strategies we have already described transform tilings into new
tilings that are in some way more complicated---larger, more obstructions, more
requirements, etc. Factorization is a strategy that identifies when pieces of a
tiling do not interact with each other, and splits them into several subtilings.
This tends to lead to combinatorial specifications that are recursive.

We say that two subsets \(S_1\) and \(S_2\) of cells of a tiling are
\emph{non-interacting} if no cell of \(S_1\) shares a row or column with any
cell of \(S_2\) and if there is no obstruction or requirement list that involves
cells in both \(S_1\) and \(S_2\). Figure~\ref{figure:factor-example} shows a
\(6 \times 6\) tiling that has 5 minimal pairwise non-interacting sets of cells:
\[
	\{(0,5), (2,5)\} \qquad \{(1,2)\} \qquad \{(3,1),(3,3)\} \qquad \{(4,4)\} \qquad \{(5,0)\},
\]
although the factorization actually performed in that figure leaves the
non-interacting sets \(\{(1,2)\}\) and \(\{(3,1),(3,3)\}\) together.
Why do we not use the full factorization? This is the magic of Combinatorial
Exploration---it has discovered that this partial factorization permits the
discovery of the proof tree from Figure~\ref{fig:av1243_1342_2143} that we are
currently discussing, whereas a human searching for this proof tree
by hand may not have gone down this path.

\begin{figure}[t]
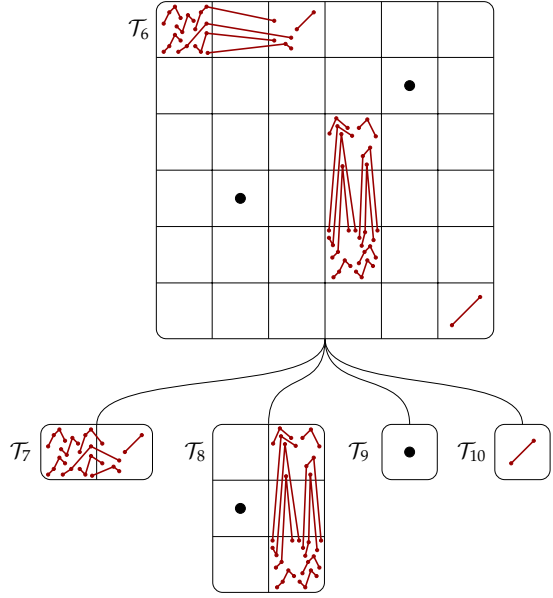

	\begin{center}
	\scalebox{0.75}{
		\begin{tikzpicture}
		   \node (t6) at (4.2, -0.5) {$\TT_6$};

		   \input{bigprooftree-large-tiling.tex}
		   \input{bigprooftree-last-two-factors-of-large-tiling.tex}

		   \begin{scope}[xshift=7.5cm, yshift=-1cm]
			   \input{bigprooftree-left-placement-subtiling.tex}
		   \end{scope}

		   \begin{scope}[xshift=-1cm]
			   \node (t8) at (6.2, -8) {$\TT_8$};
			   \input{bigprooftree-topmost-factor-of-large-tiling.tex}
		   \end{scope}
	    \end{tikzpicture}
	}
	\end{center}
	\caption{The tiling $\TT_6$ factors into four tilings.}
	\label{figure:factor-example}
\end{figure}

The counting functions of the factorization strategy are more interesting than
the other strategies we have discussed. Suppose \(\AA\) is a tiling that factors
into two tilings \(\BB^{(1)}\) and \(\BB^{(2)}\). The non-interactivity of the
cells that became \(\BB^{(1)}\) with the cells that became \(\BB^{(2)}\) imply
that each gridded permutation \(\alpha\) of size \(n\) that can be drawn on
\(\AA\) can be formed uniquely from a pair \(\beta_1, \beta_2\) where
\(\beta_1\) can be drawn on \(\BB^{(1)}\), \(\beta_2\) can be drawn
\(\BB^{(2)}\), and \(|\beta_1| + |\beta_2| = n\). Therefore,
\[
	|\AA_n| = \sum_{i=0}^n |\BB^{(1)}_i||\BB^{(2)}_{n-i}|.
\]
More generally, when a tiling \(\AA\) factors into $m$ tilings \(\BB^{(1)},
\ldots, \BB^{(m)}\) we have
\[
	|\AA_n| = \sum_{i_1 + \cdots + i_m = n} |\BB^{(1)}_{i_1}| \cdots |\BB^{(m)}_{i_m}|.
\]
However, for the first time, we are in danger of defining a strategy that does
not satisfy the productivity conditions defined in
Section~\ref{section:productivity}. Consider, for example, the factorization
shown in Figure~\ref{figure:factor-non-productive} in which a tiling is factored
into two subtilings, one containing just a point, and the other containing two
cells. Given the general counting formula above, it feels natural to define the
reliance profile function of a strategy that factors a tiling into two
subtilings to be
\[
	r(n) = (n, n).
\]

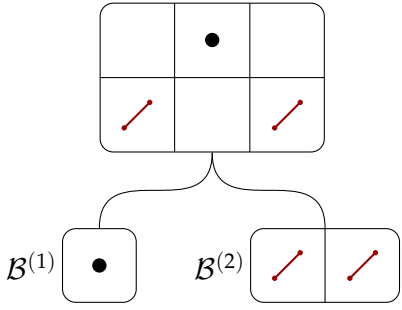
\begin{figure}
	\begin{center}
	\begin{tikzpicture}
		\node at (0,0) {
			\tiling{1.0}{3}{2}{1/1}{%
				{2/{(0.33, 0.33), (0.67, 0.67)}},%
				{2/{(2.33, 0.33), (2.67, 0.67)}}%
			}{}%
		};
		\node at (-1.5, -2.5) {\tiling{1.0}{1}{1}{0/0}{}{}};
		\node at (-2.4, -2.5) {$\BB^{(1)}$};
		\node at (1.5, -2.5) {
			\tiling{1.0}{2}{1}{}{%
				{2/{(0.33, 0.33), (0.67, 0.67)}},%
				{2/{(1.33, 0.33), (1.67, 0.67)}}%
			}{}%
		};
		\node at (0.1, -2.5) {$\BB^{(2)}$};
		\ptedge{(0, 0)}{(-0.5, 0.3)}{(-1.5, -2.5)}{(-0.5,0.8)}
		\ptedge{(0, 0)}{(-0.5, 0.3)}{(1.5, -2.5)}{(-0.5,0.8)}
    \end{tikzpicture}
	\end{center}
	\caption{An example of factoring that demonstrates that productivity
		relies on careful definition of different reliance profile functions
		for different situations.}
	\label{figure:factor-non-productive}
\end{figure}

\label{pageref:potential-productivity-problem}Figure~\ref{figure:factor-non-productive}
reveals a problem: the subtiling \(\BB^{(2)}\) has counting sequence
\(|\BB^{(2)}_n| = 2^n\) while the original tiling has a termwise strictly
smaller counting sequence \(|\AA_n| = 2^{n-1}\), violating Condition 2(a) of
Definition~\ref{definition:productive-strategy} that requires
\begin{quote}
	If \(\AA_N\) relies on \(\BB_N^{(i)}\) for some \(N \in \N\), then
	\(|\AA_n| \geq |\BB_n^{(i)}|\) for all \(n \in \N\).
\end{quote}
This problem is easily mitigated with a bit of care. In fact, the enumeration of
\(\AA_n\) does not actually depend on the enumeration of \(\BB^{(2)}_n\) because
\(|\BB^{(1)}_0| = 0\), simplifying the counting formula slightly by eliminating
the \(i=0\) term in the summation:
\[
	|\AA_n| = \sum_{i=0}^n |\BB^{(1)}_i||\BB^{(2)}_{n-i}| = \sum_{i=1}^n |\BB^{(1)}_i||\BB^{(2)}_{n-i}|.
\]
More generally, we will define the factorization strategy so that the
enumeration of \(\AA_n\) for the parent tiling \(\AA\) only depends on the
enumeration of the size \(n\) gridded permutations in a child tiling
\(\BB^{(i)}\) when absolutely required. We prove in
Subsection~\ref{subsubsection:factorization-details} that the resulting strategy
is always productive. According to this new definition, the reliance profile
function for the factorization in Figure~\ref{figure:factor-non-productive} is
\[
	r(n) = (n, n-1).
\]

Formally, consider a tiling \(\TT = ((t,u), \OO, \RR)\).
\begin{itemize}
	\renewcommand\labelitemi{--}
	\item Let \(P\) be a partition of the nonempty cells of \(\TT\) into \(m\)
	parts, and for concreteness consider the parts of \(P\) to be indexed in
	increasing order by their lexicographically smallest cell. If the cells of
	any part of \(P\) interact with the cells of any other part, then \(\TT\)
	cannot be factored according to this partition of cells. Thus, assume now
	that \(P\) is such that the parts are non-interacting, so that \(\TT\) will
	be factored into subtilings \(\BB^{(1)}, \ldots, \BB^{(m)}\). Assume that
	each of these subtilings contains at least one gridded permutation of size
	at least \(1\). Define
	\[
		S = \{i \in \{1, \ldots, m\} \,:\, |\BB^{(j)}_0| = 0\text{ for some }j \neq i\}.
	\]
	If \(i \in S\), then the enumeration of \(\TT_n\) will \emph{not} rely on
	the enumeration of \(\BB^{(i)}_n\) because \(|\BB^{(j)}_0| = 0\) for some
	\(j \neq i\).

	With such \(P\) and \(S\), we define
	\[
		d_{\Factor_{P,S}}(\TT) = (\BB^{(1)}, \ldots, \BB^{(m)}).
	\]

	In the case of an incompatible partition \(P\) or set \(S\), or where some
	\(\BB^{(i)}\) does not contain a gridded permutation of size at least
	\(1\), we define
	\[
		d_{\Factor_{P,S}}(\TT) = \DNA,
	\]
	as usual.

	\item Fix a partition \(P\) and its corresponding set \(S\). The reliance
	profile function is
	\[
		r_{\Factor_{P,S}}(n) = (r^{(1)}(n), \ldots, r^{(m)}(n))
	\]
	where
	\[
		r^{(i)}(n) =
			\begin{cases}
				n-1, & \text{if }i \in S\\
				n, & \text{if }i \not\in S
			\end{cases}.
	\]
	\item To describe the counting function we first define vectors of
	indeterminates
	\[
		b^{(i)} =
			\begin{cases}
				(b^{(i)}_0, \ldots, b^{(i)}_{n-1}), & \text{if }i \in S\\[5pt]
				(b^{(i)}_0, \ldots, b^{(i)}_{n}), & \text{if }i \not\in S\\
			\end{cases}.
	\]
	The counting functions are
	\[
			c_{\Factor_{P,S},(n)}(b^{(1)}, \ldots, b^{(m)}) = \sum_{(i_1, \ldots, i_m) \in I} b^{(1)}_{i_1} \cdots b^{(m)}_{i_m},
	\]
	where the sum is over
	\[
		I = \{(i_1, \ldots, i_m) \in \{0, \ldots, n\}^m \,:\, i_1 + \cdots + i_m = n\text{ and }i_\ell \neq n\text{ if }\ell \in S\}.
	\]
\end{itemize}

To illustrate this slightly intensive definition, the strategy applied to the
parent tiling in Figure~\ref{figure:factor-example} is \(\Factor_{P,S}\) where
\[
	P = \{ \quad \{(0,5),(2,5)\}, \quad \{(1,2),(3,1),(3,3)\}, \quad \{(4,4)\}, \quad \{(5,0)\} \quad \}
\]
and \(S = \{1,2,3,4\}\). The reliance profile function is
\[
	r_{\Factor_{P,S}}(n) = (n-1, n-1, n-1, n-1)
\]
and the counting functions are
\[
	c_{\Factor_{P,S}, (n)}(b^{(1)}, b^{(2)}, b^{(3)}, b^{(4)}) = \sum_{
		\substack{
			i_1 + i_2 + i_3 + i_4 = n\\
			0 \leq i_1, i_2, i_3, i_4 \leq n-1
		}
	} b^{(1)}_{i_1}b^{(2)}_{i_2}b^{(3)}_{i_3}b^{(4)}_{i_4}.
\]

\subsubsection{Obstruction Inferral}
\label{subsubsection:obs-inf}

It is sometimes possible to place an additional obstruction onto a tiling
without changing the underlying set of gridded permutations. For a rather
trivial example, see the tiling in Figure~\ref{figure:obs-inf-ex-1}, in which
the obstruction \((1, (1,0))\) can be added. Sometimes, but not always, this
inferred obstruction is a subobstruction of an existing obstruction.

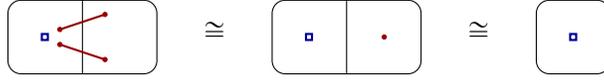
\begin{figure}
	\begin{center}
	\begin{tikzpicture}[baseline=(current bounding box.center)]
		\node at (0,0) {
			\tiling{1.0}{2}{1}{}{%
				{2/{(0.7, 0.6), (1.3, 0.8)}},%
				{2/{(0.7, 0.4), (1.3, 0.2)}}%
			}{%
				{1/{(0.5, 0.5)}}%
			}%
		};
    \end{tikzpicture}
    \quad $\cong$ \quad
    \begin{tikzpicture}[baseline=(current bounding box.center)]
		\node at (0,0) {
			\tiling{1.0}{2}{1}{}{%
				{1/{(1.5, 0.5)}}%
			}{%
				{1/{(0.5, 0.5)}}%
			}%
		};
    \end{tikzpicture}
    \quad $\cong$ \quad
    \begin{tikzpicture}[baseline=(current bounding box.center)]
		\node at (0,0) {
			\tiling{1.0}{1}{1}{}{%
			}{%
				{1/{(0.5, 0.5)}}%
			}%
		};
    \end{tikzpicture}
	\end{center}
	\caption{On the left, an example of a tiling to which the
		obstruction inferral strategy can be applied. In the middle,
		the result of applying obstruction inferral. On the right,
		an equivalent tiling that results from deleting the
		rightmost cell, which cannot contain any entries.}
	\label{figure:obs-inf-ex-1}
\end{figure}

As with some of the strategies that we have previously introduced, we will start
with a simple definition and then follow it with the more detailed discussion
about when the strategy can be applied. Formally, the strategy \(\ObsInf_h\) is
defined as follows.
\begin{itemize}
	\renewcommand\labelitemi{--}
	\item If \(\TT\) is a tiling with dimensions \(t \times u\), \(h \in
	\GG^{(t,u)} \smallsetminus \OO\), and
	\[
	  \Grid(\TT) = \Grid(((t,u), \OO \cup \{h\}, \RR)),
	\]
	then $d_{\ObsInf_h}(\TT) = ((t,u), \OO \cup \{h\}, \RR)$. Otherwise
	$d_{\ObsInf_h}(\TT) = \DNA$.
	\item The reliance profile function is $r_{\ObsInf_h}(n) = (n)$.
	\item The counting functions are $c_{\ObsInf_h, (n)}((a_0,\ldots,a_n)) = a_n$.
\end{itemize}

Like many of our previous strategies, Obstruction Inferral does not actually
change the set of gridded permutations at all, only altering the tiling
representation used for the set. As such, it is true by definition that
Obstruction Inferral is a valid equivalence strategy.

We will now discuss two ways to determine when Obstruction Inferral can be applied
to a tiling, the first being more limited but computationally easy and the second
being fully general but computationally intensive.

\paragraph{First Obstruction Inferral Case}\ \\
In discussing the point placement strategy in Section~\ref{subsubsection:point-placement}, we
mentioned in Figure~\ref{figure:first-point-placement}
that the obstruction \((132, ((0,0), (1,1), (2,0)))\)
could be replaced with the subobstruction \((12, ((0,0),(2,0)))\) because of the
point cell \((1,1)\). This phenomenon can be detected more generally. Suppose
\(h\) is an obstruction on the tiling \(\TT\) and partition the entries of \(h\)
maximally such that no two entries in different parts occur in cells that share
a row or column. Call this partition \(P\). If there exists a part \(p \in P\)
such that the subgridded permutation formed by the entries in \(p\) is contained
in every requirement in some requirement list, then the smaller obstruction
formed by deleting the entries in \(p\) from \(h\) can be added to \(\TT\)
without changing the underlying set of gridded permutations.

In the example from Figure~\ref{figure:first-point-placement}, \(P\) has two
parts, the first containing the subgridded permutation \((12, ((0,0),(2,0)))\)
and the second containing the subgridded permutation \((1, ((1,1)))\). Since
\((1,1)\) is a point cell, the subgridded permutation coming from this part of
the partition can be deleted from the obstruction.

To see that this claim is true in general, suppose the tiling \(T\) has an
obstruction \(h\), and that in the partition \(P\) defined above there is part
\(p\) such that the subobstruction \(h'\) arising from deleting the subgridded
permutation corresponding to \(p\) from \(h\) can be inferred. Let \(R\) be the
requirement list that enabled this inferral (i.e., every requirement in \(R\)
contains as a pattern the subgridded permutation formed by the entries in
\(p\)). We claim that no gridded permutation that can be drawn on \(\TT\)
contains the pattern \(h'\). To the contrary, suppose there were such a
permutation \(\pi\) that contained \(h'\). Since \(\pi \in \Grid(\TT)\), it must
also contain at least one of the requirements in \(R\). All of these requirements
contain the pattern formed by the entries in \(p\). Crucially, since the cells
involved in \(p\) share no rows nor columns with the cells involved in other
parts of \(P\), the only manner in which a gridded permutation can contain
\(h'\) and \(R\) is by containing the full obstruction \(h\).

\paragraph{Second Obstruction Inferral Case}\ \\
Theorem~\ref{theorem:is-empty-check} shows that it can be checked in
finite time whether \(\Grid(\TT) = \emptyset\) for any given tiling \(\TT\). We
can use this fact to determine precisely when an obstruction may be
inferred, although the computational burden is significant.

To infer whether an obstruction \(h\) may be added onto the tiling \(\TT
= ((t,u), \OO, \RR))\), we form a new tiling \(\TT'\) in which the potential new
obstruction is instead placed in a new requirement list:
\[
	\TT' = ((t,u), \OO, \RR \cup \{h\}).
\]

If \(\Grid(\TT')\) is empty, we can conclude that no gridded permutation that
can be drawn on \(\TT\) contains the pattern \(h\). Therefore, \(h\) may be
inferred as an obstruction on \(\TT\) without eliminating any of its gridded
permutations.

\subsection{Verification Strategies}\label{subsubsec:verification}

In Subsection~\ref{subsection:comb-specs}, we discussed the concept of
``verification strategies'' -- nullary strategies that represent when the
enumeration of a combinatorial set is known independently from the current
Combinatorial Exploration process. The proof tree for $\Av(1243, 1342, 2143)$ in
Figure~\ref{fig:av1243_1342_2143} uses
four different verification strategies, and so now that we have introduced several
strategies in the previous subsection, we will take this opportunity to trace
through the tree and see how they are applied, pointing out the verification
strategies as we encounter them.

We want to emphasize that this proof tree is the \emph{output} of successful
Combinatorial Exploration. The process of finding this tree involves the
discovery of many combinatorial rules, and sometime after all rules
in this particular tree have been discovered, the algorithm notices (using
Algorithm~\ref{algorithm:specfinder} in Subsection~\ref{subsection:tree-searcher})
that this particular subset of rules forms a combinatorial specification.

%!TEX root = combinatorial-exploration.tex

\begin{figure}
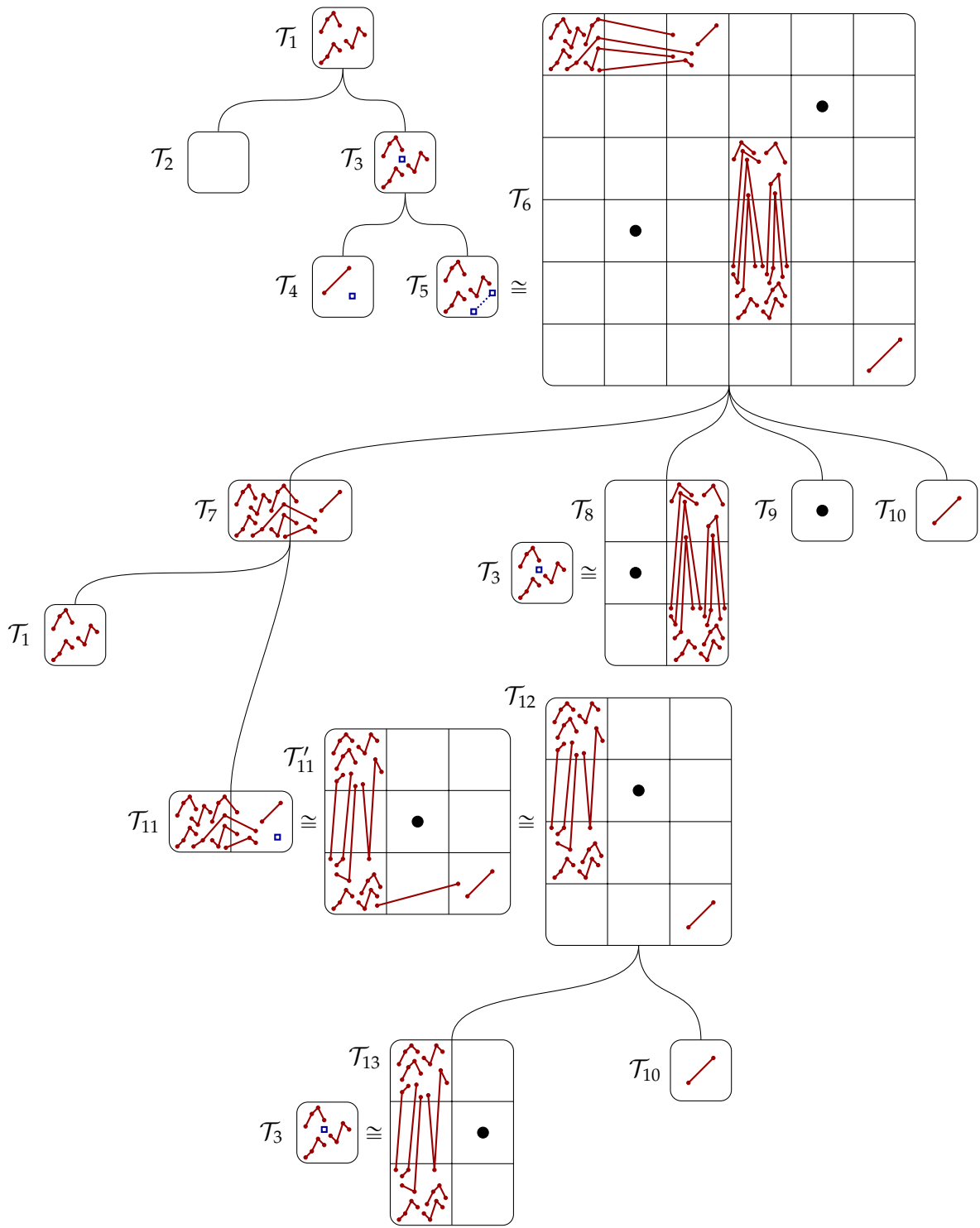

    \centering
    \begin{tikzpicture}[scale=1]

	  \input{bigprooftree-insertion.tex}

      \node (equiv) at (2.83 , -4) {$\cong$};
%      \node (t6) at (7.2, 0.8) {$\TT_6 = \fcsep{3}{\frsep{1}{\fplace{(1,(0,0))_1}{\leftarrow}{}}}$};
%      \node (t6) at (7.93, 0.3) {$\place{(12,(0,0))_2}{\rightarrow}{\TT_5})))$};

	  \begin{scope}[shift={(-1.3,0.4)}]
	  
		  \node (t6) at (4.15, -3) {$\TT_6$};
	  	
		  \input{bigprooftree-large-tiling.tex}
	
		  \input{bigprooftree-last-two-factors-of-large-tiling.tex}
	
	      \begin{scope}[xshift=7.5cm, yshift=-1cm]
		      \input{bigprooftree-left-placement.tex}
	      \end{scope}
	
	      \begin{scope}[xshift=-4cm]
			  \input{bigprooftree-topmost-factor-of-large-tiling.tex}
	%	      \node (t3) at (0.5, -10.75) {$\inso{(1,(1,0))}{\TT_7} = \TT_1$};
	          \node (t3) at (0.1, -10) {$\TT_1$};
		      \node (middleempty) at (1, -10) {
		      \tiling{1.0}{1}{1}{}%
		      {%
		        {4/{(0.15, 0.10), (0.25, 0.20), (0.35, 0.40), (0.45, 0.30)}},%
		        {4/{(0.15, 0.60), (0.25, 0.80), (0.35, 0.90), (0.45, 0.70)}},%
		        {4/{(0.55, 0.45), (0.65, 0.35), (0.75, 0.65), (0.85, 0.55)}}%
		      }
		      {%
		      }};
		      \begin{scope}[xshift=-0.5cm, yshift=-2cm]
	%		      \node (t11) at (2.7, -11.75) {$\insr{(1,(1,0))}{\TT_7} = \TT_{11}$};
			      \node (t11) at (2.6, -11) {$\TT_{11}$};
			      \node (middlenonempty) at (4, -11) {
			      \tiling{1.0}{2}{1}{}%
			      {%
			        {4/{(0.14, 0.10), (0.24, 0.20), (0.34, 0.40), (0.44, 0.30)}},%
			        {4/{(0.14, 0.60), (0.24, 0.80), (0.34, 0.90), (0.44, 0.70)}},%
			        {4/{(0.37, 0.55), (0.47, 0.45), (0.57, 0.75), (0.67, 0.65)}},%
			        {4/{(0.70, 0.20), (0.80, 0.10), (0.90, 0.43), (1.10, 0.30)}},%
			        {4/{(0.40, 0.10), (0.55, 0.20), (0.90, 0.60), (1.40, 0.35)}},%
			        {4/{(0.70, 0.50), (0.80, 0.80), (0.90, 0.90), (1.10, 0.65)}},%
			        {3/{(0.92, 0.08), (1.30, 0.24), (1.40, 0.16)}},%
			        {2/{(1.50, 0.50), (1.80, 0.80)}}%
			      }
			      {%
			        {1/{(1.75, 0.25)}}%
			      }};
			      \node (equiv4) at (5.25, -11) {$\cong$};
			      \node (t11p) at (5.1, -10) {$\TT_{11}'$};

				\node (placedmiddle) at (7, -11) {
				\tiling{1.0}{3}{3}{1/1}%
			      {%
			        {4/{(0.15, 0.10), (0.25, 0.20), (0.35, 0.40), (0.45, 0.30)}},%
			        {4/{(0.60, 0.35), (0.70, 0.55), (0.80, 0.65), (0.90, 0.45)}},%
			        {4/{(0.55, 0.20), (0.65, 0.10), (0.75, 0.40), (0.85, 0.30)}},%
			        {4/{(0.15, 2.60), (0.25, 2.80), (0.35, 2.90), (0.45, 2.80)}},%
			        {4/{(0.55, 2.70), (0.65, 2.60), (0.75, 2.90), (0.85, 2.80)}},%
			        {4/{(0.20, 2.35), (0.30, 2.55), (0.40, 2.65), (0.50, 2.45)}},%
			        {4/{(0.62, 2.10), (0.72, 0.90), (0.82, 2.50), (0.92, 2.30)}},%
			        {3/{(0.20, 0.65), (0.40, 0.55), (0.50, 2.07)}},%
			        {3/{(0.20, 0.80), (0.30, 0.90), (0.43, 2.27)}},%
			        {3/{(0.10, 0.90), (0.20, 2.15), (0.30, 2.25)}},%
			        {2/{(2.30, 0.30), (2.70, 0.70)}},%
			        {2/{(0.85, 0.15), (2.15, 0.50)}}%
			      }
			      {%
			      }};
			      
			     \node (equivT11) at (8.75, -11) {$\cong$};

				\begin{scope}[shift={(3.55cm, 0cm)}]
	%		      \node (t12) at (11.9, -11) {$\TT_{12} = \rsep{1}{\place{(1,(0,0))_1}{\leftarrow}{\TT_{11}}}$};
			      \node (t12) at (5.1, -9) {$\TT_{12}$};
			      \node (placedmiddle) at (7, -11) {
			      \tiling{1.0}{3}{4}{1/2}%
			      {%
			        {4/{(0.15, 1.10), (0.25, 1.20), (0.35, 1.40), (0.45, 1.30)}},%
			        {4/{(0.60, 1.35), (0.70, 1.55), (0.80, 1.65), (0.90, 1.45)}},%
			        {4/{(0.55, 1.20), (0.65, 1.10), (0.75, 1.40), (0.85, 1.30)}},%
			        {4/{(0.15, 3.60), (0.25, 3.80), (0.35, 3.90), (0.45, 3.80)}},%
			        {4/{(0.55, 3.70), (0.65, 3.60), (0.75, 3.90), (0.85, 3.80)}},%
			        {4/{(0.20, 3.35), (0.30, 3.55), (0.40, 3.65), (0.50, 3.45)}},%
			        {4/{(0.62, 3.10), (0.72, 1.90), (0.82, 3.50), (0.92, 3.30)}},%
			        {3/{(0.20, 1.65), (0.40, 1.55), (0.50, 3.07)}},%
			        {3/{(0.20, 1.80), (0.30, 1.90), (0.43, 3.27)}},%
			        {3/{(0.10, 1.90), (0.20, 3.15), (0.30, 3.25)}},%
			        {2/{(2.30, 0.30), (2.70, 0.70)}}%
			      }
			      {%
			      }};
	
			      \input{bigprooftree-right-placement.tex}
	
			      \node (t10) at (7.1, -15) {$\TT_{10}$};
			      \node (decr2) at (8, -15) {
			      \tiling{1.0}{1}{1}{}%
			      {%
			        {2/{(0.3, 0.3), (0.7, 0.7)}}%
			      }{}};
			    \end{scope}
			  \end{scope}
		  \end{scope}
	
	      \ptedge{(middle)}{(-0.55,0.81)}{(middleempty)}{(-0.5,0.79)}
	      \ptedge{(middle)}{(-0.55,0.81)}{(middlenonempty)}{(-0.5,0.79)}
	
	      \ptedge{(placedmiddle)}{(-0.5,-0.69)}{(nonemptyright)}{(-0.5,1.79)}
	      \ptedge{(placedmiddle)}{(-0.5,-0.69)}{(decr2)}{(-0.5,0.79)}
		\end{scope}
    \end{tikzpicture}
    \caption{A pictorial representation of the combinatorial specification found by Combinatorial Exploration for \(\Av(1243, 1342, 2143)\).}
    \label{fig:av1243_1342_2143}
  \end{figure}

We previously discussed the top part of the tree in
Figure~\ref{fig:req-ins-from-large-tree}. To the root tiling $\TT_1$, we apply
the requirement insertion strategy. Inserting $H_1 = \{(1, ((0, 0)))\}$ into
$\TT_1$ creates the rule $\TT_1 \xleftarrow{\ReqIns_{H_1}} (\TT_2,
\TT_3)$.

The tiling $\TT_2$ represents the set containing only the empty gridded
permutation of size $0$. This is clearly a set whose enumeration is known
\emph{a priori} and so to it we apply the verification strategy $V_{\TT_2}$
(recall from Section~\ref{subsection:comb-specs} that each set corresponds to a
unique verification strategy). The result is a rule $\TT_2
\xleftarrow{V_{\TT_2}} ()$.

Moving further into the tree, the insertion of $H_2 = \{(12, ((0, 0), (0,0)))\}$
into $\TT_3$ gives the rule $\TT_3 \xleftarrow{\ReqIns_{H_2}} (\TT_4, \TT_5)$.%
\footnote{Recall from our previous discussion that technically $\TT_5$ should
have a point requirement, but it is detected to be redundant and removed.
Although in practice this should be depicted as its own rule (an equivalence
strategy), we have hidden it, and several other such simplification, from this
figure.}
The set of gridded permutations represented by $\TT_4$ are precisely those
single-celled permutations that are strictly decreasing and have size at least
$1$. The enumeration of these permutations is evidently $a_0 = 0$ and $a_n = 1$
for $n \geq 1$, and since this is known independently of the structural
decomposition being described, we can apply the verification strategy
$V_{\TT_4}$ to produce the rule $\TT_4 \xleftarrow{V_{\TT_4}} ()$.

As this shows, a researcher who is applying Combinatorial Exploration has a lot
of flexibility in choosing to which combinatorial sets we can apply a
verification strategy. In this particular case, had we not employed the strategy
$V_{\TT_4}$ we could still have easily completed this part of the proof tree by
applying the point placement strategy to the requirement in $\TT_4$, then
applying the factor strategy to the result.

It is imperative to verify the two \emph{atomic sets},
the set containing only the empty permutation ($\TT_2$ here)
and the set containing only the gridded permutation of size $1$ ($\TT_9$
here). When deciding which other sets should be verified, one should keep in
mind that the goal is to shorten the search for a proof tree by identifying sets
that can be independently enumerated, and therefore eliminating the need to expand
them further.

There are many domain-specific algorithms to enumerate certain sets of
permutations in polynomial time. Here we will just briefly mention one: the
insertion encoding of Vatter~\cite{vatter:regular-insertion-encoding},
can be extended to some
single-row or single-column tilings. In these cases, one can be certain that a
combinatorial set can be enumerated, and so for the sake of computational
efficiency it makes sense to apply verification strategies to them.

For a slightly more experimental search, one might choose to verify
combinatorial sets that they simply suspect could be enumerated if needed,
either by hand or with a separate self-contained application of Combinatorial
Exploration. We have found this to be effective in the domain of
permutation patterns, where we often use the rule-of-thumb that when we are
searching for a proof tree for a permutation class $\Av(B)$ then any $1 \times
1$ tiling whose combinatorial set is a subclass $\Av(B')$ can be verified (even
if it carries requirements as in the case of $\TT_4$ above) as
long as $B'$ is obtained from $B$ by adding a pattern of length at most the
length of the longest pattern in $B$. This heuristic is not guaranteed to work;
for example, there is a permutation class that avoids two patterns of length $4$
and has an algebraic generating function such that if you add a particular third
pattern of length $4$ to its basis, the resulting class has an unknown generating
function that is conjectured to be non-D-finite~\cite{albert:c-machines}.

Continuing the traversal in the proof tree for $\Av(1243, 1342, 2143)$, we are now
at tiling $\TT_5$. After applying two point placements as shown in
Figure~\ref{figure:big-tree-t5-to-t6}, followed by the row and column
separations shown in Figure~\ref{figure:row-col-sep-second-example}
we obtain a chain of equivalence
rules between $\TT_5$ and $\TT_6$. Applying the factor strategy to $\TT_6$ gives
us four tilings, $\TT_7$, $\TT_8$, $\TT_9$ and $\TT_{10}$. The tiling $\TT_9$ is
verified as an atom, and the tiling $\TT_{10}$ represents a subclass of the
class being explored, and so as discussed above we mark it verified as well. The
tiling $\TT_8$ is already known to be equivalent to $\TT_3$, because earlier in the
process of Combinatorial Exploration, we applied point placement to $\TT_3$ and
obtained $\TT_8$. The last tiling, $\TT_7$, cannot be verified so we must
continue to decompose it. So, we apply requirement insertion with
$H = \{(1, ((1,0)))\}$ (a point in cell $(1,0)$) to it. When $H$ is
avoided, one obtains the left-hand child, which is identical to $\TT_1$, the
root. When $H$ is contained, one obtains the right-hand child, $\TT_{11}$, to
which we apply point placement and row separation. This produces the tiling
$\TT_{12}$ that factors into $\TT_{10}$ (which we have already seen and
verified) and $\TT_{13}$ which is equivalent to $\TT_3$ via a point placement.

This set of combinatorial rules forms a specification, and since all of the
rules are produced by productive strategies, we are guaranteed that the
specification can be used to produce terms of the counting sequence for
$\Av(1243, 1342, 2143)$ in polynomial time.

Below, we show the full specification on the left (omitting the names of
the strategies, which would typically be shown over the arrows, for space),
and the corresponding system of generating functions on the right. As
previously discussed, combinatorial sets that are equivalent are contracted
into an equivalence class in the specification; in this case, this occurs
three times. We let $\EE_3$, $\EE_5$, and $\EE_{11}$ denote the equivalence
classes $\{\TT_3, \TT_8, \TT_{13}\}$, $\{\TT_5, \TT_6\}$, and
$\{\TT_{11}, \TT_{11}', \TT_{12}\}$, respectively.

\begin{align*}
	\TT_1 &\leftarrow (\TT_2, \EE_3) & T_1(x) &= T_2(x) + E_3(x)\\
	\TT_2 &\leftarrow () & T_2(x) &= 1\\
	\EE_3 &\leftarrow (\TT_4, \EE_5) & E_3(x) &= T_4(x) + E_5(x)\\
	\TT_4 &\leftarrow () & T_4(x) &= x/(1-x)\\
	\EE_5 &\leftarrow (\TT_7, \EE_3, \TT_9, \TT_{10}) & E_5(x) &= T_7(x)\cdot E_3(x)\cdot T_9(x)\cdot T_{10}(x)\\
	\TT_7 &\leftarrow (\TT_1, \EE_{11}) & T_7(x) &= T_1(x) + E_{11}(x)\\
	\TT_9 &\leftarrow () & T_9(x) &= x\\
	\TT_{10} &\leftarrow () & T_{10}(x) &= 1/(1-x)\\
	\EE_{11} &\leftarrow (\EE_3, \TT_{10}) & E_{11}(x) &= E_3(x) \cdot T_{10}(x)
\end{align*}

The system of equations can be solved to find the generating function for the class:
\[
	T_1(x) = \frac{1 + x - \sqrt{1-6x+5x^2}}{2x(2-x)}.
\]

\subsection{Further Details of the Six Presented Strategies}
\label{subsection:strats-details}

In this subsection, we will provide full details of each strategy described
in Subsection~\ref{subsection:six-strategies}, and prove the productivity
of the relevant ones. Although each strategy can be easily understood on an
intuitive level from our earlier rough descriptions, the full details and
proofs often involve quite a bit of detailed argument.

\subsubsection{Requirement Insertion}

The Requirement Insertion strategy decomposes a set of gridded permutations into
those that avoid all gridded permutations from a set $H$ and those that contain
at least one gridded permutation from $H$. The formal definition of Requirement
Insertion is given in
Subsection~\ref{subsubsection:requirement-insertion-intuition} on
page~\pageref{subsubsection:requirement-insertion-intuition}. It is a
disjoint-union-type strategy, and so the proof of its productivity is rather
simple.
\begin{theorem}
\label{theorem:requirement-insertion-productive}
The Requirement Insertion strategy is productive.
\end{theorem}
\begin{proof}
	We need to verify Conditions 1 and 2 of
	Definition~\ref{definition:productive-strategy}. The reliance profile
	function of this strategy is $n \mapsto (n, n)$, which ensures that
	Condition 1 is satisfied.

	In order to check Condition 2, suppose that $\AA$ is a set of gridded
	permutations, and let $H$ be the set of patterns under consideration. Let
	$\BB$ and $\CC$ be the subsets of $\AA$ that avoid or contain $H$,
	respectively. Condition 2 requires that:
	\begin{enumerate}[(a)]
		\item $|\AA_n| \geq |\BB_n|$ and $|\AA_n| \geq |\CC_n|$ for all $n \in
		\N$;
		\item $|\AA_k| > |\BB_k|$ for some $k \in \N$ and $|\AA_\ell| >
		|\CC_\ell|$ for some $\ell \in \N$.
	\end{enumerate}
	Since every gridded permutation either avoids or contains the set $H$, the
	set $\AA$ is the disjoint union of $\BB$ and $\CC$. Moreover, the formal
	definition of Requirement Insertion requires that both $\BB$ and $\CC$ are
	nonempty in order for it apply. Both parts of Condition 2 follow immediately
	from these two facts.
\end{proof}

The above proof makes clear that once any strategy is known to be a
disjoint-union-type strategy, it is guaranteed to be productive.

\subsubsection{Obstruction and Requirement Simplification}
\label{subsubsection:obs-req-simp-details}
The equivalence strategies defined in
Subsection~\ref{subsubsection:obs-req-simp} --- Obstruction Deletion,
Requirement Deletion, and Requirement List Deletion --- were specifically
defined to only apply to a tiling when the corresponding alteration did not
change the underlying set of gridded permutations. We discussed in this section
that as a result, these three strategies are in some sense trivial. However,
they allow us to detect that the tiling representation of a set of gridded
permutations can be simplified, which is advantageous for detecting equality of
sets.

To prove that these three strategies are in fact equivalence strategies, one
would need to show that whenever $d_S(\AA) = \BB$, we have $|\AA_n| = |\BB_n|$
for all $n$. Since these strategies do not even change the underlying set of
gridded permutations, we have by definition the much stronger equality $\AA =
\BB$.

\subsubsection{Point Placement}
\label{subsubsection:point-placement-details}

The justification that the point placement strategy is an equivalence strategy
is technical and requires quite a bit of bookkeeping. The material in this
subsection is independent of the rest of this work, and so readers who wish
to skip this subsection on their first reading may freely do so.

Consider the tiling $\TT = ((t,u), \OO, \RR)$ with
$\RR = \{\RR_1, \ldots, \RR_k\}$, and suppose $\RR_1 = \{r\}$. We will describe
the tiling $\TT''$ that results from placing the point $r(I)$ of the requirement
$r$ in the extreme $d$ direction for some
$d \in \{\leftarrow, \rightarrow, \uparrow, \downarrow\}$.

Point placement results in the row and column of the placed point, $r(I)$, being
split into three rows and three columns, with the cell containing the placed
point itself becoming nine cells. For the rest of this section, suppose the
placed point is in cell $c = (c_x, c_y)$. The effect of this splitting on the
gridded permutations in $\TT$ can be thought of pictorially in the following way:
Take any gridded permutation $g \in \Grid(\TT)$, and add two new vertical lines
in column $c_x$ and two new horizontal lines in row $c_y$ such that none of these
four lines intersect any of the entries of $g$.
The result is a gridded permutation in $\GG^{(t+2,u+2)}$.
Figure~\ref{figure:multiplex-example} shows an example of this procedure. Given
a gridded permutation $g$, we will now define the set of \emph{multiplexes} of
$g$ around cell $c$, denoted $M_c(g)$, to be the set of all gridded
permutations in $\GG^{(t+2,u+2)}$ formed by adding two vertical and two
horizontal lines through $c$ in this way.

\begin{figure}
	\begin{center}
		\begin{tikzpicture}[scale=1.2, baseline=(current bounding box.center)]
			\foreach \Point in {(0.25, 0.67), (0.5, 2.33), (0.75, 1.6), %
								({1+1/7}, 2.75), ({1+2/7}, 1.2), ({1+3/7}, 1.8), %
								({1+6/7}, 2.5), ({1+5/7}, 0.33), ({1+4/7}, 1.4)}{\node
				at \Point {\point{1.5pt}};}
			\draw[dashed, very thin, gray] (0.01,0.01) grid (2.25,3.25);
			\draw[-latex, thin, gray] (-0.5,0)--(2.5,0);
			\draw[-latex, thin, gray] (0,-0.5)--(0,3.5);
		\end{tikzpicture}
		$\longrightarrow$
		\begin{tikzpicture}[scale=1.2, baseline=(current bounding box.center)]
			\foreach \Point in {(0.25, 0.67), (0.5, 2.33), (0.75, 1.6), %
								({1+1/7}, 2.75), ({1+2/7}, 1.2), ({1+3/7}, 1.8), %
								({1+6/7}, 2.5), ({1+5/7}, 0.33), ({1+4/7}, 1.4)}{\node
				at \Point {\point{1.5pt}};}
			\draw[dashed, very thin, gray] (0.01,0.01) grid (2.25,3.25);
			\draw[-latex, thin, gray] (-0.5,0)--(2.5,0);
			\draw[-latex, thin, gray] (0,-0.5)--(0,3.5);
			\draw[thin] ({1+1.5/7}, -0.25) -- ({1+1.5/7}, 3.25);
			\draw[thin] ({1+5.5/7}, -0.25) -- ({1+5.5/7}, 3.25);
			\draw[thin] (-0.25, 1.1) -- (2.25, 1.1);
			\draw[thin] (-0.25, 1.7) -- (2.25, 1.7);
		\end{tikzpicture}
		$\longrightarrow$
		\begin{tikzpicture}[scale=1.2, baseline=(current bounding box.center)]
			\foreach \Point in {(0.25, 0.67), (0.5, 4.25), (0.75, 2.75), %
								(1.5, 4.75), (2.2, 2.25), (2.4, 3.5), %
								(2.6, 2.5), (2.8, 0.33), (3.5, 4.5) %
								}{\node
				at \Point {\point{1.5pt}};}
			\draw[dashed, very thin, gray] (0.01,0.01) grid (4.25,5.25);
			\draw[-latex, thin, gray] (-0.5,0)--(4.5,0);
			\draw[-latex, thin, gray] (0,-0.5)--(0,5.5);
			\draw[thick] (1, 1) rectangle (4, 4);
		\end{tikzpicture}
	\end{center}
	\caption{On the left, a gridded permutation $g$ of size $9$. In the center, the
		same gridded permutation with two vertical lines drawn in column $c_x=1$
		and two horizontal lines drawn in row $c_y=1$. On the right, the resulting
		gridded permutation, in which the outlined $3 \times 3$ region corresponds
		to the single cell $(c_x, c_y) = (1, 1)$ which has now been split. This
		rightmost gridded permutation (which has the same underlying permutation,
		but different cell assignments) is called a \emph{multiplex} of $g$.}
	\label{figure:multiplex-example}
\end{figure}
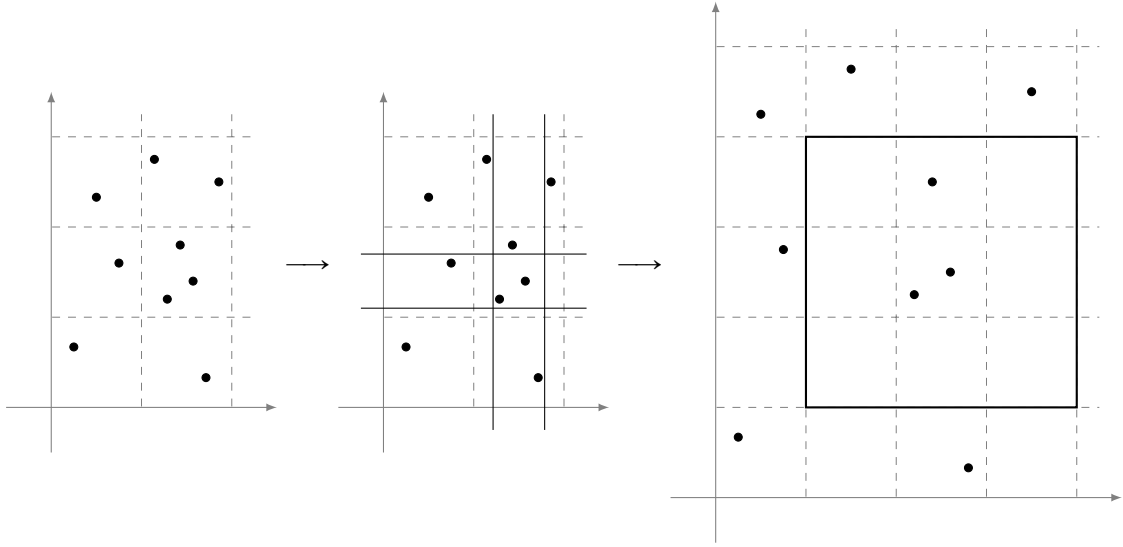

Given a multiplexed gridded permutation $g' \in M_c(g)$, the original
gridded permutation $g$ can be recovered by, essentially, removing the two
vertical lines between columns $c_x$, $c_x+1$, and $c_x+2$ and removing the
two horizontal lines between rows $c_y$, $c_y+1$, and $c_y+2$. In this context
we call $g$ the \emph{contraction} of $g'$ around $c$. To make this more
formal, consider the following map:
\[
	b_p : i \mapsto \left\{\begin{array}{ll}
		i & i < p\\
		p & i \in \{p, p+1, p+2\}\\
		i-2 & i > p+2
	\end{array}\right..
\]
With this terminology, we can define what contracting a gridded permutation
around the cell $c = (c_x, c_y)$ does to the cell that each entry lies in:
\[
	\beta_{(c_x,c_y)} : (u,v) \mapsto (b_{c_x}(u), b_{c_y}(v)).
\]
The reader is encouraged to refer back to Figure~\ref{figure:multiplex-example}
which has $c = (1,1)$, and in this case, for example,
$\beta_{(1,1)}(0, 1) = (0,1)$,
$\beta_{(1,1)}(0, 2) = (0,1)$,
$\beta_{(1,1)}(1, 4) = (1,2)$, and
$\beta_{(1,1)}(3, 2) = (1,1)$.

We can now formally define the multiplex of $g = (\pi, (c_1, \ldots, c_n))$
around the cell $c$ to be those gridded permutations whose contractions are
equal to $g$:
\[
	M_c(g) = \{g'=(\pi, (c_1', \ldots, c_n')) \in \GG : \beta_c(c_i') = c_i\text{ for all $i$ such that }1 \leq i \leq n\}.
\]

Extending this notation to sets of gridded permutations, we define
\[
	M_c(S) = \bigcup_{g \in S} M_c(g).
\]
It is important to note that because $M_c(g)$ can be thought of as the preimage
of $g$ under the contraction map, it follows that if $g \neq h$, then
$M_c(g) \cap M_c(h) = \emptyset$.

Before moving on, we will state and prove two pattern preservation lemmas.
\begin{lemma}
	\label{lemma:if-pattern-then-multiplex-pattern}
	Fix two gridded permutations $g$ and $h$ and a cell $c$. If $h \leq g$,
	then every multiplex of $g$ contains some multiplex of $h$, i.e., for all
	$g' \in M_c(g)$, there exists some $h' \in M_c(h)$ such that $h' \leq g'$.
\end{lemma}
\begin{proof}
	Fix a cell $c$ and gridded permutations $g = (\pi, (c_1, \ldots, c_n))$ and
	$h = (\sigma, (d_1, \ldots, d_k))$. Suppose that $g$ contains $h$ at the
	indices $i_1, \ldots, i_k$, implying that $\pi(i_1)\ldots\pi(i_k)$ is
	order-isomorphic to $\sigma$ and $c_{i_\ell} = d_\ell$ for
	$1 \leq \ell \leq k$.

	Let $g' \in M_c(g)$ be arbitrary. We can write
	$g' = (\pi, (c_1', \ldots, c_n'))$,
	with the property that
	$\beta_{c}(c_i') = c_i$
	for all $i$. Define $h'$ to be the gridded permutation that $g'$ contains
	at the indices $i_1, \ldots, i_k$. The underlying permutation of $h'$ is
	still $\sigma$, so
	$h' = (\sigma, (c_{i_1}', \ldots, c_{i_k}'))$.
	Because we must have $\beta_c(c_{i_\ell}') = c_{i_\ell} = d_\ell$, it follows
	that $h' \in M_c(h)$, completing the proof.
\end{proof}

\begin{lemma}
	\label{lemma:if-multiplex-pattern-then-pattern}
	Fix two gridded permutations $g$ and $h$ and a cell $c$. If some multiplex
	of $g$ contains some multiplex of $h$, then $h \leq g$.
\end{lemma}
\begin{proof}
	Fix a cell $c$ and gridded permutations $g$ and $h$. Suppose that
	$g' \in M_c(g)$ and $h' \in M_c(h)$ are such that $g'$ contains $h'$ at
	the indices $i_1, \ldots, i_k$. Thus we can write
	$g' = (\pi, (c_1', \ldots, c_n'))$ and
	$h' = (\sigma, (c_{i_1}', \ldots, c_{i_k}'))$
	where $\pi$ contains $\sigma$ at the same indices. This implies that
	$g = (\pi, (\beta_c(c_1'), \ldots, \beta_c(c_n')))$ and
	$h = (\sigma, (\beta_c(c_{i_1}'), \ldots, \beta_c(c_{i_k}')))$, from which
	we see that $g$ contains $h$ at the indices $i_1, \ldots, i_k$.
\end{proof}

We are now ready to define an intermediate tiling $\TT'$ derived from our
original tiling $\TT = ((t,u), \OO, \{\{r\}, \RR_2, \ldots, \RR_k\})$. The
tiling $\TT'$ will contain the subset of multiplexes (around cell $(c_x,c_y)$)
of the permutations in $\Grid(\TT)$ that contain exactly one point in the cell
$(c_x+1, c_y+1)$, and no other points in column $c_x+1$ or row $c_y+1$. To that
end, define $\OO' = M_c(\OO)$ and $\RR_i' = M_c(\RR_i)$ for $2 \leq i \leq k$.
To ensure that cell $(c_x+1, c_y+1)$ contains exactly one point, we define
additional obstructions $A = \{(12, (c_x+1, c_y+1)), (21, (c_x+1, c_y+1))\}$
and an additional requirement list $C = \{(1, (c_x+1, c_y+1))\}$. To ensure
that no other cells in column $c_x+1$ or row $c_y+1$ have any entries, we
define yet more obstructions
\[
	B_1 = \{(1, (c_x+1, j)) : 0 \leq j < u, j \neq c_y+1\}
\]
and
\[
	B_2 = \{(1, (i, c_y+1)) : 0 \leq i < t, i \neq c_x+1\}.
\]
The tiling $\TT'$ can now be defined:
\[
	\TT' = ((t+2, u+2), \OO' \cup A \cup B_1 \cup B_2, \{\RR_2', \ldots, \RR_k', C\}).
\]
Note that the requirement $r$ actually being placed has been dropped.

The set $\Grid(\TT')$ contains some, but not all, of the multiplexes of
gridded permutations in $\Grid(\TT)$ in a way that we will make more precise.
In particular, the multiplexes in $\Grid(\TT')$ are exactly those in which one
entry in cell $c$ has been isolated into its own cell $(c_x+1, c_y+1)$, with no
other points in that cell's row or column. Informally, these are the multiplexes
that result from picking one point in cell $c$ and adding the two vertical lines
just to the left and just to the right of this point and adding the two
horizontal lines just below and just above this point. The lemma below makes
this more precise.

\begin{lemma}
	\label{lemma:point-multiplex-in-T-prime}
	Let $g \in \Grid(\TT)$ and suppose that $g$ has precisely $\ell$ points
	in cell $c  = (c_x, c_y)$. Then,
	\[
		\left|M_c(g) \cap \Grid(\TT')\right| = \ell.
	\]
\end{lemma}
\begin{proof}
	Suppose that $g \in \Grid(\TT)$ and consider $g' \in M_c(g)$. Because $g$
	avoids all obstructions in $\OO$,
	Lemma~\ref{lemma:if-multiplex-pattern-then-pattern} shows that $g'$ avoids
	all obstructions in $\OO'$. Similarly, by
	Lemma~\ref{lemma:if-pattern-then-multiplex-pattern},
	$g'$ contains at least one requirement in each list $\RR_2', \ldots, \RR_k'$.

	It remains to show that if $g$ contains $\ell$ entries in cell $c$,
	then there are precisely $\ell$ elements of $M_c(g)$ that avoid the
	obstructions in $A \cup B_1 \cup B_2$ and contain the requirement $C$.
	These obstructions together with the requirement collectively enforce the
	property that every gridded permutation in $\Grid(\TT')$ has exactly one
	point in the cell $(c_x+1, c_y+1)$, and no other points in the same row
	or column as this cell. For each of the $\ell$ entries in cell $c$ of $g$,
	there is one multiplex $g'$ of $g$ that has that entry alone in cell
	$(c_x+1, c_y+1)$ with no other entries sharing a column or row with this
	cell, and this $g'$ is in $\Grid(\TT')$.

	Clearly, for any two of the $\ell$ entries of $g$ in cell $c$, the
	multiplexed versions that isolate these entries are different. Therefore,
	$M_c(g) \cap \Grid(\TT')$ contains $\ell$ gridded permutations, each of which
	has one of the $\ell$ entries in cell $c$ of $g$ isolated in cell
	$(c_x+1, c_y+1)$ of $g'$, with no other entries sharing a row or column.
\end{proof}

In the proof above, we discussed multiplexes of $g$ that have a single entry
isolated in cell $(c_x+1, c_y+1)$ with no other entries sharing a row or
column. These are relevant to the rest of this subsection, so we will call them
the \emph{point multiplexes} of $g$, and if the point $g(i)$ is in cell $c$ of
$g$, then we will specifically call the multiplex that isolates that point
$m_c^i(g)$. The lowercase ``$m$'' is meant to convey that $m_c^i(g)$ is one single
gridded permutation, while $M_c(g)$ is a set of gridded permutations.

Our goal now is to construct a tiling $\TT''$ by adding obstructions and one
requirement to $\TT'$ with the result that $|M_c(g) \cap \Grid(\TT'')| = 1$
for any $g \in \Grid(\TT)$. Recall that we are placing point $r(I)$ of
the original requirement $r$ in the extreme $d$ direction. For $g$ with $\ell$
entries in cell $c$, the  obstructions and requirement will invalidate $\ell-1$
of the point multiplexes, leaving only the one in which the isolated point is
the point that, out of all points playing the role of $r(I)$ in an occurrence
of $r$, is the one in the most extreme $d$ direction.

First, we must add a requirement so that not only does every permutation
in $\Grid(\TT'')$ contain some mutiplex of the requirement $r$ that we're
placing, but more specifically contains the single point multiplex $r'$ in which
the isolated point of any gridded permutation in $\Grid(\TT'')$ plays the role
of $r'(I)$ in an occurrence of $r$. (This is not yet enforcing that this
point is in the most extreme $d$ direction.) To achieve this, we add the
requirement list $\{m_c^I(r)\}$ as a new requirement list in $\TT''$.

Lastly, we need to craft new obstructions to add to $\TT''$ to enforce the
property that the isolated point in any $g' \in \Grid(\TT'')$ is the point
in the most extreme $d$ direction that plays the role of $r(I)$ in any
occurrence of $r$. This is actually simple. Consider the set
of all point multiplexes of $r$. For any $r'$ in this set with the property that
the cell $c_I$ in which the entry $r'(I)$ occurs is further the $d$ direction
than the cell $(c_x+1, c_y+1)$ that contains the isolated point, we add $r'$
as an obstruction to $\TT''$.

At this point, $\TT''$ contains all obstructions and requirements of $\TT'$
together with the new obstructions and the one new requirement described above.
The tiling $\TT''$ is the output tiling of the point placement strategy, and
so we must now prove that there is a size-preserving bijection between
$\Grid(\TT)$ and $\Grid(\TT'')$. We start with the following lemma.
\begin{lemma}
	\label{lemma:point-placement-multiplex-size-1}
	Let $g \in \Grid(\TT)$. Then,
	\[
		\left|M_c(g) \cap \Grid(\TT'')\right| = 1.
	\]
\end{lemma}
\begin{proof}
	Let $g \in \Grid(\TT)$ and recall that $M_c(g) \cap \Grid(\TT')$ was shown
	to contain one point multiplex of $g$ for each of the $\ell$ points in cell
	$c$ of $g$. Of all the points in cell $c$ of $g$ that play the role of
	$r(I)$ in an occurrence of $r$ (there is at least one such point because
	$\{r\}$ is a requirement list of $\TT$), suppose the one in the most extreme
	$d$ direction is located at index $i$. Then, by design, the point multiplex
	$m_c^i(g)$ is contained in $\Grid(\TT'')$. The other $\ell-1$ point
	multiplexes of $g$ are each isolating different points, and each of these
	other points either does not play the role of $r(I)$ in an occurrence of
	$r$, or does, but not in the most extreme $d$ direction, and therefore
	they are precluded from $\Grid(\TT'')$ by virtue of the new obstructions
	enforcing extremeness of direction or the new requirement forcing the
	isolated point to play the role of $r(I)$.
\end{proof}

The lemma above allows us to define a map $\Gamma : \Grid(\TT) \to \Grid(\TT'')$
in which $\Gamma(g)$ is defined to be the single element of $M_c(g)$ in
$\Grid(\TT'')$. Note that the cell $c$, the requirement $r$, the index of the
placed point $I$, and the direction $d$ are all implicit parameters of $\Gamma$.

\begin{theorem}
	\label{theorem:point-placement-bijection}
	The map $\Gamma$ defined above is a size-preserving bijection.
\end{theorem}
\begin{proof}
	Lemma~\ref{lemma:point-placement-multiplex-size-1} ensures that $\Gamma$ is
	well-defined, and the fact that $\Gamma$ is size-preserving follows from
	the fact that every element of $M_c(g)$ has the same size of $g$. Moreover,
	$\Gamma$ is injective as a result of the observation that if $g \neq h$
	then $M_c(g) \cap M_c(h) = \emptyset$.

	To see that $\Gamma$ is surjective, let $g' \in \Grid(\TT'')$ and let $g$ be
	such that $g' \in M_c(g)$, i.e, $g$ is the contraction of $g'$. We will
	check that $g \in \Grid(\TT)$, from which it follows that $\Gamma(g) = g'$.

	Let $o$ be any obstruction of $\TT$ and recall that every mutiplex in
	$M_c(o)$ is an obstruction of $\TT''$. Therefore, if $g$ were to contain $o$,
	Lemma~\ref{lemma:if-pattern-then-multiplex-pattern} would imply that $g'$
	contains a multiplex of $o$, a contradiction. Therefore, $g$ avoids all
	obstructions of $\TT$.

	One of the requirements of $\TT''$ is a point multiplex of the requirement $r$
	being placed, therefore $g'$ contains this point multiplex. By
	Lemma~\ref{lemma:if-multiplex-pattern-then-pattern}, it follows that $g$
	satisfies the requirement list $\{r\}$.

	Lastly, consider any of the other requirement lists $\RR_i$ of $\TT$ for
	$2 \leq i \leq k$. Since $g'$ contains some element of $M_c(\RR_i)$,
	Lemma~\ref{lemma:if-multiplex-pattern-then-pattern} once again implies
	that $g$ contains some element of $\RR_i$.

	Therefore, $g \in \Grid(\TT)$, confirming that $\Gamma$ is a size-preserving
	bijection, and therefore the point placement strategy is an equivalence
	strategy.
\end{proof}

\subsubsection{Row Separation and Column Separation}
\label{subsubsection:row-col-sep-details}

In Subsection~\ref{subsubsection:row-col-sep} we defined row separation and
column separation as equivalence strategies that can be applied when the
cells in a row or column can be split into two separate rows or columns. We will
discuss only row separation here, but the results all follow \emph{mutatis
mutandis} for column separation.

For the remainder of this subsection, let $\TT = ((t,u), \OO, \{\RR_1, \ldots,
\RR_k\})$ be a tiling. Suppose there is a row $r$ with a subset $S$ of nonempty
cells such that row separation applies. From the definition of row separation in
Subsection~\ref{subsubsection:row-col-sep}, this means that if we set $S'$ to be
the nonempty cells in row $r$ that are not in $S$, then for every pair of cells
$(c_1,c_2) \in S \times S'$, if $c_1$ is to the left of $c_2$ then $(21,
(c_1,c_2)) \in \OO$ and if $c_1$ is to the right of $c_2$ then $(12, (c_2, c_1))
\in \OO$. We call these gridded permutations the \emph{critical row patterns}
and denote the set of them by $C$.

In order to formally define the tiling $\TT'$ produced by the row separation
strategy, we first need to define a mapping $\gamma_{r,S}$ that describes how
row separation moves entries around between cells. To that end, define
\[
	\gamma_{r,S} : (i,j) \mapsto \left\{\begin{array}{ll}
		(i,j),& j < r\text{ or }(j = r\text{ and }i \in S),\\[3pt]
		(i,j+1),& j > r\text{ or }(j = r\text{ and }i \not\in S),\\
	\end{array}\right..
\]

Let $[N]$ denote the set $\{1, \ldots, N\}$.
The domain of $\gamma_{r,S}$ is $[t] \times [u]$. The range of $\gamma_{r,S}$
contains all cells in $[t] \times [u+1]$ except several in rows $r$ and $r+1$,
and those cells that it doesn't contain will contain point obstructions in
$\TT'$ (i.e., they will be empty). We define the codomain of $\gamma_{r,S}$ to
be its range.

Now, the tiling $\TT'$ that results from applying Row Separation to row $r$ and
cells $S$ is formed by applying the map $\gamma_{r,S}$ to the cells in each
obstruction and requirement of $\TT$. To that end, define the map $\Gamma_{r,S}$
on gridded permutations by
\[
	\Gamma_{r,S}((\pi, (c_1, \ldots, c_n))) = (\pi, (\gamma_{r,S}(c_1), \ldots, \gamma_{r,S}(c_n))).
\]
From here on out we will refer to $\gamma_{r,S}$ and $\Gamma_{r,S}$ just as
$\gamma$ and $\Gamma$, as $r$ and $S$ are fixed throughout.

We must take a moment to verify that for any gridded permutation $g$, the
definition of $\Gamma(g)$ actually produces a valid gridded permutation when $g$
does not contain any of the critical row patterns.
\begin{lemma}
	\label{lemma:row-col-sep-Gamma-g-valid}
	If $g$ avoids all of the critical row patterns, then the definition of
	$\Gamma(g)$ above produces a valid gridded permutation.
\end{lemma}
\begin{proof}
	The concern to address is whether the underlying pattern is consistent with
	the new cell assignments. First, note that the column assignments of every
	entry remain unchanged by $\gamma$. The row assignments stay largely
	unchanged except that some entries in row $r$ move to row $r+1$.

	This could pose a problem if $g$ contained a $12$ pattern in row $r$,
	but in $\Gamma(g)$ the $1$ moves to row $r+1$ while the $2$ stays in row
	$r$---clearly such a gridded permutation is impossible. A problem would
	similarly occur if $g$ contained a $21$ pattern in row $r$ such that, again,
	the $1$ moves to row $r+1$ under the effect of $\Gamma$ while the $2$ stays
	in row $r$. However, these two problematic patterns are precisely the critical
	row patterns, which $g$ is assumed to avoid.
\end{proof}

We extend the definition of $\Gamma$ to sets of gridded permutations under the
definition $\Gamma(A) = \{\Gamma(g): g \in A\}$.

We can finally give the formal definition of the tiling $\TT'$ that results from
row separation. We may assume without loss of generality that the obstructions
$\OO$ of $\TT$ are pairwise incomparable. Now define
\[
	\TT' = ((t,u+1), \Gamma(\OO \smallsetminus C) \cup \OO', \{\Gamma(\RR_1), \ldots, \Gamma(\RR_k)\}),
\]
where
\[
	\OO' = \{(1, (i,r)): i \not\in S\} \cup \{(1, (i,r+1)): i \in S\}
\]
is a set of size 1 obstructions setting empty the cells in row $r$ that
correspond to columns not in $S$ and the cells in row $r+1$ that correspond to
columns that are in $S$.

To prove that Row Separation is an equivalence strategy, we must exhibit a
size-preserving bijection between $\Grid(\TT)$ and $\Grid(\TT')$. In fact, in
addition to being helpful in the definition of $\TT'$, $\Gamma$ is such a
bijection (with domain and codomain appropriately defined). The following lemma
will help to establish this.
\begin{lemma}
	The cell map $\gamma$ is a bijection.
\end{lemma}
\begin{proof}
	Consider two cells $c_1 = (i_1, j_1)$ and $c_2 = (i_2, j_2)$. Suppose that
	$\gamma(c_1) = \gamma(c_2)$ and let $(a,b)$ denote this value. Since
	$\gamma$ holds the first component constant, we must have $i_1 = a = i_2$.
	If $b \leq r$ then $j_1 = b = j_2$, and if $b \geq r+2$ then $j_1 = b - 1 =
	j_2$. Lastly, if $b = r+1$, then $j_1$ and $j_2$ could equal $r$ or $r+1$,
	but this depends only on the fixed set $S$ and the quantities $i_1$ and
	$i_2$, which have already been shown to be equal. Therefore $j_1 = j_2$, and
	so $\gamma$ is injective. The surjectivity of $\gamma$ was ensured by our
	declaration that its codomain equals its range.
\end{proof}

We now state and prove two lemmas about the preservation of patterns under
the $\Gamma$ map that are similar to
Lemmas~\ref{lemma:if-pattern-then-multiplex-pattern}
and~\ref{lemma:if-multiplex-pattern-then-pattern} from
Subsection~\ref{subsubsection:point-placement-details}.

\begin{lemma}
	\label{lemma:if-pattern-then-row-sep-pattern}
	Consider a tiling $\TT$ with row $r$ and set of cells $S$ in row $r$ to which
	row separation applies. Let $h,g \in \Grid(\TT)$ with $h \leq g$.
	Then $\Gamma(h) \leq \Gamma(g)$.
\end{lemma}
\begin{proof}
	Let $h$ and $g$ be as in the hypotheses and note that $\Gamma(h)$ and
	$\Gamma(g)$ are valid gridded permutations by
	Lemma~\ref{lemma:row-col-sep-Gamma-g-valid}. Suppose that $g$ contains
	$h$ at the indices $i_1, \ldots, i_k$, so we can write
	\begin{align*}
		g &= (\pi, (c_1, \ldots, c_n)) \text{ and }\\
		h &= (\sigma, (c_{i_1}, \ldots, c_{i_k}))
	\end{align*}
	where $\pi$ contains $\sigma$ at the same indices. Then,
	\begin{align*}
		\Gamma(g) &= (\pi, (\gamma(c_1), \ldots, \gamma(c_n))) \text{ and }\\
		\Gamma(h) &= (\sigma, (\gamma(c_{i_1}), \ldots, \gamma(c_{i_k}))),
	\end{align*}
	confirming that $\Gamma(g)$ contains $\Gamma(h)$ at the same indices.
\end{proof}

\begin{lemma}
	\label{lemma:if-row-sep-pattern-then-pattern}
	Consider a tiling $\TT$ with row $r$ and set of cells $S$ to which row
	separation applies. Let $h,g \in \Grid(\TT)$ and suppose
	$\Gamma(h) \leq \Gamma(g)$. Then, $h \leq g$.
\end{lemma}
\begin{proof}
	Let $h$ and $g$ be as in the hypotheses and assume that
	$\Gamma(h) \leq \Gamma(g)$ at the indices $i_1, \ldots, i_k$. Writing
	\begin{align*}
		g &= (\pi, (c_1, \ldots, c_n)) \text{ and }\\
		h &= (\sigma, (d_1, \ldots, d_k)),
	\end{align*}
	we have
	\begin{align*}
		\Gamma(g) &= (\pi, (\gamma(c_1), \ldots, \gamma(c_n))) \text{ and }\\
		\Gamma(h) &= (\sigma, (\gamma(d_1), \ldots, \gamma(d_k))).
	\end{align*}
	As $\Gamma(g)$ contains $\Gamma(h)$ at the indices $i_1, \ldots, i_k$, we
	know that $\pi$ contains $\sigma$ at the same indices and that
	$\gamma(c_{i_\ell}) = \gamma(d_\ell)$ for $1 \leq \ell \leq k$. Since
	$\gamma$ is a bijection, we have $c_{i_\ell} = d_\ell$ for each $\ell$,
	and therefore $g$ contains $h$ also at the indices $i_1, \ldots, i_k$.
\end{proof}

With these lemmas in hand, we are ready to prove that row separation is an
equivalence strategy.
\begin{theorem}
	Row separation (and correspondingly, column separation) are equivalence
	strategies.
\end{theorem}
\begin{proof}
	In order to establish the theorem, we will show that $\Gamma : \Grid(\TT)
	\to \Grid(\TT')$, with $\Gamma$ as defined earlier, is a size-preserving
	bijection. Let us note right away that $\Gamma$ does not change the
	underlying pattern of a gridded permutation, it simply affects the cells
	that the pattern lies in; this makes clear that $\Gamma$ is size-preserving.

	We first consider whether $\Gamma$ is validly defined. Let $g \in
	\Grid(\TT)$ and define $g' = \Gamma(g)$. Since $g$ avoids the critical row
	patterns Lemma~\ref{lemma:row-col-sep-Gamma-g-valid} ensures that $g'$ is a
	validly defined gridded permutation.

	Next we will demonstrate that $g' \in \Grid(\TT')$. The obstructions of
	$\TT'$ are split into those of the form $\Gamma(\sigma)$ where $\sigma$ is
	an obstruction of $\TT$ that is not a critical row pattern, and the size 1
	obstructions defined as $\OO'$. The definition of $\Gamma$ ensures that $g'$
	will not contain any points in the cells that $\OO'$ requires be empty, and
	the contrapositive of Lemma~\ref{lemma:if-row-sep-pattern-then-pattern} shows
	that for any non-critical row pattern $\sigma$,
	since $g$ avoids $\sigma$ we must have that $g'$ avoids $\Gamma(\sigma)$.
	Similarly, Lemma~\ref{lemma:if-pattern-then-row-sep-pattern} shows that for
	each requirement list $\RR_i$, since $g$ contains $\RR_i$ we must have that
	$g'$ contains $\Gamma(\RR_i)$. Thus, $h \in \Grid(\TT')$.

	For injectivity, let $g_1 = (\pi_1, (c_1, \ldots, c_n))$ and $g_2 = (\pi_2,
	(d_1, \ldots, d_n))$. By the definition of $\Gamma$,
	\begin{align*}
		\Gamma(g_1) &= (\pi_1, (\gamma(c_1), \ldots, \gamma(c_n))),\text{ and}\\
		\Gamma(g_2) &= (\pi_2, (\gamma(d_1), \ldots, \gamma(d_n))).
	\end{align*}
	Suppose that $\Gamma(g_1) = \Gamma(g_2)$. Then, since $\Gamma$ does not
	affect the underlying pattern, $\pi_1 = \pi_2$, and we will use $\pi$ to
	denote this permutation. Moreover, $\gamma(c_i)
	= \gamma(d_i)$ for all $i$, and so by the injectivity of $\gamma$ it follows
	that $c_i = d_i$ for all $i$. Thus $g_1 = g_2$, verifying that $\Gamma$ is
	injective.

	To see that $\Gamma$ is surjective, let $g' = (\pi, (d_1, \ldots, d_n)) \in
	\Grid(\TT')$. Every cell $d_i$ is in the range of $\gamma$ because the cells
	in $[t] \times [u+1]$ that are not in the range of $\gamma$ contain size 1
	obstructions implying that no gridded permutation can have entries in those
	cells. Therefore, we may define $g = (\pi, (\gamma^{-1}(d_1), \ldots,
	\gamma^{-1}(d_n)))$ from which it is clear that $\Gamma(g) = g'$.
	\end{proof}

\subsubsection{Factor}
\label{subsubsection:factorization-details}

In this subsection, we must demonstrate the productivity of the Factor strategy
\(\Factor_{P,S}\). As discussed in an example on
page~\pageref{pageref:potential-productivity-problem}, the reliance profile
function for the Factor strategy was carefully defined in terms of the set \(S\)
to ensure productivity. For ease of exposition, we prove productivity in the
case of a partition \(P\) into two parts. Productivity for larger partitions
follows inductively.

\begin{theorem}
\label{theorm:factor-is-productive}
	The Factor strategy is productive.
\end{theorem}
\begin{proof}
	Suppose \(d_{\Factor_{P,S}}(\AA) = (\BB^{(1)}, \BB^{(2)})\). Recall that $S
	\subseteq \{1, 2\}$, and if $i \in S$ then $\AA_n$ does not rely on
	$\BB^{(i)}_n$. We consider the cases $S = \emptyset$, $S = \{1\}$, and $S =
	\{1,2\}$, while the fourth case $S = \{2\}$ is symmetric to the case $S =
	\{1\}$.

	To prove productivity, we have to check Conditions 1 and 2 in
	Definition~\ref{definition:productive-strategy} on
	page~\pageref{definition:productive-strategy}. Condition 1 requires that
	$\AA_n$ can never rely on $\BB^{(1)}$ or $\BB^{(2)}$ at some size longer
	than $n$, and that is true in all cases by the definition of the reliance
	profile function for the Factor strategy. Condition 2 requires that for
	each $i \in \{1, 2\}$, if $\AA_N$ relies on $\BB^{(i)}_N$ for some
	$N \in \N$, then:
	\begin{enumerate}[(a)]
		\item $|\AA_n| \geq |\BB_n^{(i)}|$ for all $n \in \N$, and
		\item $|\AA_\ell| > |\BB_\ell^{(i)}|$ for some $\ell \in \N$.
	\end{enumerate}

	\textbf{Case 1.} Suppose $S = \{1, 2\}$. In this case, $\AA_n$ relies on
	neither $\BB^{(1)}_n$ nor $\BB^{(2)}_n$, and so Condition 2 trivially
	holds.

	\textbf{Case 2.} Suppose $S = \{1\}$. This means that $|\BB^{(2)}_0| = 0$
	and that $\AA_n$ does not rely on $\BB^{(1)}_n$. We therefore only need to
	verify Condition 2 for $\BB^{(2)}$. First note that for all $n \in
	\N$,
	\[
		|\AA_n| = |\BB^{(1)}_0||\BB^{(2)}_n| + |\BB^{(1)}_1||\BB^{(2)}_{n-1}| + \cdots + |\BB^{(1)}_{n-1}||\BB^{(2)}_1|
	\]
	and that all terms on the right-hand side are nonnegative integers. In
	particular, $|\BB^{(1)}_0| \geq 1$ because $2 \not\in S$. Thus $|\AA_n| \geq
	|\BB^{(2)}_n|$ for all $n$, verifying Condition 2(a). To check Condition
	2(b), let $m_1 > 0$ be minimal such that $|\BB^{(1)}_{m_1}| > 0$ and let
	$m_2 > 0$ be minimal such that $|\BB^{(2)}_{m_2}| > 0$; the existence of
	$m_1$ and $m_2$ are guaranteed by the requirement in the definition of the
	Factor strategy that $\BB^{(1)}$ and $\BB^{(2)}$ each contain at least one
	object of size at least $1$. Now, we have
	\[
		|\AA_{m_1+m_2}| \geq |\BB^{(1)}_0||\BB^{(2)}_{m_1+m_2}| + |\BB^{(1)}_{m_1}||\BB^{(2)}_{m_2}|,
	\]
	and because $|\BB^{(1)}_0|$, $|\BB^{(1)}_{m_1}|$ and $|\BB^{(2)}_{m_2}|$ are
	all at least $1$, we have
	\[
		|\AA_{m_1+m_2}| > |\BB^{(2)}_{m_1+m_2}|,
	\]
	verifying Condition 2(b).

	\textbf{Case 3.} Suppose $S = \emptyset$. In this case, we are guaranteed
	that both $|\BB^{(1)}_0| \geq 1$ and $|\BB^{(2)}_0| \geq 1$. Condition 2
	must now be checked for both $\BB^{(1)}$ and $\BB^{(2)}$. We omit the
	details, as they are substantially  similar to Case 2.
\end{proof}

\subsubsection{Obstruction Inferral}
As was the case with Obstruction and Requirement Simplification in
Subsection~\ref{subsubsection:obs-req-simp-details}, there are no details that need
to be checked. By definition, the Obstruction Inferral strategy only
applies in cases where the underlying set of gridded permutations is not
modified at all, guaranteeing by definition that it is an equivalence strategy.

\subsection{Recap}
This concludes our discussion applying Combinatorial Exploration to the domain of
permutation patterns. In this work, we described only six (families of) strategies
that together can successfully find combinatorial specifications for many
permutation classes. However, we have developed a handful of additional, more
complicated strategies that we will describe in future work, and which permit
the discovery of combinatorial specifications for many more classes, including
many listed in Subsection~\ref{subsection:pp-success}. To give just a taste here,
among these additional strategies, a few examples are
\begin{itemize}[label=$\diamond$]
	\item \emph{row placement}, which is a disjoint-union-type strategy that
			splits a tiling into multiple children depending on whether a given
			row is empty, or if nonempty, which cell in the row has the smallest
			(or largest) entry; similarly, \emph{column placement};
	\item \emph{fusion}, which detects when two rows or columns can be merged
			together	 into one row or column.
			This is informally like deleting the line that separates
			two rows, and this is the strategy that introduces the catalytic
			variables discussed in Subsection~\ref{subsection:pp-success};
	\item several novel verification strategies that detect when the enumeration
			of a particular tiling can be computed independently, often by
			running Combinatorial Exploration again in a subprocess on this
			particular tiling.
\end{itemize}

%% ==== %% ==== %% ==== %% ==== %% ==== %% ==== %% ==== %% ==== %% ==== %% ==== %%
%% ==== %% ==== %% ==== %% ====   SECTION SEVEN    ==== %% ==== %% ==== %% ==== %%
%% ==== %% ==== %% ==== %% ==== %% ==== %% ==== %% ==== %% ==== %% ==== %% ==== %%

\section{Applying Combinatorial Exploration to Alternating Sign Matrices}
\label{section:asm-results}
%!TEX root = combinatorial-exploration.tex

% Get refs from here when I'm at RU https://en.wikipedia.org/wiki/Alternating_sign_matrix

\subsection{Alternating Sign Matrices}

An \emph{alternating sign matrix} (ASM) is a matrix in which every entry is $0$, $1$,
or $-1$, the sum of every row and every column is $1$, and the nonzero entries in each
row and in each column alternate between $1$ and $-1$. These conditions imply that
every alternating sign matrix is a square matrix.

Permutations can be thought of as matrices whose entries are either $0$ or $1$ 
and in which every row and every column contains a single $1$, and thus ASMs are a
generalization of permutations.

The enumeration of ASMs was first completed by Zeilberger~\cite{zeilberger:asm} who
proved that the number of $n \times n$ ASMs is
\[
	\prod_{k = 0}^{n-1} \frac{(3k + 1)!}{(n + k)!}.
\]
Kuperberg~\cite{kuperberg:asm} and Fischer~\cite{fischer:asm} later gave
shorter proofs.

The containment relation of permutations extends naturally to ASMs and pattern
avoidance has been studied on these objects.
Johansson and Linusson~\cite{Johansson:2008gr} defined what it means for an 
ASM to avoid a permutation as a pattern (which we will repeat below) and
computed the enumeration of $132$-avoiding ASMs. In this brief section, we will
show how Combinatorial Exploration can be applied to ASMs and describe a handful
of strategies that would be sufficient to recover the enumeration of the
$132$-avoiding ASMs.

\begin{definition}
	To any permutation $\pi$ of size $n$ we associate the $n \times n$ matrix
	with a $1$ in each entry $(i, \pi(i))$ and a $0$ elsewhere.
\end{definition}

We choose to index our matrices with Cartesian coordinates to match the standard
way of depicting permutations. For example, the permutation matrix of $132$
is $\scriptstyle\left[\begin{matrix}0&1&0\\0&0&1\\1&0&0\end{matrix}\right]$.

\begin{definition}
	\label{definition:asm-contains-perm}
	An ASM $M$ is said to \emph{contain} a permutation $\pi$ if it is possible to
	form the permutation matrix of $\pi$ from $M$ by deleting rows, deleting columns,
	and changing some $-1$ and $1$ entries into $0$ entries. Otherwise $M$
	is said to \emph{avoid} $\pi$.
\end{definition}
This definition will be generalized below after further definitions have been
made.

\subsection{Gridded and Partial ASMs}

In order to perform Combinatorial Exploration, we need to decide which
combinatorial sets to work with and how they will be represented (the analogue
of tilings), and then define strategies that decompose these sets. Like in the
case of permutations, we will actually work with a gridded version of the objects
at hand\footnote{This is not a requirement to use Combinatorial Exploration,
but we have found it effective. Our penchant for geometrizing objects
may be a result of our own extensive experience working with gridded
permutations.}. Because this section, and the several that follow, are meant to
serve just as ``proof-of-concepts'' for Combinatorial Exploration, we will not
attempt to make the definitions of the objects and strategies involved
completely formal.

\begin{definition}
	A \emph{gridded ASM} is an ASM in which the columns and rows have been
	partitioned into contiguous, possibly-empty parts. Less formally, a
	gridded ASM can be formed by taking an ASM and adding any number of vertical
	lines either between columns or to the extreme left or extreme right, 
	then doing the same with horizontal lines to rows. Like with gridded permutations,
	we will refer to the rectangular regions as \emph{cells}.
\end{definition}

For example,
$\left[\scriptstyle\begin{array}{r|r|rr|r|r}
	&&0 & 0 & 1 & 0\\\hline
	&&0 & 1 & 0 & 0\\
	&&1 & 0 & -1 & 1\\
	&&0 & 0 & 1 & 0\\\hline
	&&&&&
\end{array}\right]$
is a gridded ASM whose underlying ASM is
$\left[\scriptstyle\begin{array}{rrrr}
	0 & 0 & 1 & 0\\
	0 & 1 & 0 & 0\\
	1 & 0 & -1 & 1\\
	0 & 0 & 1 & 0
\end{array}\right]$. Cell $(0,2)$ of the gridded ASM contains no entries, and
cell $(2,1)$ contains six entries.

Because we will be applying strategies that pull ASMs apart into smaller
matrices (similar to the ``factor'' strategy for permutations), we also need to
define a generalization of an ASM that does not enforce the row and column sum
conditions.

\begin{definition}
	A \emph{partial ASM} is a (not necessarily square) matrix whose entries are
	all $0$, $1$, and $-1$ with the property that the nonzero entries in each
	row and column alternate between $-1$ and $1$. It is not required that
	the entries in each row and column sum to $1$, and thus it is permitted
	for the leftmost or rightmost nonzero entry in a row to be $-1$, and
	similarly for columns. \emph{Gridded partial ASMs} are defined in the
	analogous way.
\end{definition}

We can now extend Definition~\ref{definition:asm-contains-perm} by defining when
one partial ASM contains another.

\begin{definition}
	A partial ASM $M_1$ is said to \emph{contain} a partial ASM $M_2$ if it is possible to
	form $M_2$ from $M_1$ by deleting rows, deleting columns,
	and changing some $-1$ and $1$ entries into $0$ entries. Otherwise,
	$M_1$ is said to \emph{avoid} $M_2$.
\end{definition}
This containment relation extends in the natural way to containment of gridded
partial ASMs just as in the case of gridded permutations.

\subsection{ASM-Tilings}

Having now defined the combinatorial objects involved, we will define an analogue
of tilings that represent sets of gridded (partial) ASMs and can be easily
manipulated with combinatorial strategies. We will keep our definition rather
informal, prioritizing understanding over formality because this section is
only meant to serve as a proof-of-concept.

\begin{definition}
	An \emph{ASM-tiling} $\TT$ is a structure that represents a set of gridded
	partial ASMs. It is defined by four components. The first three are identical
	to the three components that define a gridded permutation tiling: dimensions
	$(t, u)$, a set $\OO$ of gridded partial ASMs called \emph{obstructions}, and
	a set $\RR = \{\RR_1, \ldots, \RR_k\}$ of sets of gridded partial ASMs called
	\emph{requirements}. The gridded partial ASMs that $\TT$ represents, 
	$\Grid(\TT)$, must each avoid all of the gridded partial ASMs in $\OO$ and
	must each contain at least one gridded partial ASM in each $\RR_i$.
	However, the fourth component in the definition adds a further condition
	that must be met by each element of $\Grid(\TT)$.
	
	The fourth component is a marking of each portion of the border of the
	ASM-tiling as either ``partial'' or ``complete''. For example,
	if the part of the border along the top of an ASM-tiling above column
	$c$ is marked ``partial'' (pictorially, we will draw that part with an
	oscillating line), this represents that gridded partial ASMs $M$ in
	$\Grid(\TT)$ may, in any columns of $M$ that correspond to column $c$
	in $\TT$, have a topmost nonzero entry equal to $-1$ (which is
	permitted in partial ASMs, but not standard (complete) ASMs).
	Similarly, if the part of the border along the bottom of column $c$
	is marked ``partial'', then columns of $M$ that correspond to
	column $c$ in $\TT$ may have a bottommost nonzero entry $-1$.
	Similar conditions apply to rows.
\end{definition}

Consider the example ASM-tiling $\TT$ shown on the left in
Figure~\ref{figure:asm-tiling-perm-example}. The dimensions of $\TT$ are
$(t,u) = (3,2)$, and there is a single requirement $r$ that is the gridded
ASM whose matrix consists of a single entry with the
value $1$ in cell $(1,0)$. There are five obstructions that consist of
two $0$ entries. Avoiding two $0$ entries vertically in a cell forces any
gridded partial ASM drawn on $\TT$ to have at most a single row of entries
in that cell (and thus in the entire row of that cell in $\TT$).
Similarly, avoiding two $0$ entries horizontally forces at most a single
column of entries. As a result, the requirement $r$ together with the
five obstructions consisting of two zeros essentially isolate the ``1''
entry much like point placement does on gridded permutation tilings.

\begin{figure}
	\centering
	\begin{tikzpicture}[baseline=(current bounding box.center)]
		\node (asmex) at (2, -2) {
			\asmtiling{1.0}{3}{2}{}%
			{% obstructions
			    {1/{(0.8, 0.52)/1}},%
			    {1/{(0.3, 0.5)/{-1}}},%
				{2/{(0.25, 1.8)/{-1}, (0.85, 1.6)/1}},%
				{2/{(0.5, 1.15)/{-1}, (2.45, 1.85)/{-1}}},%
				{3/{(2.15, 1.15)/1, (2.5, 1.4)/1, (2.85, 1.65)/1}},%
                % Vertical 00
				{2/{(0.6, 0.25)/0, (0.6, 0.75)/0}},%
				{2/{(1.85, 0.25)/0, (1.85, 0.75)/0}},%
				{2/{(2.5, 0.25)/0, (2.5, 0.75)/0}},%
				% Horizontal 00
				{2/{(1.15, 0.87)/0, (1.65, 0.87)/0}},%
				{2/{(1.25, 1.85)/0, (1.75, 1.85)/0}}%
			}
			{{1/{(1.4, 0.4)/1}}}
			{0/0, 0/2}%
			{0/1}
			{} %actual rows: 0/, actual columns: 1/
		};
		\node (asmex2) at (6.5, -2) {
			\asmtiling{1.0}{3}{2}{1/0/1}%
			{% obstructions
			    {1/{(0.5, 0.5)/2}},%
				{2/{(0.25, 1.8)/{-1}, (0.85, 1.6)/1}},%
				{2/{(0.5, 1.15)/{-1}, (2.45, 1.85)/{-1}}},%
				{3/{(2.15, 1.15)/1, (2.5, 1.4)/1, (2.85, 1.65)/1}}%
			}
			{}
            {0/0, 0/2}%
			{0/1}
			{0/0, 1/1} %actual rows: 0/, actual columns: 1/
		};
	\end{tikzpicture}
	\qquad
	$\left[\scriptstyle\begin{array}{rrr|r|rrrr}
	1 &  0 & -1 &  0 & 0 &  1 & 0 & 0 \\
	0 & -1 &  0 &  1 & 0 & -1 & 1 & 0 \\
	0 &  1 &  0 & -1 & 0 &  1 & 0 & 0 \\\hline
	0 &  0 &  0 &  1 & 0 &  0 & 0 & 0
\end{array}\right]$
	\caption{An ASM-tiling $\TT$ shown with all obstructions and requirements (left)
	and the same tiling (middle) shown with the convention that actual rows and columns have an asterisk ($\ast$) next to them, cells that can only contain $0$ have a point obstruction, and placed entries
	are bold. On the right is a partial gridded ASM in $\Grid(\TT)$.}
	\label{figure:asm-tiling-perm-example}
\end{figure}
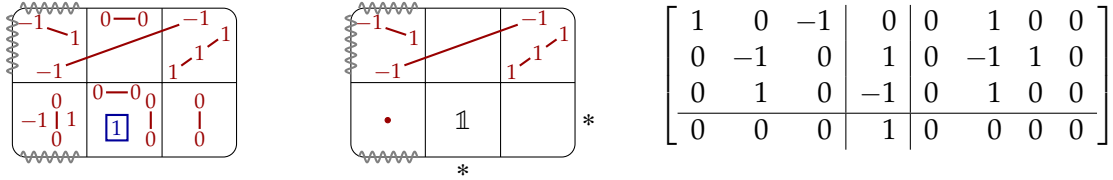

There are two size $1$ obstructions in the bottom-left cell, which forbid
any gridded partial ASM drawn on $\TT$ from containing a non-zero entry
gridded in that cell.
The other three obstructions are shown with an abbreviated notation
in which the $0$s have been omitted. For example, the size $2$ obstruction
in the top-left corner forces any gridded partial ASM drawn on $\TT$ to
avoid the pattern
$\scriptstyle\left[\begin{matrix}-1&0\\0&1\end{matrix}\right]$
in its top-left cell. The size $3$ obstruction in the top-right corner
forces any gridded partial ASM draw on $\TT$ to avoid the pattern 
$\scriptstyle\left[\begin{matrix}0&0&1\\0&1&0\\1&0&0\end{matrix}\right]$
in its top-right corner. Finally, the size $2$ obstruction that crosses
between two cells forces any gridded partial ASM drawn on $\TT$ to not
contain a $\scriptstyle\left[\begin{matrix}0&-1\\-1&0\end{matrix}\right]$
pattern in which the first column is in the top-left cell and the second
column is in the top-right cell.

Lastly, in this example, three segments of the boundary of $\TT$ have been
designated ``partial''. This implies that any columns of a gridded partial
ASM drawn on $\TT$ that are gridded into the leftmost column of $\TT$ are
permitted to have a topmost entry equal to $-1$ and a bottommost entry
equal to $-1$. Similarly, any rows that are gridded into the topmost row of
$\TT$ are permitted to have a leftmost entry equal to $-1$, but not a
rightmost entry.

In the middle of Figure~\ref{figure:asm-tiling-perm-example} the same tiling
is shown with the convention that rows of $\TT$ that can only contain a single
row of any gridded partial ASM have an asterisk ($\ast$) on the boundary, and similarly
for columns. Cells which can only contain $0$ are marked with a point obstruction
and placed entries are shown bold.
On the right is one example of
a gridded partial ASM that belongs to $\Grid(\TT)$.

\subsection{ASMs Avoiding the Pattern 132}
\label{subsection:132-ASMs}

One result of Johansson and Linusson~\cite{Johansson:2008gr} is
the enumeration of ASMs that avoid the permutation $132$ (in the sense of
Definition~\ref{definition:asm-contains-perm}). Here we show that their
result can be automatically discovered using Combinatorial Exploration
on the objects described above. We see this as justification that a
fully implemented version of Combinatorial Exploration in the ASM domain
would be likely to discover many new results, including the open
questions raised in~\cite{Johansson:2008gr}.

We are also using this opportunity to preview a significant extension
of Combinatorial Exploration that will be detailed in later work.
Figures~\ref{figure:asmf} and~\ref{figure:asmf2} depict a set of combinatorial rules that
provide sufficient structural information to compute the enumeration of
the $132$-avoiding ASMs, but the reader will immediately notice the figures
show, what appear to be, two separate proof trees, not one.

%!TEX root = combinatorial-exploration.tex

\begin{figure}
    \centering
    \begin{tikzpicture}[scale=1]

	  \input{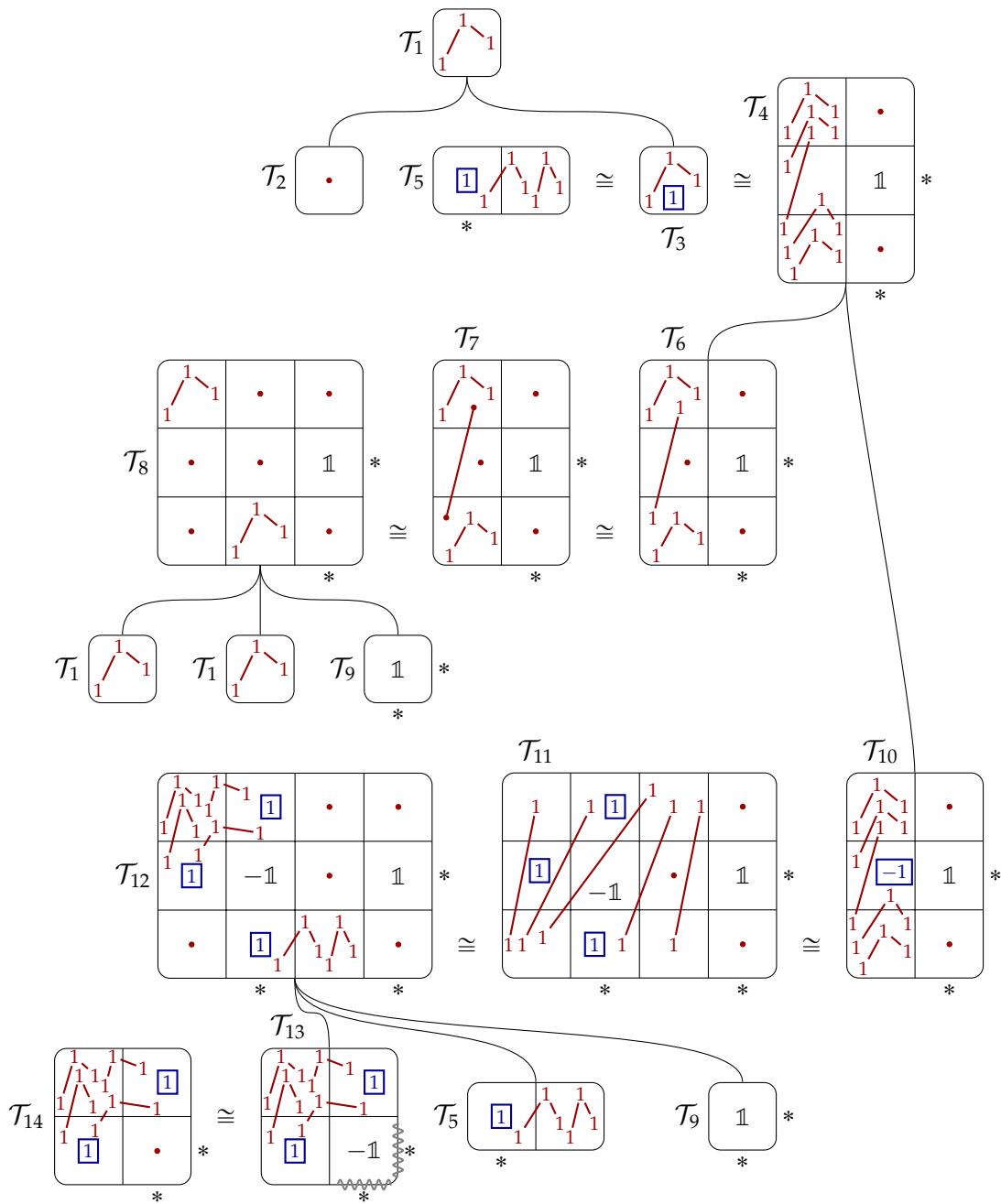}

	  \begin{scope}[shift={(-4,0.9)}]
          
          \node (t6) at (5, -5.2) {$\TT_6$};
          \node (placed1insempty) at (5.5, -7) {
      	  \asmtiling{1.0}{2}{3}{1/1/1}%
          {%
          {1/{(1.5, 0.5)/2}},%
          {1/{(1.5, 2.5)/2}},%
          {1/{(0.7, 1.5)/2}},%
          {2/{(0.2, 0.7)/1, (0.6, 2.3)/1}},%
	      {3/{(0.25, 0.15)/1, (0.55, 0.7)/1, (0.9, 0.4)/1}},%
%	      {3/{(0.15, 0.45)/1, (0.65, 1.22)/1, (0.9, 0.8)/1}},%
          {3/{(0.15, 2.2)/1, (0.45, 2.83)/1, (0.85, 2.50)/1}}%
%          {3/{(0.15, 1.7)/1, (0.5, 2.5)/1, (0.85, 2.20)/1}}%
          }
          {%
          }
          {}{} %squiggly lines
		  {0/1,1/1} %actual rows: 0/, actual columns: 1/
          };
      \node (equiv3) at (4 , -8) {$\cong$};
      \node (equiv4) at (1 , -8) {$\cong$};
      
      	  \node (t10) at (8, -11.2) {$\TT_{10}$};
          \node (placed1insnonempty) at (8.5, -13) {
      	  \asmtiling{1.0}{2}{3}{1/1/1}%
          {%
          {1/{(1.5, 0.5)/2}},%
          {1/{(1.5, 2.5)/2}},%
          {2/{(0.1, 0.8)/1, (0.5, 2.2)/1}},%
	      {3/{(0.25, 0.15)/1, (0.55, 0.7)/1, (0.9, 0.4)/1}},%
	      {3/{(0.15, 0.45)/1, (0.65, 1.22)/1, (0.9, 0.8)/1}},%
          {3/{(0.15, 2.2)/1, (0.45, 2.83)/1, (0.85, 2.50)/1}},%
          {3/{(0.15, 1.7)/1, (0.5, 2.5)/1, (0.85, 2.20)/1}}%
          }
          {%
          	{1/{(0.7, 1.55)/{-1}}}%
          }
          {}{} %squiggly lines
		  {0/1,1/1} %actual rows: 0/, actual columns: 1/
          };
          
          \node (t7) at (2, -5.2) {$\TT_7$};
          \node (placed1insemptyinf) at (2.5, -7) {
      	  \asmtiling{1.0}{2}{3}{1/1/1}%
          {%
          {1/{(1.5, 0.5)/2}},%
          {1/{(1.5, 2.5)/2}},%
          {1/{(0.7, 1.5)/2}},%
          {2/{(0.2, 0.7)/2, (0.6, 2.3)/2}},%
	      {3/{(0.25, 0.15)/1, (0.55, 0.7)/1, (0.9, 0.4)/1}},%
          {3/{(0.15, 2.2)/1, (0.45, 2.83)/1, (0.85, 2.50)/1}}%
          }
          {%
          }
          {}{} %squiggly lines
		  {0/1,1/1} %actual rows: 0/, actual columns: 1/
          };
          
          \node (t8) at (-2.8, -7) {$\TT_8$};
          \node (placed1insemptyinfsep) at (-1.0, -7) {
      	  \asmtiling{1.0}{3}{3}{2/1/1}%
          {%
          {1/{(0.5, 0.5)/2}},%
          {1/{(2.5, 0.5)/2}},%
          {1/{(1.5, 2.5)/2}},%
          {1/{(2.5, 2.5)/2}},%
          {1/{(0.5, 1.5)/2}},%
          {1/{(1.5, 1.5)/2}},%
	      {3/{(1.15, 0.2)/1, (1.45, 0.83)/1, (1.85, 0.50)/1}},%
          {3/{(0.15, 2.2)/1, (0.45, 2.83)/1, (0.85, 2.50)/1}}%
          }
          {%
          }
          {}{} %squiggly lines
		  {0/1,1/2} %actual rows: 0/, actual columns: 1/
          };
          
          \node (t1b) at (-3.8, -10) {$\TT_1$};
		  \node (rootb) at (-3, -10) {
		  \asmtiling{1.0}{1}{1}{}%
          {%
        	{3/{(0.15, 0.2)/1, (0.45, 0.83)/1, (0.85, 0.50)/1}}%
          }
          {%
          }
          {}{} %squiggly lines
          {} %actual rows: 0/, actual columns: 1/
          };
          
          \node (t1c) at (-1.8, -10) {$\TT_1$};
		  \node (rootc) at (-1, -10) {
		  \asmtiling{1.0}{1}{1}{}%
          {%
        	{3/{(0.15, 0.2)/1, (0.45, 0.83)/1, (0.85, 0.50)/1}}%
          }
          {%
          }
          {}{} %squiggly lines
          {} %actual rows: 0/, actual columns: 1/
          };
          
          \node (t9) at (0.2, -10) {$\TT_9$};
		  \node (uno) at (1, -10) {
		  \asmtiling{1.0}{1}{1}{0/0/1}%
          {%
          }
          {%
          }
          {}{} %squiggly lines
          {0/0, 1/0} %actual rows: 0/, actual columns: 1/
          };
          \ptedge{(placed1insemptyinfsep)}{(-0.5,-0.19)}{(rootb)}{(-0.5,0.8)}
          \ptedge{(placed1insemptyinfsep)}{(-0.5,-0.19)}{(rootc)}{(-0.5,0.8)}
          \ptedge{(placed1insemptyinfsep)}{(-0.5,-0.19)}{(uno)}{(-0.5,0.8)}
          
          \node (t11) at (3, -11.2) {$\TT_{11}$};
          \node (placed1insnonemptyinf) at (4.5, -13) {
      	  \asmtiling{1.0}{4}{3}{1/{0.75}/{-1},3/1/1}%
          {%
          {1/{(2.5, 1.5)/2}},%
          {1/{(3.5, 0.5)/2}},%
          {1/{(3.5, 2.5)/2}},%
          {2/{(0.1, 0.5)/1, (0.5, 2.5)/1}},%
          {2/{(0.3, 0.5)/1, (1.3, 2.5)/1}},%
          {2/{(0.6, 0.6)/1, (2.2, 2.75)/1}},%
          {2/{(1.75, 0.5)/1, (2.5, 2.5)/1}},%
          {2/{(2.5, 0.5)/1, (2.9, 2.5)/1}}%
%	      {3/{(0.25, 0.15)/1, (0.55, 0.7)/1, (0.9, 0.4)/1}},%
%	      {3/{(0.15, 0.45)/1, (0.65, 1.22)/1, (0.9, 0.8)/1}},%
%          {3/{(0.15, 2.2)/1, (0.45, 2.83)/1, (0.85, 2.50)/1}},%
%          {3/{(0.15, 1.7)/1, (0.5, 2.5)/1, (0.85, 2.20)/1}}%
          }
          {%
            {1/{(1.35, 0.5)/1}},%
            {1/{(0.55, 1.57)/1}},%
            {1/{(1.65, 2.5)/1}}%
          }
          {}{} %squiggly lines
		  {0/1,1/1,1/3} %actual rows: 0/, actual columns: 1/
          };
          \node (t12) at (-2.85, -13) {$\TT_{12}$};
          \node (placed1insnonemptyinfinf) at (-0.5, -13) {
      	  \asmtiling{1.0}{4}{3}{1/1/{-1},3/1/1}%
          {%
          {1/{(0.5, 0.5)/2}},%
          {1/{(2.5, 2.5)/2}},%
          {1/{(2.5, 1.5)/2}},%
          {1/{(3.5, 0.5)/2}},%
          {1/{(3.5, 2.5)/2}},%
	      {3/{(2.5, 0.2)/1, (2.65, 0.83)/1, (2.85, 0.40)/1}},%
	      {3/{(1.75, 0.2)/1, (2.15, 0.83)/1, (2.35, 0.40)/1}},%
          {3/{(0.1, 2.2)/1, (0.3, 2.87)/1, (0.6, 2.60)/1}},%
          {3/{(0.75, 2.45)/1, (0.86, 2.87)/1, (1.3, 2.70)/1}},%
          {3/{(0.15, 1.7)/1, (0.37, 2.6)/1, (0.57, 2.15)/1}},%
          {3/{(0.6, 1.8)/1, (0.85, 2.22)/1, (1.5, 2.12)/1}}%
          }
          {%
            {1/{(1.5, 0.5)/1}},%
            {1/{(0.5, 1.5)/1}},%
            {1/{(1.65, 2.5)/1}}%
          }
          {}{} %squiggly lines
		  {0/1,1/1,1/3} %actual rows: 0/, actual columns: 1/
          };
      \node (equiv5) at (7 , -14) {$\cong$};
      \node (equiv6) at (2 , -14) {$\cong$};
      
      	  \node (t13) at (-0.6, -15.2) {$\TT_{13}$};
          \node (placed1insnonemptyinfinfF1) at (0, -16.5) {
      	  \asmtiling{1.0}{2}{2}{1/0/{-1}}%
          {%
          {3/{(0.1, 1.2)/1, (0.3, 1.87)/1, (0.6, 1.60)/1}},%
          {3/{(0.75, 1.45)/1, (0.86, 1.87)/1, (1.3, 1.70)/1}},%
          {3/{(0.15, 0.7)/1, (0.37, 1.6)/1, (0.57, 1.15)/1}},%
          {3/{(0.6, 0.8)/1, (0.85, 1.22)/1, (1.5, 1.12)/1}}%
          }
          {%
            {1/{(0.5, 0.5)/1}},%
            {1/{(1.65, 1.5)/1}}%
          }
          {1/0}{2/0} %squiggly lines
		  {0/0,1/1} %actual rows: 0/, actual columns: 1/
          };
          \node (t5b) at (1.7, -16.5) {$\TT_5$};
          \node (placed1insnonemptyinfinfF2) at (3, -16.5) {
      	  \asmtiling{1.0}{2}{1}{}%
          {%
	      {3/{(1.5, 0.2)/1, (1.65, 0.83)/1, (1.85, 0.40)/1}},%
	      {3/{(0.75, 0.2)/1, (1.15, 0.83)/1, (1.35, 0.40)/1}}%
          }
          {%
            {1/{(0.5, 0.5)/1}}%
          }
          {}{} %squiggly lines
		  {1/0} %actual rows: 0/, actual columns: 1/
          };
          \node (t9b) at (5.2, -16.5) {$\TT_9$};
		  \node (placed1insnonemptyinfinfF3) at (6, -16.5) {
		  \asmtiling{1.0}{1}{1}{0/0/1}%
          {%
          }
          {%
          }
          {}{} %squiggly lines
          {0/0, 1/0} %actual rows: 0/, actual columns: 1/
          };
          \ptedge{(placed1insnonemptyinfinf)}{(-0.5,-0.19)}{(placed1insnonemptyinfinfF1)}{(-0.5,1.3)}
          \ptedge{(placed1insnonemptyinfinf)}{(-0.5,-0.19)}{(placed1insnonemptyinfinfF2)}{(-0.5,0.8)}
          \ptedge{(placed1insnonemptyinfinf)}{(-0.5,-0.19)}{(placed1insnonemptyinfinfF3)}{(-0.5,0.8)}
          
          \node (t14) at (-4.4, -16.5) {$\TT_{14}$};
          \node (placed1insnonemptyinfinfF1inf) at (-3, -16.5) {
      	  \asmtiling{1.0}{2}{2}{}%
          {%
          {1/{(1.5, 0.5)/2}},%
          {3/{(0.1, 1.2)/1, (0.3, 1.87)/1, (0.6, 1.60)/1}},%
          {3/{(0.75, 1.45)/1, (0.86, 1.87)/1, (1.3, 1.70)/1}},%
          {3/{(0.15, 0.7)/1, (0.37, 1.6)/1, (0.57, 1.15)/1}},%
          {3/{(0.6, 0.8)/1, (0.85, 1.22)/1, (1.5, 1.12)/1}}%
          }
          {%
            {1/{(0.5, 0.5)/1}},%
            {1/{(1.65, 1.5)/1}}%
          }
          {}{} %squiggly lines
		  {0/0,1/1} %actual rows: 0/, actual columns: 1/
          };
          \node (equiv7) at (-1.5 , -16.5) {$\cong$};
          
          \ptedge{(placed1)}{(-0.5,-0.19)}{(placed1insempty)}{(-0.5,1.8)}
          \ptedge{(placed1)}{(-0.5,-0.19)}{(placed1insnonempty)}{(-0.5,1.8)}
		\end{scope}
    \end{tikzpicture}
    \caption{The first partial proof tree for $132$-avoiding ASMs. Note that
    obstructions which use points should be interpreted to forbid both $1$ and
    $-1$ entries, as in tiling $\TT_7$. In some tilings, like $\TT_{11}$ we have
    not shown all obstructions, only the most important ones.}
    \label{figure:asmf}
  \end{figure}

%!TEX root = combinatorial-exploration.tex

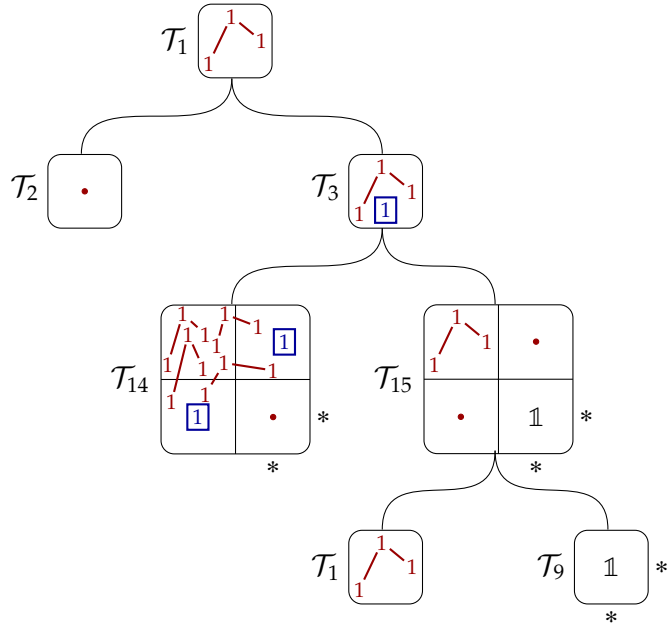
\begin{figure}
    \centering
    \begin{tikzpicture}[scale=1]

\node (t1) at (-2.8, 0) {$\TT_1$};
\node (root) at (-2, 0) {
\asmtiling{1.0}{1}{1}{}%
{%
	{3/{(0.15, 0.2)/1, (0.45, 0.83)/1, (0.85, 0.50)/1}}%
}
{%
}
{}{} %squiggly lines
{} %actual rows: 0/, actual columns: 1/
};
\node (t2) at (-4.8, -1.95) {$\TT_2$};
\node (empty) at (-4, -2) {
\asmtiling{1.0}{1}{1}{}%
{%
    {1/{(0.5, 0.5)/2}}%
}
{%
}
{}{} %squiggly lines
{} %actual rows: 0/, actual columns: 1/
};
\node (t3) at (-0.8, -1.95) {$\TT_3$};
\node (nonempty) at (0, -2) {
\asmtiling{1.0}{1}{1}{}%
{%
	{3/{(0.15, 0.2)/1, (0.45, 0.83)/1, (0.85, 0.50)/1}}%
}
{%
	{1/{(0.5, 0.25)/1}}%
}
{}{} %squiggly lines
{} %actual rows: 0/, actual columns: 1/
};
      
	\ptedge{(root)}{(-0.545,0.8)}{(empty)}{(-0.545,0.8)}
	\ptedge{(root)}{(-0.545,0.8)}{(nonempty)}{(-0.545,0.8)}
          
          \node (t14) at (-3.4, -4.5) {$\TT_{14}$};
          \node (placed1insnonemptyinfinfF1inf) at (-2, -4.5) {
          \asmtiling{1.0}{2}{2}{}%
          {%
          {1/{(1.5, 0.5)/2}},%
          {3/{(0.1, 1.2)/1, (0.3, 1.87)/1, (0.6, 1.60)/1}},%
          {3/{(0.75, 1.45)/1, (0.86, 1.87)/1, (1.3, 1.70)/1}},%
          {3/{(0.15, 0.7)/1, (0.37, 1.6)/1, (0.57, 1.15)/1}},%
          {3/{(0.6, 0.8)/1, (0.85, 1.22)/1, (1.5, 1.12)/1}}%
          }
          {%
            {1/{(0.5, 0.5)/1}},%
            {1/{(1.65, 1.5)/1}}%
          }
          {}{} %squiggly lines
		  {0/0,1/1} %actual rows: 0/, actual columns: 1/
          };
      
      	  \node (t10) at (0.1, -4.5) {$\TT_{15}$};
          \node (placed1insnonempty) at (1.5, -4.5) {
      	  \asmtiling{1.0}{2}{2}{1/0/1}%
          {%
          {1/{(1.5, 1.5)/2}},%
          {1/{(0.5, 0.5)/2}},%
          {3/{(0.15, 1.2)/1, (0.45, 1.83)/1, (0.85, 1.50)/1}}%
          }
          {%
          }
          {}{} %squiggly lines
		  {0/0,1/1} %actual rows: 0/, actual columns: 1/
          };
          
          \node (t1c) at (-0.8, -7) {$\TT_1$};
		  \node (rootc) at (0, -7) {
		  \asmtiling{1.0}{1}{1}{}%
          {%
        	{3/{(0.15, 0.2)/1, (0.45, 0.83)/1, (0.85, 0.50)/1}}%
          }
          {%
          }
          {}{} %squiggly lines
          {} %actual rows: 0/, actual columns: 1/
          };
          
          \node (t9) at (2.2, -7) {$\TT_9$};
		  \node (uno) at (3, -7) {
		  \asmtiling{1.0}{1}{1}{0/0/1}%
          {%
          }
          {%
          }
          {}{} %squiggly lines
          {0/0, 1/0} %actual rows: 0/, actual columns: 1/
          };
          \ptedge{(placed1insnonempty)}{(-0.545,0.35)}{(rootc)}{(-0.545,0.8)}
          \ptedge{(placed1insnonempty)}{(-0.545,0.35)}{(uno)}{(-0.545,0.8)}
          
          \ptedge{(nonempty)}{(-0.545,0.81)}{(placed1insnonemptyinfinfF1inf)}{(-0.545,1.3)}
          \ptedge{(nonempty)}{(-0.545,0.81)}{(placed1insnonempty)}{(-0.545,1.3)}
    \end{tikzpicture}
    \caption{The second partial proof tree for $132$-avoiding ASMs.}
    \label{figure:asmf2}
  \end{figure}

Neither of these are actually proof trees, because their corresponding
combinatorial specifications do not have the property that every set
appearing on a right-hand side also appears on exactly one left-hand
side (in terms of the proof tree perspective, this condition requires
that every set that is the child of a rule is also the parent of
exactly one rule). So, while the set of combinatorial rules formed from
taking both trees together is not a combinatorial specification, it
still contains sufficient information to find the enumeration of all
sets involved.\footnote{However, we would not be surprised if an
implemented version of Combinatorial Exploration found a genuine
combinatorial specification.} We call such a set of rules a
\emph{combinatorial forest}. These generalize combinatorial
specifications and we will describe in future work how the
Combinatorial Exploration algorithm can be modified to search for
combinatorial forests, along with a productivity-style guarantee
that they suffice to enumerate each of the sets involved.

We have surpressed some obstructions in the tilings in the two trees that
are not relevant, like in tiling $\TT_{11}$ in Figure~\ref{figure:asmf}.
We use a point in an obstruction to mean either a $1$ or a $-1$.

We will end this section by roughly describing the strategies that
produce the rules shown in the two partial proof trees in
Figures~\ref{figure:asmf} and~\ref{figure:asmf2} and then solving
the corresponding system of equations.

\begin{itemize}[label=$\diamond$]
	\item There is a disjoint-union-type strategy similar to the 
		requirement insertion strategy on gridded permutations that
		decomposes a set into two subsets depending on whether a
		specific gridded partial ASM is contained or avoided. The rule
		$\TT_1 \leftarrow (\TT_2, \TT_3)$ inserts the
		gridded ASM
		$\scriptstyle\left[\begin{array}{c}1\end{array}\right]$
		as a requirement into cell $(0,0)$. The rule
		$\TT_4 \leftarrow (\TT_6, \TT_{10})$
		inserts the requirement 
		$\scriptstyle\left[\begin{array}{c|c} & \\\hline -1 & \\\hline &\end{array}\right]$.
	\item There are equivalence strategies $\TT_3 \leftarrow (\TT_4)$
		and $\TT_{10} \leftarrow (\TT_{11})$ that are essentially the
		same as the point placement strategy for gridded permutations.
		The rule $\TT_3 \leftarrow (\TT_5)$ is also produced by a
		point-placement-style strategy, but in this case it does not
		isolate the $1$ in its own row, just in its own column.
	\item The rule $\TT_3 \leftarrow (\TT_{14}, \TT_{15})$ is a
		disjoint-union-type strategy that places the ``$1$'' requirement
		depending on whether the rightmost $1$ is in the bottom-right
		corner or not.
	\item The rules $\TT_2 \leftarrow ()$ and $\TT_9 \leftarrow ()$
		are verification strategies, as $\Grid(\TT_2)$ contains only
		the size $0$ gridded ASM and $\Grid(\TT_9)$ contains only
		the size $1$ ASM with the entry $1$.
	\item The equivalence rule $\TT_6 \leftarrow (\TT_7)$ recognizes
		that the presence of the obstruction
		$\scriptstyle\left[\begin{array}{cc|c}
			0&1&\\\hline&\\\hline 1&0
		\end{array}\right]$
		crossing from the bottom-left cell to the top-left cell
		implies the existence of the three additional obstructions
		$\scriptstyle\left[\begin{array}{cc|c}
			0&-1&\\\hline&\\\hline 1&0
		\end{array}\right]$,
		$\scriptstyle\left[\begin{array}{cc|c}
			0&1&\\\hline&\\\hline -1&0
		\end{array}\right]$,
		and
		$\scriptstyle\left[\begin{array}{cc|c}
			0&-1&\\\hline&\\\hline -1&0
		\end{array}\right]$
		that cross in the same way. This is because, for example, if
		the pattern
		$\scriptstyle\left[\begin{array}{cc|c}
			0&1&\\\hline&\\\hline -1&0
		\end{array}\right]$
		occurred, the ASM row sum property would require the presence
		of a $1$ to the left of the $-1$, which would then create the
		forbidden obstruction together with the $1$ in the top-left cell.
		Similar logic leads to the rule $\TT_{11} \leftarrow (\TT_{12})$:
		if either of the cells containing a $1$ obstruction contained
		a $-1$, then the ASM row and column sum properties would
		force them to also contain a $1$, which is forbidden.
	\item The equivalence rule $\TT_7 \leftarrow (\TT_8)$ is
		essentially the same as column separation for gridded
		permutations.
	\item The equivalence rule $\TT_{13} \leftarrow (\TT_{14})$
		observes that the mapping of taking a gridded partial ASM
		in $\Grid(\TT_{13})$ and changing the $-1$ in the bottom-right
		corner into a $0$ is a size-preserving bijection to
		$\Grid(\TT_{14})$. Note that this conversion now ensures
		that all ASMs in $\Grid(\TT_{14})$ obey the row and column
		sum properties, while this was not true for those partial
		ASMs in $\Grid(\TT_{13})$.
	\item Finally, there are three rules,
		$\TT_8 \leftarrow (\TT_1, \TT_1, \TT_9)$,
		$\TT_{12} \leftarrow (\TT_{13}, \TT_5, \TT_9)$, and
		$\TT_{15} \leftarrow (\TT_1, \TT_9)$,
		that come from a strategy similar to the factor strategy
		of gridded permutations. The first and third rule are
		rather straight-forward: each cell of the parent is in
		its own row and column, and no parts of the border are marked
		as partial, and therefore, the cells of the parent can be split
		into the children shown, all of which also have no parts of
		their own borders marked as partial. The second rule is more
		intricate. Any place where one child shares a row or column
		with another child, those shared parts of the boundary
		must get marked as partial in the children. This explains
		the partial boundary of $\TT_{13}$. At first it seems like
		$\TT_5$ and $\TT_9$ should also each have a part of their
		boundaries marked as partial (the left half of the top
		border of $\TT_5$ and the left border of $\TT_9$), but in each
		case, the $-1$ cell that stayed with $\TT_{13}$ implies that
		the row and column sum conditions actually remain satisfied
		on $\TT_5$ and $\TT_9$.
\end{itemize}

From this combinatorial forest (the union of the rules in each of the
partial proof trees), we obtain an algebraic system of equations as
outlined in Section~\ref{section:transfer-tools}, which can be solved
to discover that the generating function for $\TT_1$ is
\[
	T_1(x) = \frac{3-x-\sqrt{1-6x+x^2}}{2}
\]
and therefore the counting sequence for the $132$-avoiding ASMs is
the large Schr\"oder numbers:
\[
	1, 1, 2, 6, 22, 90, 394, 1806, 8558, \ldots.
\]

%% ==== %% ==== %% ==== %% ==== %% ==== %% ==== %% ==== %% ==== %% ==== %% ==== %%
%% ==== %% ==== %% ==== %% ====   SECTION SEVEN    ==== %% ==== %% ==== %% ==== %%
%% ==== %% ==== %% ==== %% ==== %% ==== %% ==== %% ==== %% ==== %% ==== %% ==== %%

\section{Applying Combinatorial Exploration to Polyominoes}
\label{section:polyomino-results}
%!TEX root = combinatorial-exploration.tex

\subsection{Polyominoes}
In this section a \emph{square} refers to a region $[i, i+1] \times [j, j+1]$
in the Euclidean plane. A \emph{polyomino} is a finite union of squares with a
connected interior. So, for example, two squares that only share a corner
would not be a polyomino. Two polyominoes are considered equal if one can be
transformed into the other by a vertical or a horizontal translation of the
plane.%
\footnote{Some texts consider two polyominoes equal if one can be
obtained from the other by rotation, but we consider these to be
different.}
The word ``cell'' is more common for what we call a ``square'', but we
will use ``cell'' for another concept below.

Polyominoes have been used for modeling physical phenomena such as
crystal-growth (see Dhar~\cite{Dhar}) and percolation (see 
Broadbent and Hammersley~\cite{broadbent_hammersley_1957},
Conway and Guttmann~\cite{Conway_2dim}, and
Rensburg~\cite{Rensburg}).
Figure~\ref{figure:smallpolys} shows the nonempty
polyominoes with at most three squares.

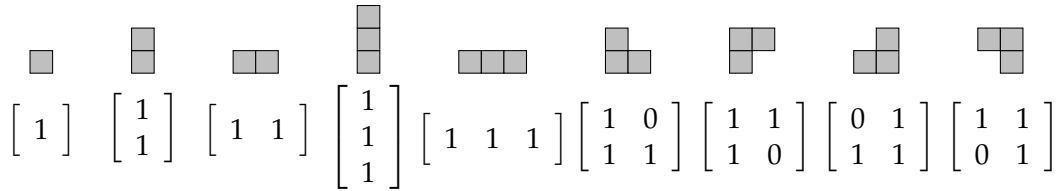
\begin{figure}
	\begin{center}
	\begin{tikzpicture}[scale=0.3, baseline=(current bounding box.center)]
    \foreach \x/\y in {0/0}
      {%
        \filldraw[fill=lightgray] (\x, \y) rectangle ++(1,1);
      }  
    \node at (0.5,-2.5){  
    $\left[\scriptstyle\begin{array}{c}
	1 
    \end{array}\right]$
    };
    
    \begin{scope}[shift={(4.5,0)}]
    \foreach \x/\y in {0/0, 0/1}
      {%
        \filldraw[fill=lightgray] (\x, \y) rectangle ++(1,1);
      }
    \node at (0.5,-2.5){  
    $\left[\scriptstyle\begin{array}{c}
	1 \\
	1
    \end{array}\right]$
    };
    \end{scope}
    
    \begin{scope}[shift={(9,0)}]
    \foreach \x/\y in {0/0, 1/0}
      {%
        \filldraw[fill=lightgray] (\x, \y) rectangle ++(1,1);
      }
    \node at (1.1,-2.5){  
    $\left[\scriptstyle\begin{array}{cc}
	1 & 1\\
    \end{array}\right]$
    };
    \end{scope}
    
    \begin{scope}[shift={(14.5,0)}]
    \foreach \x/\y in {0/0, 0/1, 0/2}
      {%
        \filldraw[fill=lightgray] (\x, \y) rectangle ++(1,1);
      }
    \node at (0.5,-2.9){  
    $\left[\scriptstyle\begin{array}{c}
	1 \\
	1 \\
	1
    \end{array}\right]$
    };
    \end{scope}
    
    \begin{scope}[shift={(19,0)}]
    \foreach \x/\y in {0/0, 1/0, 2/0}
      {%
        \filldraw[fill=lightgray] (\x, \y) rectangle ++(1,1);
      }
    \node at (1.5,-2.9){  
    $\left[\scriptstyle\begin{array}{ccc}
	1 & 1 & 1
    \end{array}\right]$
    };
    \end{scope}
    
    \begin{scope}[shift={(25.5,0)}]
    \foreach \x/\y in {0/0, 1/0, 0/1}
      {%
        \filldraw[fill=lightgray] (\x, \y) rectangle ++(1,1);
      }
    \node at (1.1,-2.9){  
    $\left[\scriptstyle\begin{array}{cc}
	1 & 0 \\
	1 & 1
    \end{array}\right]$
    };
    \end{scope}
    
    \begin{scope}[shift={(31,0)}]
    \foreach \x/\y in {0/0, 0/1, 1/1}
      {%
        \filldraw[fill=lightgray] (\x, \y) rectangle ++(1,1);
      }
    \node at (1.1,-2.9){  
    $\left[\scriptstyle\begin{array}{cc}
	1 & 1 \\
	1 & 0
    \end{array}\right]$
    };
    \end{scope}
    
    \begin{scope}[shift={(36.5,0)}]
    \foreach \x/\y in {0/0, 1/0, 1/1}
      {%
        \filldraw[fill=lightgray] (\x, \y) rectangle ++(1,1);
      }
    \node at (1.1,-2.9){  
    $\left[\scriptstyle\begin{array}{cc}
	0 & 1 \\
	1 & 1
    \end{array}\right]$
    };
    \end{scope}
    
    \begin{scope}[shift={(42,0)}]
    \foreach \x/\y in {1/0, 0/1, 1/1}
      {%
        \filldraw[fill=lightgray] (\x, \y) rectangle ++(1,1);
      }
    \node at (1.1,-2.9){  
    $\left[\scriptstyle\begin{array}{cc}
	1 & 1 \\
	0 & 1
    \end{array}\right]$
    };
    \end{scope}

    \end{tikzpicture}
	\end{center}
	\caption{Nonempty polyominoes with at most three squares and their matrix representations}
	\label{figure:smallpolys}
\end{figure}

There are at least two natural definitions of the size of a polyomino,
each of which turns the set of polyominoes into a combinatorial set.
These are \emph{area}, the number of squares in the polyomino, as well
as \emph{perimeter}, the distance along the border. These two definitions
are sometimes used simultaneously, leading to bivariate generating
functions. It is worth noting that the total number of polyominoes of
a given size (either area or perimeter) is unknown, a stark contrast to
the case of permutations.%
\footnote{If we take the area as the size function then the enumeration of
polyominoes is $1, 1, 2, 6, 19, 63, 216, \dotsc$
and is sequence A001168 on the OEIS. While the general formula for these
numbers, $a(n)$, is unknown, it is known that the limit $a(n)^{1/n}$
exists (see Klarner~\cite{klarner_1967}), and is estimated to be close
to $4.06$ (Jensen~\cite{Iwan}).}

However, as is the case with permutations and many other combinatorial
objects, several subsets of polyominoes have been intensely studied,
including those that are convex,
directed and convex, column-convex, directed column-convex (see
Bousquet-M\'{e}lou and F\'{e}dou~\cite{bousquet-melou:95-convex} and
Bousquet-M\'{e}lou~\cite{bousquet-melou:96-convex}), inscribed in a
rectangle (see Goupil, Cloutier, and Nouboud~\cite{goupil:rectangle}),
tree-like (see Aleksandrowicz, Asinowski,
Barequet~\cite{Aleksandrowicz:poly-perm}),
centered convex, Z-convex, $4$-stack, bi-centered (see Fedou,
Rinaldi, and Drosini~\cite{fedou:4-stack}),
and L-convex (see Castiglione, Frosini, Restivo, and
Rinaldi~\cite{castiglione:L-convex} and Guttmann and
Kotesovec~\cite{guttmann:L-convex-asymptotics}).

Many of the subsets that have been studied can be defined as the
set of polyominoes that avoid certain patterns, and in this section we
present a proof-of-concept to demonstrate that Combinatorial Exploration
can be effective in the domain of polyominoes as well. Later in this section,
we will derive a proof tree for the set of Ferrers diagrams (whose enumeration
is equal to the number of integer partitions) and describe, but not show,
a proof tree for the L-convex polyominoes (also sometimes called moon
polyominoes).

We start by using the matrix-based description of polyominoes given 
by Frosini, Guerrini, and Rinaldi~\cite{Frosinipatterns} in which squares
of a polyomino are represented by a $1$ and the remaining locations of the
bounding box are represented by a $0$. Figure~\ref{figure:smallpolys}
shows the matrix representations for each polyomino with at most $3$
squares, and Figure~\ref{figure:larger-poly} shows a larger example.
From this description, we can define when one polyomino contains another.

\begin{figure}
	\begin{center}
	\begin{tikzpicture}[scale=0.3, baseline=(current bounding box.center)]
    \foreach \x/\y in {0/3, 1/2, 1/3, 1/4, 2/0, 2/1, 2/2, 2/4, 3/0, 3/1, 3/2, 3/3, 3/4, 4/0, 4/4, 5/0, 5/1, 5/2}
      {%
        \filldraw[fill=lightgray] (\x, \y) rectangle ++(1,1);
      }  
    
    \begin{scope}[shift={(10,0)}]
    \foreach \x/\y in {0/3, 1/2, 1/3, 1/4, 2/0, 2/1, 2/2, 2/4, 3/0, 3/1, 3/2, 3/3, 3/4, 4/0, 4/4, 5/0, 5/1, 5/2}
      {%
        \filldraw[fill=lightgray] (\x, \y) rectangle ++(1,1);
        \node at (\x+0.5, \y+0.5) {$\scriptstyle 1$};
      }
    \foreach \x/\y in {0/0, 0/1, 0/2, 0/4, 1/0, 1/1, 2/3, 4/1, 4/2, 4/3, 5/3, 5/4}
      {%
        \node at (\x+0.5, \y+0.5) {$\scriptstyle 0$};
      }
    \end{scope}
    
    \node at (27,2.5) {
    $\left[\scriptstyle\begin{array}{cccccc}
	0 & 1 & 1 & 1 & 1 & 0 \\
	1 & 1 & 0 & 1 & 0 & 0 \\
	0 & 1 & 1 & 1 & 0 & 1 \\
	0 & 0 & 1 & 1 & 0 & 1 \\
	0 & 0 & 1 & 1 & 1 & 1
    \end{array}\right]$
    };

    \end{tikzpicture}
	\end{center}
	\caption{A polyomino, with its matrix drawn on top and shown separately.}
	\label{figure:larger-poly}
\end{figure}
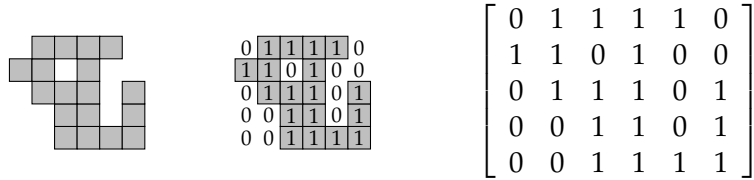

\begin{definition}[\cite{Frosinipatterns}]
	A polyomino $P_1$ is said to \emph{contain} a polyomino $P_2$ if it
	is possible to form the matrix for $P_2$ from the matrix for $P_1$
	by deleting rows and deleting columns.
\end{definition}
We may extend this definition slightly by allowing $P_2$ to
be a matrix that does not satisfy the bounding box criteria, e.g., 
the polyominoes that avoid
$P_2 = \left[\begin{matrix}1&0\end{matrix}\right]$
are those that do not have a square to the left of a non-square (within
the bounding box of the polyomino). In this context, we call matrices
such as $P_2$ \emph{polyomino patterns}.

Following the development of Combinatorial Exploration in the domain of
alternating sign matrices in Section~\ref{section:asm-results}, we will
now define a gridded version of polyominoes and a tiling-like structure
that represents the sets of gridded polyominoes upon which we will
perform Combinatorial Exploration. 

\subsection{Gridded Polyominoes}
\begin{definition}
	A \emph{gridded polyomino} is a polyomino in which the columns and
	rows have been partitioned into contiguous, possibly-empty parts.
	Less formally, a gridded polyomino can be formed by taking the
	matrix representation of a polyomino
	and adding any number of vertical lines either between columns or to
	the extreme left or extreme right, then doing the same with horizontal
	lines to rows. Like with gridded permutations and gridded ASMs,
	we will refer to the rectangular regions as \emph{cells}.
	Figure~\ref{figure:gridded-poly-example} shows an example of a
	gridded polyomino.
\end{definition}

\begin{figure}
	\begin{center}
		
	\begin{tikzpicture}[scale=0.3, baseline=(current bounding box.center)]
    \foreach \x/\y in {0/2, 1/1, 1/2}
      {%
        \filldraw[fill=lightgray] (\x, \y) rectangle ++(1,1);
      }
    \begin{scope}[shift={(2.75,0)}]
    \foreach \x/\y in {0/0, 0/1, 1/0, 1/1, 1/2}
      {%
        \filldraw[fill=lightgray] (\x, \y) rectangle ++(1,1);
      }
    \end{scope}
    \begin{scope}[shift={(5.25,0)}]
    \foreach \x/\y in {1/0, 1/1}
      {%
        \filldraw[fill=lightgray] (\x, \y) rectangle ++(1,1);
      }
    \end{scope}
    
    \begin{scope}[shift={(2.75,-1.5)}]
    \foreach \x/\y in {0/0, 1/0}
      {%
        \filldraw[fill=lightgray] (\x, \y) rectangle ++(1,1);
      }
    \end{scope}
    \begin{scope}[shift={(5.25,-1.5)}]
    \foreach \x/\y in {0/0, 1/0}
      {%
        \filldraw[fill=lightgray] (\x, \y) rectangle ++(1,1);
      }
    \end{scope}
    
    \begin{scope}[shift={(0,3.5)}]
    \foreach \x/\y in {1/0}
      {%
        \filldraw[fill=lightgray] (\x, \y) rectangle ++(1,1);
      }
    \end{scope}
    \begin{scope}[shift={(2.75,3.5)}]
    \foreach \x/\y in {0/0, 1/0}
      {%
        \filldraw[fill=lightgray] (\x, \y) rectangle ++(1,1);
      }
    \end{scope}
    \begin{scope}[shift={(5.25,3.5)}]
    \foreach \x/\y in {0/0}
      {%
        \filldraw[fill=lightgray] (\x, \y) rectangle ++(1,1);
      }
    \end{scope}
    
    \draw (-0.25,-1.75) -- (-0.25,4.75);
    \draw (2.25,-1.75) -- (2.25,4.75);
    \draw (2.5,-1.75) -- (2.5,4.75);
    \draw (5,-1.75) -- (5,4.75);
    \draw (-0.75,-0.25) -- (7.5,-0.25);
    \draw (-0.75,3.25) -- (7.5,3.25);

    \end{tikzpicture}
    \quad\quad\quad\quad
    $\left[\scriptstyle\begin{array}{|cc||cc|cc}
	0 & 1 & 1 & 1 & 1 & 0 \\\hline
	1 & 1 & 0 & 1 & 0 & 0 \\
	0 & 1 & 1 & 1 & 0 & 1 \\
	0 & 0 & 1 & 1 & 0 & 1 \\\hline
	0 & 0 & 1 & 1 & 1 & 1
    \end{array}\right]$
	\end{center}
	\caption{A gridding of the polyomino from Figure~\ref{figure:larger-poly}.}
	\label{figure:gridded-poly-example}
\end{figure}
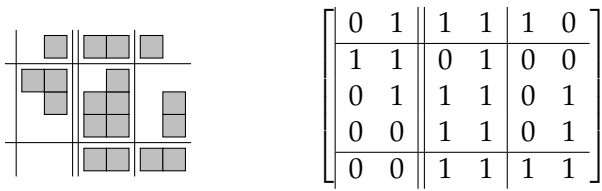

The containment relation of polyominoes and the notion of a polyomino
pattern extend in a natural way to gridded analogues in the same way as
with gridded permutations.

\subsection{Polyomino Tilings}

For the remainder of this section, we will focus on enumerating polyominoes
by area (the number of squares); only small changes would be needed to set
up Combinatorial Exploration to enumerate by perimeter.

\begin{definition}
	A \emph{poly-tiling} $\TT$ is a structure that represents a set of gridded
	polyominoes. It is defined by four components. The first three are identical
	to the three components that define a gridded permutation tiling: dimensions
	$(t, u)$, a set $\OO$ of gridded polyomino patterns called
	\emph{obstructions}, and a set $\RR = \{\RR_1, \ldots, \RR_k\}$ of sets
	of gridded polyomino patterns called \emph{requirements}. The gridded
	polyominoes that $\TT$ represents, $\Grid(\TT)$, must each avoid all of
	the gridded polyomino patterns in $\OO$ and must each contain at least one
	gridded polyomino pattern in each $\RR_i$.
	
	The fourth component, new to this domain, is called a \emph{tracking list}.
	The tracking list designates a subset of the rows of the tiling. This
	component does not change which gridded polyominoes are in $\Grid(\TT)$.
	Instead, it signals to the combinatorial strategies that
	we are not just interested in the sequence $a_n$ for the number
	of polyominoes in $\Grid(\TT)$ with area $n$, but the more refined
	sequence $a_{n,k}$ for the number of polyominoes in $\Grid(\TT)$ with
	area $n$ and with $k$ squares in the rows indicated by the tracking list.
	Informally, when applying strategies to a tiling, the strategy will look
	at the tracking list to determine whether they can be applied, and if
	so what the children will be. Correspondingly, the counting functions for
	the strategy will involve two indices, $n$ and $k$.
	
	This is a simplified version of the strategies briefly discussed
	in Subsection~\ref{subsection:pp-success} that lead to catalytic
	variables in the systems of equations for generating functions, and
	which will be discussed further in future work.
\end{definition}

Consider the example poly-tiling $\TT$ shown on the left in
Figure~\ref{figure:poly-tiling-example}. The dimensions of $\TT$ are
$(t, u) = (2,2)$, and there is a single requirement, the gridded polyomino
pattern that contains $0$ in the cell $(1,0)$. The cells $(1,0)$, $(1,1)$ each contain the four
obstructions
$\left[\begin{matrix}0&0\end{matrix}\right]$,
$\left[\begin{matrix}0&1\end{matrix}\right]$,
$\left[\begin{matrix}1&0\end{matrix}\right]$, and
$\left[\begin{matrix}1&1\end{matrix}\right]$.
These force any gridded polyomino in $\Grid(\TT)$ to have at most one column
gridded into those cells. The vertical versions of these four obstructions
appear in cells $(0,0)$, $(1,0)$, and they enforce a similar condition on the rows.
The eight obstructions and one requirement in cell $(1,0)$ enforce that
any gridded polyomino in $\Grid(\TT)$ has a $0$ gridded into this cell and
no other entries. There are four remaining obstructions, two local
to the cell $(0,1)$ and two crossing between cells. This tiling has
an empty tracking list. The second subfigure in Figure~\ref{figure:poly-tiling-example} shows the same tiling with the same kind of abbreviations as we
used for ASM-tilings. The proof tree for Ferrers diagrams in
Figure~\ref{figure:ferrers-proof-tree} has many examples of tilings that do
have tracking lists, shown by drawing an oscillating line around the rows in
the list.

\begin{figure}[t]
	\centering
	\begin{tikzpicture}[baseline=(current bounding box.center)]
		\node (polyex) at (2, -2) {
			\polytiling{1}{2}{2}{}%
			{% obstructions
				% Vertical in (0,0)
                {2/{(0.32, 0.15)/1, (0.32, 0.5)/1}},%
                {2/{(0.44, 0.15)/1, (0.44, 0.5)/0}},%
                {2/{(0.56, 0.15)/0, (0.56, 0.5)/1}},%
                {2/{(0.68, 0.15)/0, (0.68, 0.5)/0}},%
				% Horizontal in (1,0)
				{2/{(1.1, 0.80)/0, (1.4, 0.80)/0}},%
				{2/{(1.1, 0.60)/0, (1.4, 0.60)/1}},%
				{2/{(1.1, 0.40)/1, (1.4, 0.40)/0}},%
				{2/{(1.1, 0.20)/1, (1.4, 0.20)/1}},%
				% Vertical in (1,0)
                {2/{(1.89, 0.15)/1, (1.89, 0.5)/1}},%
                {2/{(1.77, 0.15)/1, (1.77, 0.5)/0}},%
                {2/{(1.65, 0.15)/0, (1.65, 0.5)/1}},%
                {2/{(1.53, 0.15)/0, (1.53, 0.5)/0}},%
				% Horizontal in (1,1)
				{2/{(1.35, 1.80)/0, (1.65, 1.80)/0}},%
				{2/{(1.35, 1.60)/0, (1.65, 1.60)/1}},%
				{2/{(1.35, 1.40)/1, (1.65, 1.40)/0}},%
				{2/{(1.35, 1.20)/1, (1.65, 1.20)/1}},%
				{5/{(0.35, 1.2)/1, (0.65, 1.2)/0, (0.65, 0.8)/1, (0.35, 0.8)/1, (0.35, 1.2)/1}},%
				{5/{(0.85, 1.55)/0, (1.15, 1.55)/1, (1.15, 1.15)/1, (0.85, 1.15)/0, (0.85, 1.55)/0}},%
				{3/{(0.35, 1.88)/1, (0.63, 1.88)/0, (0.91, 1.88)/1}},%
				{3/{(0.15, 1.80)/1, (0.15, 1.5)/0, (0.15, 1.20)/1}}%
			}
			{{1/{(1.7, 0.8)/0}}}
			{}%
			{}
			{} %actual rows: 0/, actual columns: 1/
		};
		\node (polyex2) at (6, -2) {
			\polytiling{1.0}{2}{2}{1/0/0}%
			{% obstructions
				{5/{(0.35, 1.2)/1, (0.65, 1.2)/0, (0.65, 0.8)/1, (0.35, 0.8)/1, (0.35, 1.2)/1}},%
				{5/{(0.85, 1.55)/0, (1.15, 1.55)/1, (1.15, 1.15)/1, (0.85, 1.15)/0, (0.85, 1.55)/0}},%
				{3/{(0.35, 1.88)/1, (0.63, 1.88)/0, (0.91, 1.88)/1}},%
				{3/{(0.15, 1.80)/1, (0.15, 1.5)/0, (0.15, 1.20)/1}}%
			}
			{}
			{}%
			{}
			{0/0, 1/1} %actual rows: 0/, actual columns: 1/
		};
	\end{tikzpicture}
	\quad\quad\quad
	$\left[\scriptstyle\begin{array}{rrr|r}
	0 & 0 & 1 & 1  \\
	0 & 1 & 1 & 0  \\
	1 & 1 & 1 & 1  \\\hline
	1 & 0 & 1 & 0
    \end{array}\right]$
    \quad\quad\quad
    \begin{tikzpicture}[scale=0.3, baseline=(current bounding box.center)]
    \foreach \x/\y in {0/0, 1/0, 2/0, 1/1, 2/1, 2/2}
      {%
        \filldraw[fill=lightgray] (\x, \y) rectangle ++(1,1);
      }
    \begin{scope}[shift={(3.5,0)}]
    \foreach \x/\y in {0/0, 0/2}
      {%
        \filldraw[fill=lightgray] (\x, \y) rectangle ++(1,1);
      }
    \end{scope}
    \begin{scope}[shift={(0,-1.5)}]
    \foreach \x/\y in {0/0, 2/0}
      {%
        \filldraw[fill=lightgray] (\x, \y) rectangle ++(1,1);
      }
    \end{scope}
    \draw (3.25,-1.75) -- (3.25,3.25);
    \draw (-0.25,-0.25) -- (4.75,-0.25);

    \end{tikzpicture}

	\caption{A tiling $\TT$ with all obstructions and requirements shown
	explicitly; the same tiling with abbreviations; a gridded polyomino on the
	tiling in matrix form; the same polyomino in traditional form.}
	\label{figure:poly-tiling-example}
\end{figure}
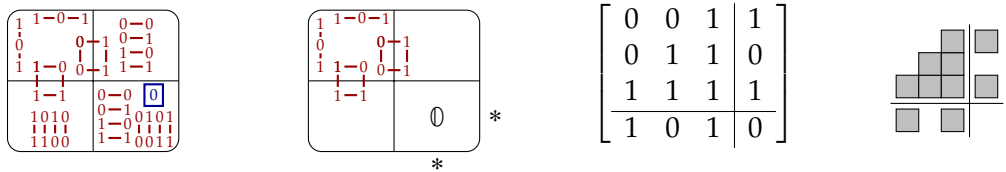

The third subfigure of Figure~\ref{figure:poly-tiling-example} shows the matrix
representation of a gridded polyomino in $\Grid(\TT)$. It has three columns
that have been gridded into the first column of $\TT$ and one gridded into
the second, and it has one row that has been gridded into the first row of $\TT$,
and three gridded into the second. The last subfigure of
Figure~\ref{figure:poly-tiling-example} shows the same gridded polyomino in the
traditional form.

\subsection{Ferrers Diagrams}

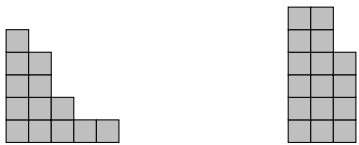
\begin{figure}
	\centering
	\begin{tikzpicture}[scale=0.3, baseline=(current bounding box.center)]
    \foreach \x/\y in {0/0, 0/1, 0/2, 0/3, 0/4, 1/0, 1/1, 1/2, 1/3, 2/0, 2/1, 3/0, 4/0}
      {%
        \filldraw[fill=lightgray] (\x, \y) rectangle ++(1,1);
      }
    \begin{scope}[shift={(12.5,0)}]
    \foreach \x/\y in {0/0, 0/1, 0/2, 0/3, 0/4, 0/5, 1/0, 1/1, 1/2, 1/3, 1/4, 1/5, 2/0, 2/1, 2/2, 2/3}
      {%
        \filldraw[fill=lightgray] (\x, \y) rectangle ++(1,1);
      }
    \end{scope}
    \end{tikzpicture}
	\caption{Two Ferrers diagrams.}
	\label{figure:ferrers-example}
\end{figure}

A Ferrers diagram is a polyomino in which all rows are aligned to the left
side of the bounding box and each row is at least as big as the one above
it. Figure~\ref{figure:ferrers-example} shows two examples of Ferrers
diagrams. Ferrers diagrams with $n$ squares are in bijection with integer
partitions of $n$: the partition $p_1 \geq p_2 \geq \ldots \geq p_\ell$ of
$n$ maps to the Ferrers diagram whose row lengths, starting from the bottom,
are $p_1, p_2, \ldots, p_\ell$. As a result, the enumeration of Ferrers diagrams
with $n$ squares is known, and the generating function satisfies an
algebraic differential equation, but does not satisfy a linear differential
equation.

\begin{figure}[tbhp]
	\centering
\begin{tikzpicture}[scale=0.9]

% LEVEL 0
\begin{scope}[yshift=100]
\node (t1) at (-3.45, 0) {$\TT_1$};
\node (root) at (-2.5, 0) {
\polytiling{1.0}{1}{1}{}%
			{% obstructions
				{2/{(0.25, 0.85)/0, (0.6, 0.85)/1}},%
				{2/{(0.85, 0.65)/1, (0.85, 0.2)/0}}%
			}
			{}
			{}{}
			{} %actual rows: 0/, actual columns: 1/
};

% LEVEL 1
\node (t0) at (-5.45, -1.95) {$\TT_2$};
\node (empty) at (-4.5, -2) {
\polytiling{1.0}{1}{1}{}%
			{% obstructions
			{1/{(0.5, 0.5)/2}}%
			}
			{}
			{}{}
			{} %actual rows: 0/, actual columns: 1/
};
\node (t2) at (-1.45, -1.95) {$\TT_3$};
\node (nonempty) at (-0.5, -2) {
\polytiling{1.0}{1}{1}{}%
			{% obstructions
				{2/{(0.25, 0.85)/0, (0.6, 0.85)/1}},%
				{2/{(0.85, 0.65)/1, (0.85, 0.2)/0}}%
			}
			{{1/{(0.25, 0.35)/1}}}
			{}{}
			{} %actual rows: 0/, actual columns: 1/
};
\node (t2') at (3, -0.65) {$\TT_3'$};
\node (nonempty') at (3, -2) {
\polytiling{1.0}{3}{2}{1/0/1}%
			{% obstructions
			    {1/{(2.5, 0.5)/1}},%
			    {0/{(0.5, 0.5)/0}},%
				{2/{(0.55, 1.5)/0, (0.85, 1.5)/1}},%x
				{2/{(0.55, 1.15)/0, (1.5, 1.15)/1}},%x
				{2/{(0.55, 1.85)/0, (2.45, 1.85)/1}},%x
				{2/{(1.5, 1.5)/0, (2.45, 1.5)/1}},%x
				{2/{(2.15, 1.15)/0, (2.45, 1.15)/1}},%x
				{2/{(0.15, 1.675)/1, (0.15, 1.275)/0}},%
				{2/{(1.15, 1.675)/1, (1.15, 1.275)/0}},%
				{2/{(2.85, 1.6755)/1, (2.85, 1.275)/0}},%
				{2/{(2.65, 1.5)/1, (2.65, 0.7)/0}}%
			}
			{}
			{}{}
			{0/0, 1/1} %actual rows: 0/, actual columns: 1/
};
\node (t2'') at (7, -0.65) {$\TT_3''$};
\node (nonempty'') at (7, -2) {
\polytiling{1.0}{2}{2}{1/0/1}%
			{% obstructions
			    {0/{(0.5, 0.5)/0}},%
				{2/{(0.55, 1.5)/0, (0.85, 1.5)/1}},%x
				{2/{(0.55, 1.15)/0, (1.5, 1.15)/1}},%x
				{2/{(0.15, 1.675)/1, (0.15, 1.275)/0}},%
				{2/{(1.85, 1.675)/1, (1.85, 1.275)/0}}%
			}
			{}
			{}{}
			{0/0, 1/1} %actual rows: 0/, actual columns: 1/
};

\node (equiv1) at (0.6 , -1.95) {$\cong$};
\node (equiv2) at (5.35 , -1.95) {$\cong$};

	\ptedge{(root)}{(-0.5,0.78)}{(empty)}{(-0.5,0.815)}
	\ptedge{(root)}{(-0.5,0.78)}{(nonempty)}{(-0.5,0.815)}
	
% LEVEL 2
\node (t3) at (5.4, -5) {$\TT_4$};
\node (firstass) at (7, -5) {
\polytiling{1.0}{2}{2}{1/0/1}%
			{% obstructions
			    {0/{(0.5, 0.5)/0}},%
				{2/{(0.55, 1.5)/0, (0.85, 1.5)/1}},%x
				{2/{(0.55, 1.15)/0, (1.5, 1.15)/1}},%x
				{2/{(0.15, 1.675)/1, (0.15, 1.275)/0}},%
				{2/{(1.85, 1.675)/1, (1.85, 1.275)/0}}%
			}
			{}
			{0}{}
			{0/0, 1/1} %actual rows: 0/, actual columns: 1/
};
\end{scope}
	\ptedge{(nonempty'')}{(-0.5,0.26)}{(firstass)}{(-0.5,1.35)}

% LEVEL 3
\begin{scope}[xshift=28.6, yshift=12]
\node (t4) at (-7.6, -3.5) {$\TT_5$};
\node (firstass4) at (-6, -3.5) {
\polytiling{1.0}{2}{1}{1/0/1}%
			{% obstructions
			    {0/{(0.5, 0.5)/0}}%
			}
			{}
			{0}{}
			{0/0, 1/1} %actual rows: 0/, actual columns: 1/
};
\node (t5) at (-0.1, -5.5) {$\TT_6$};
\node (firstass5) at (1.5, -5.5) {
\polytiling{1.0}{2}{3}{1/0/1, 1/1/0}%
			{% obstructions
			    {0/{(0.5, 0.5)/0}},%
			    {0/{(1.5, 2.5)/1}},%
				{2/{(0.55, 1.5)/0, (0.85, 1.5)/1}},%x
				{2/{(0.55, 2.5)/0, (0.85, 2.5)/1}},%x
				{2/{(0.15, 2.675)/1, (0.15, 2.275)/0}},%
				{2/{(0.25, 2.675)/1, (0.25, 1.5)/0}}%
			}
			{}
			{0}{}
			{0/0, 0/1, 1/1} %actual rows: 0/, actual columns: 1/
};
\node (t6) at (4.4, -5.5) {$\TT_7$};
\node (firstass6) at (6, -5.5) {
\polytiling{1.0}{2}{3}{1/0/1, 1/1/1}%
			{% obstructions
			    {0/{(0.5, 0.5)/0}},%
			    {0/{(0.5, 1.5)/0}},%
				{2/{(0.55, 2.5)/0, (0.85, 2.5)/1}},%x
				{2/{(0.15, 2.675)/1, (0.15, 2.275)/0}},%
				{2/{(0.55, 2.15)/0, (1.5, 2.15)/1}},%x
				{2/{(1.85, 2.675)/1, (1.85, 2.275)/0}}%
			}
			{}
			{0}{}
			{0/0, 0/1, 1/1} %actual rows: 0/, actual columns: 1/
};
\end{scope}
	\ptedge{(firstass)}{(-0.5,0.26)}{(firstass4)}{(-0.5,0.825)}
	\ptedge{(firstass)}{(-0.5,0.26)}{(firstass5)}{(-0.5,1.87)}
	\ptedge{(firstass)}{(-0.5,0.26)}{(firstass6)}{(-0.5,1.87)}

% LEVEL 4
\begin{scope}[xshift=-55, yshift=65]
\node (t9) at (-3.85, -12.5) {$\TT_{10}$};
\node (t9t) at (-2.75, -12.5) {
\polytiling{1.0}{1}{3}{0/0/1, 0/1/0}%
			{% obstructions
			    {0/{(0.5, 2.5)/1}}%
			}
			{}
			{0}{}
			{0/0, 0/1, 1/0} %actual rows: 0/, actual columns: 1/
};

\node (t10) at (-1.1, -12.5) {$\TT_{11}$};
\node (t10t) at (1.2, -12.5) {
\polytiling{1.0}{3}{3}{1/0/1, 1/1/0, 2/0/1, 2/1/0}%
			{% obstructions
			    {0/{(0.5, 0.5)/0}},%
			    {0/{(1.5, 2.5)/1}},%
			    {0/{(2.5, 2.5)/1}},%
				{2/{(0.55, 1.5)/0, (0.85, 1.5)/1}},%x
				{2/{(0.55, 2.5)/0, (0.85, 2.5)/1}},%x
				{2/{(0.15, 2.675)/1, (0.15, 2.275)/0}},%
				{2/{(0.25, 2.675)/1, (0.25, 1.5)/0}}%
			}
			{}
			{0}{}
			{0/0, 0/1, 1/1, 1/2} %actual rows: 0/, actual columns: 1/
};
\node (t11) at (4.2, -12.5) {$\TT_{12}$};
\node (t11t) at (6.5, -12.5) {
\polytiling{1.0}{3}{3}{1/0/1, 1/1/1, 2/0/1, 2/1/0}%
			{% obstructions
			    {0/{(0.5, 0.5)/0}},%
			    {0/{(0.5, 1.5)/0}},%
			    {0/{(2.5, 2.5)/1}},%
				{2/{(0.55, 2.5)/0, (0.85, 2.5)/1}},%x
				{2/{(0.15, 2.675)/1, (0.15, 2.275)/0}},%
				{2/{(1.85, 2.675)/1, (1.85, 2.275)/0}},%
				{2/{(0.55, 2.15)/0, (1.5, 2.15)/1}}%x
			}
			{}
			{0}{}
			{0/0, 0/1, 1/1, 1/2} %actual rows: 0/, actual columns: 1/
};
\end{scope}

	\ptedge{(firstass5)}{(-0.5,-0.27)}{(t9t)}{(-0.5,1.87)}
	\ptedge{(firstass5)}{(-0.5,-0.27)}{(t10t)}{(-0.5,1.87)}
	\ptedge{(firstass5)}{(-0.5,-0.27)}{(t11t)}{(-0.5,1.87)}

% LEVEL 5
\begin{scope}[xshift=-50, yshift=50]
\node (t10f1) at (-4.6, -16.5) {$\TT_6$};
\node (t10f1t) at (-3, -16.5) {
\polytiling{1.0}{2}{3}{1/0/1, 1/1/0}%
			{% obstructions
			    {0/{(0.5, 0.5)/0}},%
			    {0/{(1.5, 2.5)/1}},%
				{2/{(0.55, 1.5)/0, (0.85, 1.5)/1}},%x
				{2/{(0.55, 2.5)/0, (0.85, 2.5)/1}},%x
				{2/{(0.15, 2.675)/1, (0.15, 2.275)/0}},%
				{2/{(0.25, 2.675)/1, (0.25, 1.5)/0}}%
			}
			{}
			{0}{}
			{0/0, 0/1, 1/1} %actual rows: 0/, actual columns: 1/
};
\node (t10f2) at (-1, -15.5) {$\TT_8$};
\node (t10f2t) at (0, -15.5) {
\polytiling{1.0}{1}{1}{0/0/1}%
			{% obstructions
			}
			{}
			{0}{}
			{0/0, 1/0} %actual rows: 0/, actual columns: 1/
};
\node (t11f1) at (1.4, -16.5) {$\TT_7$};
\node (t11f1t) at (3, -16.5) {
\polytiling{1.0}{2}{3}{1/0/1, 1/1/0}%
			{% obstructions
			    {0/{(0.5, 0.5)/0}},%
			    {0/{(1.5, 2.5)/1}},%
				{2/{(0.55, 1.5)/0, (0.85, 1.5)/1}},%x
				{2/{(0.55, 2.5)/0, (0.85, 2.5)/1}},%x
				{2/{(0.15, 2.675)/1, (0.15, 2.275)/0}},%
				{2/{(0.25, 2.675)/1, (0.25, 1.5)/0}}%
			}
			{}
			{0}{}
			{0/0, 0/1, 1/1} %actual rows: 0/, actual columns: 1/
};
\node (t11f2) at (4.75, -15.5) {$\TT_8$};
\node (t11f2t) at (5.75, -15.5) {
\polytiling{1.0}{1}{1}{0/0/1}%
			{% obstructions
			}
			{}
			{0}{}
			{0/0, 1/0} %actual rows: 0/, actual columns: 1/
};
\end{scope}
\node (t3') at (5.4, -14.25) {$\TT_4$};
\node (firstass') at (7, -14.25) {
\polytiling{1.0}{2}{2}{1/0/1}%
			{% obstructions
			    {0/{(0.5, 0.5)/0}},%
				{2/{(0.55, 1.5)/0, (0.85, 1.5)/1}},%x
				{2/{(0.55, 1.15)/0, (1.5, 1.15)/1}},%x
				{2/{(0.15, 1.675)/1, (0.15, 1.275)/0}},%
				{2/{(1.85, 1.675)/1, (1.85, 1.275)/0}}%
			}
			{}
			{0}{}
			{0/0, 1/1} %actual rows: 0/, actual columns: 1/
};

	\ptedge{(firstass6)}{(-0.5,-0.27)}{(firstass')}{(-0.5,1.35)}
	\ptedge{(t10t)}{(-0.75,-0.27)}{(t10f1t)}{(-0.5,1.87)}
	\ptedge{(t10t)}{(-0.75,-0.27)}{(t10f2t)}{(-0.5,0.825)}
	\ptedge{(t11t)}{(-0.75,-0.27)}{(t11f1t)}{(-0.5,1.87)}
	\ptedge{(t11t)}{(-0.75,-0.27)}{(t11f2t)}{(-0.5,0.825)}

\begin{scope}[xshift=50, yshift=375]
% LEVEL 6
\node (t7b) at (-9, -18.25) {$\TT_8$};
\node (t7bt) at (-8, -18.25) {
\polytiling{1.0}{1}{1}{0/0/1}%
			{% obstructions
			}
			{}
			{0}{}
			{0/0,1/0} %actual rows: 0/, actual columns: 1/
};
\node (t8) at (-6.3, -18.25) {$\TT_9$};
\node (t8t) at (-4, -18.25) {
\polytiling{1.0}{3}{1}{1/0/1, 2/0/1}%
			{% obstructions
			    {0/{(0.5, 0.5)/0}}%
			}
			{}
			{0}{}
			{0/0, 1/1, 1/2} %actual rows: 0/, actual columns: 1/
};

    \ptedge{(firstass4)}{(-0.5,0.815)}{(t8t)}{(-0.5,0.825)}
    \ptedge{(firstass4)}{(-0.5,0.815)}{(t7bt)}{(-0.5,0.825)}
    
% LEVEL 7
\node (t8f1) at (-9.65, -20.25) {$\TT_5$};
\node (t8f1t) at (-8, -20.25) {
\polytiling{1.0}{2}{1}{1/0/1}%
			{% obstructions
			    {0/{(0.5, 0.5)/0}}%
			}
			{}
			{0}{}
			{0/0, 1/1} %actual rows: 0/, actual columns: 1/
};
\node (t8f2) at (-5.75, -20.25) {$\TT_8$};
\node (t8f2t) at (-4.75, -20.25) {
\polytiling{1.0}{1}{1}{0/0/1}%
			{% obstructions
			}
			{}
			{0}{}
			{0/0, 1/0} %actual rows: 0/, actual columns: 1/
};

    \ptedge{(t8t)}{(-0.75,0.78)}{(t8f1t)}{(-0.5,0.825)}
    \ptedge{(t8t)}{(-0.75,0.78)}{(t8f2t)}{(-0.5,0.825)}
\end{scope}
	
\end{tikzpicture}
	\caption{A proof tree for Ferrers diagrams.}
	\label{figure:ferrers-proof-tree}
\end{figure}
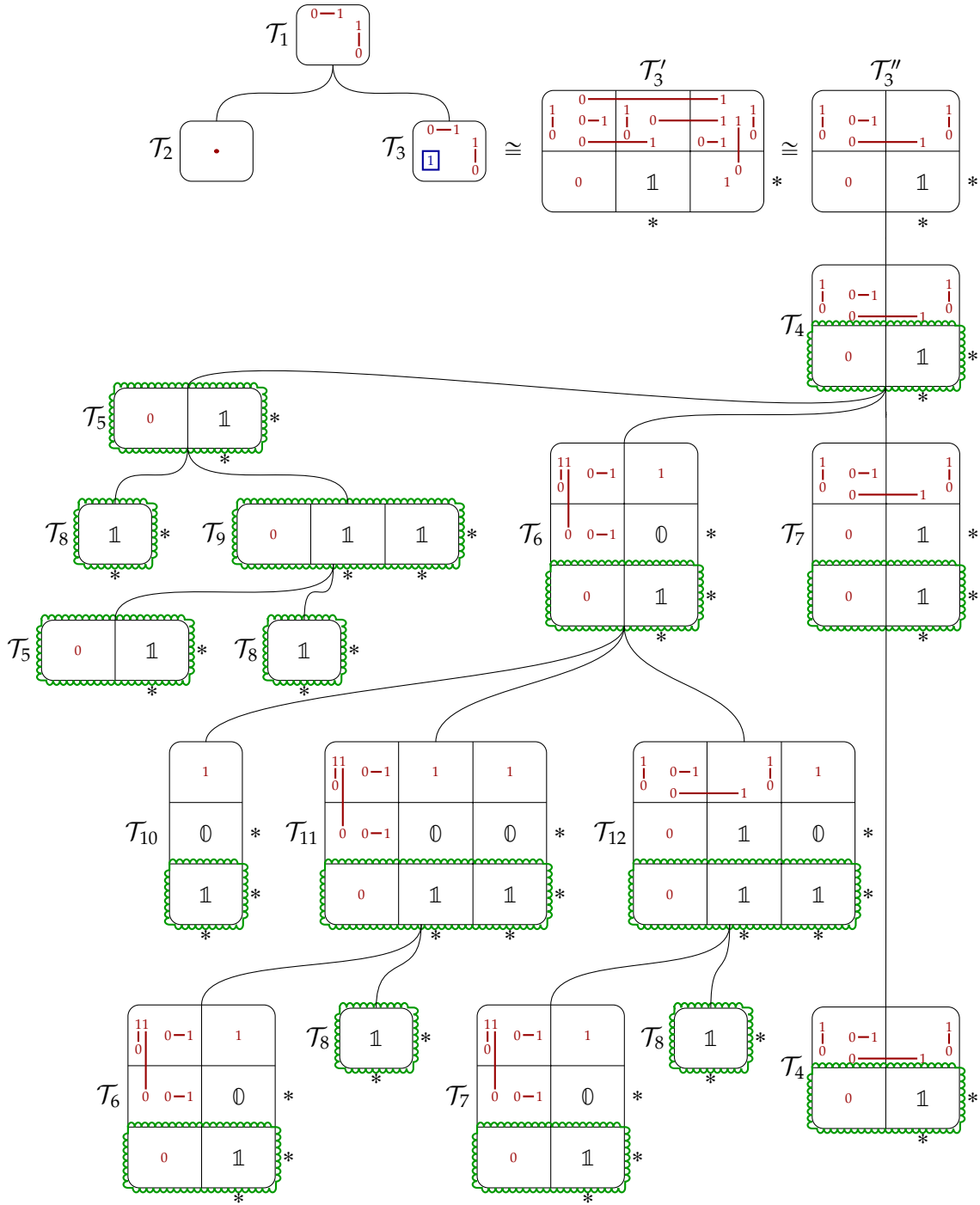

Ferrers diagrams are precisely those polyominoes that avoid the
polyomino patterns
$\beta_1 = \left[\begin{matrix}0&1\end{matrix}\right]$
and
$\beta_2 = \left[\begin{matrix}1\\0\end{matrix}\right]$.
Figure~\ref{figure:ferrers-proof-tree} shows a proof tree for the set
of Ferrers diagrams. As with other domains, we have used visual
abbreviations to simplify the pictures of poly-tilings.

We will now briefly explain the strategies that produce each rule of the
tree.
\begin{itemize}[label=$\diamond$]
	\item The rule $\TT_1 \leftarrow (\TT_2, \TT_3)$ is produced
			by a requirement-insertion-style strategy that splits into
			the case where the polyomino is empty and the case where it
			contains at least one square.
	\item A point-placement-style strategy places the requirement in $\TT_3$
			first as bottommost as possible, then as rightmost as possible,
			to get $\TT_3'$ (some implicit simplifications have been
			applied, just as we often did in the permutation domain). A
			different strategy observes that the rightmost column in 
			$\TT_3'$ cannot contain any entries, as it would then contain a
			$\beta_2$ pattern, and therefore $\TT_3$ and $\TT_3''$ are
			equinumerous.
	\item The rule $\TT_3'' \leftarrow (\TT_4)$ simply adds the tracking
			list for the bottommost row. Although it may at first feel
			backward, this is a valid strategy because we can certainly
			obtain the counting sequence for $\TT_3''$ if we know the
			counting sequence for $\TT_4$; indeed, we just ignore the extra
			index corresponding to the tracking list.
	\item The rules $\TT_4 \leftarrow (\TT_5, \TT_6, \TT_7)$,
			$\TT_5 \leftarrow (\TT_8, \TT_9)$, and
			$\TT_6 \leftarrow (\TT_{10}, \TT_{11}, \TT_{12})$
			are disjoint-union-type strategies that split into cases
			depending on what the extreme value in a particular cell in
			a row or column of height or width $1$ is. For example, in
			the first rule mentioned above, either the top row of $\TT_4$
			contains no entries at all (giving $\TT_5$) or the entry of
			that row that is as far down and to the right as possible is
			a $0$ (giving $\TT_6$), or it is a $1$ (giving $\TT_7$).
	\item The rules $\TT_9 \leftarrow (\TT_5, \TT_8)$,
			$\TT_{11} \leftarrow (\TT_6, \TT_8)$, and
			$\TT_{12} \leftarrow (\TT_7, \TT_8)$ are similar to the factor
			strategy on permutations. They detect when a particular square
			placed into a row and column of height and width $1$ can be 
			removed from every gridded polyomino on the tiling without
			violating the connectedness property.
	\item Two verification strategies give the rules $\TT_{2} \leftarrow ()$ and $\TT_{8} \leftarrow ()$ because their generating functions are easily seen to be $1$ and $xy$, respectively.
		A more complicated verification strategy gives the rule $\TT_{10} \leftarrow ()$ by
			detecting that there are no gridded polyominoes in
			$\Grid(\TT_{10})$ because the placed $0$ will lead to a violation
			of the bounding box condition, i.e., that the topmost and bottommost
			rows and the leftmost and rightmost columns all contain at least
			one square.
	\item Lastly, and most interestingly, the rule $\TT_7 \leftarrow (\TT_4)$
			is produced by a strategy that notices that the two bottommost rows
			of $\TT_7$ must be identical. Therefore, they can be ``merged''
			together. To see that this is a valid strategy, let $a_{n,k}$
			denote the number of gridded polyominoes in $\Grid(\TT_7)$ with
			$n$ squares, $k$ of which are in the tracked row, and let
			$b_{n,k}$ be the corresponding sequence for $\Grid(\TT_4)$.
			Then, we have the counting formula
			\(
				a_{n,k} = b_{n-k,k}
			\)
			and the generating function equation
			\(
				A(x,y) = B(x, xy).
			\)
\end{itemize}

As discussed in Section~\ref{section:transfer-tools}, this proof tree
provides a polynomial-time counting algorithm for the number of
Ferrers diagrams of size $n$. Moreover, the system of equations for
the generating functions of the combinatorial sets is the following.%
\footnote{The variable $y$ tracks the number of squares of each
gridded polyomino that are in rows identified by the tracking list
of the tiling.} 
\begin{align*}
	T_1(x) &= T_2(x) + T_3(x) & 								T_7(x,y) &= T_4(x, xy)\\
	T_2(x) &= 1 & 												T_8(x,y) &= xy \\
	T_3(x) &= T_4(x,1) & 										T_9(x,y) &= T_5(x,y) \cdot T_8(x,y)\\
	T_4(x,y) &= T_5(x,y) + T_6(x,y) + T_7(x,y) &				T_{10}(x,y) &= 0\\
	T_5(x,y) &= T_8(x,y) + T_9(x,y) &							T_{11}(x,y) &= T_6(x,y) \cdot T_8(x,y)\\
	T_6(x,y) &= T_{10}(x,y) + T_{11}(x,y) + T_{12}(x,y) & 		T_{12}(x,y) &= T_7(x,y) \cdot T_8(x,y)
\end{align*}
This system can be solved to find the known generating function
\[
	T_1(x) = \prod_{i=1}^\infty \frac{1}{1-x^i}.
\]

We have also derived, by hand, a proof tree for the large class of
\emph{L-convex polyominoes}, which also sometimes go by the name
\emph{moon polyominoes}. These can be defined by avoiding the
four polyomino patterns
\[
	\left[\begin{matrix}
		1 & 0 & 1
	\end{matrix}\right],
	\qquad
	\left[\begin{matrix}
		1 \\ 0 \\ 1
	\end{matrix}\right],
	\qquad
	\left[\begin{matrix}
		1 & 0 \\ 0 & 1
	\end{matrix}\right],
	\qquad
	\left[\begin{matrix}
		0 & 1 \\ 1 & 0
	\end{matrix}\right],
\]
and their enumeration is also already known as sequence A126764 in the
OEIS~\cite{oeis}. Like with Ferrers diagrams, their generating
function is non-D-finite.
The proof tree is too large to show here, involving $47$
combinatorial sets. Nevertheless, it serves as further proof that
Combinatorial Exploration would be effective in the domain
of polyominoes.

%% ==== %% ==== %% ==== %% ==== %% ==== %% ==== %% ==== %% ==== %% ==== %% ==== %%
%% ==== %% ==== %% ==== %% ====   SECTION EIGHT    ==== %% ==== %% ==== %% ==== %%
%% ==== %% ==== %% ==== %% ==== %% ==== %% ==== %% ==== %% ==== %% ==== %% ==== %%

\section{Applying Combinatorial Exploration to Pattern-Avoiding Set Partitions}
\label{section:set-partitions-results}
%!TEX root = combinatorial-exploration.tex

\subsection{Set Partitions}

A set partition of size $n$ is a decomposition of the set
$[n] = \{1,\ldots,n\}$ into a set of nonempty disjoint subsets
$B_1, \ldots, B_k$ called \emph{blocks} whose union is $[n]$.
For readability, we write set partitions without commas and braces,
using $/$ to separate the blocks. The order of the blocks and the order
within the blocks do not matter, so for example the set partitions
$134 / 28 / 5 / 79$ and $97 / 82 / 341 / 5$ are considered the same; the
first of these is written in \emph{standard form}, in which the entries
within the blocks are written in increasing order, and the blocks are
written in increasing order of their smallest elements. The five set
partitions of size $3$, written in standard form, are
\[
	123, \qquad 1 / 23, \qquad 12 / 3, \qquad 13 / 2, \qquad 1 / 2 / 3.
\]

An alternative way to express a set partition $\sigma = B_1 / \cdots / B_k$
is with a word $w = w(1) \ldots w(n)$ over the positive integers such
that $w(i) = j$ if $i \in B_j$. The five partitions of size $3$ can be
expressed in this way as
\[
	111, \qquad 122, \qquad 112, \qquad 121, \qquad 123.
\]

Words of this kind are called \emph{restricted growth functions}
and can be defined by the conditions
\begin{enumerate}
	\item $w(1) = 1$,
	\item for all $i \geq 2$, $w(i) \leq 1 + \max(\{w(1), \ldots, w(i - 1)\})$.
\end{enumerate}

There have been several different notions of pattern avoidance studied for set
partitions. For example, the definition used by Klazar~\cite{klazar:ababfree,%
klazar:countingpatternfreesetpartitions,%
klazar:countingpatternfreesetpartitionstwo} and more recently Bloom and
Saracino~\cite{bloom:patternavoidancealaklazar},
Chen, Deng, Du, Stanley, and Yan~\cite{chen:crossingsandnesting},
Riordan~\cite{riordan:distributionofcrossing}, and
Touchard~\cite{touchard:surunproblemedeconfigurations} says
that $\sigma$ contains $\tau$ if $\tau$ can be obtained by deleting entries
from blocks of $\sigma$ and standardizing the remaining entries. With this
notion, the set partition $145 / 23$ contains $12 / 34$ as can be seen by
deleting the $1$
and standardizing (and rearranging the blocks, which is permitted because
the blocks of a set partition are unordered). A second definition
involving arc-diagrams has been studied as well by,
for example Bloom and Elizalde~\cite{bloom:patternavoidanceinmatching}.

In this section we work with a third notion of pattern avoidance that
uses the correspondence
with restricted growth functions. This notion was investigated by
Sagan~\cite{sagan:patternavoidanceinsetpartitions} (who also considered
Klazar's version of pattern avoidance) and then studied
further by
Campbell, Dahlberg, Dorward, Gerhard, Grubb, Purcell, and
Sagan~\cite{campbell:rgfpatternandstatistics},
Dahlberg, Dorward, Gerhard, Grubb, Purcell, Reppuhn, and
Sagan~\cite{dahlberg:setpartitionpatternandstatistics},
and Jel\'{\i}nek, Mansour, and 
Shattuck~\cite{jelinek:onmultiplepatternavoidingsetpartitions}. 
In this third version of pattern
avoidance, we say that $\sigma$ contains $\tau$ if
there is a (not necessarily consecutive) subsequence of $w(\sigma)$ that
is order-isomorphic to $w(\tau)$. In this notion, it is no longer true
that $145 / 23$ contains $12 / 34$ because the corresponding restricted
growth functions are $12211$ and $1122$, and the former does not contain
the latter as a subsequence. As in previous sections, we can generalize
this containment notion to words that are not themselves restricted
growth functions, e.g., the set partition $12211$ contains the word $21$.
As before, we will call these words that are not necessarily actual set partitions
\emph{patterns}.

At the end of this section, we will give two proof trees, one for
set partitions avoiding the pattern $1212$ and the other for set
partitions avoiding the pattern $111$. The set partitions avoiding
$1212$ are often called \emph{non-crossing} set partitions, and we
will re-discover that they are counted by the Catalan numbers. Set
partitions avoiding $111$ have each part of size at most $2$; for
these we write down a system of differential equations and
solve to re-discover that the generating function is D-finite.

\subsection{Gridded Set Partitions}

As in the previous sections, we will define a gridded version of set
partitions, an analogous version of pattern containment, and then define
a tiling-like object to represent sets of gridded set partitions. The
easiest way to do this for this short proof-of-concept is to think of
a set partition as a $0/1$ matrix that contains a $1$ in the $(i,j)$
entry (indexed with Cartesian coordinates) if block $B_i$, written in
standard form, contains the entry $j$. All other entries are $0$.
For example, the set partition $134 / 2$ is represented by the matrix
\(
	\scriptstyle\left[\begin{matrix}
		1 & 0\\
		1 & 0\\
		0 & 1\\
		1 & 0
	\end{matrix}\right].
\)

All such matrices have exactly one $1$ in each row. The restricted
growth function property can be phrased in the following way: the
$(1,1)$ entry must be $1$, and for all $i > 1$, if the lowest $1$ in
column $i$ is at location $(i,j)$, then subcolumn consisting of the
bottommost $j-1$ entries of column $i-1$ must contain at least one $1$.

Having defined the matrix representation of a set partition, we can
now lean on previous sections by defining a gridded set partition as
simply the matrix representation with vertical and horizontal lines
added in the usual way. Gridded set partition patterns (those that do
not necessarily obey the restricted growth function condition) are defined
similarly, and the definition of pattern avoidance can be easily
extended to these gridded analogues.

\subsection{Set Partition Tilings}

\begin{definition}
	An \emph{SP-tiling} $\TT$ is a structure that represents a set of gridded
	set partitions. It is defined by three components, identical to those
	that define a gridded permutation tiling: dimensions
	$(t, u)$, a set $\OO$ of gridded set partition patterns called
	\emph{obstructions}, and a set $\RR = \{\RR_1, \ldots, \RR_k\}$ of sets
	of gridded set partition patterns called \emph{requirements}. The gridded
	set partitions that $\TT$ represents, $\Grid(\TT)$, must each avoid all of
	the gridded set partition patterns in $\OO$ and must each contain at least one
	gridded set partition pattern in each $\RR_i$.
\end{definition}

Consider the example SP-tiling $\TT$ shown on the left in
Figure~\ref{figure:SP-tiling-example}. The dimensions of $\TT$ are
$(t, u) = (2,2)$, and there is a single requirement, the size $1$ gridded
set partition in the cell $(1,0)$. Like gridded permutation tilings, we
can depict obstructions and requirements using dots with lines between them
to signify the locations of the $1$s. For example, the obstruction in the top-right corner
is the gridded set partition $1223$ (in restricted growth function notation)
with all entries in the cell $(1,1)$.

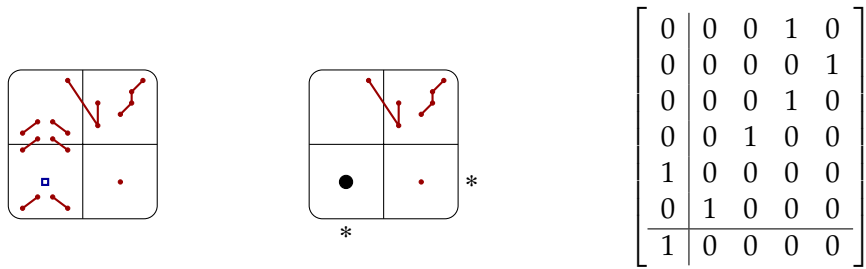
\begin{figure}
	\begin{center}
		\begin{tikzpicture}[baseline=(current bounding box.center)]
		\node (polyex) at (2, -2) {
			\sptiling{1}{2}{2}{}%
			{% obstructions
				% Local to (0,0)
                {2/{(0.20, 0.15), (0.40, 0.30)}},%
                {2/{(0.60, 0.30), (0.80, 0.15)}},%
                % Local to (0,1)
                {2/{(0.20, 1.15), (0.40, 1.30)}},%
                {2/{(0.60, 1.30), (0.80, 1.15)}},%
                % Crossing between (0,0) and (0,1)
                {2/{(0.20, 0.925), (0.40, 1.0725)}},%
                {2/{(0.60, 1.0725), (0.80, 0.925)}},%
                % Local to (1,0)
                {1/{(1.50, 0.5)}},%
                % Local to (1,1)
                {4/{(1.5, 1.4), (1.65, 1.55), (1.65, 1.7), (1.8, 1.85)}},%
                % Crossing between (0,1) and (1,1)
                {3/{(0.8, 1.85), (1.2, 1.25), (1.2, 1.55)}}%
			}
			{{1/{(0.5, 0.5)}}}
			{}%
			{}
			{} %actual rows: 0/, actual columns: 1/
		};
		\node (polyex2) at (6, -2) {
			\sptiling{1.0}{2}{2}{0/0}%
			{% obstructions
			    % Local to (1,0)
                {1/{(1.50, 0.5)}},%
                % Local to (1,1)
                {4/{(1.5, 1.4), (1.65, 1.55), (1.65, 1.7), (1.8, 1.85)}},%
                % Crossing between (0,1) and (1,1)
                {3/{(0.8, 1.85), (1.2, 1.25), (1.2, 1.55)}}%
			}
			{}
			{}%
			{}
			{0/0, 1/0} %actual rows: 0/, actual columns: 1/
		};
	\end{tikzpicture}
	\quad\quad\quad\quad\quad
	$\left[\scriptstyle\begin{array}{r|rrrr}
	0 & 0 & 0 & 1 & 0 \\
	0 & 0 & 0 & 0 & 1 \\
	0 & 0 & 0 & 1 & 0 \\
	0 & 0 & 1 & 0 & 0 \\
	1 & 0 & 0 & 0 & 0\\
	0 & 1 & 0 & 0 & 0\\\hline
	1 & 0 & 0 & 0 & 0
    \end{array}\right]$
	\end{center}
	\caption{An SP-tiling with all obstructions and requirements shown,
	followed by an abbreviated versions, and finally a gridded set partition on the
	tiling with underlying set partition $13/2/4/57/6$.
	The corresponding restricted growth function is $1213454$.}
	\label{figure:SP-tiling-example}
\end{figure}

The six size $2$ obstructions, together with the size $1$ requirement,
force every gridded set partition in $\Grid(\TT)$ to have at most one column
gridded into the first column of $\TT$. The size $1$ obstruction forces
all gridded set partitions in $\Grid(\TT)$ to have no entries gridded into
cell $(1,0)$ (although this is actually already implied by the
constraint that the matrix of a set partition has exactly one $1$ per row).
Lastly, there is a crossing size $3$ obstruction with
underlying word $221$ and the previously mentioned size $4$ obstruction
with underlying word $1223$.

The middle of Figure~\ref{figure:SP-tiling-example} shows the visual
abbreviations that we use with SP-tilings: we draw an asterisk below
a column to show that it has width $1$, similarly for rows, and use a solid black point to
show a cell with precisely one element of the set partition.
The right side of the figure shows a gridded set partition in
$\Grid(\TT)$.

\subsection{Non-Crossing Set Partitions}

\begin{figure}
	\centering
\begin{tikzpicture}[scale=1]

% LEVEL 0
\node (t1) at (-1.25, 0) {$\TT_1$};
\node (root) at (-0.5, 0) {
			\sptiling{1}{1}{1}{}%
			{% obstructions
				% Local to (0,0)
                {4/{(0.35, 0.20), (0.35, 0.60), (0.65, 0.40), (0.65, 0.80)}}%
			}
			{}
			{}%
			{}
			{} %actual rows: 0/, actual columns: 1/
};

% LEVEL 1
\node (t0) at (-3.25, -2) {$\TT_2$};
\node (empty) at (-2.5, -2) {
			\sptiling{1}{1}{1}{}%
			{% obstructions
				% Local to (0,0)
                {1/{(0.5, 0.50)}}%
			}
			{}
			{}%
			{}
			{} %actual rows: 0/, actual columns: 1/
};
\node (t2) at (0.75, -2) {$\TT_3$};
\node (nonempty) at (1.5, -2) {
			\sptiling{1}{1}{1}{}%
			{% obstructions
				% Local to (0,0)
                {4/{(0.35, 0.20), (0.35, 0.60), (0.65, 0.40), (0.65, 0.80)}}%
			}
			{{1/{(0.55, 0.22)}}}
			{}%
			{}
			{} %actual rows: 0/, actual columns: 1/
};
\node (t3) at (4.0, -0.75) {$\TT_3'$};
\node (nonempty') at (4.0, -2) {
\sptiling{1.0}{2}{2}{0/0}%
			{% obstructions
			    % Local to (1,0)
                {1/{(1.50, 0.5)}},%
                % Local to (1,1)
                {4/{(1.35, 1.20), (1.35, 1.60), (1.65, 1.40), (1.65, 1.80)}},%
                % Crossing between (0,1) and (1,1)
                {3/{(0.8, 1.60), (1.2, 1.40), (1.2, 1.80)}}%
			}
			{}
			{}%
			{}
			{0/0, 1/0} %actual rows: 0/, actual columns: 1/
};

\node (equiv1) at (2.5 , -2) {$\cong$};
\node (equiv2) at (-0.5 , -10) {$\cong$};
\node (equiv3) at (2.5 , -10) {$\cong$};
\node (equiv4) at (5.5 , -10) {$\cong$};

	\ptedge{(root)}{(-0.5,0.815)}{(empty)}{(-0.5,0.8)}
	\ptedge{(root)}{(-0.5,0.815)}{(nonempty)}{(-0.5,0.8)}
	
% LEVEL 2
\begin{scope}[xshift=-4cm, yshift=-3.5cm]
\node (t6) at (0.75, -7) {$\TT_6$};
\node (t6t) at (2, -7) {
\sptiling{1.0}{2}{2}{0/0}%
			{% obstructions
			    % Local to (1,0)
                {1/{(1.50, 0.5)}},%
                % Local to (1,1)
                {4/{(1.35, 1.20), (1.35, 1.60), (1.65, 1.40), (1.65, 1.80)}},%
                % Crossing between (0,1) and (1,1)
                {3/{(0.8, 1.60), (1.2, 1.40), (1.2, 1.80)}}%
			}
			{{1/{(0.5, 1.5)}}}
			{}%
			{}
			{0/0, 1/0} %actual rows: 0/, actual columns: 1/
};

\node (t6') at (5, -3.7) {$\TT_6$'};
\node (t6t') at (5, -6) {
\sptiling{1.0}{2}{4}{0/0, 0/2}%
			{% obstructions
			    % Local to (1,0)
                {1/{(1.50, 0.5)}},%
                % Local to (0,1)
                {1/{(0.50, 1.5)}},%
                % Local to (1,2)
                {1/{(1.50, 2.5)}},%
                % Local to (1,1)
                {4/{(1.35, 1.20), (1.35, 1.60), (1.65, 1.40), (1.65, 1.80)}},%
                % Local to (1,3)
                {4/{(1.35, 3.20), (1.35, 3.60), (1.65, 3.40), (1.65, 3.80)}},%
                % Crossing between (0,3) and (1,3)
                {3/{(0.8, 3.60), (1.2, 3.40), (1.2, 3.80)}},%
                % Crossing between (1,1) and (1,3)
                {2/{(1.2, 1.60), (1.2, 3.20)}}%
			}
			{}
			{}%
			{}
			{0/0, 0/2, 1/0} %actual rows: 0/, actual columns: 1/
};

\node (t6'') at (8, -3.7) {$\TT_6''$};
\node (t6t'') at (8, -6) {
\sptiling{1.0}{2}{4}{0/0, 0/2}%
			{% obstructions
			    % Local to (1,0)
                {1/{(1.50, 0.5)}},%
                % Local to (0,1)
                {1/{(0.50, 1.5)}},%
                % Local to (1,2)
                {1/{(1.50, 2.5)}},%
                % Local to (1,1)
                {4/{(1.35, 1.20), (1.35, 1.60), (1.65, 1.40), (1.65, 1.80)}},%
                % Local to (1,3)
                {4/{(1.35, 3.20), (1.35, 3.60), (1.65, 3.40), (1.65, 3.80)}},%
                % Crossing between (0,3) and (1,3)
                {3/{(0.8, 3.60), (1.2, 3.40), (1.2, 3.80)}},%
                % Crossing between (1,1) and (1,3)
                {2/{(1.2, 1.60), (1.2, 3.20)}},%
                {2/{(1.85, 1.60), (1.65, 3.20)}}%
			}
			{}
			{}%
			{}
			{0/0, 0/2, 1/0} %actual rows: 0/, actual columns: 1/
};

\node (t6''') at (11.5, -3.7) {$\TT_6'''$};
\node (t6t''') at (11.5, -6) {
\sptiling{1.0}{3}{4}{0/0, 0/2}%
			{% obstructions
			    % Local to (1,0)
                {1/{(1.50, 0.5)}},%
                % Local to (0,1)
                {1/{(0.50, 1.5)}},%
                % Local to (1,2)
                {1/{(1.50, 2.5)}},%
                % Local to (1,3)
                {1/{(1.50, 3.25)}},%
                % Local to (2,0)
                {1/{(2.50, 0.5)}},%
                % Local to (2,1)
                {1/{(2.50, 1.5)}},%
                % Local to (2,2)
                {1/{(2.50, 2.5)}},%
                % Local to (1,1)
                {4/{(1.35, 1.20), (1.35, 1.60), (1.65, 1.40), (1.65, 1.80)}},%
                % Local to (2,3)
                {4/{(2.35, 3.20), (2.35, 3.60), (2.65, 3.40), (2.65, 3.80)}},%
                % Crossing between (0,3) and (2,3)
                {3/{(0.8, 3.60), (2.2, 3.40), (2.2, 3.80)}}%
            }
            {}
			{}%
			{}
			{0/0, 0/2, 1/0} %actual rows: 0/, actual columns: 1/
};

\node (t3') at (6.75, -10) {$\TT_3'$};
\node (nonempty'') at (8, -10) {
\sptiling{1.0}{2}{2}{0/0}%
			{% obstructions
			    % Local to (1,0)
                {1/{(1.50, 0.5)}},%
                % Local to (1,1)
                {4/{(1.35, 1.20), (1.35, 1.60), (1.65, 1.40), (1.65, 1.80)}},%
                % Crossing between (0,1) and (1,1)
                {3/{(0.8, 1.60), (1.2, 1.40), (1.2, 1.80)}}%
			}
			{}
			{}%
			{}
			{0/0, 1/0} %actual rows: 0/, actual columns: 1/
};

\node (t1'') at (10.25, -10) {$\TT_1$};
\node (root'') at (11, -10) {
			\sptiling{1}{1}{1}{}%
			{% obstructions
				% Local to (0,0)
                {4/{(0.35, 0.20), (0.35, 0.60), (0.65, 0.40), (0.65, 0.80)}}%
			}
			{}
			{}%
			{}
			{} %actual rows: 0/, actual columns: 1/
};

\node (t5') at (12.25, -10) {$\TT_5$};
\node (t5t') at (13, -10) {
			\sptiling{1}{1}{1}{0/0}%
			{% obstructions
			}
			{}
			{}%
			{}
			{0/0, 1/0} %actual rows: 0/, actual columns: 1/
};

\end{scope}

% LEVEL 3
\begin{scope}[xshift=2cm, yshift=2.5cm]
\node (t4) at (-5.25, -7) {$\TT_4$};
\node (t4t) at (-4, -7) {
\sptiling{1.0}{2}{2}{0/0}%
			{% obstructions
			    % Local to (1,0)
                {1/{(1.50, 0.5)}},%
                % Local to (0,1)
                {1/{(0.5, 1.5)}},%
                % Local to (1,1)
                {4/{(1.35, 1.20), (1.35, 1.60), (1.65, 1.40), (1.65, 1.80)}}%
			}
			{}
			{}%
			{}
			{0/0, 1/0} %actual rows: 0/, actual columns: 1/
};

\node (t1') at (-3.75, -9.5) {$\TT_1$};
\node (root') at (-3, -9.5) {
			\sptiling{1}{1}{1}{}%
			{% obstructions
				% Local to (0,0)
                {4/{(0.35, 0.20), (0.35, 0.60), (0.65, 0.40), (0.65, 0.80)}}%
			}
			{}
			{}%
			{}
			{} %actual rows: 0/, actual columns: 1/
};

\node (t5) at (-5.75, -9.5) {$\TT_5$};
\node (t5t) at (-5, -9.5) {
			\sptiling{1}{1}{1}{0/0}%
			{% obstructions
			}
			{}
			{}%
			{}
			{0/0, 1/0} %actual rows: 0/, actual columns: 1/
};
\end{scope}

%	\ptedge{(nonempty')}{(-0.5,0.4)}{(t4t)}{(-0.5,1.2)}
%	\ptedge{(nonempty')}{(-0.5,0.4)}{(t6t)}{(-0.5,1.2)}
	
%	\draw (1,-3) to [out=225, in=25] (-3.5,-5.5) to [out = 200, in=160] (-2,-9);
%	\draw (1,-3) to [out=210, in=25] (-0.5,-3.5) to [out = 200, in=160] (-2,-6);
	\draw (4,-3) to [out=242, in=80] (-2,-3.5);
	\draw (4,-3) to [out=280, in=69] (-2,-9.5);
	
	\ptedge{(t4t)}{(-0.5,0.4)}{(t5t)}{(-0.5,0.8)}
	\ptedge{(t4t)}{(-0.5,0.4)}{(root')}{(-0.5,0.8)}
	
	\ptedge{(t6t''')}{(-0.5,-0.7)}{(nonempty'')}{(-0.5,1.2)}
	\ptedge{(t6t''')}{(-0.5,-0.7)}{(root'')}{(-0.5,0.8)}
	\ptedge{(t6t''')}{(-0.5,-0.7)}{(t5t')}{(-0.5,0.8)}
	
\end{tikzpicture}
	\caption{A proof tree for non-crossing set partitions.}
	\label{figure:av-1212}
\end{figure}

Figure~\ref{figure:av-1212} shows a proof tree for the set partitions that
avoid the pattern $1212$. The strategies used are simple and will be
familiar after having read the section on strategies for gridded permutations.
\begin{itemize}[label=$\diamond$]
	\item The rules $\TT_1 \leftarrow (\TT_2, \TT_3)$ and 
			$\TT_3' \leftarrow (\TT_4, \TT_6)$ are produced by requirement
			insertion strategies. The first inserts into cell $(0,0)$ and
			the second inserts into cell $(0,1)$.
	\item The rules $\TT_4 \leftarrow (\TT_1, \TT_5)$ and
			$\TT_6''' \leftarrow (\TT_1, \TT_3', \TT_5)$ are produced by
			a factor-like strategy that splits apart cells that do not
			interact in their rows and columns and have no obstructions or requirements
			crossing between them.
	\item The rules $\TT_3 \leftarrow (\TT_3')$ and
			$\TT_6 \leftarrow (\TT_6')$ place a requirement into its own row,
			much like the point placement strategy for permutations.
	\item The rule $\TT_6' \leftarrow (\TT_6'')$ infers that all gridded
			set partitions in $\Grid(\TT_6')$ also avoid the $21$ obstruction
			crossing between the cells $(1,3)$ and $(1,1)$, as a result of the
			vertical size $2$ obstruction and the restricted growth function
			property. Because of this new obstruction
			the rule $\TT_6'' \leftarrow (\TT_6''')$ is produced by a
			column separation strategy.
\end{itemize}

From this proof tree, we can write down an algebraic system of equations
whose solution identifies that $T_1(x)$ is the Catalan generating function
as expected.

\subsection{Set Partitions with Parts of Size at Most $2$}

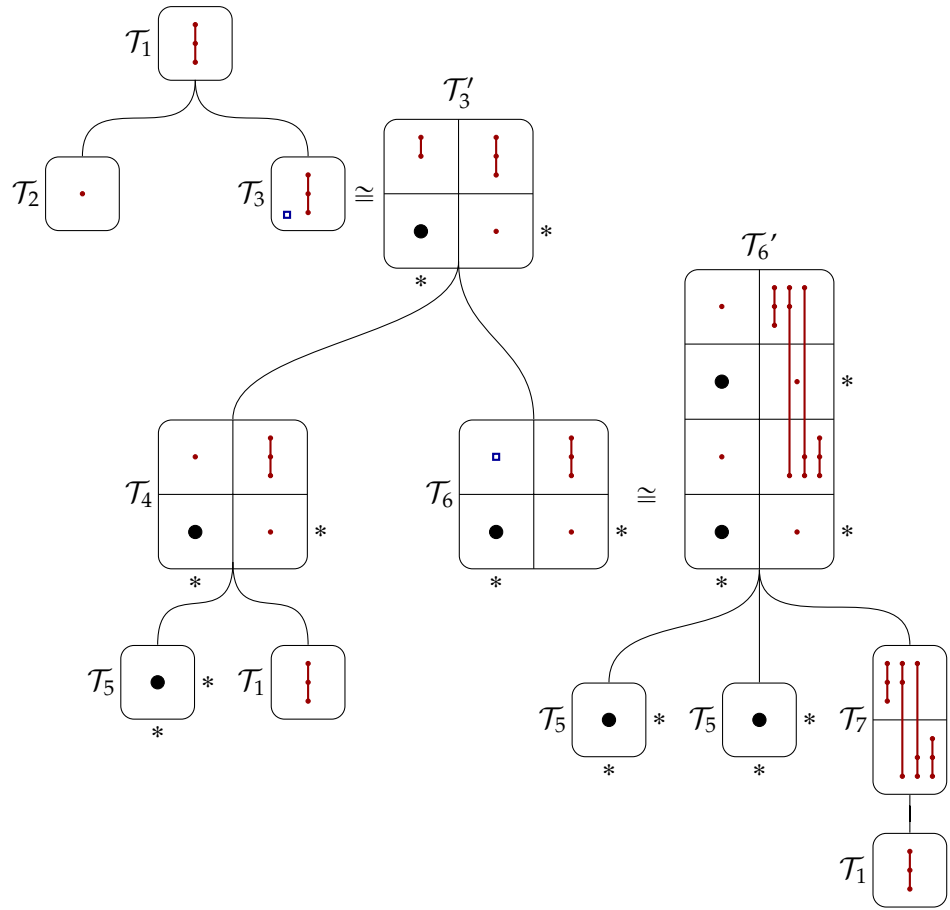
\begin{figure}
	\centering
\begin{tikzpicture}[scale=1]

% LEVEL 0
\node (t1) at (-3.25, 0) {$\TT_1$};
\node (root) at (-2.5, 0) {
			\sptiling{1}{1}{1}{}%
			{% obstructions
				% Local to (0,0)
                {3/{(0.5, 0.25), (0.5, 0.50), (0.5, 0.75)}}%
			}
			{}
			{}%
			{}
			{} %actual rows: 0/, actual columns: 1/
};

% LEVEL 1
\node (t0) at (-4.75, -2) {$\TT_2$};
\node (empty) at (-4, -2) {
			\sptiling{1}{1}{1}{}%
			{% obstructions
				% Local to (0,0)
                {1/{(0.5, 0.50)}}%
			}
			{}
			{}%
			{}
			{} %actual rows: 0/, actual columns: 1/
};
\node (t2) at (-1.75, -2) {$\TT_3$};
\node (nonempty) at (-1.0, -2) {
			\sptiling{1}{1}{1}{}%
			{% obstructions
				% Local to (0,0)
                {3/{(0.5, 0.25), (0.5, 0.50), (0.5, 0.75)}}%
			}
			{{1/{(0.22, 0.22)}}}
			{}%
			{}
			{} %actual rows: 0/, actual columns: 1/
};
\node (t3) at (1, -0.65) {$\TT_3'$};
\node (nonempty') at (1, -2) {
\sptiling{1.0}{2}{2}{0/0}%
			{% obstructions
			    % Local to (1,0)
                {1/{(1.50, 0.5)}},%
                % Local to (1,1)
                {3/{(1.5, 1.25), (1.5, 1.50), (1.5, 1.75)}},%
                % Local to (0,1)
                {2/{(0.5, 1.50), (0.5, 1.75)}}%
			}
			{}
			{}%
			{}
			{0/0, 1/0} %actual rows: 0/, actual columns: 1/
};

\node (equiv1) at (-0.25 , -2) {$\cong$};
\node (equiv2) at (3.5 , -6) {$\cong$};

	\ptedge{(root)}{(-0.5,0.815)}{(empty)}{(-0.5,0.8)}
	\ptedge{(root)}{(-0.5,0.815)}{(nonempty)}{(-0.5,0.8)}
	
% LEVEL 2
\begin{scope}[xshift=0cm, yshift=1cm]
\node (t6) at (0.75, -7) {$\TT_6$};
\node (t6t) at (2, -7) {
\sptiling{1.0}{2}{2}{0/0}%
			{% obstructions
			    % Local to (1,0)
                {1/{(1.50, 0.5)}},%
                % Local to (1,1)
                {3/{(1.5, 1.25), (1.5, 1.50), (1.5, 1.75)}}%
			}
			{{1/{(0.5, 1.5)}}}
			{}%
			{}
			{0/0, 1/0} %actual rows: 0/, actual columns: 1/
};

\node (t6') at (5, -3.7) {$\TT_6$'};
\node (t6t') at (5, -6) {
\sptiling{1.0}{2}{4}{0/0, 0/2}%
			{% obstructions
			    % Local to (1,0)
                {1/{(1.50, 0.5)}},%
                % Local to (0,1)
                {1/{(0.50, 1.5)}},%
                % Local to (0,3)
                {1/{(0.50, 3.5)}},%
                % Local to (1,2)
                {1/{(1.50, 2.5)}},%
                % Local to (1,1)
                {3/{(1.8, 1.25), (1.8, 1.50), (1.8, 1.75)}},%
                % Local to (1,3)
                {3/{(1.2, 3.25), (1.2, 3.50), (1.2, 3.75)}},%
                % Crossing between (1,1) and (1,3)
                {3/{(1.4, 1.25), (1.4, 3.50), (1.4,3.75)}},%
                {3/{(1.6, 1.25), (1.6, 1.50), (1.6,3.75)}}%
			}
			{}
			{}%
			{}
			{0/0, 0/2, 1/0} %actual rows: 0/, actual columns: 1/
};

\node (t7) at (6.25, -10) {$\TT_7$};
\node (t7t) at (7, -10) {
\sptiling{1.0}{1}{2}{}%
			{% obstructions
			    % Local to (1,1)
                {3/{(0.8, 0.25), (0.8, 0.50), (0.8, 0.75)}},%
                % Local to (1,3)
                {3/{(0.2, 1.25), (0.2, 1.50), (0.2, 1.75)}},%
                % Crossing between (1,1) and (1,3)
                {3/{(0.4, 0.25), (0.4, 1.50), (0.4,1.75)}},%
                {3/{(0.6, 0.25), (0.6, 0.50), (0.6,1.75)}}%
			}
			{}
			{}%
			{}
			{} %actual rows: 0/, actual columns: 1/
};

\node (t5'') at (2.25, -10) {$\TT_5$};
\node (t5t'') at (3, -10) {
			\sptiling{1}{1}{1}{0/0}%
			{% obstructions
			}
			{}
			{}%
			{}
			{0/0, 1/0} %actual rows: 0/, actual columns: 1/
};

\node (t5') at (4.25, -10) {$\TT_5$};
\node (t5t') at (5, -10) {
			\sptiling{1}{1}{1}{0/0}%
			{% obstructions
			}
			{}
			{}%
			{}
			{0/0, 1/0} %actual rows: 0/, actual columns: 1/
};

\end{scope}

% LEVEL 3
\begin{scope}[xshift=2cm, yshift=1cm]
\node (t4) at (-5.25, -7) {$\TT_4$};
\node (t4t) at (-4, -7) {
\sptiling{1.0}{2}{2}{0/0}%
			{% obstructions
			    % Local to (1,0)
                {1/{(1.50, 0.5)}},%
                % Local to (0,1)
                {1/{(0.5, 1.5)}},%
                % Local to (1,1)
                {3/{(1.5, 1.25), (1.5, 1.50), (1.5, 1.75)}}%
			}
			{}
			{}%
			{}
			{0/0, 1/0} %actual rows: 0/, actual columns: 1/
};

\node (t1') at (-3.75, -9.5) {$\TT_1$};
\node (root') at (-3, -9.5) {
			\sptiling{1}{1}{1}{}%
			{% obstructions
				% Local to (0,0)
                {3/{(0.5, 0.25), (0.5, 0.50), (0.5, 0.75)}}%
			}
			{}
			{}%
			{}
			{} %actual rows: 0/, actual columns: 1/
};

\node (t5) at (-5.75, -9.5) {$\TT_5$};
\node (t5t) at (-5, -9.5) {
			\sptiling{1}{1}{1}{0/0}%
			{% obstructions
			}
			{}
			{}%
			{}
			{0/0, 1/0} %actual rows: 0/, actual columns: 1/
};
\end{scope}

\node (t1'') at (6.25, -11) {$\TT_1$};
\node (root'') at (7, -11) {
			\sptiling{1}{1}{1}{}%
			{% obstructions
				% Local to (0,0)
                {3/{(0.5, 0.25), (0.5, 0.50), (0.5, 0.75)}}%
			}
			{}
			{}%
			{}
			{} %actual rows: 0/, actual columns: 1/
};

	\ptedge{(nonempty')}{(-0.5,0.4)}{(t4t)}{(-0.5,1.3)}
	\ptedge{(nonempty')}{(-0.5,0.4)}{(t6t)}{(-0.5,1.3)}
	
	\ptedge{(t4t)}{(-0.5,0.4)}{(t5t)}{(-0.5,0.8)}
	\ptedge{(t4t)}{(-0.5,0.4)}{(root')}{(-0.5,0.8)}
	
	\ptedge{(t6t')}{(-0.5,-0.7)}{(t7t)}{(-0.5,1.3)}
	\ptedge{(t6t')}{(-0.5,-0.7)}{(t5t'')}{(-0.5,0.8)}
	\ptedge{(t6t')}{(-0.5,-0.7)}{(t5t')}{(-0.5,0.8)}
	
	\ptedge{(t7t)}{(-0.5,0.3)}{(root'')}{(-0.5,0.8)}
	
\end{tikzpicture}
	\caption{A proof tree for set partitions with parts of size at most $2$.}
	\label{figure:av-111}
\end{figure}

Figure~\ref{figure:av-111} shows a proof tree for the set partitions that
avoid the pattern $111$. They are the set partitions with the property that 
every block has size $1$ or $2$. The ordinary generating function for their
counting sequence is D-finite, but not algebraic, and we will see how this
proof tree allows the defining linear differential equation to be computed.
This sequence is A000085 in the OEIS~\cite{oeis}.

All of the rules in this proof tree except for $\TT_7 \leftarrow (\TT_1)$
are produced by requirement insertion, point placement, and factor
strategies as seen in the previous example. The rule
$\TT_7 \leftarrow (\TT_1)$ is produced by a strategy similar to the
``fusion'' strategy mentioned at the end of
Section~\ref{section:pp-results}. The gridded set partitions in
$\Grid(\TT_7)$ can be formed uniquely by starting with a gridded set
partition in $\Grid(\TT_1)$ and adding a horizontal line between any
pair of entries, or at the extreme bottom and top. In this way,
a gridded set partition of size $n$ in $\Grid(\TT_1)$ produces $n+1$
gridded set partitions of size $n$ in $\Grid(\TT_7)$. This leads
to the generating function equation $T_7(x) = \ds\frac{d}{dx}(xT_1(x))$.%
\footnote{Alternatively, one can write an algebraic system of equations
that uses two catalytic variables, but no derivatives.}

The entire system of equations is
\begin{align*}
T_1(x) &= T_2(x) + T_3(x)\\
T_2(x) &= 1\\
T_3(x) &= T_4(x) + T_6(x)\\
T_4(x) &= T_1(x) \cdot T_5(x)\\
T_5(x) &= x\\
T_6(x) &= T_5(x)^2 \cdot T_7(x)\\
T_7(x) &= \frac{d}{dx}\left(x \cdot T_1(x)\right).	
\end{align*}
We can deduce from this that
\[
	T_1(x) = 1 + (x+x^2)T_1(x) + x^3\frac{d}{dx}T_1(x),
\]
proving that $T_1(x)$ is D-finite.

%% ==== %% ==== %% ==== %% ==== %% ==== %% ==== %% ==== %% ==== %% ==== %% ==== %%
%% ==== %% ==== %% ==== %% ====    SECTION NINE    ==== %% ==== %% ==== %% ==== %%
%% ==== %% ==== %% ==== %% ==== %% ==== %% ==== %% ==== %% ==== %% ==== %% ==== %%

\section{Further Algorithmic Considerations}
\label{section:algorithmic}
%!TEX root = combinatorial-exploration.tex

In this section we aim to give a small taste of the extensive algorithmic
design work required to actually implement Combinatorial Exploration in an
effective manner.  Our implementation of Combinatorial Exploration, which we
call the \CSS{}, is open-source and can be found on
Github~\cite{comb-spec-searcher}. Our code is fully compatible with the alternative Python implementation
called \texttt{PyPy}, a free just-in-time compiler that in our experience tends
to speed up computations by a factor of around 5 at the cost of some increased
memory usage.
The \CSS{} is completely domain-agnostic, making no reference to
permutations or any other combinatorial objects. In order to employ the \CSS{} to
perform Combinatorial Exploration, one must implement their own Python classes
and functions representing their combinatorial objects and strategies, and
plug these into the \CSS{}. For more details about how this is done, one can
refer to the ``README'' file in the \CSS{} repository~\cite{comb-spec-searcher} which shows
a sample implementation of Combinatorial Exploration for the simple domain of
binary words. Additionally, our \texttt{Tilings} repository~\cite{tilings} on
Github contains our implementation of Combinatorial Exploration for permutations
discussed in Section~\ref{section:pp-results}.

The user initiates Combinatorial Exploration by specifying the combinatorial
set that they wish to be enumerated, and a set of strategies to be used, which
we call a \emph{strategy pack}. The strategy pack separates the strategies into
several different groups depending on how one wishes them to be used and in what
order. The groups are called \emph{inferral strategies}, \emph{initial strategies}, 
\emph{expansion strategies}, and \emph{verification strategies}.

The \CSS{} was designed with the \emph{separation of concerns} design principle
in mind, making the software as flexible as possible for future development.
Work is delegated to a handful of separate components that work together to manage
the flow of Combinatorial Exploration. The most important of these components,
which we discuss in more detail below, are
\begin{itemize}[label=$\diamond$]
	\item the \Queue{}, which manages the order in which combinatorial sets
			are decomposed by strategies;
	\item the \ClassDB{}, which is in charge of storing the combinatorial
			sets discovered, including their compression and decompression
			to improve memory usage;
	\item the \EquivDB{}, a union-find data structure that tracks
			the equivalence classes of combinatorial sets;
	\item the \RuleDB{}, which stores all combinatorial rules discovered during
			the exploration process;
	\item the \SpecSearcher{}, which searches for a subset of combinatorial rules
			that form a combinatorial specification.
\end{itemize}

The \Queue{} controls the order in which combinatorial sets are considered and in
which the strategies in the strategy pack are applied to them. It
fundamentally has two operations: a method for adding a combinatorial set to the
queue, and a method that tells \CSS{} which combinatorial set should be expanded
next and, moreover, which strategy should be used for this expansion. The
default \Queue{} is actually composed of three separate queues: the
\emph{working queue}, the \emph{current queue}, and the \emph{next queue}.

When a combinatorial set is added to the \Queue{} for the first time, it is added
to the working queue. When \CSS{} is ready for a new combinatorial set and
requests one from the \Queue{}, the \Queue{} first takes from the working queue if
it is nonempty. Each combinatorial set $\AA$ taken from the working queue is
immediately expanded using the strategies in the pack that are designated as
inferral strategies. These are equivalence strategies for which, if they apply,
then we no longer want to try applying any additional strategies to the parent
$\AA$. For permutations, the row and column separation strategies discussed in
Subsection~\ref{subsubsection:row-col-sep} is effective as an inferral
strategy---if a tiling has two columns that separate, we want that separation
to occur right away, and then there is no need to do further work on the
unseparated version. If an inferral strategy applies, the child set is added
to the working queue and $\AA$ is removed from the \Queue{} permanently.

If no inferral strategy successfully applies, $\AA$ is then expanded using the
strategies in the pack that are designated as initial strategies. The children
of any strategy that applies are placed in the working queue, and after $\AA$
has been expanded with all initial strategies, it is moved to the next queue. 
Good candidates for initial strategies are those that decompose a set into
simpler sets that are likely to already exist in the universe like, for instance,
the factor strategy of Subsection~\ref{subsubsection:factor}. Treating the
factor strategy as an initial strategy means we will attempt to factor a tiling
long before we apply other strategies that tend to push deeper into the universe,
like the requirement insertion strategy of
Subsection~\ref{subsubsection:requirement-insertion-intuition}.

If the working queue is empty, the next combinatorial set to be expanded will be
taken from the current queue. When a combinatorial set is taken from the current
queue, it is expanded using the strategies in the pack designated as
expansion strategies. The expansion strategies are further split into subsets of
strategies
$S_1$, $S_2$, $\ldots$, $S_k$. The first time a combinatorial set is taken from
the current queue it is expanded using the strategies in set $S_1$ and added
back to the current queue, the second time using the strategies in $S_2$ and so
on. Once the last subset $S_k$ has been applied, we discard the combinatorial
set from \Queue{} as it has been fully expanded. This allows fine-grained
control over how exploration of the universe of sets proceeds simply by
modifying the strategy pack. This can be useful when performing
Combinatorial Exploration on a set for which one has some intuition about
which strategies are more likely to lead to effective decompositions and
which are not.

Finally, when both the working queue and the current queue are empty, all sets
in the next queue are moved to the current queue, and expansion continues.
The lifecycle of an individual combinatorial set $\AA$ is outlined in
Figure~\ref{fig:queueflowchart}.

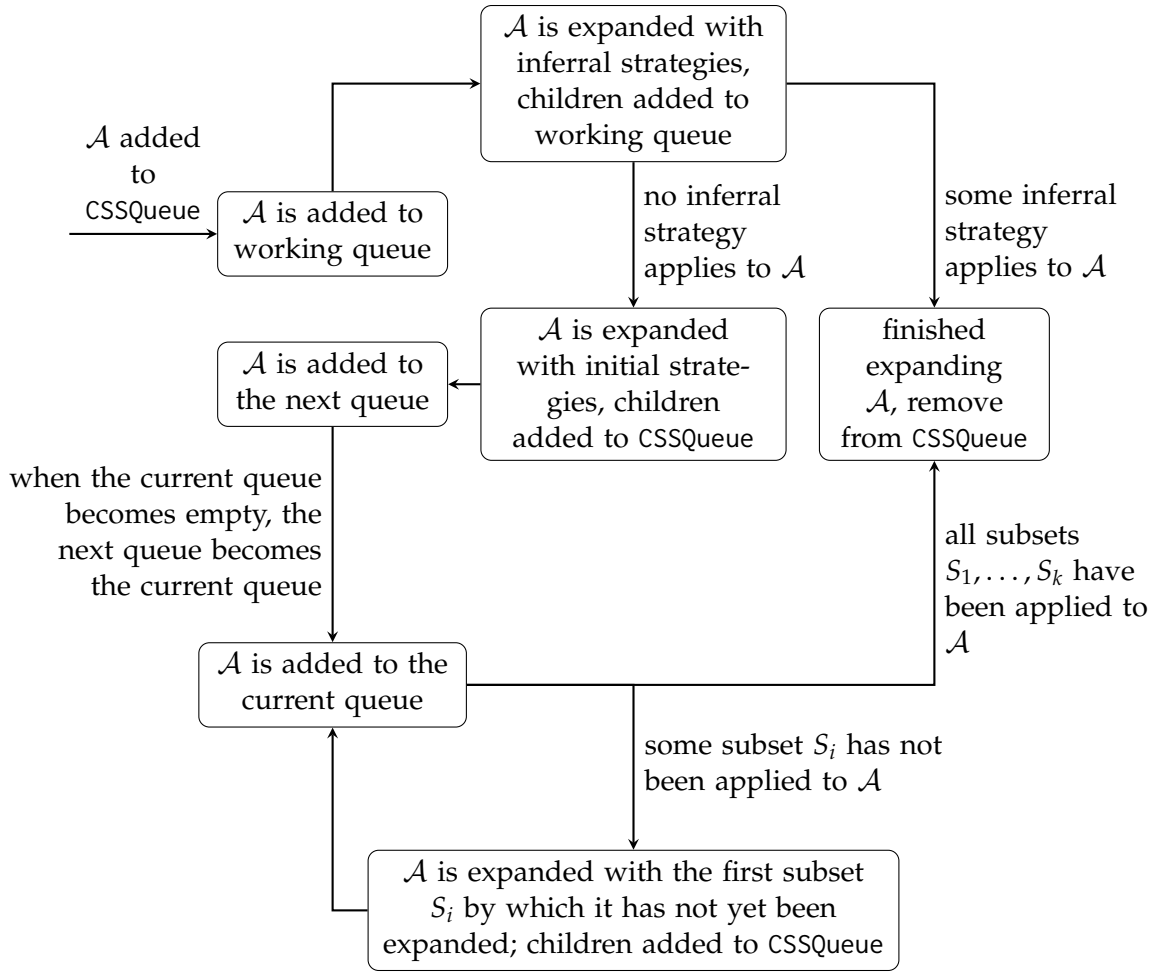
\begin{figure}[t]
	\centering
	\begin{tikzpicture}[node distance=2cm]
		% NODES
		\node (working) at (0, 0) [rectangle, rounded corners, minimum
		width=3cm, minimum height=1cm, text centered, draw=black, text width =
		2.8cm] {$\AA$ is added to working queue};
		\node (inferral) at (4, 2) [rectangle, rounded corners, minimum
		width=4cm, minimum height=1cm, text centered, draw=black, text width =
		3.8cm] {$\AA$ is expanded with inferral strategies, children added to
		working queue};
		\node (initial) at (4, -2) [rectangle, rounded corners, minimum
		width=4cm, minimum height=1cm, text centered, draw=black, text width =
		3.8cm] {$\AA$ is expanded with initial strategies, children added to
		\Queue};
		\node (next) at (0, -2) [rectangle, rounded corners, minimum width=3cm,
		minimum height=1cm, text centered, draw=black, text width = 2.8cm]
		{$\AA$ is added to the next queue};
		\node (current) at (0, -6) [rectangle, rounded corners, minimum
		width=3.5cm, minimum height=1cm, text centered, draw=black, text width =
		3.3cm] {$\AA$ is added to the current queue};
		\node (expansion) at (4, -9) [rectangle, rounded corners, minimum
		width=7cm, minimum height=1cm, text centered, draw=black, text width =
		6.8cm] {$\AA$ is expanded with the first subset $S_i$ by which it has
		not yet been expanded; children added to \Queue};
		\node (finished) at (8, -2) [rectangle, rounded corners, minimum
		width=3cm, minimum height=1cm, text centered, draw=black, text width =
		2.8cm] {finished expanding $\AA$, remove from \Queue};
		% ARROWS
		\draw [thick, ->, >=stealth] (-3.5, 0) -- node[anchor=south, text
		centered, text width = 2cm] {$\AA$ added to \Queue} (working);
		\draw [thick, ->, >=stealth] (working) |- (inferral);
		\draw [thick, ->, >=stealth] (inferral) -| node[anchor=west,
		text width = 2.5cm, yshift=-2cm] {some \mbox{inferral} strategy \mbox{applies}
		to $\AA$} (finished);
		\draw [thick, ->, >=stealth] (inferral) -- node[anchor=west,
		text width = 2.5cm] {no \mbox{inferral} strategy \mbox{applies} to
		$\AA$} (initial);
		\draw [thick, ->, >=stealth] (initial) -- (next);
		\draw [thick, ->, >=stealth] (next) -- node[anchor=east,
		text width = 5cm, align=right] {when the current queue becomes empty,
		the next queue becomes the current queue} (current);
		\draw [thick, ->, >=stealth] (current) -| node[anchor=west,
		text width = 3cm, yshift=1.3cm] {all subsets
		$S_1, \ldots, S_k$ have been applied to $\AA$} (finished);
		\draw [thick, ->, >=stealth] (current) -| node[anchor=north,
		xshift=2.15cm, yshift=-0.5cm, text width = 4cm] {some
		subset $S_i$ has not been applied to $\AA$} (expansion);
		\draw [thick, ->, >=stealth] (expansion) -| (current);
	\end{tikzpicture}
	\caption{The lifecycle of a combinatorial set $\AA$ in the default \Queue{}
	employed by \CSS{}}
	\label{fig:queueflowchart}
\end{figure}

Every time a combinatorial set $\AA$ is discovered for the first time it is
assigned a unique integer label and stored in the \ClassDB. In practice there
can be millions of combinatorial sets discovered in a single run of the \CSS{},
which imposes a significant memory burden. To mitigate this, our implementation
includes the option of storing a compressed version of the
combinatorial set, and the \ClassDB{} manages the corresponding compression and
decompression as needed.

Recall that Subsection~\ref{subsection:equivalencestrategies} describes how
equivalence strategies may be used to determine that two combinatorial sets have
the same enumeration. As such rules violate the productivity conditions, sets
are collected into equivalence classes. The \EquivDB{} is a union-find data
structure that efficiently represents the equivalence classes of combinatorial
sets. It contains routines for efficiently merging two equivalence classes,
finding a representative for a given equivalence class, and finding a shortest
path of equivalence rules between two sets in the same equivalence class.

Immediately after a combinatorial set $\AA$ is assigned a label by the
\ClassDB{}, the \CSS{} checks if any of the verification strategies in
the strategy pack apply to $\AA$. If so, this information is passed to the
\EquivDB{} because all equivalent classes may also be considered verified.

To expand a combinatorial set $\AA$ with a strategy $S$, the decomposition
strategy $d_S$ is applied to $\AA$. If $S$ successfully applies to $\AA$,
it creates a rule $\AA \xleftarrow{S} (\BB^{(1)}, \ldots, \BB^{(k)})$.
%Depending on the strategy used, it is important to check
%whether not any of the combinatorial sets are precisely the empty set. For
%example, say the decomposition function of $S$ teaches us that $\AA = \BB^{(1)}
%\sqcup \BB^{(2)}$ but we discover that $\BB^{(2)} = \emptyset$, and so $\AA$ is
%equivalent $\BB^{(1)}$. Leaving this rule with the empty $\BB^{(2)}$ could lead
%to \CSS{} returning a tautological specification if there were some other
%strategies recognising the equivalence between $\AA$ and $\BB^{(1)}$. Instead we
%treat $S$ as an equivalence strategy and update the \EquivDB{} of the
%equivalence.
%
%The $\ClassDB$ keeps track of whether or not combinatorial sets are empty. It is
%often possible to avoid this potentially heavy check for a combinatorial set.
%For example, if we find $\AA \xleftarrow{S} (\BB^{(1)}, \ldots, \BB^{(k)})$ but
%knew by other means that $\AA = \BB^{(1)}$ then we know $\BB^{(2)}, \ldots,
%\BB^{(k)}$ are empty. For many strategies, such as cartesian product and
%equivalence strategies, knowing that the parent is not empty is sufficient to
%tell you that all the children are not empty. These pre-checks are employed by
%\CSS{} before relying on the more expensive check on the combinatorial sets. 
Each of the child sets' assigned labels are placed into
the \Queue{} as described above. The rule is then stored as a tuple
$(a, (b_{1}, \ldots, b_{k}), S)$ in the \RuleDB{}, where $a$ and the $b_i$
are the integer labels assigned to $\AA$ and the $\BB_i$ and $S$ is the
strategy that produced the rule.

Searching for a combinatorial specification is done on the level of equivalence
classes of the equivalence relation $\eqrel$ discussed in
Subsection~\ref{subsection:equivalencestrategies}. This task is performed by the
\SpecSearcher{}, which first asks the \RuleDB{} for all of the rules found so
far. These rules on sets are converted into rules on the equivalence classes
of $\eqrel$ by consulting the \EquivDB{}. Finally
Algorithm~\ref{algorithm:specfinder} from Subsection~\ref{subsection:tree-searcher}
is used to decide if a specification exists. If so, 
Algorithm~\ref{algorithm:specgetter} chooses one. The items in this
specification are
equivalence classes of integer labels, and so the \RuleDB{}, \EquivDB{}, and
\ClassDB{} work together to convert this back into a specification on
combinatorial sets.

The advantage of following the separation of concerns design principle is that
the behavior of any one of these components can be easily altered
without interfering with the work
done by the others. For example, in Subsection~\ref{subsection:132-ASMs}
we briefly introduced the notion of a combinatorial forest, a generalization
of a combinatorial specification. The \CSS{} can be configured to search instead
for combinatorial forests by simply implementing a new \RuleDB{} object
that has the same functions, but whose internal logic is different. Another
example is found in \'Ardal's thesis~\cite{ardal:proof-numbers-thesis}, in which an alternative
\Queue{} based on the proof-number search algorithm commonly used on
game trees is used to replace our breadth-first approach.

%% ==== %% ==== %% ==== %% ==== %% ==== %% ==== %% ==== %% ==== %% ==== %% ==== %%
%% ==== %% ==== %% ==== %% ====     SECTION TEN    ==== %% ==== %% ==== %% ==== %%
%% ==== %% ==== %% ==== %% ==== %% ==== %% ==== %% ==== %% ==== %% ==== %% ==== %%

\section{Final Notes}
\label{section:final-notes}
%!TEX root = combinatorial-exploration.tex

Despite the length of this work, we have only scratched the surface
of what Combinatorial Exploration can do. While the example domains
presented here have all involved sets described by the occurrence or
avoidance of certain kinds of patterns, that is just an artifact of our
own experience and interests, and is not a requirement
for Combinatorial Exploration to apply.

We have mentioned throughout
a number of extensions and generalizations that will lead to the
automatic discovery of combinatorial specifications for even more
combinatorial sets. Some of these have already been completed, or
are in progress, and will be discussed in forthcoming work. Others
would benefit from the attention of experts in fields other than ours.
We will briefly summarize these ideas here and mention a few additional ones
as well.

\begin{enumerate}
	\item In Section~\ref{section:asm-results}, we enumerated the
			$132$-avoiding alternating sign matrices by discovering
			a \emph{combinatorial forest}, which is a set of rules that
			does not constitute a combinatorial specification because there
			may be sets on the right-hand sides of rules that appear more than
			once or not at all on the left-hand sides of rules.Nonetheless,
			they carry sufficient structural information to enumerate all
			of the combinatorial sets involved. Although examples of
			these have appeared sporadically in the literature, to
			our knowledge they have never been formally defined nor
			systematically studied. We have devised an algorithmic
			approach that allows us to efficiently search a set of
			rules for a combinatorial forest. In our experience, they tend to
			lead to clever structural descriptions that would be
			challenging for a human to discover by hand.
			
			This is already implemented in the Combinatorial Exploration code
			base~\cite{comb-spec-searcher}, and we have used it to find
			a combinatorial specification for $\Av(1342)$. This class was
			first enumerated by B\'ona~\cite{bona:1342} by constructing a
			bijection between it and a set of non-permutation objects; we
			believe ours is the first direct enumeration. Future work will
			explore the concept more, lay theoretical foundations,
			and demonstrate how the use of forests makes Combinatorial
			Exploration even more effective.
	\item In several places, we have mentioned combinatorial specifications
			that use strategies that lead to extra indices in the counting
			formulas and ``catalytic'' variables in systems of equations,
			e.g., fusion for permutation patterns, and the strategy that
			merges two identical rows for polyominoes. In order to fully
			understand these results, the formalized versions of strategies
			and specifications developed in this work must be extended to
			multivariate versions, and the notions of reliance graphs and
			productivity must be suitably adapted.
	\item Section~\ref{section:transfer-tools} discussed that fast uniform
			random sampling can often be performed when a combinatorial
			specification has been discovered. Like the other transfer tools
			discussed, the strategies involved must have certain properties
			for this to be possible.
	\item Suppose that one has combinatorial specifications $S_1$ and $S_2$
			for two different combinatorial sets $\AA$ and $\BB$ that may
			be from different domains. If $S_1$ and $S_2$ are structurally
			similar in the sense that they employ strategies with the same
			counting functions in precisely the same ways, we call $S_1$
			and $S_2$ \emph{parallel specifications}. From any two parallel
			specifications we automatically obtain a size-preserving bijection
			between $\AA$ and $\BB$. This work can be found in Eliasson's
			thesis~\cite{eliasson:sliding-thesis}.
	\item Our heavy formalization of strategies and specifications may
			make it possible to employ a theorem verification
			system such as Lean~\cite{lean} or Agda~\cite{agda} to
			formally verify combinatorial specifications and their
			enumerations.
	\item While we have given short proofs-of-concept for three
			domains---alternating sign matrices, polyominoes, and set
			partitions---true success in applying Combinatorial
			Exploration in these domains is likely to require expertise
			from researchers who have more experience with these
			objects than we do. We also believe that Combinatorial
			Exploration would be effective in other domains such as
			inversion sequences, ascent sequences, and more.
\end{enumerate}

\subsection*{Acknowledgments}
We would like to thank the many undergraduate and masters students with
whom we had invaluable discussions and who contributed to this work,
including (with references to their theses when applicable):
Annija Apine,
Ragnar P\'all \'Ardal~\cite{ardal:proof-numbers-thesis},
Arnar Bjarni Arnarsson~\cite{arnarson:bsc-thesis, arnarson:msc-thesis},
Alfur Birkir Bjarnason~\cite{arnarson:bsc-thesis},
Jon Steinn Eliasson~\cite{eliasson:sliding-thesis},
Unnar Freyr Erlendsson~\cite{arnarson:bsc-thesis, erlendsson:msc-thesis},
Kolbeinn P\'all Erlingsson~\cite{gunnarsson:bsc-thesis},
Bjarki \'Ag\'ust Gu{\dh}mundsson,
Bj\"orn Gunnarsson~\cite{gunnarsson:bsc-thesis},
Sigur{\dh}ur Helgason~\cite{helgason:bsc-thesis},
Kristmundur \'Ag\'ust J\'onsson~\cite{gunnarsson:bsc-thesis},
T\'omas Ken Magn\'usson~\cite{magnusson:msc-thesis},
James Robb~\cite{helgason:bsc-thesis},
\'O{\dh}inn Hjaltason Schi\"oth,
Murray Tannock,
and
Sigurj\'on Freyr Viktorsson~\cite{arnarson:bsc-thesis}.

We'd also like to thank Vince Vatter for providing improved code for processing our sampled data into  heatmap images. We are further grateful to the referees for their suggestions, which have improved this work.

CB's research is supported by grant 207178-052 from the Icelandic Research Fund. 

\'EN's research is supported by the Reykjavik University Research Fund.

JP's research is supported by grant 713579 from the Simons Foundation.

%\bibliographystyle{acm}
%\bibliography{/Users/jay/Dropbox/Research/refs/bib.bib,paper.bib}
%\bibliography{paper.bib}

\printbibliography

% may need this for uploading to arxiv: https://github.com/plk/biblatex/wiki/biblatex-and-the-arXiv

\end{document}
% using bibtex (not biblatex)
% to move refs from your own bib file to the paper-specific one:
% bibexport -o paper.bib .texpadtmp/combinatorial-exploration.aux